\documentclass[11pt]{amsart}
\usepackage{amssymb}
\usepackage{amsmath}
\usepackage{fancyhdr}
\usepackage[british]{babel}
\usepackage{geometry}
\usepackage{enumitem}
\usepackage{algpseudocode}
\usepackage{dsfont}
\usepackage{xstring}
\usepackage{colortbl}
\usepackage[section]{algorithm}
\usepackage{graphicx}
\usepackage[font={footnotesize}]{caption}
\usepackage[usenames,dvipsnames,table]{xcolor}
\usepackage{pstricks,tikz}
\usepackage[h]{esvect}
\usepackage[
		bookmarksopen=true,
		bookmarksopenlevel=1,
		colorlinks=true,
		linkcolor=darkblue,
        linktoc=page,
		citecolor=darkblue,
]{hyperref}

\title[]{Hypergraph \texorpdfstring{$F$}{F}-designs for arbitrary $F$}
\date{\today}
\author[S.~Glock, D.~K\"uhn, A.~Lo and D.~Osthus]{Stefan Glock, Daniela K\"uhn, Allan Lo and Deryk Osthus}

\thanks{This preprint has been merged with `The existence of designs via iterative absorption' (arXiv:1611.06827v1) into a single paper `The existence of designs via iterative absorption: hypergraph $F$-designs for arbitrary $F$' (arXiv:1611.06827v3).
\\The research leading to these results was partially supported by the EPSRC, grant nos. EP/N019504/1 (D.~K\"uhn) and EP/P002420/1 (A.~Lo),
by the Royal Society and the Wolfson Foundation (D.~K\"uhn) as well as by the European Research Council
under the European Union's Seventh Framework Programme (FP/2007--2013) / ERC Grant
Agreement no. 306349 (S.~Glock and D.~Osthus).}

\newtheorem{theorem}[algorithm]{Theorem}
\newtheorem{prop}[algorithm]{Proposition}
\newtheorem{lemma}[algorithm]{Lemma}
\newtheorem{cor}[algorithm]{Corollary}
\newtheorem{fact}[algorithm]{Fact}

\theoremstyle{definition}

\newtheorem{defin}[algorithm]{Definition}

\newtheorem{example}[algorithm]{Example}

\newtheoremstyle{claimstyle}{5pt}{5pt}{\em}{5pt}{\em}{:}{5pt}{}
\theoremstyle{claimstyle}
\newtheorem{claim}{Claim}

\numberwithin{equation}{section}

\definecolor{darkblue}{rgb}{0,0,0.5}

\def\noproof{{\unskip\nobreak\hfill\penalty50\hskip2em\hbox{}\nobreak\hfill%
       $\square$\parfillskip=0pt\finalhyphendemerits=0\par}\goodbreak}
\def\endproof{\noproof\bigskip}

\def\noclaimproof{{\unskip\nobreak\hfill\penalty50\hskip2em\hbox{}\nobreak\hfill%
       $-$\parfillskip=0pt\finalhyphendemerits=0\par}\goodbreak}

\def\endclaimproof{\noclaimproof\medskip}

\newdimen\margin
\def\textno#1&#2\par{
   \margin=\hsize
   \advance\margin by -4\parindent
          \setbox1=\hbox{\sl#1}
   \ifdim\wd1 < \margin
      $$\box1\eqno#2$$
   \else
      \bigbreak
      \hbox to \hsize{\indent$\vcenter{\advance\hsize by -3\parindent
      \it\noindent#1}\hfil#2$}
      \bigbreak
   \fi}

\def\proof{\removelastskip\penalty55\medskip\noindent\setcounter{claim}{0}{\bf Proof. }} 

\def\lateproof#1{\removelastskip\penalty55\medskip\noindent\setcounter{claim}{0}{\bf Proof of #1. }} 

\DeclareMathOperator{\Ima}{Im}
\def\claimproof{\removelastskip\penalty55\medskip\noindent{\em Proof of claim: }}

\begin{document}

\def\COMMENT#1{}
\def\TASK#1{}

\def\eps{{\varepsilon}}
\newcommand{\ex}{\mathbb{E}}
\newcommand{\pr}{\mathbb{P}}
\newcommand{\cB}{\mathcal{B}}
\newcommand{\cE}{\mathcal{E}}
\newcommand{\cS}{\mathcal{S}}
\newcommand{\cF}{\mathcal{F}}
\newcommand{\bF}{\mathbb{F}}
\newcommand{\bZ}{\mathbb{Z}}
\newcommand{\cH}{\mathcal{H}}
\newcommand{\cC}{\mathcal{C}}
\newcommand{\cM}{\mathcal{M}}
\newcommand{\bN}{\mathbb{N}}
\newcommand{\bR}{\mathbb{R}}
\def\O{\mathcal{O}}
\newcommand{\cP}{\mathcal{P}}
\newcommand{\cQ}{\mathcal{Q}}
\newcommand{\cR}{\mathcal{R}}
\newcommand{\cJ}{\mathcal{J}}
\newcommand{\cL}{\mathcal{L}}
\newcommand{\cK}{\mathcal{K}}
\newcommand{\cD}{\mathcal{D}}
\newcommand{\cI}{\mathcal{I}}
\newcommand{\cV}{\mathcal{V}}
\newcommand{\cT}{\mathcal{T}}
\newcommand{\cU}{\mathcal{U}}
\newcommand{\cZ}{\mathcal{Z}}
\newcommand{\1}{{\bf 1}_{n\not\equiv \delta}}
\newcommand{\eul}{{\rm e}}
\newcommand{\Erd}{Erd\H{o}s}
\newcommand{\cupdot}{\mathbin{\mathaccent\cdot\cup}}
\newcommand{\whp}{whp }

\newcommand{\doublesquig}{%
  \mathrel{%
    \vcenter{\offinterlineskip
      \ialign{##\cr$\rightsquigarrow$\cr\noalign{\kern-1.5pt}$\rightsquigarrow$\cr}%
    }%
  }%
}

\newcommand{\defn}{\emph}

\newcommand\restrict[1]{\raisebox{-.5ex}{$|$}_{#1}}

\newcommand{\prob}[1]{\mathrm{\mathbb{P}}(#1)}
\newcommand{\expn}[1]{\mathrm{\mathbb{E}}#1}
\def\gnp{G_{n,p}}
\def\G{\mathcal{G}}
\def\lflr{\left\lfloor}
\def\rflr{\right\rfloor}
\def\lcl{\left\lceil}
\def\rcl{\right\rceil}

\newcommand{\brackets}[1]{\left(#1\right)}
\def\sm{\setminus}
\newcommand{\Set}[1]{\{#1\}}
\newcommand{\set}[2]{\{#1\,:\;#2\}}
\newcommand{\krq}[2]{K^{(#1)}_{#2}}
\newcommand{\ind}[1]{$(\ast)_{#1}$}
\newcommand{\indcov}[1]{$(\#)_{#1}$}
\def\In{\subseteq}

\begin{abstract}  \noindent
We solve the existence problem for $F$-designs for arbitrary $r$-uniform
hypergraphs~$F$. In particular, this shows that, given any $r$-uniform hypergraph~$F$, the trivially necessary divisibility
conditions are sufficient to guarantee a decomposition of any sufficiently
large complete $r$-uniform hypergraph $G=K_n^{(r)}$ into edge-disjoint copies
of~$F$, which answers a question asked e.g.~by Keevash. The graph case $r=2$ forms one of the cornerstones of design theory and was proved by Wilson in 1975.
The case when~$F$ is complete corresponds to the existence of
block designs, a problem going back to the 19th century, which was first
settled by Keevash.

More generally, our results extend to~$F$-designs of quasi-random
hypergraphs~$G$ and of~hypergraphs $G$ of suitably large
minimum degree. Our approach builds on results and methods we recently introduced
in our new proof of the existence conjecture for block designs.
\end{abstract}

\maketitle

\section{Introduction}\label{sec:intro}

\subsection{Background}
A \defn{hypergraph} $G$ is a pair $(V,E)$, where $V=V(G)$ is the vertex set of $G$ and the edge set $E$ is a set of subsets of $V$. We often identify $G$ with $E$, in particular, we let $|G|:=|E|$. We say that $G$ is an \defn{$r$-graph} if every edge has size $r$. We let $\krq{r}{n}$ denote the complete $r$-graph on $n$ vertices.

Let $G$ and $F$ be $r$-graphs.  An \defn{$F$-decomposition of $G$} is a collection $\cF$ of copies of $F$ in $G$ such that every edge of $G$ is contained in exactly one of these copies. (Throughout the paper, we always assume that $F$ is non-empty
without mentioning this explicitly.)
More generally, an \defn{$(F,\lambda)$-design of $G$} is a collection $\cF$ of distinct copies of $F$ in $G$ such that every edge of $G$ is contained in exactly $\lambda$ of these copies.
Such a design can only exist if $G$ satisfies certain divisibility conditions
(e.g.~if $F$ is a graph triangle and $\lambda=1$, then $G$ must have even vertex degrees and the
number of edges must be a multiple of three). If $F$ is complete, such designs are also referred to as block designs.

The question of the existence of such designs goes back to the 19th century.
The first general result was due to Kirkman~\cite{Ki}, who proved the existence of
Steiner triple systems (i.e. triangle decompositions of complete graphs)
under the appropriate divisibility conditions.
In a ground-breaking series of papers which transformed the area,
Wilson~\cite{W1,W2,W3,WilsonBCC} solved the existence problem in the graph setting (i.e.~when $r=2$) by showing that
the trivially necessary divisibility conditions imply the existence of $(F,\lambda)$-designs in $\krq{2}{n}$ for sufficiently large $n$.
More generally, the existence conjecture postulated that the necessary divisibility conditions
are also sufficient to ensure the existence of block designs with given parameters in $\krq{r}{n}$.

Answering a question of Erd\H{o}s and Hanani~\cite{EH}, R\"odl~\cite{Ro} was able to give an approximate solution to the existence conjecture by constructing near optimal packings of
edge-disjoint copies of $\krq{r}{f}$ in $\krq{r}{n}$, i.e.~packings which cover almost all the edges of~$\krq{r}{n}$.
(For this, he introduced his now famous R\"odl nibble method, which has since
had a major impact in many areas.)
More recently, Kuperberg, Lovett and Peled~\cite{KLP} were able to prove probabilistically
 the existence of non-trivial designs for a large range of parameters
 (but their result requires that $\lambda$ is comparatively large).
 Apart from this, progress for $r \ge 3$ was mainly limited to explicit constructions for
 rather restrictive parameters (see e.g.~\cite{CD,T}).

In a recent breakthrough, Keevash~\cite{Ke} proved the existence of
$(\krq{r}{f},\lambda)$-designs in $\krq{r}{n}$ for arbitrary (but fixed)
$r,f$ and $\lambda$, provided $n$ is sufficiently large.
In particular, his result implies the existence of Steiner systems
for any admissible range of parameters as long as
$n$ is sufficiently large compared to $f$ (Steiner systems are block designs with $\lambda=1$).
The approach in~\cite{Ke} involved randomised algebraic constructions
and yielded a far-reaching generalisation to block designs in quasirandom $r$-graphs.
This in turn was extended in~\cite{GKLO}, where we developed a non-algebraic approach
based on iterative absorption, which additionally yielded resilience versions and the existence of block designs in
hypergraphs of large minimum degree.
This naturally raises the question of whether $F$-designs also exist for arbitrary $r$-graphs~$F$.
Here, we answer this affirmatively, by building on methods and results from~\cite{GKLO}.

\subsection{\texorpdfstring{$F$}{F}-designs in quasirandom hypergraphs}

We now describe the degree conditions which are trivially necessary for the
existence of an $F$-design in an $r$-graph $G$.
For a set $S\In V(G)$ with $0\le |S|\le r$, the $(r-|S|)$-graph $G(S)$ has vertex set $V(G)\sm S$ and contains all $(r-|S|)$-subsets of $V(G)\sm S$ that together with $S$ form an edge in $G$. ($G(S)$ is often called the \defn{link graph of $S$}.) Let $\delta(G)$ and $\Delta(G)$ denote the minimum and maximum $(r-1)$-degree of an $r$-graph $G$, respectively, that is, the minimum/maximum value of $|G(S)|$ over all $S\In V(G)$ of size $r-1$.
For a (non-empty) $r$-graph $F$, we define the \defn{divisibility vector of $F$} as $Deg(F):=(d_0,\dots,d_{r-1})\in \bN^r$, where $d_{i}:=\gcd\set{|F(S)|}{S\in \binom{V(F)}{i}}$, and we set $Deg(F)_i:=d_i$ for $0\le i\le r-1$.\COMMENT{As long as $F$ is not edgeless, this is well-defined.} Note that $d_0=|F|$.
So if $F$ is the Fano plane, we have $Deg(F)=(7,3,1)$.

Given $r$-graphs $F$ and $G$, $G$ is called \defn{$(F,\lambda)$-divisible} if $Deg(F)_i\mid \lambda |G(S)|$ for all $0\le i\le r-1$ and all $S\in \binom{V(G)}{i}$.
Note that $G$ must be $(F,\lambda)$-divisible in order to admit an $(F,\lambda)$-design.
For simplicity, we say that $G$ is \defn{$F$-divisible} if $G$ is $(F,1)$-divisible.
Thus $F$-divisibility of $G$ is necessary for the existence of an $F$-decomposition
of $G$.

As a special case, the following result implies that $(F,\lambda)$-divisibility is
sufficient to guarantee the existence of an $(F,\lambda)$-design
when $G$ is complete and $\lambda$ is not too large.
This answers a question asked e.g.~by Keevash~\cite{Ke}.

In fact, rather than requiring $G$ to be complete, it suffices that $G$
is quasirandom in the following sense.
An $r$-graph $G$ on $n$ vertices is called \defn{$(c,h,p)$-typical} if for any set $A$ of $(r-1)$-subsets of $V(G)$ with $|A|\le h$ we have $|\bigcap_{S\in A}G(S)|=(1\pm c)p^{|A|}n$.
Note that this is what one would expect in a random $r$-graph
with edge probability~$p$.

\begin{theorem}[$F$-designs in typical hypergraphs]\label{thm:design}
For all $f,r\in \bN$ with $f>r$ and all $c,p\in (0,1]$ with
\begin{align*}
c \le 0.9(p/2)^{h}/(q^r4^q), \mbox{ where }q:=2f\cdot f! \mbox{ and } h:=2^r\binom{q+r}{r},
\end{align*}
there exist $n_0\in \bN$ and $\gamma>0$ such that the following holds for all $n\ge n_0$.
Let $F$ be any $r$-graph on $f$ vertices and let $\lambda\in \bN$ with $\lambda\le \gamma n$. Suppose that $G$ is a $(c,h,p)$-typical $r$-graph on $n$ vertices. Then $G$ has an $(F,\lambda)$-design if it is $(F,\lambda)$-divisible.
\end{theorem}
The main result in~\cite{Ke} is also stated in the setting of typical $r$-graphs,
but additionally requires that $c \ll 1/h \ll p,1/f$ and that $\lambda=\O(1)$ and $F$ is complete.
The case when $F$ is complete and $\lambda$ is bounded
is also a special case of our recent result on designs in supercomplexes
(see Theorem~1.4 in~\cite{GKLO}).
Previous results in the case when $r \ge 3$ and $F$ is not complete are very sporadic
-- for instance Hanani~\cite{hanani} settled the problem if $F$ is an octahedron
(viewed as a $3$-uniform hypergraph)
and $G$ is complete.

As a very special case, Theorem~\ref{thm:design} resolves a conjecture of
Archdeacon on self-dual embeddings of random graphs in orientable surfaces:
as proved in~\cite{Archdeacon}, a graph has such an embedding if it has a decomposition into $K:=\krq{2}{4}$ and $K':=\krq{2}{5}$.
Suppose $G$ is a $(c,h,p)$-typical $2$-graph on $n$ vertices with an even number of edges
and $1/n \ll c \ll 1/h\ll p$ (which almost surely holds for the binomial random graph
$G_{n,p}$ if we remove at most one edge).
Now remove a suitable number of copies of
$K$ from $G$ to ensure that the leftover $G'$ satisfies $16 \mid |G'|$. Let $F$ be the vertex-disjoint union of $K$ and $K'$. Since $Deg(F)_1=1$, $G'$ is $F$-divisible. Thus we can apply Theorem~\ref{thm:design} to obtain an $F$-decomposition of $G'$.\COMMENT{http://www.cems.uvm.edu/TopologicalGraphTheoryProblems/partcomp.htm
\newline
It doesn't seem to easy to do this using the clique paper results?} If the number of edges is odd, a similar argument yields self-dual embeddings in non-orientable surfaces.\COMMENT{Take out one copy of $K_6$}

In Section~\ref{sec:main proofs}, we will deduce Theorem~\ref{thm:design} from a more general result on
$F$-decompositions in supercomplexes $G$ (Theorem~\ref{thm:main complex}).
(The condition of $G$ being a supercomplex is considerably less restrictive than typicality.)
Moreover, the $F$-designs we obtain will have the additional property that $|V(F')\cap V(F'')|\le r$ for all distinct $F',F''$ which are included in the design.
It is easy to see that with this additional property the bound on $\lambda$
in Theorem~\ref{thm:design} is best possible up to the value of~$\gamma$.\COMMENT{the number of $F$-copies is
$O(\lambda \binom{n}{r}) =O( \binom{n}{r+1})$}

We can also deduce the following result which yields `near-optimal' $F$-packings in typical $r$-graphs which are not divisible.\COMMENT{To get the better bound here, don't use Theorem~\ref{thm:design} but Theorem~\ref{thm:typical separated dec}} (An $F$-packing in $G$ is a collection of edge-disjoint copies of $F$ in~$G$.)

\begin{theorem}\label{thm:near optimal}
For all $f,r\in \bN$ with $f>r$ and all $c,p\in (0,1]$ with
\begin{align*}
c \le 0.9p^{h}/(q^r4^q), \mbox{ where }q:=2f\cdot f! \mbox{ and } h:=2^r\binom{q+r}{r}, 
\end{align*}
there exist $n_0,C\in \bN$ such that the following holds for all $n\ge n_0$.
Let $F$ be any $r$-graph on $f$ vertices. Suppose that $G$ is a $(c,h,p)$-typical $r$-graph on $n$ vertices. Then $G$ has an $F$-packing $\cF$ such that the leftover $L$ consisting of all uncovered edges satisfies $\Delta(L)\le C$.\COMMENT{Could also write $n_0$ instead of $C$, but seems a bit odd to use $n_0$ as bound for maxdeg}
\end{theorem}

\subsection{\texorpdfstring{$F$}{F}-designs in hypergraphs of large minimum degree}
Once the existence question is settled, a next natural step is to seek~$F$-designs
and $F$-decompositions in $r$-graphs of large minimum degree.
Our next result gives a bound on the minimum degree which ensures
an $F$-decomposition  for `weakly regular' $r$-graphs~$F$. These are defined as follows.
\begin{defin}[weakly regular] \label{def:weakly regular}
Let $F$ be an $r$-graph. We say that $F$ is \defn{weakly $(s_0,\dots,s_{r-1})$-regular} if for all $0\le i\le r-1$ and all $S\in\binom{V(F)}{i}$, we have $|F(S)|\in \Set{0,s_i}$.\COMMENT{Of course $s_0$ is a bit redundant because there is only one $S$ of size $0$. But as $r$ could be one it is nicer to have $0$th entry.} We simply say that $F$ is \defn{weakly regular} if it is weakly $(s_0,\dots,s_{r-1})$-regular for suitable $s_i$'s.
\end{defin}
So for example, cliques, the Fano plane and the octahedron are all weakly regular
but a $3$-uniform tight or loose cycle is not.
\begin{theorem}[$F$-decompositions in hypergraphs of large minimum degree]\label{thm:min deg}
Let $F$ be a weakly regular $r$-graph on $f$ vertices. Let $$c^\diamond_{F}:=\frac{r!}{3\cdot 14^r f^{2r}}.$$ There exists an $n_0\in\mathbb{N}$ such that the following holds for all $n \ge n_0$.
Suppose that $G$ is an $r$-graph on $n$ vertices with $\delta(G)\ge (1-c^\diamond_{F})n$.
Then $G$ has an $F$-decomposition if it is $F$-divisible.
\end{theorem}
Note that Theorem~\ref{thm:min deg} implies that every packing
of edge-disjoint copies of $F$ into $\krq{r}{n}$ with overall maximum degree at most
$c^\diamond_{F}n$ can be extended into an $F$-decomposition of $\krq{r}{n}$
(provided $\krq{r}{n}$ is $F$-divisible).

An analogous (but significantly worse) constant $c^\diamond_{F}$ for $r$-graphs $F$ which are not weakly regular immediately follows from the case $p=1$ of
Theorem~\ref{thm:design}.
These results lead to the concept of the `decomposition threshold' $\delta_F$ of a given $r$-graph $F$.

\begin{defin}[Decomposition threshold] \label{decthreshold}
\emph{Given an $r$-graph $F$, let $\delta_F$ be the infimum of all $\delta\in[0,1]$ with the following property: There exists $n_0\in \bN$ such that for all $n\ge n_0$,
every $F$-divisible $r$-graph $G$ on $n$ vertices with $\delta(G)\ge \delta n$ has an $F$-decomposition.}
\end{defin}
By Theorem~\ref{thm:min deg}, we have $\delta_F \le 1-c_F^\diamond$ whenever $F$ is weakly regular.
As noted in~\cite{GKLO}, for all $r,f,n_0\in \bN$, there exists an $r$-graph $G_n$ on $n\ge n_0$ vertices with $\delta(G_n)\ge (1-b_r \frac{\log{f}}{f^{r-1}})n$ such that $G_n$ does not contain a single copy of $\krq{r}{f}$, where $b_r>0$ only depends on $r$. (This can be seen by adapting a construction from~\cite{KMV} which is based on a result from~\cite{RS}.)

Previously, the only positive result for the hypergraph case $r \ge 3$
was due to Yuster~\cite{Y2}, who showed that
if $T$ is a linear $r$-uniform hypertree, then every $T$-divisible $r$-graph $G$ on $n$ vertices with minimum vertex degree at least
$(\frac{1}{2^{r-1}}+o(1))\binom{n}{r-1}$ has a $T$-decomposition.
This is asymptotically best possible for nontrivial $T$. Moreover, the result implies that $\delta_T\le 1/2^{r-1}$.

For the graph case $r=2$, much more is known about the decomposition threshold:
the results in~\cite{BKLO,GKLMO} establish a close connection between $\delta_F$
and the fractional decomposition
threshold $\delta_F^\ast$  (which is defined as in Definition~\ref{decthreshold},
but with an
$F$-decomposition replaced by a fractional $F$-decomposition).
In particular, the results in~\cite{BKLO,GKLMO}
imply that  $\delta_F\le \max\Set{\delta_F^\ast,1-1/(\chi(F)+1)}$ and that
$\delta_F=\delta_F^\ast$ if $F$ is a complete graph.

Together with recent results on the fractional decomposition threshold for cliques in~\cite{BKLMO,D},
this gives the best current bounds on $\delta_F$ for general~$F$.
It would be very interesting to establish a similar connection in the hypergraph case.

Also, for bipartite graphs the decomposition threshold was completely determined
in~\cite{GKLMO}. It would be interesting to see if this can be generalised
to $r$-partite $r$-graphs.
On the other hand, even the decomposition threshold of a graph triangle is still unknown
(a beautiful conjecture of Nash-Williams~\cite{NW} would imply
that the value is $3/4$).

\subsection{Counting}
An approximate $F$-decomposition of $\krq{r}{n}$ is a set of edge-disjoint copies
of $F$ in $\krq{r}{n}$ which together cover almost all edges of $\krq{r}{n}$.
Given good bounds on the number of approximate $F$-decompositions of
$\krq{r}{n}$ whose set of leftover edges forms a typical $r$-graph, one can apply Theorem~\ref{thm:design} to obtain
corresponding bounds on the number of $F$-decompositions in $\krq{r}{n}$ (see~\cite{Ke,Ke2} for the clique case).
Such bounds on the number of approximate $F$-decompositions can be achieved by considering either a random greedy $F$-removal
process or an associated $F$-nibble removal process.

\subsection{Outline of the paper}\label{subsec:outline}
As mentioned earlier, our main result (Theorem~\ref{thm:main complex})
actually concerns $F$-decompositions in so-called supercomplexes. We will define supercomplexes
in Section~\ref{sec:super and main} and derive Theorems~\ref{thm:design},~\ref{thm:near optimal}
and~\ref{thm:min deg} from Theorem~\ref{thm:main complex} in
Section~\ref{sec:main proofs}.
The definition of a supercomplex $G$ involves mainly the distribution of cliques
of size $f$ in $G$ (where $f=|V(F)|$).
The notion is weaker than usual notions of quasirandomness.
This has two main advantages:
firstly, our proof is by induction on $r$, and working with this weaker notion is essential to make the induction proof work.
Secondly, this allows us to deduce Theorems~\ref{thm:design},~\ref{thm:near optimal}
and~\ref{thm:min deg} from a single statement.

However, Theorem~\ref{thm:main complex} applies only to $F$-decompositions of a supercomplex $G$ for
weakly regular $r$-graphs $F$ (which allows us to deduce Theorem~\ref{thm:min deg} but not Theorem~\ref{thm:design}).
To deal with this, in Section~\ref{sec:main proofs} we first provide an explicit construction which shows that every $r$-graph $F$ can be `perfectly' packed into
a suitable weakly regular $r$-graph $F^*$.
In particular, $F^*$ has an $F$-decomposition.
The idea is then to apply Theorem~\ref{thm:main complex} to find an $F^*$-decomposition in $G$. Unfortunately, $G$ may not be $F^*$-divisible.
To overcome this, in Section~\ref{sec:make divisible} we show that we can remove a small set of
copies of $F$ from $G$ to achieve that the leftover $G'$ of $G$ is now
$F^*$-divisible (see Lemma~\ref{lem:make divisible} for the statement).
This now implies Theorem~\ref{thm:design} for $F$-decompositions,
i.e.~for $\lambda=1$.
However, by repeatedly applying Theorem~\ref{thm:main complex} in a suitable way, we can actually allow $\lambda$ to be as large as required in Theorem~\ref{thm:design}.

It thus remains to prove Theorem~\ref{thm:main complex} itself.
We achieve this via the iterative absorption method.
The idea is to iteratively extend a packing of edge-disjoint copies of $F$
until the set $H$ of uncovered edges is very small.
This final set can then be `absorbed' into an $r$-graph $A$ we set aside at the beginning of the proof
(in the sense that $A \cup H$ has an $F$-decomposition).
This iterative approach to decompositions was first introduced in~\cite{KKO,KO} in the context of
Hamilton decompositions of graphs.
(Absorption itself was pioneered earlier for spanning structures e.g.~in~\cite{Kriv,RRS},
but as remarked e.g.~in~\cite{Ke}, such direct approaches are not feasible in the
decomposition setting.)

This approach relies on being able to find a suitable approximate $F$-decomposition in each iteration, whose existence we derive in Section~\ref{sec:bounded}.
The iteration process is underpinned by a so-called `vortex', which consists of
an appropriate nested sequence of vertex subsets of $G$
(after each iteration, the current set of uncovered edges is constrained to the
next vertex subset in the sequence). These vortices are discussed in Section~\ref{sec:vortices}.
The final absorption step  is described in Section~\ref{sec:absorbers}.

As mentioned earlier, the current proof builds on the framework introduced
in~\cite{GKLO}.
In fact, several parts of the argument in~\cite{GKLO} can either be used directly
or can be straightforwardly
adapted to the current setting.\COMMENT{in journal version: ...so we do not repeat them here.}
In particular, this applies to the Cover down lemma (Lemma~\ref{lem:cover down}),
which is the key result that allows the iteration to work.
Thus in the current paper we concentrate on the parts which involve significant new ideas (e.g.~the absorption process).
For details of the parts which can be straightforwardly adapted, we refer
to the appendix.\COMMENT{in journal version: ...of the arxiv version of the current paper.}
Altogether, this illustrates the versatility of our framework and we thus believe that
it can be developed in further settings.

As a byproduct of the construction of the weakly regular $r$-graph $F^*$ outlined above,
we prove the existence of resolvable clique decompositions in complete partite $r$-graphs $G$ (see Theorem~\ref{thm:partite designs}). The construction is explicit and exploits the property that all square submatrices of so-called Cauchy matrices over finite fields are invertible.
We believe this construction to be of independent interest.
A natural question leading on from the current work would be to obtain such resolvable decompositions
also in the general (non-partite) case. For decompositions of $\krq{2}{n}$ into $\krq{2}{f}$,
this is due to Ray-Chaudhuri and Wilson~\cite{RCW}. For recent progress see~\cite{DukesLing,LRV}.

\section{Notation}\label{sec:notation}

\subsection{Basic terminology}

We let $[n]$ denote the set $\Set{1,\dots,n}$, where $[0]:=\emptyset$. Moreover, let $[n]_0:=[n]\cup\Set{0}$ and $\bN_0:=\bN\cup \Set{0}$. As usual, $\binom{n}{i}$ denotes the binomial coefficient, where we set $\binom{n}{i}:=0$ if $i>n$ or $i<0$. Moreover, given a set $X$ and $i\in\bN_0$, we write $\binom{X}{i}$ for the collection of all $i$-subsets of $X$. Hence, $\binom{X}{i}=\emptyset$ if $i>|X|$. If $F$ is a collection of sets, we define $\bigcup F:=\bigcup_{f\in F}f$. We write $A \cupdot B$ for the union of $A$ and $B$ if we want to emphasise that $A$ and $B$ are disjoint.

We write $X\sim B(n,p)$ if $X$ has binomial distribution with parameters $n,p$, and we write $bin(n,p,i):=\binom{n}{i}p^i(1-p)^{n-i}$. So by the above convention, $bin(n,p,i)=0$ if $i>n$ or $i<0$.

We say that an event holds \defn{with high probability (whp)} if the probability that it holds tends to $1$ as $n\to\infty$ (where $n$ usually denotes the number of vertices).

We write $x\ll y$ to mean that for any $y\in (0,1]$ there exists an $x_0\in (0,1)$ such that for all $x\le x_0$ the subsequent statement holds. Hierarchies with more constants are defined in a similar way and are to be read from the right to the left. We will always assume that the constants in our hierarchies are reals in $(0,1]$. Moreover, if $1/x$ appears in a hierarchy, this implicitly means that $x$ is a natural number. More precisely, $1/x\ll y$ means that for any $y\in (0,1]$ there exists an $x_0\in \bN$ such that for all $x\in \bN$ with $x\ge x_0$ the subsequent statement holds.

We write $a=b\pm c$ if $b-c\le a\le b+c$. Equations containing $\pm$ are always to be interpreted from left to right, e.g. $b_1\pm c_1=b_2\pm c_2$ means that $b_1-c_1\ge b_2-c_2$ and $b_1+c_1\le b_2+c_2$.

When dealing with multisets, we treat multiple appearances of the same element as distinct elements. In particular, two subsets $A,B$ of a multiset can be disjoint even if they both contain a copy of the same element, and if $A$ and $B$ are disjoint, then the multiplicity of an element in the union $A\cup B$ is obtained by adding the multiplicities of this element in $A$ and $B$ (rather than just taking the maximum).

\subsection{Hypergraphs and complexes}
Let $G$ be an $r$-graph. Note that $G(\emptyset)=G$. For a set $S\In V(G)$ with $|S|\le r$ and $L\In G(S)$, let $S\uplus L:=\set{S\cup e}{e\in L}$. Clearly, there is a natural bijection between $L$ and $S\uplus L$.

For $i\in[r-1]_0$, we define $\delta_{i}(G)$ and $\Delta_i(G)$ as the minimum and maximum value of $|G(S)|$ over all $i$-subsets $S$ of $V(G)$, respectively. As before, we let $\delta(G):=\delta_{r-1}(G)$ and $\Delta(G):=\Delta_{r-1}(G)$. Note that $\delta_0(G)=\Delta_0(G)=|G(\emptyset)|=|G|$.\COMMENT{For a $0$-graph, $\Delta$ is not defined. So should be careful not to use it in this case.}

For two $r$-graphs $G$ and $G'$, we let $G-G'$ denote the $r$-graph obtained from $G$ by deleting all edges of $G'$.
We write $G_1+G_2$ to mean the vertex-disjoint union of $G_1$ and $G_2$, and $t\cdot G$ to mean the vertex-disjoint union of $t$ copies of $G$.

Let $F$ and $G$ be $r$-graphs. An \defn{$F$-packing in $G$} is a set $\cF$ of edge-disjoint copies of $F$ in $G$. We let $\cF^{(r)}$ denote the $r$-graph consisting of all covered edges of $G$, i.e.~$\cF^{(r)}=\bigcup_{F'\in \cF}F'$.

A \defn{multi-$r$-graph} $G$ consists of a set of vertices $V(G)$ and a multiset of edges $E(G)$, where each $e\in E(G)$ is a subset of $V(G)$ of size $r$. We will often identify a multi-$r$-graph with its edge set.
For $S\In V(G)$, let $|G(S)|$ denote the number of edges of $G$ that contain $S$ (counted with multiplicities). If $|S|=r$, then $|G(S)|$ is called the \defn{multiplicity of $S$ in $G$}. We say that $G$ is \defn{$F$-divisible} if $Deg(F)_{|S|}$ divides $|G(S)|$ for all $S\In V(G)$ with $|S|\le r-1$. An $F$-decomposition of $G$ is a collection $\cF$ of copies of $F$ in $G$ such that every edge $e\in G$ is covered precisely once. (Thus if $S\In V(G)$ has size $r$, then there are precisely $|G(S)|$ copies of $F$ in $\cF$ in which $S$ forms an edge.)

\begin{defin}
A \defn{complex} $G$ is a hypergraph which is closed under inclusion, that is, whenever $e' \In e \in G$ we have $e' \in G$.
If $G$ is a complex and $i\in\bN_0$, we write $G^{(i)}$ for the $i$-graph on $V(G)$ consisting of all $e \in G$ with $|e|=i$. We say that a complex is empty if $\emptyset\notin G^{(0)}$, that is, if $G$ does not contain any edges.
\end{defin}

Suppose $G$ is a complex and $e \In V(G)$. Define $G(e)$ as the complex on vertex set $V(G)\sm e$ containing all sets $e'\In V(G)\sm e$ such that $e \cup e' \in G$. Clearly, if $e\notin G$, then $G(e)$ is empty. Observe that if $|e|=i$ and $r\ge i$, then $G^{(r)}(e)=G(e)^{(r-i)}$. We say that $G'$ is a \defn{subcomplex} of $G$ if $G'$ is a complex and a subhypergraph of $G$.

For a set $U$, define $G[U]$ as the complex on $U\cap V(G)$ containing all $e\in G$ with $e\In U$. Moreover, for an $r$-graph $H$, let $G[H]$ be the complex on $V(G)$ with edge set $$G[H]:=\set{e\in G}{\binom{e}{r}\In H},$$ and define $G-H:=G[G^{(r)}-H]$. So for $i\in[r-1]$, $G[H]^{(i)}=G^{(i)}$. For $i>r$, we might have $G[H]^{(i)} \subsetneqq G^{(i)}$. Moreover, if $H\In G^{(r)}$, then $G[H]^{(r)}=H$. Note that for an $r_1$-graph $H_1$ and an $r_2$-graph $H_2$, we have $(G[H_1])[H_2]=(G[H_2])[H_1]$.
Also, $(G-H_1)-H_2=(G-H_2)-H_1$, so we may write this as $G-H_1-H_2$.

If $G_1$ and $G_2$ are complexes, we define $G_1\cap G_2$ as the complex on vertex set $V(G_1)\cap V(G_2)$ containing all sets $e$ with $e\in G_1$ and $e\in G_2$. We say that $G_1$ and $G_2$ are \defn{$i$-disjoint} if $G_1^{(i)}\cap G_2^{(i)}$ is empty.

For any hypergraph $H$, let $H^{\le}$ be the complex on $V(H)$ \defn{generated by $H$}, that is, $$H^{\le}:=\set{e\In V(H)}{\exists e'\in H\mbox{ such that } e\In e'}.$$

For an $r$-graph $H$, we let $H^{\leftrightarrow}$ denote the complex on $V(H)$ that is \defn{induced by $H$}, that is, $$H^{\leftrightarrow}:=\set{e\In V(H)}{\binom{e}{r}\In H}.$$ Note that $H^{\leftrightarrow(r)}=H$ and for each $i\in[r-1]_0$, $H^{\leftrightarrow(i)}$ is the complete $i$-graph on $V(H)$. We let $K_n$ denote the the complete complex on $n$ vertices.

\section{Decompositions of supercomplexes} \label{sec:super and main}

\subsection{Supercomplexes} \label{subsec:super}

We prove our main decomposition theorem for so-called `supercomplexes', which were introduced in \cite{GKLO}. The crucial property appearing in the definition is
that of `regularity', which means that every $r$-set of a given complex $G$ is contained in roughly the same number of $f$-sets (where  $f=|V(F)|$). 
If we view $G$ as a complex which is induced by some $r$-graph, this means that every edge lies in roughly the same number of cliques of size $f$. It turns out that this set of conditions is
appropriate  even when $F$ is not a clique. 

A key advantage of the notion of a supercomplex is that the conditions are 
very flexible, which will enable us to `boost' their parameters (see 
Lemma~\ref{lem:boost complex}
below).
The following definitions are the same as in \cite{GKLO}.

\begin{defin}\label{def:complex}
Let $G$ be a complex on $n$ vertices, $f\in \bN$ and $r\in[f-1]_0$, $0\le \eps,d,\xi\le 1$. We say that $G$ is
\begin{enumerate}[label={\rm(\roman*)}]
\item \defn{$(\eps,d,f,r)$-regular}, if for all $e\in G^{(r)}$ we have $$|G^{(f)}(e)|=(d\pm \eps)n^{f-r};$$\label{def:complex:regular}
\item \defn{$(\xi,f,r)$-dense}, if for all $e\in G^{(r)}$, we have $$|G^{(f)}(e)|\ge \xi n^{f-r};$$
\item \defn{$(\xi,f,r)$-extendable}, if $G^{(r)}$ is empty or there exists a subset $X\In V(G)$ with $|X|\ge \xi n$ such that for all $e\in \binom{X}{r}$, there are at least $\xi n^{f-r}$ $(f-r)$-sets $Q\In V(G)\sm e$ such that $\binom{Q\cup e}{r}\sm\Set{e}\In G^{(r)}$.
\end{enumerate}
We say that $G$ is a \defn{full $(\eps,\xi,f,r)$-complex} if $G$ is
\begin{itemize}
\item $(\eps,d,f,r)$-regular for some $d\ge \xi$,
\item $(\xi,f+r,r)$-dense,
\item $(\xi,f,r)$-extendable.
\end{itemize}
We say that $G$ is an \defn{$(\eps,\xi,f,r)$-complex} if there exists an $f$-graph $Y$ on $V(G)$ such that $G[Y]$ is a full $(\eps,\xi,f,r)$-complex. Note that $G[Y]^{(r)}=G^{(r)}$ (recall that $r<f$).
\end{defin}

\begin{defin}{(supercomplex)}\label{def:supercomplex}
Let $G$ be a complex. We say that $G$ is an \defn{$(\eps,\xi,f,r)$-supercomplex} if for every $i\in[r]_0$ and every set $B\In G^{(i)}$ with $1\le|B|\le 2^i$, we have that $\bigcap_{b\in B}G(b)$ is an $(\eps,\xi,f-i,r-i)$-complex.
\end{defin}

In particular, taking $i=0$ and $B=\Set{\emptyset}$ implies that every $(\eps,\xi,f,r)$-supercomplex is also an $(\eps,\xi,f,r)$-complex.
Moreover, the above definition ensures that if $G$ is a supercomplex and $b,b'\in G^{(i)}$, then $G(b)\cap G(b')$ is also a supercomplex (cf.~Proposition~\ref{prop:hereditary}).

The following examples from \cite{GKLO} demonstrate that the definition of supercomplexes generalises the notion of typicality.

\begin{example}\label{ex:complete}
Let $1/n\ll 1/f$ and $r\in [f-1]$. Then the complete complex $K_n$ is a $(0,0.99/f!,f,r)$-supercomplex.
\end{example}

\begin{example}\label{ex:typical super}
Suppose that $1/n\ll c,p,1/f$, that $r\in [f-1]$ and that $G$ is a $(c,2^r\binom{f+r}{r},p)$-typical $r$-graph on $n$ vertices. Then $G^{\leftrightarrow}$ is an $(\eps,\xi,f,r)$-supercomplex, where $$\eps:=2^{f-r+1}c/(f-r)!\quad\mbox{and}\quad \xi:=(1-2^{f+1}c)p^{2^r\binom{f+r}{r}}/{f!}.$$
\end{example}

As mentioned above, the following lemma allows us to `boost' the regularity parameters (and thus deduce results with `effective' bounds). 
It is an easy consequence of our Boost lemma (Lemma~\ref{lem:boost}). The key to the proof is that we can (probabilistically) choose some $Y\In G^{(f)}$ so that the parameters of $G[Y]$ in Definition~\ref{def:complex}\ref{def:complex:regular} are better than those of $G$, i.e.~the resulting distribution of $f$-sets is more uniform.

\begin{lemma}[\cite{GKLO}]\label{lem:boost complex}
Let $1/n\ll \eps,\xi,1/f$ and $r\in[f-1]$ with $2(2\sqrt{\eul})^r \eps \le \xi$. Let $\xi':=0.9(1/4)^{\binom{f+r}{f}}\xi$. If $G$ is an $(\eps,\xi,f,r)$-complex on $n$ vertices, then $G$ is an $(n^{-1/3},\xi',f,r)$-complex. In particular, if $G$ is an $(\eps,\xi,f,r)$-supercomplex, then it is a $(2n^{-1/3},\xi',f,r)$-supercomplex.\COMMENT{$2$ accounts for $v(G(b))<v(G)$}
\end{lemma}

\subsection{The main complex decomposition theorem}\label{subsec:main thm}

The statement of our main complex decomposition theorem involves the concept of `well separated' decompositions.
This did not appear in \cite{GKLO}, but is crucial for our inductive proof to work in the context of $F$-decompositions.

\begin{defin}[well separated]\label{def:well separated}
Let $F$ be an $r$-graph and let $\cF$ be an $F$-packing (in some $r$-graph $G$). We say that $\cF$ is \defn{$\kappa$-well separated} if the following hold:
\begin{enumerate}[label={\rm (WS\arabic*)}]
\item for all distinct $F',F''\in \cF$, we have $|V(F')\cap V(F'')|\le r$.\label{separatedness:1}
\item for every $r$-set $e$, the number of $F'\in \cF$ with $e\In V(F')$ is at most $\kappa$.\label{separatedness:2}
\end{enumerate}
We simply say that $\cF$ is \defn{well separated} if \ref{separatedness:1} holds.
\end{defin}

For instance, any $\krq{r}{f}$-packing is automatically $1$-well separated. Moreover, if an $F$-packing $\cF$ is $1$-well separated, then for all distinct $F',F''\in \cF$, we have $|V(F')\cap V(F'')|< r$. On the other hand, if $F$ is not complete, we cannot require $|V(F')\cap V(F'')|< r$ in \ref{separatedness:1}: this would make it impossible to find an
$F$-decomposition of $\krq{r}{n}$.\COMMENT{Take any copy $F'$ of $F$. Then there is some $r$-set inside $V(F')$ which is not covered by $F'$, and no other copy is allowed to cover it either.}
The notion of being well-separated is a natural relaxation of this requirement,
we discuss this in more detail after stating Theorem~\ref{thm:main complex}.

We now define $F$-divisibility and $F$-decompositions for complexes~$G$
(rather than $r$-graphs~$G$).

\begin{defin}\label{def:complex dec}
Let $F$ be an $r$-graph and $f:=|V(F)|$. A complex $G$ is \defn{$F$-divisible} if $G^{(r)}$ is $F$-divisible. An \defn{$F$-packing in $G$} is an $F$-packing $\cF$ in $G^{(r)}$ such that $V(F')\in G^{(f)}$ for all $F'\in \cF$. Similarly, we say that $\cF$ is an \defn{$F$-decomposition of $G$} if $\cF$ is an $F$-packing in $G$ and $\cF^{(r)}=G^{(r)}$.
\end{defin}
Note that this implies that every copy $F'$ of $F$ used in an $F$-packing in $G$ is `supported' by a clique, i.e.~$G^{(r)}[V(F')]\cong \krq{r}{f}$.

We can now state our main complex decomposition theorem.
\begin{theorem}[Main complex decomposition theorem]\label{thm:main complex}
For all $r\in \bN$, the following is true.
\begin{itemize}
\item[\ind{r}] Let $1/n\ll 1/\kappa,\eps \ll \xi,1/f$ and $f>r$. Let $F$ be a weakly regular $r$-graph on $f$ vertices and
let $G$ be an $F$-divisible $(\eps,\xi,f,r)$-supercomplex on $n$ vertices. Then $G$ has a $\kappa$-well separated $F$-decomposition.
\end{itemize}
\end{theorem}

We will prove \ind{r} by induction on $r$ in Section~\ref{sec:main proofs}. We do not make any attempt to optimise the values that we obtain for $\kappa$.

We now motivate Definitions~\ref{def:well separated} and~\ref{def:complex dec}. 
This involves the following additional concepts, which are also convenient later.
\begin{defin}
Let $f:=|V(F)|$ and suppose that  $\cF$ is a well separated $F$-packing. 
We let $\cF^{\le}$ denote the complex generated by the $f$-graph $\set{V(F')}{F'\in \cF}$.
We say that well separated $F$-packings $\cF_1,\cF_2$ are \defn{$i$-disjoint} if $\cF_1^\le,\cF_2^\le$ are $i$-disjoint (or equivalently, if $|V(F')\cap V(F'')|< i$ for all $F'\in \cF_1$ and $F''\in \cF_2$).
\end{defin}
Note that if $F$ is a well-separated $F$-packing, then  the $f$-graph $\set{V(F')}{F'\in \cF}$ is simple.\COMMENT{If $\cF$ is well separated, this in particular implies that $V(F')\neq V(F'')$ for all distinct $F',F''\in \cF$. if $f>r$, this is clear. If $f=r$ and thus $F=\krq{r}{r}$, two copies of $F$ cannot have the same vertex set as they would not be edge-disjoint.} 
Moreover, observe that \ref{separatedness:2} is equivalent to the condition 
$\Delta_r(\cF^{\le(f)})\le \kappa$. 
Furthermore, if $\cF$ is a well separated $F$-packing in a complex $G$, then $\cF^{\le}$ is a subcomplex of $G$ by Definition~\ref{def:complex dec}. Clearly, we have $\cF^{(r)}\In \cF^{\le(r)}$, but in general equality does not hold. 
On the other hand, if $\cF$ is an $F$-decomposition of $G$, then $\cF^{(r)}=G^{(r)}$ which implies $\cF^{(r)}=\cF^{\le(r)}$.

We now discuss \ref{separatedness:1}.
During our proof, we will need to find an $F$-packing which covers a given set of edges. This gives rise to the following task of `covering down locally'.

\emph{$(\star)$ Given a set $S\In V(G)$ of size $1\le i\le r-1$, find an $F$-packing $\cF$ which covers all edges of $G$ that contain $S$.} 

(This is crucial in the proof of Lemma~\ref{lem:cover down}. Moreover, a two-sided version of this involving sets $S$, $S'$ is needed to construct parts of our absorbers, see Section~\ref{subsec:transformers}.) 

A natural approach to achieve $(\star)$ is as follows: Let $T\in \binom{V(F)}{i}$. Suppose that by using the main theorem inductively, we can find an $F(T)$-decomposition $\cF'$ of $G(S)$. We now wish to obtain $\cF$ by `extending' $\cF'$ as follows: For each copy $F'$ of $F(T)$ in $\cF'$, we define a copy $F'_{\triangleleft}$ of $F$ by `adding $S$ back', that is, $F'_{\triangleleft}$ has vertex set $V(F')\cup S$ and $S$ plays the role of $T$ in $F'_{\triangleleft}$. Then $F'_{\triangleleft}$ covers all edges $e$ with $S\In e$ and $e\sm S\in F'$. Since $\cF'$ is an $F(T)$-decomposition of $G(S)$, the union of all $F'_{\triangleleft}$ would indeed cover all edges of $G$ that contain $S$, as desired. There are two issues with this `extension' though. Firstly, it is not clear that $F'_{\triangleleft}$ is a subgraph of $G$. Secondly, for distinct $F',F''\in \cF'$, it is not clear that $F'_{\triangleleft}$ and $F''_{\triangleleft}$ are edge-disjoint. Definition~\ref{def:complex dec} (and the succeeding remark) allows us to resolve the first issue. Indeed, if $\cF'$ is an $F(T)$-decomposition of 
the complex $G(S)$, then from $V(F')\in G(S)^{(f-i)}$, we can deduce $V(F'_{\triangleleft})\in G^{(f)}$ and thus that $F'_{\triangleleft}$ is a subgraph of $G^{(r)}$.

We now consider the second issue.
This does not arise if $F$ is a clique. Indeed, in that case $F(T)$ is a copy of $\krq{r-i}{f-i}$, and thus for distinct $F',F''\in \cF'$ we have $|V(F')\cap V(F'')|<r-i$. Hence $|V(F'_{\triangleleft})\cap V(F''_{\triangleleft})|<r-i+|S|=r$, i.e.~$F'_{\triangleleft}$ and $F''_{\triangleleft}$ are edge-disjoint. If however $F$ is not a clique%
\COMMENT{
(more precisely, if there is no $T$ such that $F(T)$ is a clique)}, 
then $F',F''\in \cF'$ can overlap in $r-i$ or more vertices (they could in fact have the same vertex set), and the above argument does not work.
We will show that under the assumption that $\cF'$ is well separated, we can overcome this issue and still carry out the above `extension'.
(Moreover,  the resulting $F$-packing $\cF$ will in fact be well separated itself, see Definition~\ref{def:link extension} and Proposition~\ref{prop:S cover}).
For this it is useful to note that $F(T)$ is an $(r-i)$-graph, and thus we already have $|V(F')\cap V(F'')|\le r-i$ if $\cF'$ is well separated.

The reason why we also include \ref{separatedness:2} in Definition~\ref{def:well separated} is as follows. Suppose we have already found a well separated $F$-packing $\cF_1$ in $G$ and now want to find another well separated $F$-packing $\cF_2$ such that we can combine $\cF_1$ and $\cF_2$. If we find $\cF_2$ in $G-\cF_1^{(r)}$, then $\cF_1^{(r)}$ and $\cF_2^{(r)}$ are edge-disjoint and thus $\cF_1\cup \cF_2$ will be an $F$-packing in $G$, but it is not necessarily well separated. We therefore find $\cF_2$ in $G-\cF_1^{(r)}-\cF_1^{\le(r+1)}$. This ensures that $\cF_1$ and $\cF_2$ are $(r+1)$-disjoint, which in turn implies that $\cF_1\cup \cF_2$ is indeed well separated, as required.
But in order to be able to construct $\cF_2$, we need to ensure that $G-\cF_1^{(r)}-\cF_1^{\le(r+1)}$ is still a supercomplex, which is true if $\Delta(\cF_1^{(r)})$ and $\Delta(\cF_1^{\le(r+1)})$ are small (cf.~Proposition~\ref{prop:noise}). The latter in turn is ensured by \ref{separatedness:2}
 via Fact~\ref{fact:ws}.

Finally, we discuss why we prove Theorem~\ref{thm:main complex} for weakly regular $r$-graphs $F$. Most importantly, the `regularity' of the degrees will be crucial for the construction of our absorbers (most notably in Lemma~\ref{lem:colouring}). Beyond that, weakly regular graphs also have useful closure properties (cf.~Proposition~\ref{prop:link divisibility}): they are closed under taking link graphs and divisibility is inherited by link graphs in a natural way.

We prove Theorem~\ref{thm:main complex} in Sections~\ref{sec:bounded}--\ref{sec:absorbers} and~\ref{subsec:main complex thm}.
As described in Section~\ref{subsec:outline}, we generalise this to arbitrary $F$ via Lemma~\ref{lem:regularisation} (proved in Section~\ref{subsec:resolvable}) and Lemma~\ref{lem:make divisible} (proved in Section~\ref{sec:make divisible}):
Lemma~\ref{lem:regularisation} shows that for every given $r$-graph $F$, there is a weakly regular $r$-graph $F^\ast$ which has an $F$-decomposition. 
Lemma~\ref{lem:make divisible} then complements this by showing that every $F$-divisible $r$-graph $G$ can be transformed into an $F^\ast$-divisible $r$-graph $G'$ by removing a sparse $F$-decomposable subgraph of $G$.

\section{Tools}\label{sec:tools}

\subsection{Basic tools}

We will often use the following `handshaking lemma' for $r$-graphs: Let $G$ be an $r$-graph and $0\le i\le k\le r-1$. Then for every $S\in\binom{V(G)}{i}$ we have
\begin{align}
|G(S)|=\binom{r-i}{r-k}^{-1}\sum_{T\in \binom{V(G)}{k}\colon S\In T}|G(T)|.\label{handshaking}
\end{align}

\begin{prop}\label{prop:divisible existence}
Let $F$ be an $r$-graph. Then there exist infinitely many $n\in\bN$ such that $\krq{r}{n}$ is $F$-divisible.
\end{prop}

\proof
Let $p:=\prod_{i=0}^{r-1}Deg(F)_i$. We will show that for every $a\in \bN$, if we let $n=r!ap+r-1$ then $\krq{r}{n}$ is $F$-divisible. Clearly, this implies the claim. In order to see that $\krq{r}{n}$ is $F$-divisible, it is sufficient to show that $p\mid \binom{n-i}{r-i}$ for all $i\in [r-1]_0$. It is easy to see that this holds for the above choice of $n$.\COMMENT{$\binom{n-i}{r-i}=\frac{(n-i)\cdots(n-r+1)}{(r-i)!}=\frac{(n-i)\cdots(n-r+2)r!ap }{(r-i)!}$}
\endproof

The following proposition shows that the class of weakly regular uniform hypergraphs is closed under taking link graphs.

\begin{prop}\label{prop:link divisibility}
Let $F$ be a weakly regular $r$-graph and let $i\in[r-1]$. Suppose that $S\in\binom{V(F)}{i}$ and that $F(S)$ is non-empty. Then $F(S)$ is a weakly regular $(r-i)$-graph and $Deg(F(S))_j=Deg(F)_{i+j}$ for all $j\in[r-i-1]_0$.
\end{prop}

\proof
Let $s_0,\dots,s_{r-1}$ be such that $F$ is weakly $(s_0,\dots,s_{r-1})$-regular. Note that since $F$ is non-empty, we have $s_j>0$ for all $j\in[r-1]_0$ (and the $s_i$'s are unique).
Consider $j\in[r-i-1]_0$. For all $T\in \binom{V(F(S))}{j}$, we have $|F(S)(T)|=|F(S\cup T)|\in\Set{0,s_{i+j}}$. Hence, $F(S)$ is weakly $(s_i,\dots,s_{r-1})$-regular. Since $F$ is non-empty, we have $Deg(F)=(s_0,\dots,s_{r-1})$, and since $F(S)$ is non-empty too by assumption, we have $Deg(F(S))=(s_i,\dots,s_{r-1})$. Therefore, $Deg(F(S))_j=Deg(F)_{i+j}$ for all $j\in[r-i-1]_0$.
\endproof

We now list some useful properties of well separated $F$-packings.

\begin{fact}\label{fact:ws}
Let $G$ be a complex and $F$ an $r$-graph on $f>r$ vertices. Suppose that $\cF$ is a $\kappa$-well separated $F$-packing (in $G$) and $\cF'$ is a $\kappa'$-well separated $F$-packing (in $G$). Then the following hold.
\begin{enumerate}[label={\rm (\roman*)}]
\item $\Delta(\cF^{\le(r+1)})\le \kappa(f-r)$.\COMMENT{Every $r$-set $e$ is contained in at most $\kappa$ $f$-sets in $\cF^{\le(f)}$ by \ref{separatedness:2}. For each of these, we get $f-r$ $(r+1)$-sets in $\cF^{\le(r+1)}$ which contain $e$}\label{fact:ws:maxdeg}
\item If $\cF^{(r)}$ and $\cF'^{(r)}$ are edge-disjoint and $\cF$ and $\cF'$ are $(r+1)$-disjoint, then $\cF\cup \cF'$ is a $(\kappa+\kappa')$-well separated $F$-packing (in $G$). \label{fact:ws:1}
\item If $\cF$ and $\cF'$ are $r$-disjoint, then $\cF\cup \cF'$ is a $\max\Set{\kappa,\kappa'}$-well separated $F$-packing (in $G$).\label{fact:ws:2}
\end{enumerate}
\end{fact}

\subsection{Some properties of supercomplexes}

The following properties of supercomplexes were proved in~\cite{GKLO}.

\begin{prop}[\cite{GKLO}]\label{prop:hereditary}
Let $G$ be an $(\eps,\xi,f,r)$-supercomplex and let $B\In G^{(i)}$ with $1\le |B|\le 2^i$ for some $i\in[r]_0$. Then $\bigcap_{b\in B}G(b)$ is an $(\eps,\xi,f-i,r-i)$-supercomplex.\COMMENT{Only needed for transformers, could still leave it here because it is quite an important fact?}
\end{prop}

\begin{prop}[\cite{GKLO}]\label{prop:noise}
Let $f,r'\in \bN$ and $r\in \bN_0$ with $f>r$ and $r'\ge r$. Let $G$ be a complex on $n\ge r2^{r+1}$ vertices and let $H$ be an $r'$-graph on $V(G)$ with $\Delta(H)\le \gamma n$. Then the following hold:
\begin{enumerate}[label={\rm (\roman*)}]
\item If $G$ is $(\eps,d,f,r)$-regular, then $G-H$ is $(\eps+2^r\gamma,d,f,r)$-regular.\label{noise:regular}
\item If $G$ is $(\xi,f,r)$-dense, then $G-H$ is $(\xi-2^r\gamma,f,r)$-dense.\label{noise:dense}
\item If $G$ is $(\xi,f,r)$-extendable, then $G-H$ is $(\xi-2^r\gamma,f,r)$-extendable.\label{noise:extendable}
\item If $G$ is an $(\eps,\xi,f,r)$-complex, then $G-H$ is an $(\eps+2^r\gamma,\xi-2^{r}\gamma,f,r)$-complex.\label{noise:complex}
\item If $G$ is an $(\eps,\xi,f,r)$-supercomplex, then $G-H$ is an $(\eps+2^{2r+1}\gamma,\xi-2^{2r+1}\gamma,f,r)$-supercomplex.\label{noise:supercomplex}
\end{enumerate}\COMMENT{Only \ref{noise:supercomplex} used in main paper}
\end{prop}

\begin{cor}[\cite{GKLO}]\label{cor:random supercomplex}
Let $1/n\ll \eps,\gamma,\xi,p,1/f$ and $r\in[f-1]$. Let $$\xi':=0.95\xi p^{2^r\binom{f+r}{r}} \ge 0.95\xi p^{(8^f)}\mbox{ and } \gamma':=1.1\cdot 2^{r}\frac{\binom{f+r}{r}}{(f-r)!}\gamma.$$ Suppose that $G$ is an $(\eps,\xi,f,r)$-supercomplex on $n$ vertices and that $H\In G^{(r)}$ is a random subgraph obtained by including every edge of $G^{(r)}$ independently with probability $p$. Then \whp the following holds: for all $L\In G^{(r)}$ with $\Delta(L)\le \gamma n$, $G[H\bigtriangleup L]$ is a $(3\eps+\gamma',\xi'-\gamma',f,r)$-supercomplex.
\end{cor}

\COMMENT{Only needed for min degree version.}

\subsection{Rooted Embeddings}
We now prove a result (Lemma~\ref{lem:rooted embedding}) which allows us to find edge-disjoint embeddings of graphs with a prescribed `root embedding'. Let $T$ be an $r$-graph and suppose that $X\In V(T)$ is such that $T[X]$ is empty. A \defn{root of $(T,X)$} is a set $S\In X$ with $|S|\in [r-1]$ and $|T(S)|>0$.

For an $r$-graph $G$, we say that $\Lambda\colon X\to V(G)$ is a \defn{$G$-labelling of $(T,X)$} if $\Lambda$ is injective. Our aim is to embed $T$ into $G$ such that the roots of $(T,X)$ are embedded at their assigned position.
 More precisely, given a $G$-labelling $\Lambda$ of $(T,X)$, we say that $\phi$ is a \defn{$\Lambda$-faithful embedding of $(T,X)$ into $G$} if $\phi$ is an injective homomorphism from $T$ to $G$ with $\phi{\restriction_{X}}=\Lambda$. Moreover, for a set $S\In V(G)$ with $|S|\in [r-1]$, we say that $\Lambda$ \defn{roots $S$} if $S\In \Ima(\Lambda)$ and $|T(\Lambda^{-1}(S))|>0$, i.e.~if $\Lambda^{-1}(S)$ is a root of $(T,X)$.

The \defn{degeneracy of $T$ rooted at $X$} is the smallest $D$ such that there exists an ordering $v_1,\dots,v_{k}$ of the vertices of $V(T)\sm X$ such that for every $\ell\in[k]$, we have $$|T[X\cup \Set{v_1,\dots,v_\ell}](v_\ell)|\le D,$$ i.e.~every vertex is contained in at most $D$ edges which lie to the left of that vertex in the ordering.

We need to be able to embed many copies of $(T,X)$ simultaneously (with different labellings) into a given host graph $G$ such that the different embeddings are edge-disjoint. In fact, we need a slightly stronger disjointness criterion. Ideally, we would like to have that two distinct embeddings intersect in less than $r$ vertices. However, this is in general not possible because of the desired rooting. We therefore introduce the following concept of a \defn{hull}. We will ensure that the hulls are edge-disjoint, which will be sufficient for our purposes.
Given $(T,X)$ as above, the \defn{hull of $(T,X)$} is the $r$-graph $T'$ on $V(T)$ with $e\in T'$ if and only if $e\cap X=\emptyset$ or $e\cap X$ is a root of $(T,X)$. Note that $T\In T'\In K_{V(T)}^{(r)}-K_X^{(r)}$, where $K_Z^{(r)}$ denotes the complete $r$-graph with vertex set $Z$.\COMMENT{Let $e\in T$. If $e\cap X=\emptyset$, then $e\in T'$. If $e\cap X\neq\emptyset$, then $|e\cap X|\in [r-1]$ since $T[X]$ is empty. Moreover, $|T(e\cap X)|>0$ since $e\cap X\In e$. Thus, $e\cap X$ is a root and hence $e\in T'$.} Moreover, the roots of $(T',X)$ are precisely the roots of $(T,X)$.\COMMENT{Every root of $T$ is also a root of $T'$ as $T\In T'$. Suppose now that $S$ is a root of $(T',X)$. Hence, there is $e\in T'$ with $S\In e$. In particular, $e\cap X\neq \emptyset$. Thus, $e\cap X$ is a root of $(T,X)$, i.e.~$|T(e\cap X)|>0$, since $S\In e\cap X$, this implies $|T(S)|>0$, i.e.~$S$ is a root of $(T,X)$.}

\begin{lemma}\label{lem:rooted embedding}
Let $1/n\ll \gamma \ll \xi,1/t,1/D$ and $r\in [t]$. Suppose that $\alpha \in (0,1]$ is an arbitrary scalar (which might depend on $n$) and let $m\le \alpha\gamma n^r$ be an integer. For every $j\in[m]$, let $T_j$ be an $r$-graph on at most $t$ vertices and $X_j\In V(T_j)$ such that $T_j[X_j]$ is empty and $T_j$ has degeneracy at most $D$ rooted at $X_j$. Let $G$ be an $r$-graph on $n$ vertices such that for all $A\In\binom{V(G)}{r-1}$ with $|A|\le D$, we have $|\bigcap_{S\in A}G(S)|\ge \xi n$.
Let $O$ be an $(r+1)$-graph on $V(G)$ with $\Delta(O)\le \gamma n$.
For every $j\in[m]$, let $\Lambda_j$ be a $G$-labelling of $(T_j,X_j)$. Suppose that for all $S\In V(G)$ with $|S|\in[r-1]$, we have that
\begin{align}
|\set{j\in[m]}{\Lambda_j \mbox{ roots }S}|\le \alpha \gamma n^{r-|S|}-1.\label{rooting nice}
\end{align}
Then for every $j\in [m]$, there exists a $\Lambda_j$-faithful embedding $\phi_j$ of $(T_j,X_j)$ into $G$ such that the following hold:
\begin{enumerate}[label=\rm{(\roman*)}]
\item for all distinct $j,j'\in [m]$, the hulls of $(\phi_j(T_j),\Ima(\Lambda_j))$ and $(\phi_{j'}(T_{j'}),\Ima(\Lambda_{j'}))$ are edge-disjoint;\label{rooted embedding:disjoint}
\item for all $j\in[m]$ and $e\in O$ with $e\In \Ima(\phi_j)$, we have $e\In \Ima(\Lambda_j)$;\label{rooted embedding:f-disjoint}
\item $\Delta(\bigcup_{j\in[m]}\phi_j(T_j))\le \alpha\gamma^{(2^{-r})} n$.\label{rooted embedding:maxdeg}
\end{enumerate}
\end{lemma}

Note that \ref{rooted embedding:disjoint} implies that $\phi_1(T_1),\dots,\phi_m(T_m)$ are edge-disjoint. We also remark that the $T_j$ do not have to be distinct; in fact, they could all be copies of a single $r$-graph $T$.

\proof
For $j\in[m]$ and a set $S\In V(G)$ with $|S|\in[r-1]$, let $$root(S,j):=|\set{j'\in[j]}{\Lambda_{j'} \mbox{ roots }S}|.$$ We will define $\phi_1,\dots,\phi_m$ successively. Once $\phi_j$ is defined, we let $K_j$ denote the hull of $(\phi_j(T_j),\Ima(\Lambda_j))$. Note that $\phi_j(T_j)\In K_j$ and that $K_j$ is not necessarily a subgraph of $G$.

Suppose that for some $j\in[m]$, we have already defined $\phi_1,\dots,\phi_{j-1}$ such that $K_1,\dots,K_{j-1}$ are edge-disjoint, \ref{rooted embedding:f-disjoint} holds for all $j'\in[j-1]$, and the following holds for $G_{j}:=\bigcup_{j'\in[j-1]}K_{j'}$, all $i\in[r-1]$ and all $S\in \binom{V(G)}{i}$:
\begin{align}
|G_{j}(S)|\le \alpha\gamma^{(2^{-i})}n^{r-i}+ (root(S,j-1)+1)2^t.\label{embedding degree bound}
\end{align}
Note that \eqref{embedding degree bound} together with \eqref{rooting nice} implies that for all $i\in[r-1]$ and all $S\in \binom{V(G)}{i}$, we have
\begin{align}
|G_{j}(S)|\le 2\alpha\gamma^{(2^{-i})}n^{r-i}.\label{embedding degree bound new}
\end{align}\COMMENT{$(root(S,j-1)+1)2^t\le \alpha \gamma n^{r-i}2^t  \le  \alpha\gamma^{(2^{-i})}n^{r-i} $ as $\gamma 2^t\le \gamma^{(2^{-i})}$}

We will now define a $\Lambda_j$-faithful embedding $\phi_j$ of $(T_j,X_j)$ into $G$ such that $K_j$ is edge-disjoint from $G_j$, \ref{rooted embedding:f-disjoint} holds for $j$, and \eqref{embedding degree bound} holds with $j$ replaced by $j+1$. For $i\in[r-1]$, define $BAD_i:=\set{S\in\binom{V(G)}{i}}{|G_j(S)|\ge \alpha\gamma^{(2^{-i})}n^{r-i}}$.
We view $BAD_i$ as an $i$-graph. We claim that for all $i\in[r-1]$,
\begin{align}
\Delta(BAD_i)\le \gamma^{(2^{-r})}n.\label{degeneracy embedding max degree}
\end{align}\COMMENT{Deliberately no $\alpha$ here}
Consider $i\in[r-1]$ and suppose that there exists some $S\in\binom{V(G)}{i-1}$ such that $|BAD_i(S)|>\gamma^{(2^{-r})}n$. We then have that
\begin{align*}
|G_j(S)|&=\frac{1}{r-i+1}\sum_{v\in V(G)\sm S}|G_j(S\cup \Set{v})|\ge r^{-1} \sum_{v\in BAD_i(S)}|G_j(S\cup \Set{v})|\\
          &\ge r^{-1}|BAD_i(S)|\alpha \gamma^{(2^{-i})}n^{r-i}  \ge r^{-1}\gamma^{(2^{-r})}n \alpha \gamma^{(2^{-i})}n^{r-i}=r^{-1}\alpha \gamma^{(2^{-r}+2^{-i})}n^{r-(i-1)}.
\end{align*}
This contradicts \eqref{embedding degree bound new} if $i-1>0$ since $2^{-r}+2^{-i}<2^{-(i-1)}$. If $i=1$, then $S=\emptyset$ and we have $|G_j|\ge r^{-1}\alpha\gamma^{(2^{-r}+2^{-1})}n^{r}$, which is also a contradiction since $|G_{j}|\le m\binom{t}{r}\le \binom{t}{r} \alpha\gamma n^r$ and $2^{-r}+2^{-1}<1$ (as $r\ge 2$ if $i\in [r-1]$). This proves~\eqref{degeneracy embedding max degree}.

We now embed the vertices of $T_j$ such that the obtained embedding $\phi_j$ is $\Lambda_j$-faithful. First, embed every vertex from $X_j$ at its assigned position. Since $T_j$ has degeneracy at most $D$ rooted at $X_j$, there exists an ordering $v_1,\dots,v_{k}$ of the vertices of $V(T_j)\sm X_j$ such that for every $\ell\in[k]$, we have
\begin{align}
|T_j[X_j\cup \Set{v_1,\dots,v_\ell}](v_\ell)|\le D.\label{degeneracy ordering}
\end{align}
Suppose that for some $\ell\in [k]$, we have already embedded $v_1,\dots,v_{\ell-1}$. We now want to define $\phi_j(v_\ell)$.
Let $U:=\set{\phi_j(v)}{v\in X_j\cup \Set{v_1,\dots,v_{\ell-1}}}$ be the set of vertices which have already been used as images for $\phi_j$.
Let $A$ contain all $(r-1)$-subsets $S$ of $U$ such that $\phi_j^{-1}(S)\cup \Set{v_\ell}\in T_j$. We need to choose $\phi_j(v_\ell)$ from the set $(\bigcap_{S\in A}G(S))\sm U$ in order to complete $\phi_j$ to an injective homomorphism from $T_j$ to $G$. By~\eqref{degeneracy ordering}, we have $|A|\le D$. Thus, by assumption, $|\bigcap_{S\in A}G(S)|\ge \xi n$.

For $i\in[r-1]$, let $O_i$ consist of all vertices $x\in V(G)$ such that there exists some $S\in \binom{U}{i-1}$ such that $S\cup \Set{x}\in BAD_i$ (so $BAD_1= \binom{O_1}{1}$). We have $$|O_i|\le \binom{|U|}{i-1}\Delta(BAD_i)\overset{\eqref{degeneracy embedding max degree}}{\le} \binom{t}{i-1}\gamma^{(2^{-r})}n.$$
Let $O_r$ consist of all vertices $x\in V(G)$ such that $S\cup \Set{x}\in G_j$ for some $S\in \binom{U}{r-1}$. By~\eqref{embedding degree bound new}, we have that $|O_r|\le \binom{|U|}{r-1}\Delta(G_j)\le  \binom{t}{r-1}2\alpha \gamma^{(2^{-(r-1)})}n \le \binom{t}{r-1} \gamma^{(2^{-r})}n$.\COMMENT{Here need $\alpha\le 1$}
Finally, let $O_{r+1}$ be the set of all vertices $x\in V(G)$ such that there exists some $S\in \binom{U}{r}$ such that $S\cup \Set{x}\in O$. By assumption, we have $|O_{r+1}|\le \binom{|U|}{r}\Delta(O) \le \binom{t}{r}\gamma n$.

Crucially, we have $$|\bigcap_{S\in A}G(S)|-|U|-\sum_{i=1}^{r+1}|O_i| \ge \xi n-t-2^t\gamma^{(2^{-r})}n >0.$$ Thus, there exists a vertex $x\in V(G)$ such that $x\notin U\cup O_1\cup \dots \cup O_{r+1}$ and $S\cup \Set{x}\in G$ for all $S\in A$. Define $\phi_j(v_\ell):=x$.

Continuing in this way until $\phi_j$ is defined for every $v\in V(T_j)$ yields an injective homomorphism from $T_j$ to $G$. By definition of $O_{r+1}$, \ref{rooted embedding:f-disjoint} holds for $j$.  Moreover, by definition of $O_r$, $K_j$ is edge-disjoint from $G_j$. It remains to show that \eqref{embedding degree bound} holds with $j$ replaced by $j+1$. Let $i\in[r-1]$ and $S\in \binom{V(G)}{i}$. If $S\notin BAD_i$, then we have $|G_{j+1}(S)|\le |G_{j}(S)|+\binom{t-i}{r-i}\le \alpha\gamma^{(2^{-i})}n^{r-i}+2^t$, so \eqref{embedding degree bound} holds. Now, assume that $S\in BAD_i$. If $S\In \Ima(\Lambda_{j})$ and $|T_j(\Lambda_{j}^{-1}(S))|>0$, then $root(S,j)=root(S,j-1)+1$ and thus $|G_{j+1}(S)|\le |G_{j}(S)|+\binom{t-i}{r-i}\le \alpha \gamma^{(2^{-i})}n^{r-i}+ (root(S,j-1)+1)2^t+\binom{t-i}{r-i} \le \alpha\gamma^{(2^{-i})}n^{r-i}+ (root(S,j)+1)2^t$ and \eqref{embedding degree bound} holds. Suppose next that $S\not\In \Ima(\Lambda_{j})$. We claim that $S\not\In V(\phi_j(T_j))$. Suppose, for a contradiction, that $S\In V(\phi_j(T_j))$. Let $\ell:=\max\set{\ell'\in[k]}{\phi_j(v_{\ell'})\in S}$. (Note that the maximum exists since $(S\cap V(\phi_j(T_j)))\sm \Ima(\Lambda_{j})$ is not empty.) Hence, $x:=\phi_j(v_\ell)\in S$. Recall that when we defined $\phi_j(v_\ell)$, $\phi_j(v)$ had already been defined for all $v\in X_j\cup \Set{v_1,\dots,v_{\ell-1}}$ and hence $S\sm\Set{x}\In U$. But since $S\in BAD_i$, we have $x\in O_i$, in contradiction to $x=\phi_j(v_\ell)$. Thus, $S\not\In V(\phi_j(T_j))=V(K_j)$, which clearly implies that $|G_{j+1}(S)|=|G_j(S)|$ and \eqref{embedding degree bound} holds. The last remaining case is if $S\In \Ima(\Lambda_{j})$ but $|T_j(\Lambda_{j}^{-1}(S))|=0$. But then $S$ is not a root of $(\phi_j(T_j),\Ima(\Lambda_j))$ and thus not a root of $(K_j,\Ima(\Lambda_j))$. Hence $|K_j(S)|=0$ and therefore $|G_{j+1}(S)|=|G_j(S)|$ as well.

Finally, if $j=m$, then the fact that \eqref{embedding degree bound} holds with $j$ replaced by $j+1$ together with \eqref{rooting nice} implies that $\Delta(\bigcup_{j\in[m]}\phi_j(T_j))\le 2\alpha \gamma^{(2^{-(r-1)})}n \le \alpha \gamma^{(2^{-r})} n$.
\endproof

\section{Approximate \texorpdfstring{$F$}{F}-decompositions}\label{sec:bounded}

The majority of the edges which are covered during our iterative absorption procedure are covered by `approximate' $F$-decompositions of certain parts of $G$, i.e.~$F$-packings which cover almost all the edges in these parts. For cliques, the existence of such packings was first proved by R\"odl~\cite{Ro}, introducing what is now called the `nibble' technique. 
Here, we derive a result on approximate $F$-decompositions which is suitable for our needs (Lemma~\ref{lem:F nibble}).

We will derive the $F$-nibble lemma (Lemma~\ref{lem:F nibble}) from the special case when $F$ is a clique. This in turn was derived in~\cite{GKLO} from a result in~\cite{AY} which allows us to assume that the leftover of an approximate clique decomposition has appropriately bounded maximum degree.

\begin{lemma}[Boosted nibble lemma, \cite{GKLO}]\label{lem:boosted nibble}
Let $1/n\ll \gamma,\eps \ll \xi,1/f$ and $r\in[f-1]$. Let $G$ be a complex on $n$ vertices such that $G$ is $(\eps,d,f,r)$-regular and $(\xi,f+r,r)$-dense for some $d\ge \xi$. Then $G$ contains a $\krq{r}{f}$-packing $\cK$ such that $\Delta(G^{(r)}-\cK^{(r)})\le \gamma n$.
\end{lemma}

Crucially, we do not need to assume that $\eps\ll \gamma$ in Lemma~\ref{lem:boosted nibble}. The reason for this is the so-called Boost lemma from \cite{GKLO}, which allows us to `boost' the regularity parameters of a suitable complex and which is an important ingredient in the proof of both Lemma~\ref{lem:boosted nibble} and Lemma~\ref{lem:F nibble}.

\begin{lemma}[Boost lemma, \cite{GKLO}]\label{lem:boost}
Let $1/n\ll \eps,\xi,1/f$ and $r\in[f-1]$ such that $2(2\sqrt{\eul})^r \eps \le \xi$. Let $\xi':=0.9(1/4)^{\binom{f+r}{f}}\xi$. Suppose that $G$ is a complex on $n$ vertices and that $G$ is $(\eps,d,f,r)$-regular and $(\xi,f+r,r)$-dense for some $d\ge \xi$. Then there exists $Y\In G^{(f)}$ such that $G[Y]$ is $(n^{-(f-r)/2.01},d/2,f,r)$-regular and $(\xi',f+r,r)$-dense.
\end{lemma}

We now prove an $F$-nibble lemma which allows us to find $\kappa$-well separated approximate $F$-decompositions in supercomplexes.

\begin{lemma}[$F$-nibble lemma]\label{lem:F nibble}
Let $1/n\ll 1/\kappa \ll \gamma,\eps \ll \xi,1/f$ and $r\in[f-1]$.\COMMENT{Don't need $1/\kappa\ll \eps$, but $1/\kappa\ll \gamma$ is essential} Let $F$ be an $r$-graph on $f$ vertices.\COMMENT{Can be any $r$-graph as we do not use induction} Let $G$ be a complex on $n$ vertices such that $G$ is $(\eps,d,f,r)$-regular and $(\xi,f+r,r)$-dense for some $d\ge \xi$. Then $G$ contains a $\kappa$-well separated $F$-packing $\cF$ such that $\Delta(G^{(r)}-\cF^{(r)})\le \gamma n$.
\end{lemma}

Let $F$ be an $r$-graph on $f$ vertices. Given a collection $\cK$ of edge-disjoint copies of $\krq{r}{f}$, we define the \defn{$\cK$-random $F$-packing $\cF$} as follows: For every $K\in \cK$, choose a random bijection from $V(F)$ to $V(K)$ and let $F_K$ be a copy of $F$ on $V(K)$ embedded by this bijection. Let $\cF:=\set{F_K}{K\in \cK}$.

Clearly, if $\cK$ is a $\krq{r}{f}$-decomposition of a complex $G$, then the $\cK$-random $F$-packing $\cF$ is a $1$-well separated $F$-packing in $G$. Moreover, writing $p:=1-|F|/\binom{f}{r}$, we have $|\cF^{(r)}|=|F||\cK|=|F||G^{(r)}|/\binom{f}{r}=(1-p)|G^{(r)}|$, and for every $e\in G^{(r)}$, we have $\prob{e\in G^{(r)}-\cF^{(r)}}=p$. As turns out, the leftover $G^{(r)}-\cF^{(r)}$ behaves essentially like a $p$-random subgraph of $G^{(r)}$ (cf.~Lemma~\ref{lem:random packing}). Our strategy to prove Lemma~\ref{lem:F nibble} is thus as follows: We apply Lemma~\ref{lem:boosted nibble} to $G$ to obtain a $\krq{r}{f}$-packing $\cK_1$ such that $\Delta(G^{(r)}-\cK_1^{(r)})\le \gamma n$. The leftover here is negligible, so assume for the moment that $\cK_1$ is a $\krq{r}{f}$-decomposition. We then choose a $\cK_1$-random $F$-packing $\cF_1$ in $G$ and continue the process with $G-\cF_1^{(r)}$. In each step, the leftover decreases by a factor of $p$. Thus after $\log_{p}\gamma$ steps, the leftover will have maximum degree at most $\gamma n$.

\begin{lemma}\label{lem:random packing}
Let $1/n\ll \eps \ll \xi,1/f$ and $r\in[f-1]$. Let $F$ be an $r$-graph on $f$-vertices with $p:=1-|F|/\binom{f}{r}\in (0,1)$. Let $G$ be an $(\eps,d,f,r)$-regular and $(\xi,f+r,r)$-dense complex on $n$ vertices for some $d\ge \xi$. Suppose that $\cK$ is a $\krq{r}{f}$-decomposition of $G$. Let $\cF$ be the $\cK$-random $F$-packing in $G$. Then whp the following hold for $G':=G-\cK^{\le(r+1)}-\cF^{(r)}$.
\begin{enumerate}[label=\rm{(\roman*)}]
\item $G'$ is $(2\eps,p^{\binom{f}{r}-1}d,f,r)$-regular;\label{f nibble trick regular}
\item $G'$ is $(0.9p^{\binom{f+r}{r}-1}\xi,f+r,r)$-dense;\label{f nibble trick dense}
\item $\Delta(G'^{(r)})\le 1.1p\Delta(G^{(r)})$.\label{f nibble trick maxdeg}
\end{enumerate}
\end{lemma}

Since the assertions follow easily from the definitions, we omit the proof here. We refer to Appendix~\ref{app:tools} for the details.

\lateproof{Lemma~\ref{lem:F nibble}}
Let $p:=1-|F|/\binom{f}{r}$. If $F=\krq{r}{f}$, then we are done by Lemma~\ref{lem:boosted nibble}. We may thus assume that $p\in (0,1)$. Choose $\eps'>0$ such that $1/n\ll \eps' \ll 1/\kappa \ll \gamma, \eps  \ll p,1-p, \xi,1/f$. We will now repeatedly apply Lemma~\ref{lem:boosted nibble}. More precisely, let $\xi_0:=0.9(1/4)^{\binom{f+r}{f}}\xi$ and define $\xi_{j}:=(0.5p)^{j\binom{f+r}{r}}\xi_0$ for $j\ge 1$.
For every $j\in[\kappa]_0$, we will find $\cF_j$ and $G_j$ such that the following hold:
\begin{enumerate}[label=\rm{(\alph*)}]
\item$\hspace{-5pt}_{j}$ $\cF_j$ is a $j$-well separated $F$-packing in $G$ and $G_j\In G-\cF_j^{(r)}$;\label{iterative nibble 1}
\item$\hspace{-5pt}_{j}$ $\Delta(L_j)\le j\eps' n$, where $L_j:=G^{(r)}-\cF_j^{(r)}-G_j^{(r)}$;\label{iterative nibble 5}
\item$\hspace{-5pt}_{j}$ $G_j$ is $(2^{(r+1)j}\eps',d_j,f,r)$-regular and $(\xi_j,f+r,r)$-dense for some $d_j\ge \xi_j$;\label{iterative nibble 2}
\item$\hspace{-5pt}_{j}$ $\cF_{j}^\le$ and $G_j$ are $(r+1)$-disjoint;\label{iterative nibble 3}
\item$\hspace{-5pt}_{j}$ $\Delta(G_j^{(r)})\le (1.1p)^{j}n$.\label{iterative nibble 4}
\end{enumerate}
First, apply Lemma~\ref{lem:boost} to $G$ in order to find $Y\In G^{(f)}$ such that $G_0:=G[Y]$ is $(\eps',d/2,f,r)$-regular and $(\xi_0,f+r,r)$-dense. Hence, \ref{iterative nibble 1}$_{0}$--\ref{iterative nibble 4}$_{0}$ hold with $\cF_0:=\emptyset$. Also note that $\cF_\kappa$ will be a $\kappa$-well separated $F$-packing in $G$ and $\Delta(G^{(r)}-\cF_\kappa^{(r)})\le \Delta(L_\kappa)+\Delta(G_\kappa^{(r)})\le  \kappa\eps' n+(1.1p)^{\kappa}n\le \gamma n$, so we can take $\cF:=\cF_\kappa$.

Now, assume that for some $j\in[\kappa]$, we have found $\cF_{j-1}$ and $G_{j-1}$ and now need to find $\cF_{j}$ and $G_{j}$.
By \ref{iterative nibble 2}$_{j-1}$, $G_{j-1}$ is $(\sqrt{\eps'},d_{j-1},f,r)$-regular and $(\xi_{j-1},f+r,r)$-dense for some $d_{j-1}\ge \xi_{j-1}$.
Thus, we can apply Lemma~\ref{lem:boosted nibble}\COMMENT{$\sqrt{\eps'}\ll \xi_{j-1}$} to obtain a $\krq{r}{f}$-packing $\cK_j$ in $G_{j-1}$ such that $\Delta(L_j')\le \eps' n$, where $L_j':=G_{j-1}^{(r)}-\cK_{j}^{(r)}$.
Let $G_j':=G_{j-1}-L_j'$. Clearly, $\cK_j$ is a $\krq{r}{f}$-decomposition of $G_j'$. Moreover, by \ref{iterative nibble 2}$_{j-1}$ and Proposition~\ref{prop:noise} we have that $G_j'$ is $(2^{(r+1)(j-1)+r}\eps',d_{j-1},f,r)$-regular and $(0.9\xi_{j-1},f+r,r)$-dense.
By Lemma~\ref{lem:random packing}, there exists a $1$-well separated $F$-packing $\cF_j'$ in $G_j'$ such that the following hold for $G_j:=G_j'-\cF_j'^{(r)}-\cK_j^{\le(r+1)}=G_j'-\cF_j'^{(r)}-\cF_j'^{\le(r+1)}$:
\begin{enumerate}[label=\rm{(\roman*)}]
\item $G_j$ is $(2^{(r+1)(j-1)+r+1}\eps',p^{\binom{f}{r}-1}d_{j-1},f,r)$-regular;\label{almost random leftover 1}
\item $G_j$ is $(0.81 p^{\binom{f+r}{r}-1}\xi_{j-1},f+r,r)$-dense;\label{almost random leftover 2}
\item $\Delta(G_j^{(r)})\le 1.1p\Delta(G_j'^{(r)})$.\label{almost random leftover 3}
\end{enumerate}
Let $\cF_j:=\cF_{j-1}\cup \cF_j'$ and $L_j:=G^{(r)}-\cF_j^{(r)}-G_j^{(r)}$. Note that $\cF_{j-1}^{(r)}\cap \cF_j'^{(r)}=\emptyset$ by \ref{iterative nibble 1}$_{j-1}$. Moreover, $\cF_{j-1}$ and $\cF_j'$ are $(r+1)$-disjoint by \ref{iterative nibble 3}$_{j-1}$. Thus, $\cF_j$ is $(j-1+1)$-well separated by Fact~\ref{fact:ws}\ref{fact:ws:1}. Moreover, using \ref{iterative nibble 1}$_{j-1}$, we have $$G_j\In G_{j-1}-\cF_j'^{(r)}\In G-\cF_{j-1}^{(r)}-\cF_j'^{(r)},$$ thus \ref{iterative nibble 1}$_{j}$ holds.
Observe that $L_{j}\sm L_{j-1}\In L_j'$.\COMMENT{$L_{j}\sm L_{j-1}\In (\cF_{j-1}^{(r)}\cup G_{j-1}^{(r)})-(\cF_j^{(r)}\cup G_j^{(r)}) \In G_{j-1}^{(r)}-(\cF_j'^{(r)}\cup G_j^{(r)})=G_{j-1}^{(r)}-G_{j}'^{(r)}$} Thus, we clearly have $\Delta(L_j)\le \Delta(L_{j-1})+ \Delta(L_j')\le j\eps' n$, so \ref{iterative nibble 5}$_{j}$ holds.
Moreover, \ref{iterative nibble 2}$_{j}$ follows directly from \ref{almost random leftover 1} and \ref{almost random leftover 2},\COMMENT{need that $0.81p^{\binom{f+r}{r}-1}\xi_{j-1}\ge \xi_j$, iff $0.81p^{\binom{f+r}{r}-1} \ge (0.5p)^{\binom{f+r}{r}}$} and \ref{iterative nibble 4}$_{j}$ follows from \ref{iterative nibble 4}$_{j-1}$ and \ref{almost random leftover 3}.\COMMENT{$\Delta(G_j^{(r)})\le 1.1p\Delta(G_j'^{(r)})\le 1.1p \Delta(G_{j-1}^{(r)})\le 1.1p (1.1p)^{j-1}n\le (1.1p)^j n$} To see \ref{iterative nibble 3}$_{j}$, observe that $\cF_{j-1}^\le$ and $G_j$ are $(r+1)$-disjoint by \ref{iterative nibble 3}$_{j-1}$ and since $G_j\In G_{j-1}$, and $\cF_{j}'^\le$ and $G_j$ are $(r+1)$-disjoint by definition of $G_j$.
Thus, \ref{iterative nibble 1}$_{j}$--\ref{iterative nibble 4}$_{j}$ hold and the proof is completed.
\endproof

\section{Vortices}\label{sec:vortices}

A vortex is best thought of as a sequence of nested `random-like' subsets of the vertex set of a supercomplex $G$. In our approach, the final set of the vortex has bounded size.

The main results of this section are Lemmas~\ref{lem:get vortex}~and~\ref{lem:almost dec}, where the first one shows that vortices exist, and the latter one shows that given a vortex, we can find an $F$-packing covering all edges which do not lie inside the final vortex set.
We now give the formal definition of what it means to be a `random-like' subset. 

\begin{defin}\label{def:regular subset}
Let $G$ be a complex on $n$ vertices. We say that $U$ is \defn{$(\eps,\mu,\xi,f,r)$-random in $G$} if there exists an $f$-graph $Y$ on $V(G)$ such that the following hold:
\begin{enumerate}[label={\rm(R\arabic*)}]
\item $U\In V(G)$ with $|U|=\mu n\pm n^{2/3}$;\label{random:objects}
\item there exists $d\ge \xi$ such that for all $x\in[f-r]_0$ and all $e\in G^{(r)}$, we have that $$|\set{Q\in G[Y]^{(f)}(e)}{|Q\cap U|=x}|=(1\pm \eps)bin(f-r,\mu,x)dn^{f-r};$$\label{random:binomial}
\item for all $e\in G^{(r)}$ we have $|G[Y]^{(f+r)}(e)[U]|\ge \xi (\mu n)^{f}$;\label{random:dense}
\item for all $h\in[r]_0$ and all $B\In G^{(h)}$ with $1\le |B|\le 2^{h}$ we have that $\bigcap_{b\in B}G(b)[U]$ is an $(\eps,\xi,f-h,r-h)$-complex.\label{random:intersections}
\end{enumerate}
\end{defin}

Having defined what it means to be a `random-like' subset, we can now define what a vortex is.

\begin{defin}[Vortex]
Let $G$ be a complex. An \defn{$(\eps,\mu,\xi,f,r,m)$-vortex in $G$} is a sequence $U_0\supseteq U_1 \supseteq \dots \supseteq U_\ell$ such that
\begin{enumerate}[label={\rm(V\arabic*)}]
\item $U_0=V(G)$;
\item $|U_i|=\lfloor \mu |U_{i-1}|\rfloor$ for all $i\in[\ell]$;\label{vortex:size}
\item $|U_\ell|=m$;
\item for all $i\in[\ell]$, $U_i$ is $(\eps,\mu,\xi,f,r)$-random in $G[U_{i-1}]$;\label{vortex:untwisted}
\item for all $i\in[\ell-1]$, $U_i\sm U_{i+1}$ is $(\eps,\mu(1-\mu),\xi,f,r)$-random in $G[U_{i-1}]$.\label{vortex:twisted}
\end{enumerate}
\end{defin}

As shown in \cite{GKLO}, a vortex can be found in a supercomplex by repeatedly taking random subsets.

\begin{lemma}[\cite{GKLO}]\label{lem:get vortex}
Let $1/m'\ll \eps\ll \mu,\xi,1/f$ such that $\mu\le 1/2$ and $r\in[f-1]$. Let $G$ be an $(\eps,\xi,f,r)$-supercomplex on $n\ge m'$ vertices. Then there exists a $(2\sqrt{\eps},\mu,\xi-\eps,f,r,m)$-vortex in $G$ for some $\mu m' \le m \le m'$.
\end{lemma}

The following is the main lemma of this section. Given a vortex in a supercomplex $G$, it allows us to cover all edges of $G^{(r)}$ except possibly some from inside the final vortex set (see Lemma~7.13 in~\cite{GKLO} for the corresponding result in the case when $F$ is a clique).

\begin{lemma}\label{lem:almost dec}
Let $1/m\ll 1/\kappa \ll \eps \ll \mu \ll \xi,1/f$ and $r\in[f-1]$. Assume that \ind{k} is true for all $k\in[r-1]$. Let $F$ be a weakly regular $r$-graph on $f$ vertices. Let $G$ be an $F$-divisible $(\eps,\xi,f,r)$-supercomplex and $U_0 \supseteq U_1 \supseteq \dots \supseteq U_\ell$ an $(\eps,\mu,\xi,f,r,m)$-vortex in $G$. Then there exists a $4\kappa$-well separated $F$-packing $\cF$ in $G$ which covers all edges of $G^{(r)}$ except possibly some inside~$U_\ell$.
\end{lemma}

The proof of Lemma~\ref{lem:almost dec} consists of an `iterative absorption' procedure, where the key ingredient is the Cover down lemma (Lemma~\ref{lem:cover down}). Roughly speaking, given a supercomplex $G$ and a `random-like' subset $U\In V(G)$, the Cover down lemma allows us to find a `partial absorber' $H\In G^{(r)}$ such that for any sparse $L\In G^{(r)}$, $H\cup L$ has an $F$-packing which covers all edges of $H\cup L$ except possibly some inside $U$. Together with the $F$-nibble lemma (Lemma~\ref{lem:F nibble}), this allows us to cover all edges of $G$ except possibly some inside $U$ whilst using only few edges inside $U$. Indeed, set aside $H$ as above, which is reasonably sparse. Then apply the Lemma~\ref{lem:F nibble} to $G-G^{(r)}[U]-H$ to obtain an $F$-packing $\cF_{nibble}$ with a very sparse leftover $L$. Combine $H$ and $L$ to find an $F$-packing $\cF_{clean}$ whose leftover lies inside $U$.

Now, if $U_0 \supseteq U_1 \supseteq \dots \supseteq U_\ell$ is a vortex, then $U_1$ is `random-like' in $G$ and thus we can cover all edges which are not inside $U_1$ by using only few edges inside $U_1$ (and in this step we forbid edges inside $U_2$ from being used.) Then $U_2$ is still `random-like' in the remainder of $G[U_1]$, and hence we can iterate until we have covered all edges of $G$ except possibly some inside $U_\ell$. The proof of Lemma~\ref{lem:almost dec} is very similar to that of Lemma~7.13 in \cite{GKLO}, thus we omit it here. The details can be found in Appendix~\ref{app:vortices}.

We record the following easy tools from~\cite{GKLO} for later use.

\begin{fact}[\cite{GKLO}]\label{fact:random trivial}
The following hold.
\begin{enumerate}[label={\rm(\roman*)}]
\item If $G$ is an $(\eps,\xi,f,r)$-supercomplex, then $V(G)$ is $(\eps/\xi,1,\xi,f,r)$-random in $G$.\label{fact:random trivial:itself}
\item If $U$ is $(\eps,\mu,\xi,f,r)$-random in $G$, then $G[U]$ is an $(\eps,\xi,f,r)$-supercomplex.\label{fact:random trivial:subcomplex}
\end{enumerate}
\end{fact}

\begin{prop}[\cite{GKLO}]\label{prop:find subset}
Let $1/n\ll \eps \ll \mu_1,\mu_2,1-\mu_2,\xi,1/f$ and $r\in[f-1]$. Let $G$ be a complex on $n$ vertices and let $U\In V(G)$ be of size $\lfloor\mu_1 n \rfloor$ and $(\eps,\mu_1,\xi,f,r)$-random in $G$. Then there exists $\tilde{U}\In U$ of size $\lfloor\mu_2 |U|\rfloor$ such that
\begin{enumerate}[label={\rm(\roman*)}]
\item $\tilde{U}$ is $(\eps+|U|^{-1/6},\mu_2,\xi-|U|^{1/6},f,r)$-random in $G[U]$ and\label{random subset:untwisted}
\item $U\sm \tilde{U}$ is $(\eps+|U|^{-1/6},\mu_1(1-\mu_2),\xi-|U|^{1/6},f,r)$-random in $G$.\label{random subset:twisted}
\end{enumerate}
\end{prop}

\begin{prop}[\cite{GKLO}]\label{prop:random noise}
Let $1/n\ll \eps \ll \mu,\xi,1/f$ such that $\mu\le 1/2$ and $r\in[f-1]$. Suppose that $G$ is a complex on $n$ vertices and $U$ is $(\eps,\mu,\xi,f,r)$-random in $G$. Suppose that $L\In G^{(r)}$ and $O\In G^{(r+1)}$ satisfy $\Delta(L)\le \eps n$ and $\Delta(O)\le \eps n$. Then $U$ is still $(\sqrt{\eps},\mu,\xi-\sqrt{\eps},f,r)$-random in $G-L-O$.
\end{prop}

\subsection{The Cover down lemma}\label{subsec:cover down}

Recall that the Cover down lemma allows us to replace a given leftover $L$ with a new leftover which is restricted to some small set of vertices $U$. We now provide the formal statement.  

\begin{defin}\label{def:dense wrt}
Let $G$ be a complex on $n$ vertices and $H\In G^{(r)}$. We say that $G$ is \defn{$(\xi,f,r)$-dense with respect to $H$} if for all $e\in G^{(r)}$, we have $|G[H\cup \Set{e}]^{(f)}(e)|\ge \xi n^{f-r}$.\COMMENT{For $H=G^{(r)}$, this is the normal dense.}
\end{defin}

\begin{lemma}[Cover down lemma]\label{lem:cover down}
Let $1/n\ll 1/\kappa \ll \gamma \ll \eps \ll \nu\ll \mu,\xi,1/f$ and $r\in[f-1]$ with $\mu\le 1/2$. Assume that \ind{i} is true for all $i\in[r-1]$ and that $F$ is a weakly regular $r$-graph on $f$ vertices. Let $G$ be a complex on $n$ vertices and suppose that $U$ is $(\eps,\mu,\xi,f,r)$-random in $G$.
Let $\tilde{G}$ be a complex on $V(G)$ with $G\In \tilde{G}$ such that $\tilde{G}$ is $(\eps,f,r)$-dense with respect to $G^{(r)}- G^{(r)}[\bar{U}]$, where $\bar{U}:=V(G)\sm U$.

Then there exists a subgraph $H^{\ast}\In G^{(r)}-G^{(r)}[\bar{U}]$ with $\Delta(H^\ast)\le \nu n$ such that for any $L \In \tilde{G}^{(r)}$ with $\Delta(L)\le \gamma n$ and $H^\ast \cup L$ being $F$-divisible and any $(r+1)$-graph $O$ on $V(G)$ with $\Delta(O)\le \gamma n$, there exists a $\kappa$-well separated $F$-packing in $\tilde{G}[H^\ast \cup L]-O$ which covers all edges of $H^\ast\cup L$ except possibly some inside $U$.
\end{lemma}

Roughly speaking, the proof of the Cover down lemma proceeds as follows. 
Suppose that we have already chosen $H^\ast$  and that $L$ is any sparse (leftover) $r$-graph. For an edge $e\in H^\ast \cup L$, we refer to $|e\cap U|$ as its type. Since $L$ is very sparse, we can greedily cover all edges of $L$ in a first step. 
In particular, this covers all type-$0$-edges. We will now continue and cover all type-$1$-edges. Note that every type-$1$-edge contains a unique $S\in\binom{V(G)\sm U}{r-1}$. For a given set $S\in\binom{V(G)\sm U}{r-1}$, we would like to cover all remaining edges of $H^*$ that contain $S$ simultaneously. Assuming a suitable choice of $H^\ast$, this can be
achieved as follows. Let $L_S$ be the link graph of $S$ after the first step. Let $T\in \binom{V(F)}{r-1}$ be such that $F(T)$ is non-empty. By Proposition~\ref{prop:link divisibility}, $L_S$ will be $F(T)$-divisible. Thus, by \ind{1}, $L_S$ has a $\kappa$-well separated $F(T)$-decomposition $\cF_S'$. 
Proposition~\ref{prop:S cover} below implies that we can extend $\cF_S'$ to a $\kappa$-well separated $F$-packing $\cF_S$ which covers all edges that contain $S$. 

However, in order to cover all type-$1$-edges, we need to obtain such a packing $\cF_S$ for every $S\in\binom{V(G)\sm U}{r-1}$, and these packings are to be $r$-disjoint for their union to be a $\kappa$-well separated $F$-packing again. The real
difficulty thus lies in choosing $H^\ast$ in such a way that the link graphs $L_S$ do not interfere too much with each other, and then to choose the decompositions $\cF_S'$ sequentially.
We would then continue to cover all type-$2$-edges using \ind{2}, etc., until we finally cover all type-$(r-1)$-edges using \ind{r-1}. The only remaining edges are then type-$r$-edges, which are contained in $U$, as desired. 

Proving the Cover down lemma for cliques presented one of the main challenges
in~\cite{GKLO}. However, with Proposition~\ref{prop:S cover} in hand, the proof
carries over to general (weakly regular)\COMMENT{That's important for divisibility issues.} $F$ without significant modifications, and is thus omitted here. The full proof of the Cover down lemma (Lemma~\ref{lem:cover down}) can be found in Appendix~\ref{app:covering down}.

We now show how the notion of well separated $F$-packings allows us to `extend' a decomposition of a link complex to a packing which covers all edges that contain a given set $S$ (cf. the discussion in Section~\ref{subsec:main thm}).

\begin{defin}\label{def:link extension}
Let $F$ be an $r$-graph, $i\in[r-1]$ and assume that $T\in \binom{V(F)}{i}$ is such that $F(T)$ is non-empty. Let $G$ be a complex and $S\in \binom{V(G)}{i}$. Suppose that $\cF'$ is a well separated $F(T)$-packing in $G(S)$. We then define $S \triangleleft \cF'$ as follows: For each $F'\in \cF'$, let $F'_{\triangleleft}$ be an (arbitrary) copy of $F$ on vertex set $S\cup V(F')$ such that $F'_{\triangleleft}(S)=F'$. Let $$S \triangleleft \cF':=\set{F'_{\triangleleft}}{F'\in \cF'}.$$
\end{defin}

The following proposition is crucial and guarantees that the above extension yields a packing which covers the desired set of edges. It replaces Fact 10.1\COMMENT{check pointer once accepted} of \cite{GKLO} in the proof of the Cover down lemma. It is also used in the construction of so-called `transformers' (see Section~\ref{subsec:transformers}). 

\begin{prop}\label{prop:S cover}
Let $F$, $r$, $i$, $T$, $G$, $S$ be as in Definition~\ref{def:link extension}. Let $L\In G(S)^{(r-i)}$. Suppose that $\cF'$ is a $\kappa$-well separated $F(T)$-decomposition of $G(S)[L]$. Then $\cF:=S \triangleleft \cF'$ is a $\kappa$-well separated $F$-packing in $G$ and $\set{e\in \cF^{(r)}}{S\In e}=S\uplus L$.
\end{prop}

In particular, if $L=G(S)^{(r-i)}$, i.e.~if $\cF'$ is a $\kappa$-well separated $F(T)$-decomposition of $G(S)$, then $\cF$ is a $\kappa$-well separated $F$-packing in $G$ which covers all $r$-edges of $G$ that contain~$S$.

\proof
We first check that $\cF$ is an $F$-packing in $G$. Let $f:=|V(F)|$. For each $F'\in \cF'$, we have $V(F')\in G(S)[L]^{(f-i)}\In G(S)^{(f-i)}$. Hence, $V(F'_{\triangleleft})\in G^{(f)}$. In particular, $G^{(r)}[V(F'_{\triangleleft})]$ is a clique and thus $F'_{\triangleleft}$ is a subgraph of $G^{(r)}$. Suppose, for a contradiction, that for distinct $F',F''\in \cF'$, $F'_{\triangleleft}$ and $F''_{\triangleleft}$ both contain $e\in G^{(r)}$. By \ref{separatedness:1} we have that $|V(F')\cap V(F'')|\le r-i$, and thus we must have $e=S\cup (V(F')\cap V(F''))$. Since $V(F')\cap V(F'')\in G(S)[L]$, we have $e\sm S\in G(S)[L]^{(r-i)}$, and thus $e\sm S$ belongs to at most one of $F'$ and $F''$. Without loss of generality, assume that $e\sm S\notin F'$. Then we have $e\sm S\notin F'_{\triangleleft}(S)$ and thus $e\notin F'_{\triangleleft}$, a contradiction. Thus, $\cF$ is an $F$-packing in $G$.

We next show that $\cF$ is $\kappa$-well separated. Clearly, for distinct $F',F''\in \cF'$, we have $|V(F'_{\triangleleft})\cap V(F''_{\triangleleft})|\le r-i+|S|=r$, so \ref{separatedness:1} holds. To check \ref{separatedness:2}, consider $e\in \binom{V(G)}{r}$. Let $e'$ be an $(r-i)$-subset of $e\sm S$. By definition of $\cF$, we have that the number of $F'_{\triangleleft}\in \cF$ with $e\In V(F'_{\triangleleft})$ is at most the number of $F'\in \cF'$ with $e'\In V(F')$, where the latter is at most $\kappa$ since $\cF'$ is $\kappa$-well separated.

Finally, we check that $\set{e\in \cF^{(r)}}{S\In e}=S\uplus L$. Let $e$ be any $r$-set with $S\In e$. By Definition~\ref{def:link extension}, we have $e\in \cF^{(r)}$ if and only if $e\sm S\in \cF'^{(r-i)}$. Since $\cF'$ is an $F(T)$-decomposition of $G(S)[L]^{(r-i)}=L$, we have $e\sm S\in \cF'^{(r-i)}$ if and only if $e\sm S\in L$. Thus, $e\in \cF^{(r)}$ if and only if $e\in S\uplus L$.
\endproof

\section{Absorbers}\label{sec:absorbers}

In this section we show that for any (divisible) $r$-graph $H$ in a supercomplex $G$, we can find an `exclusive' absorber $r$-graph $A$ (as discussed in Section~\ref{subsec:outline}, one may think of $H$ as a potential leftover from an approximate $F$-decomposition and $A$ will be set aside earlier to absorb $H$ into an $F$-decomposition). The following definition makes this precise. The main result of this section is Lemma~\ref{lem:absorbing lemma}, which constructs an absorber provided that $F$ is weakly regular.

\begin{defin}[Absorber]
Let $F$, $H$ and $A$ be $r$-graphs. We say that $A$ is an \defn{$F$-absorber for $H$} if $A$ and $H$ are edge-disjoint and both $A$ and $A\cup H$ have an $F$-decomposition. More generally, if $G$ is a complex and $H\In G^{(r)}$, then $A\In G^{(r)}$ is a \defn{$\kappa$-well separated $F$-absorber for $H$ in $G$} if $A$ and $H$ are edge-disjoint and there exist $\kappa$-well separated $F$-packings $\cF_{\circ}$ and $\cF_{\bullet}$ in $G$ such that $\cF_{\circ}^{(r)}=A$ and $\cF_{\bullet}^{(r)}=A\cup H$.
\end{defin}

\begin{lemma}[Absorbing lemma]\label{lem:absorbing lemma}
Let $1/n\ll 1/\kappa \ll \gamma, 1/h , \eps\ll \xi,1/f$ and $r\in[f-1]$. Assume that \ind{i} is true for all $i\in[r-1]$. Let $F$ be a weakly regular $r$-graph on $f$ vertices, let $G$ be an $(\eps,\xi,f,r)$-supercomplex on $n$ vertices and let $H$ be an $F$-divisible subgraph of $G^{(r)}$ with $|H|\le h$. Then there exists a $\kappa$-well separated $F$-absorber $A$ for $H$ in $G$ with $\Delta(A)\le \gamma n$.
\end{lemma}

We now briefly discuss the case $r=1$. For the case $F=\krq{1}{f}$, a construction of an $F$-absorber for any $F$-divisible $r$-graph $H$ in a supercomplex $G$ is given in \cite{GKLO}. It is easy to see that this absorber is $1$-well separated. Essentially the same construction also works if $F$ contains some isolated vertices.
Thus, for the remainder of this section, we will assume that $r\ge 2$.

The building blocks of our absorbers will be so-called `transformers', first introduced in~\cite{BKLO}.
Roughly speaking, a transformer $T$ can be viewed as transforming a given leftover graph $H$ into a new leftover $H'$ (where we set aside $T$ and $H'$ earlier).

\begin{defin}[Transformer]
Let $F$ be an $r$-graph, $G$ a complex and assume that $H,H'\In G^{(r)}$. A subgraph $T\In G^{(r)}$ is a \defn{$\kappa$-well separated $(H,H';F)$-transformer in $G$} if $T$ is edge-disjoint from both $H$ and $H'$ and there exist $\kappa$-well separated $F$-packings $\cF$ and $\cF'$ in $G$ such that $\cF^{(r)}=T\cup H$ and $\cF'^{(r)}=T\cup H'$.
\end{defin}

Our `Transforming lemma' (Lemma~\ref{lem:transformer}) guarantees the existence of a transformer for $H$ and $H'$ if $H'$ is \emph{obtained from $H$ by identifying vertices} (modulo deleting some isolated vertices from $H'$).
To make this more precise, given a multi-$r$-graph $H$ and $x,x'\in V(H)$, we say that $x$ and $x'$ are \defn{identifiable} if $|H(\Set{x,x'})|=0$, that is, if identifying $x$ and $x'$ does not create an edge of size less than $r$.
For multi-$r$-graphs $H$ and $H'$, we write $H\doublesquig H'$ if there is a sequence $H_0,\dots,H_t$ of multi-$r$-graphs such that $H_0\cong H$, $H_t$ is obtained from $H'$ by deleting isolated vertices, and for every $i\in[t]$, there are two identifiable vertices $x,x'\in V(H_{i-1})$ such that $H_i$ is obtained from $H_{i-1}$ by identifying $x$ and $x'$.

If $H$ and $H'$ are (simple) $r$-graphs and $H\doublesquig H'$, we just write $H\rightsquigarrow H'$ to indicate the fact that during the identification steps, only vertices $x,x'\in V(H_{i-1})$ with $H_{i-1}(\Set{x})\cap H_{i-1}(\Set{x'})=\emptyset$ were identified (i.e.~if we did not create multiple edges).

Clearly, $\doublesquig$ is a reflexive and transitive relation on the class of multi-$r$-graphs, and $\rightsquigarrow$ is a reflexive and transitive relation on the class of $r$-graphs.

It is easy to see that $H\rightsquigarrow H'$ if and only if there is an \emph{edge-bijective homomorphism} from $H$ to $H'$ (see Proposition~\ref{prop:simple identification facts}\ref{fact:identification equivalence simple}). Given $r$-graphs $H,H'$, a \defn{homomorphism from $H$ to $H'$} is a map $\phi\colon V(H)\to V(H')$ such that $\phi(e)\in H'$ for all $e\in H$. Note that this implies that $\phi{\restriction_{e}}$ is injective for all $e\in H$.
We let $\phi(H)$ denote the subgraph of $H'$ with vertex set $\phi(V(H))$ and edge set $\set{\phi(e)}{e\in H}$. We say that $\phi$ is \defn{edge-bijective} if $|H|=|\phi(H)|=|H'|$. For two $r$-graphs $H$ and $H'$, we write $H\overset{\phi}{\rightsquigarrow} H'$ if $\phi$ is an edge-bijective homomorphism from $H$ to $H'$.

We now record a few simple observations about the relation $\rightsquigarrow$ for future reference.

\begin{prop}\label{prop:simple identification facts}
The following hold.
\begin{enumerate}[label=\rm{(\roman*)}]
\item $H\rightsquigarrow H'$ if and only if there exists $\phi$ such that $H\overset{\phi}{\rightsquigarrow} H'$.\label{fact:identification equivalence simple}
\item Let $H_1,H_1',\dots,H_t,H_t'$ be $r$-graphs such that $H_1,\dots,H_t$ are vertex-disjoint and $H_1',\dots,H_t'$ are edge-disjoint and $H_i\cong H_i'$ for all $i\in[t]$. Then $$H_1+\dots+H_t \rightsquigarrow H_1'\cupdot \cdots \cupdot H_t'.$$\label{fact:disjoint identification}
\item If $H\rightsquigarrow H'$ and $H$ is $F$-divisible, then $H'$ is $F$-divisible.\label{fact:identification divisible}\COMMENT{\proof Let $H\overset{\phi}{\rightsquigarrow} H'$. Let $i\in[r-1]_0$ and $S'\in \binom{V(H')}{i}$. Let $\cS$ consist of all $S\in \binom{V(H)}{i}$ such that $\phi(S)=S'$. Then $|H'(S')|=\sum_{S\in \cS}|H(S)|\equiv 0\mod{Deg(F)_i}$.
\endproof}
\end{enumerate}
\end{prop}

\subsection{Transformers}\label{subsec:transformers}

The following lemma guarantees the existence of a transformer from $H$ to $H'$ if $F$ is weakly regular and $H \rightsquigarrow H'$. The proof relies inductively on the assertion of the main complex decomposition theorem (Theorem~\ref{thm:main complex}).

\begin{lemma}[Transforming lemma]\label{lem:transformer}
Let $1/n\ll 1/\kappa \ll \gamma , 1/h, \eps \ll \xi,1/f$ and $2\le r<f$. Assume that \ind{i} is true for all $i\in[r-1]$. Let $F$ be a weakly regular $r$-graph on $f$ vertices, let $G$ be an $(\eps,\xi,f,r)$-supercomplex on $n$ vertices and let $H,H'$ be vertex-disjoint $F$-divisible subgraphs of $G^{(r)}$ of order at most $h$ and such that $H\rightsquigarrow H'$. Then there exists a $\kappa$-well separated $(H,H';F)$-transformer $T$ in $G$ with $\Delta(T)\le \gamma n$.
\end{lemma}

A key operation in the proof of Lemma~\ref{lem:transformer} is the ability to find `localised transformers'. Let $i\in[r-1]$ and let $S\In V(H)$, $S' \In V(H')$ and $S^\ast\In V(F)$ be sets of size $i$. For an $(r-i)$-graph $L$ in the link graph of both $S$ and $S'$, we can view an $F(S^\ast)$-decomposition $\cF_L$ of $L$ (which exists by \ind{r-i}\COMMENT{if $L$ is $F(S^\ast)$-divisible and a supercomplex}) as a localised transformer between $S\uplus L$ and $S'\uplus L$. Indeed, similarly to the situation described in Sections~\ref{subsec:main thm} and~\ref{subsec:cover down}, we can extend $\cF_L$ `by adding $S$ back' to obtain an $F$-packing $\cF$ which covers all edges of $S\uplus L$. By `mirroring' this extension, we can also obtain an $F$-packing $\cF'$ which covers all edges of $S'\uplus L$ (see Definition~\ref{def:link extension two sided} and Proposition~\ref{prop:S cover two sided}). 
To make this more precise, we introduce the following notation.

\begin{defin}
Let $V$ be a set and let $V_1,V_2$ be disjoint subsets of $V$ having equal size. Let $\phi\colon V_1\to V_2$ be a bijection. For a set $S\In V\sm V_2$, define $\phi(S):=(S\sm V_1)\cup \phi(S\cap V_1)$. Moreover, for an $r$-graph $R$ with $V(R)\In V\sm V_2$, we let $\phi(R)$ be the $r$-graph on $\phi(V(R))$ with edge set $\set{\phi(e)}{e\in R}$.
\end{defin}

The following facts are easy to see.

\begin{fact}\label{fact:simple proj}
Suppose that $V$, $V_1$, $V_2$ and $\phi$ are as above. Then the following hold for every $r$-graph $R$ with $V(R)\In V\sm V_2$:
\begin{enumerate}[label=\rm{(\roman*)}]
\item $\phi(R)\cong R$;
\item if $R=R_1\cupdot \dots \cupdot R_k$, then $\phi(R)=\phi(R_1)\cupdot \dots \cupdot \phi(R_k)$ and thus $\phi(R_1)=\phi(R)-\phi(R_2)-\dots-\phi(R_k)$. \label{simple proj:union}
\end{enumerate}
\end{fact}

The following definition is a two-sided version of Definition~\ref{def:link extension}.

\begin{defin}\label{def:link extension two sided}
Let $F$ be an $r$-graph, $i\in[r-1]$ and assume that $S^\ast\in \binom{V(F)}{i}$ is such that $F(S^\ast)$ is non-empty. Let $G$ be a complex and assume that $S_1,S_2\in \binom{V(G)}{i}$ are disjoint and that a bijection $\phi\colon S_1\to S_2$ is given. Suppose that $\cF'$ is a well separated $F(S^\ast)$-packing in $G(S_1)\cap G(S_2)$. We then define $S_1 \triangleleft \cF' \triangleright S_2$ as follows: For each $F'\in \cF'$ and $j\in\Set{1,2}$, let $F'_{j}$ be a copy of $F$ on vertex set $S_j\cup V(F')$ such that $F'_{j}(S_j)=F'$ and such that $\phi(F_1')=F_2'$. Let
\begin{align*}
\cF_1:=\set{F_1'}{F'\in \cF'};\\
\cF_2:=\set{F_2'}{F'\in \cF'};\\
S_1 \triangleleft \cF' \triangleright S_2:=(\cF_1,\cF_2).
\end{align*}
\end{defin}

The next proposition is proved using its one-sided counterpart, Proposition~\ref{prop:S cover}. As in Proposition~\ref{prop:S cover}, the notion of well separatedness (Definition~\ref{def:well separated}) is crucial here.

\begin{prop}\label{prop:S cover two sided}
Let $F$, $r$, $i$, $S^\ast$, $G$, $S_1$, $S_2$ and $\phi$ be as in Definition~\ref{def:link extension two sided}. Suppose that $L\In G(S_1)^{(r-i)}\cap G(S_2)^{(r-i)}$ and that $\cF'$ is a $\kappa$-well separated $F(S^\ast)$-decomposition of $(G(S_1)\cap G(S_2))[L]$. Then the following holds for $(\cF_1,\cF_2)=S_1 \triangleleft \cF' \triangleright S_2$:
\begin{enumerate}[label=\rm{(\roman*)}]
\item for $j\in[2]$, $\cF_j$ is a $\kappa$-well separated $F$-packing in $G$ with $\set{e\in \cF_j^{(r)}}{S_j\In e}=S_j\uplus L$;\label{prop:S cover two sided:packing}
\item $V(\cF_1^{(r)})\In V(G)\sm S_2$ and $\phi(\cF_1^{(r)})=\cF_2^{(r)}$.\label{prop:S cover two sided:mirror}
\end{enumerate}
\end{prop}

\proof
Let $j\in[2]$. Since $(G(S_1)\cap G(S_2))[L]\In G(S_j)$, we can view $\cF_j$ as $S_j\triangleleft \cF'$ (cf.~Definition~\ref{def:link extension}). Moreover, since $(G(S_1)\cap G(S_2))[L]^{(r-i)}=L=G(S_j)[L]^{(r-i)}$, we can conclude that $\cF'$ is a $\kappa$-well separated $F(S^\ast)$-decomposition of $G(S_j)[L]$. Thus, by Proposition~\ref{prop:S cover}, $\cF_j$ is a $\kappa$-well separated $F$-packing in $G$ with $\set{e\in \cF_j^{(r)}}{S_j\In e}=S_j\uplus L$.

Moreover, we have $V(\cF_1^{(r)})\In \bigcup_{F'\in \cF' }V(F'_1)\In V(G)\sm S_2$ and by Fact~\ref{fact:simple proj}\ref{simple proj:union}$$\phi(\cF_1^{(r)})=\phi(\mathop{\dot{\bigcup}}_{F'\in \cF' }F'_1)=\mathop{\dot{\bigcup}}_{F'\in \cF' }\phi(F'_1)=\mathop{\dot{\bigcup}}_{F'\in \cF' }F'_2=\cF_2^{(r)}.$$
\endproof

We now sketch the proof of Lemma~\ref{lem:transformer}. Given Proposition~\ref{prop:S cover two sided}, the details are very similar to the proof of Lemma~8.5 in~\cite{GKLO} and thus omitted here. The full proof of Lemma~\ref{lem:transformer} can be found in Appendix~\ref{app:transformers}.

Suppose for simplicity that $H'$ is simply a copy of $H$, i.e.~$H'=\phi(H)$ where $\phi$ is an isomorphism from $H$ to $H'$. We aim to construct an $(H,H';F)$-transformer. In a first step, for every edge $e\in H$, we introduce a set $X_e$ of $|V(F)|-r$ new vertices and let $F_e$ be a copy of $F$ such that $V(F_e)=e\cup X_e$ and $e\in F_e$. Let $T_1:=\bigcup_{e\in H}F_e[X_e]$ and $R_1:=\bigcup_{e\in H}F_e-T_1-H$. Clearly, $\set{F_e}{e\in H}$ is an $F$-decomposition of $H\cup R_1 \cup T_1$. By Fact~\ref{fact:simple proj}\ref{simple proj:union}, we also have that $\set{\phi(F_e)}{e\in H}$ is an $F$-decomposition of $H'\cup \phi(R_1) \cup T_1$. Hence, $T_1$ is an $(H\cup R_1,H'\cup \phi(R_1);F)$-transformer.
Note that at this stage, it would suffice to find an $(R_1,\phi(R_1);F)$-transformer $T_1'$, as then $T_1\cup T_1'\cup R_1 \cup \phi(R_1)$ would be an $(H,H';F)$-transformer. The crucial difference now to the original problem is that every edge of $R_1$ contains at most $r-1$ vertices from $V(H)$. On the other hand, every edge in $R_1$ contains at least one vertex in $V(H)$ as otherwise it would belong to $T_1$.
We view this as Step~$1$ and will now proceed inductively. After Step~$i$, we will have an $r$-graph $R_i$ and an $(H\cup R_i,H'\cup \phi(R_i);F)$-transformer $T_i$ such that every edge $e\in R_i$ satisfies $1\le |e\cap V(H)|\le r-i$. Thus, after Step~$r$ we can terminate the process as $R_r$ must be empty and thus $T_r$ is an $(H,H';F)$-transformer.

In Step~$i+1$, where $i\in [r-1]$, we use \ind{i} inductively as follows. Let $R_i'$ consist of all edges of $R_i$ which intersect $V(H)$ in $r-i$ vertices.
We decompose $R_i'$ into `local' parts.
For every edge $e\in R_i'$, there exists a unique set $S\in \binom{V(H)}{r-i}$ such that $S\In e$. For each $S\in \binom{V(H)}{r-i}$, let $L_S:=R_i'(S)$. Note that the `local' parts $S\uplus L_S$ form a decomposition of $R_i'$. The problem of finding $R_{i+1}$ and $T_{i+1}$ can be reduced to finding a `localised transformer' between $S\uplus L_S$ and $\phi(S)\uplus L_S$ for every $S$, as described above.
At this stage, by Proposition~\ref{prop:link divisibility}, $L_S$ will automatically be $F(S^\ast)$-divisible, where $S^\ast\in \binom{V(F)}{r-i}$ is such that $F(S^\ast)$ is non-empty. If we were given an $F(S^\ast)$-decomposition $\cF_S'$ of $L_S$, we could use Proposition~\ref{prop:S cover two sided} to extend $\cF_S'$ to an $F$-packing $\cF_S$ which covers all edges of $S\uplus L_S$, and all new edges created by this extension intersect $S$ (and $V(H)$) in at most $r-i-1$ vertices, as desired. It is possible to combine these localised transformers with $T_i$ and $R_i$ in such a way that we obtain $T_{i+1}$ and $R_{i+1}$.

Unfortunately, $(G(S)\cap G(\phi(S)))[L_S]$ might not be a supercomplex (one can think of $L_S$ as some leftover from previous steps) and so $\cF_S'$ may not exist. However, by Proposition~\ref{prop:hereditary}, we have that $G(S)\cap G(\phi(S))$ is a supercomplex.\COMMENT{(This step is the reason why we use supercomplexes instead of just $()$-complexes)} Thus we can (randomly) choose a suitable $i$-subgraph $A_S$ of $(G(S)\cap G(\phi(S)))^{(i)}$ such that $A_S$ is $F(S^\ast)$-divisible and edge-disjoint from $L_S$. Instead of building a localised transformer for $L_S$ directly, we will now build one for $A_S$ and one for $A_S\cup L_S$, using \ind{i} both times to find the desired $F(S^\ast)$-decomposition. These can then be combined into a localised transformer for~$L_S$.

\subsection{Canonical multi-\texorpdfstring{$r$}{r}-graphs}\label{subsec:canonical}

Roughly speaking, the aim of this section is to show that any $F$-divisible $r$-graph $H$ can be transformed into a canonical multigraph $M_h$ which does not depend on the structure of $H$. However, it turns out that for this we need to move to a `dual' setting, where we consider $\nabla H$ which is obtained from $H$ by applying an $F$-extension operator $\nabla$. This operator allows us to switch between multi-$r$-graphs (which arise naturally in the construction but are not present in the complex $G$ we are decomposing) and (simple) $r$-graphs (see e.g.~Fact~\ref{fact:id after ext}).

Given a multi-$r$-graph $H$ and a set $X$ of size $r$, we say that $\psi$ is an \defn{$X$-orientation of $H$} if $\psi$ is a collection of bijective maps $\psi_e\colon X\to e$, one for each $e\in H$. (For $r=2$ and $X=\Set{1,2}$, say, this coincides with the notion of an oriented multigraph, e.g.~by viewing $\psi_e(1)$ as the tail and $\psi_e(2)$ as the head of~$e$, where parallel edges can be oriented in opposite directions.)

Given an $r$-graph $F$ and a distinguished edge $e_0\in F$, we introduce the following `extension' operators $\tilde{\nabla}_{(F,e_0)}$ and $\nabla_{(F,e_0)}$.

\begin{defin}[Extension operators $\tilde{\nabla}$ and $\nabla$]\label{def:extensions}
 Given a (multi-)$r$-graph $H$ with an $e_0$-orientation $\psi$, let $\tilde{\nabla}_{(F,e_0)}(H,\psi)$ be obtained from $H$ by extending every edge of $H$ into a copy of $F$, with $e_0$ being the rooted edge. More precisely, let $Z_e$ be vertex sets of size $|V(F)\sm e_0|$ such that $Z_e\cap Z_{e'}=\emptyset$ for all distinct (but possibly parallel) $e,e'\in H$ and $V(H)\cap Z_e=\emptyset$ for all $e\in H$. For each $e\in H$, let $F_e$ be a copy of $F$ on vertex set $e\cup Z_e$ such that $\psi_e(v)$ plays the role of $v$ for all $v\in e_0$ and $Z_{e}$ plays the role of $V(F)\sm e_0$. Then $\tilde{\nabla}_{(F,e_0)}(H,\psi):=\bigcup_{e\in H}F_e$.
Let $\nabla_{(F,e_0)}(H,\psi):=\tilde{\nabla}_{(F,e_0)}(H,\psi)-H$. 
\end{defin}

Note that $\nabla_{(F,e_0)}(H,\psi)$ is a (simple) $r$-graph even if $H$ is a multi-$r$-graph.
If $F$, $e_0$ and $\psi$ are clear from the context, or if we only want to motivate an argument before giving the formal proof, we just write $\tilde{\nabla}H$ and $\nabla H$.

\begin{fact}\label{fact:extensions}
Let $F$ be an $r$-graph and $e_0\in F$. Let $H$ be a multi-$r$-graph and let $\psi$ be any $e_0$-orientation of $H$. Then the following hold:
\begin{enumerate}[label=\rm{(\roman*)}]
\item $\tilde{\nabla}_{(F,e_0)}(H,\psi)$ is $F$-decomposable;\label{fact:extensions:decomposable}
\item $\nabla_{(F,e_0)}(H,\psi)$ is $F$-divisible if and only if $H$ is $F$-divisible.\label{fact:extensions:divisible}\COMMENT{Since $H\cup \nabla_{(F,e_0)}(H,\psi)=\tilde{\nabla}_{(F,e_0)}(H,\psi)$ is $F$-divisible}
\end{enumerate}
\end{fact}

The goal of this subsection is to show that for every $h\in \bN$, there is a multi-$r$-graph $M_h$ such that for any $F$-divisible $r$-graph $H$ on at most $h$ vertices, we have
\begin{align}
\nabla(\nabla(H+t\cdot F)+s\cdot F)  &\rightsquigarrow \nabla M_h \label{identification formula}
\end{align}
for suitable $s,t\in \bN$. The multigraph $M_h$ is \defn{canonical} in the sense that it does not depend on $H$, but only on $h$. The benefit is, very roughly speaking, that it allows us to transform any given leftover $r$-graph $H$ into the empty $r$-graph, which is trivially decomposable, and this will enable us to construct an absorber for $H$. Indeed, to see that \eqref{identification formula} allows us to transform $H$ into the empty $r$-graph, let $$H':=\nabla(\nabla(H+t\cdot F)+s\cdot F)=\nabla\nabla H + t\cdot \nabla\nabla F + s\cdot \nabla F$$ and observe that the $r$-graph $T:=\nabla H + t\cdot \tilde{\nabla} F + s\cdot F$ `between' $H$ and $H'$ can be chosen\COMMENT{subject to orientations and everything} in such a way that
\begin{align*}
T\cup H &=\tilde{\nabla} H + t\cdot \tilde{\nabla} F +s\cdot F,\\
T\cup H' &=\tilde{\nabla} (\nabla H) + t\cdot (\tilde{\nabla} (\nabla F) \cupdot  F) + s\cdot\tilde{\nabla} F,
\end{align*}
i.e.~$T$ is an $(H,H';F)$-transformer (cf.~Fact~\ref{fact:extensions}\ref{fact:extensions:decomposable}).
Hence, together with \eqref{identification formula} and Lemma~\ref{lem:transformer}, this means that we can transform $H$ into $\nabla M_h$. Since $M_h$ does not depend on $H$, we can also transform the empty $r$-graph into $\nabla M_h$, and by transitivity we can transform $H$ into the empty graph, which amounts to an absorber for $H$ (the detailed proof of this can be found in Section~\ref{subsec:prove absorbing lemma}).

We now give the rigorous statement of~\eqref{identification formula}, which is the main lemma of this subsection.

\begin{lemma}\label{lem:identifaction}
Let $r\ge 2$ and assume that \ind{i} is true for all $i\in[r-1]$. Let $F$ be a weakly regular $r$-graph and $e_0\in F$. Then for all $h\in \bN$, there exists a multi-$r$-graph $M_h$ such that for any $F$-divisible $r$-graph $H$ on at most $h$ vertices, we have
\begin{align*}
\nabla_{(F,e_0)}(\nabla_{(F,e_0)}(H+t\cdot F,\psi_1)+s\cdot F,\psi_3)  &\rightsquigarrow \nabla_{(F,e_0)} (M_h,\psi_2)
\end{align*}
for suitable $s,t\in \bN$, where $\psi_1$ and $\psi_2$ can be arbitrary $e_0$-orientations of $H+t\cdot F$ and $M_h$, respectively, and $\psi_3$ is an $e_0$-orientation depending on these.
\end{lemma}

The above graphs $\nabla(\nabla(H+t\cdot F)+s\cdot F)$ and $\nabla M_h$ will be part of our $F$-absorber for~$H$. We therefore need to make sure that we can actually find them in a supercomplex~$G$. This requirement is formalised by the following definition.

\begin{defin}\label{def:ext in complex}
Let $G$ be a complex, $X\In V(G)$, $F$ an $r$-graph with $f:=|V(F)|$ and $e_0\in F$. Suppose that $H\In G^{(r)}$ and that $\psi$ is an $e_0$-orientation of $H$. By \defn{extending $H$ with a copy of $\nabla_{(F,e_0)}(H,\psi)$ in $G$ (whilst avoiding $X$)} we mean the following: for each $e\in H$, let $Z_e\in G^{(f)}(e)$ be such that $Z_e\cap (V(H)\cup X)=\emptyset$ for every $e\in H$ and $Z_e\cap Z_{e'}=\emptyset$ for all distinct $e,e'\in H$. For each $e\in H$, let $F_e$ be a copy of $F$ on vertex set $e\cup Z_e$ (so $F_e\In G^{(r)}$) such that $\psi_e(v)$ plays the role of $v$ for all $v\in e_0$ and $Z_{e}$ plays the role of $V(F)\sm e_0$. Let $H^\nabla:=\bigcup_{e\in H}F_e-H$ and $\cF:=\set{F_e}{e\in H}$ be the output of this.
\end{defin}

For our purposes, the set $|V(H)\cup X|$ will have a small bounded size compared to $|V(G)|$. Thus, if the $G^{(f)}(e)$ are large enough (which is the case e.g.~in an $(\eps,\xi,f,r)$-supercomplex), then the above extension can be carried out simply by picking the sets $Z_e$ one by one.

\begin{fact}\label{fact:ext in complex}
Let $(H^{\nabla},\cF)$ be obtained by extending $H\In G^{(r)}$ with a copy of $\nabla_{(F,e_0)}(H,\psi)$ in $G$. Then $H^\nabla\In G^{(r)}$ is a copy of $\nabla_{(F,e_0)}(H,\psi)$ and $\cF$ is a $1$-well separated $F$-packing in $G$ with $\cF^{(r)}=H\cup H^{\nabla}$ such that for all $F'\in \cF$, $|V(F')\cap V(H)|\le r$.
\end{fact}

For a partition $\cP=\Set{V_{x}}_{x\in X}$ whose classes are indexed by a set $X$, we define $V_{Y}:=\bigcup_{x\in Y}V_x$ for every subset $Y\In X$. Recall that for a multi-$r$-graph $H$ and $e\in \binom{V(H)}{r}$, $|H(e)|$ denotes the multiplicity of $e$ in $H$.
For multi-$r$-graphs $H,H'$, we write $H\overset{\cP}{\doublesquig} H'$ if $\cP=\Set{V_{x'}}_{x'\in V(H')}$ is a partition of $V(H)$ such that
\begin{enumerate}[label=\rm{(I\arabic*)}]
\item for all $x'\in V(H')$ and $e\in H$, $|V_{x'}\cap e|\le 1$;\label{identification:independent}\COMMENT{i.e.~each $V_{x'}$ is a strongly independent set in $H$}
\item for all $e'\in\binom{V(H')}{r}$, $\sum_{e\in \binom{V_{e'}}{r}}|H(e)|=|H'(e')|$.\label{identification:bijective}
\end{enumerate}
Given $\cP$, define $\phi_{\cP}\colon V(H)\to V(H')$ as $\phi_{\cP}(x):=x'$ where $x'$ is the unique $x'\in V(H')$ such that $x\in V_{x'}$. Note that by \ref{identification:independent}, we have $|\set{\phi_{\cP}(x)}{x\in e}|=r$ for all $e\in H$. Further, by \ref{identification:bijective}, there exists a bijection $\Phi_{\cP}\colon H\to H'$ between the multi-edge-sets of $H$ and $H'$ such that for every edge $e\in H$, the image $\Phi_{\cP}(e)$ is an edge consisting of the vertices $\phi_{\cP}(x)$ for all $x\in e$.
It is easy to see that $H\doublesquig H'$ if and only if there is some $\cP$ such that $H\overset{\cP}{\doublesquig} H'$.

The extension operator $\nabla$ is well behaved with respect to the identification relation $\doublesquig$ in the following sense: if $H\doublesquig H'$, then $\nabla H\rightsquigarrow \nabla H'$. More precisely, let $H$ and $H'$ be multi-$r$-graphs and suppose that $H\overset{\cP}{\doublesquig} H'$. Let $\phi_{\cP}$ and $\Phi_{\cP}$ be defined as above. Let $F$ be an $r$-graph and $e_0 \in F$. For any $e_0$-orientation $\psi'$ of $H'$, we define an $e_0$-orientation $\psi$ of $H$ \defn{induced by $\psi'$} as follows:
for every $e\in H$, let $e':=\Phi_{\cP}(e)$ be the image of $e$ with respect to $\overset{\cP}{\doublesquig}$. We have that $\phi_{\cP}{\restriction_{e}}\colon e\to e'$ is a bijection. We now define the bijection $\psi_{e}\colon e_0\to e$ as $\psi_e:=\phi_{\cP}{\restriction_{e}}^{-1}\circ \psi'_{e'}$, where $\psi'_{e'}\colon e_0\to e'$. Thus, the collection $\psi$ of all $\psi_e$, $e\in H$, is an $e_0$-orientation of $H$. It is easy to see that $\psi$ satisfies the following.

\begin{fact}\label{fact:id after ext}
Let $F$ be an $r$-graph and $e_0\in F$. Let $H,H'$ be multi-$r$-graphs and suppose that $H\doublesquig H'$. Then for any $e_0$-orientation $\psi'$ of $H'$, we have $\nabla_{(F,e_0)}(H,\psi) \rightsquigarrow \nabla_{(F,e_0)}(H',\psi')$, where $\psi$ is induced by $\psi'$.
\end{fact}

We now define the multi-$r$-graphs which will serve as the canonical multi-$r$-graphs $M_h$ in \eqref{identification formula}. For $r\in \bN$, let $\cM_r$ contain all pairs $(k,m)\in \bN_0^2$ such that $\frac{m}{r-i}\binom{k-i}{r-1-i}$ is an integer for all $i\in[r-1]_0$.

\begin{defin}[Canonical multi-$r$-graph]\label{def:loop graph}
Let $F^\ast$ be an $r$-graph and $e^\ast\in F^\ast$. Let $V':=V(F^\ast)\sm e^\ast$. If $(k,m)\in \cM_r$, define the multi-$r$-graph $M^{(F^\ast,e^\ast)}_{k,m}$ on vertex set $[k]\cupdot V'$ such that for every $e\in\binom{[k]\cup V'}{r}$, the multiplicity of $e$ is
$$|M^{(F^\ast,e^\ast)}_{k,m}(e)|=\begin{cases}
0 &\mbox{if } e\In [k]; \\
\frac{m}{r-|e\cap [k]|}\binom{k-|e\cap [k]|}{r-1-|e\cap [k]|} &\mbox{if }|e\cap [k]|>0,|e\cap V'|>0; \\
0 &\mbox{if }e\In V',e\notin F^\ast;\\
\frac{m}{r}\binom{k}{r-1} &\mbox{if }e\In V',e\in F^\ast.
\end{cases}$$
\end{defin}

We will require the graph $F^\ast$ in Definition~\ref{def:loop graph} to have a certain symmetry property with respect to $e^\ast$, which we now define. We will prove the existence of a suitable ($F$-decomposable) symmetric $r$-extender in Lemma~\ref{lem:extenders}.

\begin{defin}[symmetric $r$-extender]\label{def:extenders}
We say that $(F^\ast,e^\ast)$ is a \defn{symmetric $r$-extender} if $F^\ast$ is an $r$-graph, $e^\ast\in F^\ast$ and the following holds:
\begin{enumerate}[label=(SE)]
\item for all $e'\in \binom{V(F^\ast)}{r}$ with $e'\cap e^\ast\neq\emptyset$, we have $e'\in F^\ast$.\label{extender3}
\end{enumerate}
\end{defin}

Note that if $(F^\ast,e^\ast)$ is a symmetric $r$-extender, then the operators $\tilde{\nabla}_{(F^\ast,e^\ast)},\nabla_{(F^\ast,e^\ast)}$ are labelling-invariant, i.e.~$\tilde{\nabla}_{(F^\ast,e^\ast)}(H,\psi_1)\cong \tilde{\nabla}_{(F^\ast,e^\ast)}(H,\psi_2)$ and $\nabla_{(F^\ast,e^\ast)}(H,\psi_1)\cong \nabla_{(F^\ast,e^\ast)}(H,\psi_2)$ for all $e^\ast$-orientations $\psi_1,\psi_2$ of a multi-$r$-graph $H$. We therefore simply write $\tilde{\nabla}_{(F^\ast,e^\ast)}H$ and $\nabla_{(F^\ast,e^\ast)}H$ in this case.

To prove Lemma~\ref{lem:identifaction} we introduce so called strong colourings.
Let $H$ be an $r$-graph and $C$ a set. A map $c\colon V(H)\to C$ is a \defn{strong $C$-colouring of $H$} if for all distinct $x,y\in V(H)$ with $|H(\Set{x,y})|>0$, we have $c(x)\neq c(y)$, that is, no colour appears twice in one edge. For $\alpha\in C$, we let $c^{-1}(\alpha)$ denote the set of all vertices coloured $\alpha$. For a set $C'\In C$, we let $c^{\In}(C'):=\set{e\in H}{C'\In c(e)}$. We say that $c$ is \defn{$m$-regular} if $|c^{\In}(C')|=m$ for all $C'\in \binom{C}{r-1}$. For example, an $r$-partite $r$-graph $H$ trivially has a strong $|H|$-regular $[r]$-colouring.

\begin{fact}\label{fact:double count}
Let $H$ be an $r$-graph and let $c$ be a strong $m$-regular $[k]$-colouring of $H$. Then $|c^{\In}(C')|=\frac{m}{r-i}\binom{k-i}{r-1-i}$ for all $i\in[r-1]_0$ and all $C'\in \binom{[k]}{i}$.
\end{fact}

\begin{lemma}\label{lem:unique multigraph}
Let $(F^\ast,e^\ast)$ be a symmetric $r$-extender. Suppose that $H$ is an $r$-graph and suppose that $c$ is a strong $m$-regular $[k]$-colouring of $H$. Then $(k,m)\in \cM_r$ and $$\nabla_{(F^\ast,e^\ast)}H\doublesquig M^{(F^\ast,e^\ast)}_{k,m}.$$
\end{lemma}

\proof
By Fact~\ref{fact:double count}, $(k,m)\in \cM_r$, thus $M^{(F^\ast,e^\ast)}_{k,m}$ is defined. Recall that $M^{(F^\ast,e^\ast)}_{k,m}$ has vertex set $[k]\cup V'$, where $V':=V(F^\ast)\sm e^\ast$.
Let $V(H)\cup \bigcup_{e\in H} Z_{e}$ be the vertex set of $\nabla_{(F^\ast,e^\ast)}H$ as in Definition~\ref{def:extensions}, with $Z_{e}=\set{z_{e,v}}{v\in V'}$. We define a partition $\cP$ of $V(H)\cup \bigcup_{e\in H} Z_{e}$ as follows: for all $i\in[k]$, let $V_i:=c^{-1}(i)$. For all $v\in V'$, let $V_{v}:=\set{z_{e,v}}{e\in H}$. We now claim that $\nabla_{(F^\ast,e^\ast)}H\overset{\cP}{\doublesquig} M^{(F^\ast,e^\ast)}_{k,m}.$

Clearly, $\cP$ satisfies \ref{identification:independent} because $c$ is a strong colouring of $H$.
For a set $e'\in \binom{[k]\cup V'}{r}$, define $$S_{e'}:=\set{e''\in \nabla_{(F^\ast,e^\ast)}H}{e''\In V_{e'}}.$$
Since $\nabla_{(F^\ast,e^\ast)}H$ is simple, in order to check \ref{identification:bijective}, it is enough to show that for all $e'\in \binom{[k]\cup V'}{r}$, we have $|S_{e'}|=|M^{(F^\ast,e^\ast)}_{k,m}(e')|$. We distinguish three cases.

\medskip
\noindent{\emph{Case 1:} $e'\In [k]$}
\medskip

In this case, $|M^{(F^\ast,e^\ast)}_{k,m}(e')|=0$. Since $V_{e'}\In V(H)$ and $(\nabla_{(F^\ast,e^\ast)}H)[V(H)]$ is empty, we have $S_{e'}=\emptyset$, as desired.

\medskip
\noindent{\emph{Case 2:} $e'\In V'$}
\medskip

In this case, $S_{e'}$ consists of all edges of $\nabla_{(F^\ast,e^\ast)}H$ which play the role of $e'$ in $F^\ast_{e}$ for some $e\in H$. Hence, if $e'\notin F^\ast$, then $|S_{e'}|=0$, and if $e'\in F^\ast$, then $|S_{e'}|=|H|$. Fact~\ref{fact:double count} applied with $i=0$ yields $|H|=\frac{m}{r}\binom{k}{r-1}$, as desired.

\medskip
\noindent{\emph{Case 3:} $|e'\cap [k]|>0$ and $|e'\cap V'|>0$}
\medskip

We claim that $|S_{e'}|=|c^{\In}(e'\cap [k])|$.
In order to see this, we define a bijection $\pi\colon c^{\In}(e'\cap [k]) \to S_{e'}$ as follows: for every $e\in H$ with $e'\cap [k]\In c(e)$, define $$\pi(e):=(e\cap c^{-1}(e'\cap [k]))\cup \set{z_{e,v}}{v\in e'\cap V'}.$$ We first show that $\pi(e)\in S_{e'}$. Note that $e\cap {c^{-1}(e'\cap [k])}$ is a subset of $e$ of size $|e'\cap [k]|$ and $\set{z_{e,v}}{v\in e'\cap V'}$ is a subset of $Z_{e}$ of size $|e'\cap V'|$. Hence, $\pi(e)\in \binom{V(F^\ast_{e})}{r}$ and $|\pi(e)\cap e|=|e'\cap [k]|>0$. Thus, by \ref{extender3}, we have $\pi(e)\in F^\ast_{e}\In \nabla_{(F^\ast,e^\ast)}H$. (This is in fact the crucial point where we need \ref{extender3}.)
Moreover, $$\pi(e)\In c^{-1}(e'\cap [k])\cup \set{z_{e,v}}{v\in e'\cap V'}\In V_{e'\cap [k]}\cup V_{e'\cap V'}=V_{e'}.$$ Therefore, $\pi(e)\in S_{e'}$. It is straightforward to see that $\pi$ is injective.\COMMENT{Suppose $\pi(e_1)=\pi(e_2)$. This implies $\set{z_{e_1,v}}{v\in e'\cap V'}=\set{z_{e_2,v}}{v\in e'\cap V'}\neq \emptyset$, which is only possible if $e_1=e_2$.} Finally, for every $e''\in S_{e'}$, we have $e''=\pi(e)$, where $e\in H$ is the unique edge of $H$ with $e''\in F^{\ast}_e$.\COMMENT{$\pi(e)=(e\cap c^{-1}(e'\cap [k]))\cup \set{z_{e,v}}{v\in e'\cap V'}$. Since $e''\In V_{e'}$, we have $e\cap e''\In V_{e'\cap [k]}=c^{-1}(e'\cap [k])$. Moreover, $e''\cap Z_e\In \set{z_{e,v}}{v\in e'\cap V'}$. Thus, $e''=(e''\cap e)\cup (e''\cap Z_e)\In \pi(e)$. Since $|\pi(e)|=r$, the claim follows.} This establishes our claim that $\pi$ is bijective and hence $|S_{e'}|=|c^{\In}(e'\cap [k])|$.
Since $1\le |e'\cap [k]|\le r-1$, Fact~\ref{fact:double count} implies that
$$|S_{e'}|=|c^{\In}(e'\cap [k])|=\frac{m}{r-|e'\cap [k]|}\binom{k-|e'\cap [k]|}{r-1-|e'\cap [k]|}=|M^{(F^\ast,e^\ast)}_{k,m}(e')|,$$ as required.
\endproof

Next, we establish the existence of suitable strong regular colourings.
As a tool we need the following result about decompositions of very dense multi-$r$-graphs (which we will apply with $r-1$ playing the role of $r$). We omit the proof as it is essentially the same as that of Corollary~8.16\COMMENT{Check crossref once accepted} in \cite{GKLO}.

\begin{lemma}\label{lem:multi dec}
Let $r\in \bN$ and assume that \ind{r} is true. Let $1/n\ll 1/h,1/f$ with $f>r$, let $F$ be a weakly regular\COMMENT{makes it nicer to refer to \ind{r}} $r$-graph on $f$ vertices and assume that $\krq{r}{n}$ is $F$-divisible. Let $m\in \bN$. Suppose that $H$ is an $F$-divisible multi-$r$-graph on~$[h]$ with multiplicity at most~$m-1$ and let $K$ be the complete multi-$r$-graph on $[n]$ with multiplicity~$m$. Then $K-H$ has an $F$-decomposition.
\end{lemma}

The next lemma guarantees the existence of a suitable strong regular colouring. For this, we apply Lemma~\ref{lem:multi dec} to the shadow of $F$. For an $r$-graph $F$, define the \defn{shadow $F^{sh}$ of $F$} to be the $(r-1)$-graph on $V(F)$ where an $(r-1)$-set $S$ is an edge if and only if $|F(S)|>0$. We need the following fact.

\begin{fact}\label{fact:shadow}
If $F$ is a weakly $(s_0,\dots,s_{r-1})$-regular $r$-graph, then $F^{sh}$ is a weakly $(s_0',\dots,s_{r-2}')$-regular $(r-1)$-graph, where $s_i':=\frac{r-i}{s_{r-1}}s_i$ for all $i\in[r-2]_0$.
\end{fact}

\proof
Let $i\in[r-2]_0$. For every $T\in \binom{V(F)}{i}$, we have $|F^{sh}(T)|=\frac{r-i}{s_{r-1}}|F(T)|$ since every edge of $F$ which contains $T$ contains $r-i$ edges of $F^{sh}$ which contain $T$, but each such edge of $F^{sh}$ is contained in $s_{r-1}$ such edges of $F$. This implies the claim.
\endproof

\begin{lemma}\label{lem:colouring}
Let $r\ge 2$ and assume that \ind{r-1} holds. Let $F$ be a weakly regular $r$-graph. Then for all $h\in \bN$, there exist $k,m\in \bN$ such that for any $F$-divisible $r$-graph $H$ on at most $h$ vertices, there exists $t\in \bN$ such that $H+t\cdot F$ has a strong $m$-regular $[k]$-colouring.
\end{lemma}

\proof
Let $f:=|V(F)|$ and suppose that $F$ is weakly $(s_0,\dots,s_{r-1})$-regular. Thus, for every $S\in \binom{V(F)}{r-1}$, we have
\begin{align}
|F(S)|=\begin{cases}
s_{r-1} &\mbox{if } S\in F^{sh}; \\
0 &\mbox{otherwise}.
\end{cases}\label{shadow reduction}
\end{align}

By Proposition~\ref{prop:divisible existence}, we can choose $k\in \bN$ such that $1/k\ll 1/h,1/f$ and such that $\krq{r-1}{k}$ is $F^{sh}$-divisible.
Let $G$ be the complete multi-$(r-1)$-graph on $[k]$ with multiplicity $m':=h+1$ and let $m:=s_{r-1} m'$.

Let $H$ be any $F$-divisible $r$-graph on at most $h$ vertices. By adding isolated vertices to $H$ if necessary, we may assume that $V(H)=[h]$. We first define a multi-$(r-1)$-graph $H'$ on $[h]$ as follows: For each $S\in \binom{[h]}{r-1}$, let the multiplicity of $S$ in $H'$ be $|H'(S)|:=|H(S)|$. Clearly, $H'$ has multiplicity at most~$h$. Observe that for each $S\In [h]$ with $|S|\le r-1$, we have
\begin{align}
|H'(S)|=(r-|S|)|H(S)|.\label{multiplicity}
\end{align}
Note that since $H$ is $F$-divisible, we have that $s_{r-1}\mid |H(S)|$ for all $S\in \binom{[h]}{r-1}$. Thus, the multiplicity of each $S\in \binom{[h]}{r-1}$ in $H'$ is divisible by $s_{r-1}$. Let $H''$ be the multi-$(r-1)$-graph on $[h]$ obtained from $H'$ by dividing the multiplicity of each $S\in \binom{[h]}{r-1}$ by $s_{r-1}$. Hence, by~\eqref{multiplicity}, for all $S\In [h]$ with $|S|\le r-1$, we have
\begin{align}
|H''(S)|=\frac{|H'(S)|}{s_{r-1}}=\frac{r-|S|}{s_{r-1}}|H(S)|.\label{multiplicity division}
\end{align}
For each $S\in \binom{[k]}{r-1}$ with $S\not\In [h]$, we set $|H''(S)|:=|H(S)|:=0$. Then \eqref{multiplicity division} still holds.

We claim that $H''$ is $F^{sh}$-divisible. Recall that by Fact~\ref{fact:shadow},
\begin{align*}
F^{sh}\mbox{ is weakly }(\frac{r}{s_{r-1}}s_0,\dots,\frac{r-i}{s_{r-1}}s_i,\dots,\frac{2}{s_{r-1}}s_{r-2})\mbox{-regular.}
\end{align*}
Let $i\in[r-2]_0$ and let $S\in \binom{[h]}{i}$. We need to show that $|H''(S)|\equiv 0 \mod{Deg(F^{sh})_i}$, where $Deg(F^{sh})_i=\frac{r-i}{s_{r-1}}s_i$.
Since $H$ is $F$-divisible, we have $|H(S)|\equiv 0 \mod{s_i}$. Together with \eqref{multiplicity division}, we deduce that $|H''(S)|\equiv 0 \mod{\frac{r-i}{s_{r-1}}s_i}$.
Hence, $H''$ is $F^{sh}$-divisible.
Therefore, by Lemma~\ref{lem:multi dec} (with $k,m',r-1,F^{sh}$ playing the roles of $n,m,r,F$) and our choice of~$k$, $G-H''$ has an $F^{sh}$-decomposition $\cF$ into $t$ edge-disjoint copies $F_1',\dots,F_t'$ of~$F^{sh}$.

We will show that $t$ is as required in Lemma~\ref{lem:colouring}. To do this, let $F_1,\dots,F_t$ be vertex-disjoint copies of $F$ which are also vertex-disjoint from $H$.
We will now define a strong $m$-regular $[k]$-colouring $c$ of $$H^+:=H\cup \bigcup_{j\in[t]}F_j.$$

Let $c_0$ be the identity map on $V(H)=[h]$, and for each $j\in [t]$, let
\begin{align}
c_j\colon V(F_j)\to V(F_j')\mbox{ be an isomorphism from }F_j^{sh}\mbox{ to }F_j'\label{colouring of Fjs}
\end{align}
(recall that $V(F_j^{sh})=V(F_j)$).
Since $H,F_1,\dots,F_t$ are vertex-disjoint and $V(H)\cup \bigcup_{j\in[t]}V(F_j')\In[k]$, we can combine $c_0,c_1,\dots,c_t$ to a map $$c\colon V(H^+)\to [k],$$ i.e.~for $x\in V(H^+)$, we let $c(x):=c_j(x)$, where either $j$ is the unique index for which $x\in V(F_j)$ or $j=0$ if $x\in V(H)$.
For every edge $e\in H^+$, we have $e\In V(H)$ or $e\In V(F_j)$ for some $j\in[t]$, thus $c{\restriction_e}$ is injective. Therefore, $c$ is a strong $[k]$-colouring of~$H^+$.

It remains to check that $c$ is $m$-regular. Let $C\in \binom{[k]}{r-1}$. Clearly, $|c^{\In}(C)|=\sum_{j=0}^t|c_j^{\In}(C)|$. Since every $c_j$ is a bijection, we have
\begin{align*}
|c_0^{\In}(C)|&=|\set{e\in H}{c_0^{-1}(C)\In e}|=|H(c_0^{-1}(C))|=|H(C)|\quad\mbox{and}\\
|c_j^{\In}(C)|
&=|F_j(c_j^{-1}(C))|\overset{\eqref{shadow reduction}}{=}\begin{cases}
s_{r-1} &\mbox{if } c_j^{-1}(C)\in F_j^{sh}\overset{\eqref{colouring of Fjs}}{\Leftrightarrow} C\in F_j'; \\
0 &\mbox{otherwise}.
\end{cases}
\end{align*}
Thus, we have $|c^{\In}(C)|=|H(C)|+s_{r-1}|J(C)|$, where $$J(C):=\set{j\in[t]}{C\in F_j'}.$$
Now crucially, since $\cF$ is an $F^{sh}$-decomposition of $G-H''$, we have that $|J(C)|$ is equal to the multiplicity of $C$ in $G-H''$, i.e. $|J(C)|=m'-|H''(C)|$.
Thus, $$|c^{\In}(C)|=|H(C)|+s_{r-1}|J(C)|\overset{\eqref{multiplicity division}}{=} s_{r-1}(|H''(C)|+|J(C)|)= s_{r-1}m'=m,$$ completing the proof.
\endproof

Before we can prove Lemma~\ref{lem:identifaction}, we need to show the existence of a symmetric $r$-extender $F^\ast$ which is $F$-decomposable. For some $F$ we could actually take $F^\ast=F$ (e.g.~if $F$ is a clique). For general (weakly regular) $r$-graphs $F$, we will use the Cover down lemma (Lemma~\ref{lem:cover down}) to find $F^\ast$. At first sight, appealing to the Cover down lemma may seem rather heavy handed, but a direct construction seems to be quite difficult.

\begin{lemma}\label{lem:extenders}
Let $F$ be a weakly regular $r$-graph, $e_0\in F$ and assume that \ind{i} is true for all $i\in[r-1]$. There exists a symmetric $r$-extender $(F^\ast,e^\ast)$ such that $F^\ast$ has an $F$-decomposition $\cF$ with $e^\ast\in F'\in \cF$ and $e^\ast$ plays the role of $e_0$ in $F'$.
\end{lemma}

\proof
Let $f:=|V(F)|$. By Proposition~\ref{prop:divisible existence}, we can choose $n\in \bN$ and $\gamma,\eps,\nu,\mu>0$ such that $1/n\ll \gamma \ll \eps \ll \nu \ll \mu \ll 1/f$ and such that $\krq{r}{n}$ is $F$-divisible.
By Example~\ref{ex:complete}, $K_n$ is a $(0,0.99/f!,f,r)$-supercomplex. By Fact~\ref{fact:random trivial}\ref{fact:random trivial:itself} and Proposition~\ref{prop:find subset}, there exists $U\In V(K_n)$ of size $\lfloor \mu n\rfloor$ which is $(\eps,\mu,0.9/f!,f,r)$-random in $K_n$.
Let $\bar{U}:=V(K_n)\sm U$. Using \ref{random:binomial} of Definition~\ref{def:regular subset}, it is easy to see that $K_n$ is $(\eps,f,r)$-dense with respect to $K_n^{(r)}-K_n^{(r)}[\bar{U}]$ (see~Definition~\ref{def:dense wrt}). Thus, by the Cover down lemma (Lemma~\ref{lem:cover down}), there exists a subgraph $H^\ast$ of $K_n^{(r)}-K_n^{(r)}[\bar{U}]$ with $\Delta(H^\ast)\le \nu n$ and the following property: for all $L\In K_n^{(r)}$ such that $\Delta(L)\le \gamma n$ and $H^\ast \cup L$ is $F$-divisible, $H^\ast \cup L$ has an $F$-packing which covers all edges except possibly some inside $U$.

Let $F'$ be a copy of $F$ with $V(F')\In \bar{U}$.
Let $G_{nibble}:=K_n-H^\ast-F'$. By Proposition~\ref{prop:noise}\ref{noise:supercomplex}, $G_{nibble}$ is a $(2^{2r+2}\nu,0.8/f!,f,r)$-supercomplex. Thus, by Lemma~\ref{lem:F nibble}, there exists an $F$-packing $\cF_{nibble}$ in $G_{nibble}^{(r)}$ such that $\Delta(L)\le \gamma n$, where $L:=G_{nibble}^{(r)}-\cF_{nibble}^{(r)}$. Clearly, $H^\ast \cup L=\krq{r}{n}-\cF_{nibble}^{(r)}-F'$ is $F$-divisible. Thus, there exists an $F$-packing $\cF^\ast$ in $H^\ast \cup L$ which covers all edges of $H^\ast \cup L$ except possibly some inside $U$. Let $\cF:=\Set{F'}\cup \cF_{nibble}\cup \cF^\ast$. Let $F^\ast:=\cF^{(r)}$ and let $e^\ast$ be the edge in $F'$ which plays the role of $e_0$.

Clearly, $\cF$ is an $F$-decomposition of $F^\ast$ with $e^\ast\in F'\in \cF$ and $e^\ast$ plays the role of $e_0$ in $F'$. It remains to check \ref{extender3}. Let $e'\in \binom{V(\krq{r}{n})}{r}$ with $e'\cap e^\ast\neq \emptyset$. Since $e^\ast \In \bar{U}$, $e'$ cannot be inside $U$. Thus, $e'$ is covered by $\cF$ and we have $e'\in F^\ast$.
\endproof

Note that $|V(F^\ast)|$ is quite large here, in particular $1/|V(F^\ast)|\ll 1/f$ for $f=|V(F)|$. This means that $G$ being an $(\eps,\xi,f,r)$-supercomplex does not necessarily allow us to extend a given subgraph $H$ of $G^{(r)}$ to a copy of $\nabla_{(F^\ast,e^\ast)}H$ as described in Definition~\ref{def:ext in complex}. Fortunately, this will in fact not be necessary, as $F^\ast$ will only serve as an abstract auxiliary graph and will not appear as a subgraph of the absorber. (This is crucial since otherwise we would not be able to prove our main theorems with explicit bounds, let alone the bounds given in~Theorem~\ref{thm:min deg}.)

We are now ready to prove Lemma~\ref{lem:identifaction}.

\lateproof{Lemma~\ref{lem:identifaction}}
Given $F$ and $e_0$, we first apply Lemma~\ref{lem:extenders} to obtain a symmetric $r$-extender $(F^\ast,e^\ast)$ such that $F^\ast$ has an $F$-decomposition $\cF$ with $e^\ast\in F'\in \cF$ and $e^\ast$ plays the role of $e_0$ in $F'$. For given $h\in \bN$, let $k,m\in \bN$ be as in Lemma~\ref{lem:colouring}. Clearly, we may assume that there exists an $F$-divisible $r$-graph on at most $h$ vertices.\COMMENT{e.g. the empty one} Together with Lemma~\ref{lem:unique multigraph}, this implies that $(k,m)\in \cM_r$. Define $$M_h:=M^{(F^\ast,e^\ast)}_{k,m}.$$

Now, let $H$ be any $F$-divisible $r$-graph on at most $h$ vertices. By Lemma~\ref{lem:colouring}, there exists $t\in \bN$ such that $H+t\cdot F$ has a strong $m$-regular $[k]$-colouring. By Lemma~\ref{lem:unique multigraph}, we have $$\nabla_{(F^\ast,e^\ast)}(H+t\cdot F)\doublesquig M_h.$$
Let $\psi_1$ be any $e_0$-orientation of $H+t\cdot F$. Observe that since $e^\ast$ plays the role of $e_0$ in $F'$, $\nabla_{(F^\ast,e^\ast)}(H+t\cdot F)$ can be decomposed into a copy of $\nabla_{(F,e_0)}(H+t\cdot F,\psi_1)$ and $s$ copies of $F$ (where $s=|H+t\cdot F|\cdot|\cF\sm\Set{F'}|$). Hence, we have $$\nabla_{(F,e_0)}(H+t\cdot F,\psi_1)+s\cdot F\rightsquigarrow \nabla_{(F^\ast,e^\ast)}(H+t\cdot F)$$ by Proposition~\ref{prop:simple identification facts}\ref{fact:disjoint identification}. Thus, $\nabla_{(F,e_0)}(H+t\cdot F,\psi_1)+s\cdot F\doublesquig M_h$ by transitivity of $\doublesquig$.
Finally, let $\psi_2$ be any $e_0$-orientation of $M_h$. By Fact~\ref{fact:id after ext}, there exists an $e_0$-orientation $\psi_3$ of $\nabla_{(F,e_0)}(H+t\cdot F,\psi_1)+s\cdot F$ such that $$\nabla_{(F,e_0)}(\nabla_{(F,e_0)}(H+t\cdot F,\psi_1)+s\cdot F,\psi_3)  \rightsquigarrow \nabla_{(F,e_0)} (M_h,\psi_2).$$
\endproof

\subsection{Proof of the Absorbing lemma}\label{subsec:prove absorbing lemma}

As discussed at the beginning of Section~\ref{subsec:canonical}, we can now combine Lemma~\ref{lem:transformer} and Lemma~\ref{lem:identifaction} to construct the desired absorber by concatenating transformers between certain auxiliary $r$-graphs, in particular the extension $\nabla M_h$ of the canonical multi-$r$-graph $M_h$. It is relatively straightforward to find these auxiliary $r$-graphs within a given supercomplex $G$. The step when we need to find $\nabla M_h$ is the reason why the definition of a supercomplex includes the notion of extendability.

\lateproof{Lemma~\ref{lem:absorbing lemma}}
If $H$ is empty, then we can take $A$ to be empty, so let us assume that $H$ is not empty. In particular, $G^{(r)}$ is not empty. Recall also that we assume $r\ge 2$. Let $e_0\in F$ and let $M_h$ be as in Lemma~\ref{lem:identifaction}. Fix any $e_0$-orientation $\psi$ of $M_h$. By Lemma~\ref{lem:identifaction}, there exist $t_1,t_2,s_1,s_2,\psi_1,\psi_2,\psi_1',\psi_2'$ such that
\begin{align}
\nabla_{(F,e_0)}(\nabla_{(F,e_0)}(H+t_1\cdot F,\psi_1)+s_1\cdot F,\psi_1')  &\rightsquigarrow \nabla_{(F,e_0)} (M_h,\psi); \label{identifications1}\\
\nabla_{(F,e_0)}(\nabla_{(F,e_0)}(t_2\cdot F,\psi_2)+s_2\cdot F,\psi_2')  &\rightsquigarrow \nabla_{(F,e_0)} (M_h,\psi). \label{identifications2}
\end{align}
We can assume that $1/n\ll 1/\ell$ where $\ell:=\max\Set{|V(M_h)|,t_1,t_2,s_1,s_2}$.

Since $G$ is $(\xi,f+r,r)$-dense, there exist disjoint $Q_{1,1},\dots,Q_{1,t_1},Q_{2,1},\dots,Q_{2,t_2}\in G^{(f)}$ which are also disjoint from $V(H)$. For $i\in[2]$ and $j\in[t_i]$, let $F_{i,j}$ be a copy of $F$ with $V(F_{i,j})=Q_{i,j}$. Let $H_1:=H\cup \bigcup_{j\in[t_1]} F_{1,j}$ and $H_2:=\bigcup_{j\in[t_2]} F_{2,j}$ and for $i\in[2]$, define $$\cF_i:=\set{F_{i,j}}{j\in[t_i]}.$$ So $H_1$ is a copy of $H+t_1\cdot F$ and $H_2$ is a copy of $t_2\cdot F$. In fact, we will from now on assume (by redefining $\psi_i$ and $\psi_i'$) that for $i\in[2]$, we have
\begin{align}
\nabla_{(F,e_0)}(\nabla_{(F,e_0)}(H_i,\psi_i)+s_i\cdot F,\psi_i')  &\rightsquigarrow \nabla_{(F,e_0)} (M_h,\psi).
\end{align}
For $i\in[2]$, let $(H_i',\cF_i')$ be obtained by extending $H_i$ with a copy of $\nabla_{(F,e_0)}(H_i,\psi_i)$ in $G$ (cf.~Definition~\ref{def:ext in complex}). We can assume that $H_1'$ and $H_2'$ are vertex-disjoint by first choosing $H_1'$ whilst avoiding $V(H_2)$ and subsequently choosing $H_2'$ whilst avoiding $V(H_1')$. (To see that this is possible we can e.g.~use the fact that $G$ is $(\eps,d,f,r)$-regular for some $d\ge \xi$.)

There exist disjoint $Q_{1,1}',\dots,Q_{1,s_1}',Q_{2,1}',\dots,Q_{2,s_2}'\in G^{(f)}$ which are also disjoint from $V(H_1')\cup V(H_2')$. For $i\in[2]$ and $j\in[s_i]$, let $F_{i,j}'$ be a copy of $F$ with $V(F_{i,j}')=Q_{i,j}'$.
For $i\in[2]$, let
\begin{align*}
H_i''&:=H_i'\cup \bigcup_{j\in[s_i]}F_{i,j}';\\
\cF_i''&:=\set{F'_{i,j}}{j\in[s_i]}.
\end{align*}
 Since $H_i''$ is a copy of $\nabla_{(F,e_0)}(H_i,\psi_i)+s_i\cdot F$, we can assume (by redefining $\psi_i'$) that
\begin{align}
\nabla_{(F,e_0)}(H_i'',\psi_i')  &\rightsquigarrow \nabla_{(F,e_0)} (M_h,\psi).\label{identifications}
\end{align}
For $i\in[2]$, let $(H_i''',\cF_i''')$ be obtained by extending $H_i''$ with a copy of $\nabla_{(F,e_0)}(H_i'',\psi_i')$ in $G$ (cf.~Definition~\ref{def:ext in complex}). We can assume that $H_1'''$ and $H_2'''$ are vertex-disjoint.

Since $G$ is $(\xi,f,r)$-extendable, it is straightforward to find a copy $M'$ of $\nabla_{(F,e_0)} (M_h,\psi)$ in $G^{(r)}$ which is vertex-disjoint from $H_1'''$ and $H_2'''$.

Since $H_i'''$ is a copy of $\nabla_{(F,e_0)}(H_i'',\psi_i')$, by~\eqref{identifications} we have $H_i'''\rightsquigarrow M'$ for $i\in[2]$. Using Fact~\ref{fact:extensions}\ref{fact:extensions:divisible} repeatedly, we can see that both $H_1'''$ and $H_2'''$ are $F$-divisible. Together with Proposition~\ref{prop:simple identification facts}\ref{fact:identification divisible}, this implies that $M'$ is $F$-divisible as well.

Let $T_1:=(H_1-H)\cup H_1''$ and $T_2:=H_2\cup H_2''$. For $i\in[2]$, let $$\cF_{i,1}:=\cF_i'\cup \cF_i''\mbox{ and }\cF_{i,2}:=\cF_i\cup \cF_i'''.$$
We claim that $\cF_{1,1},\cF_{1,2},\cF_{2,1},\cF_{2,2}$ are $2$-well separated $F$-packings in $G$ such that
\begin{align}
\cF_{1,1}^{(r)}=T_1\cup H,\quad
\cF_{1,2}^{(r)}=T_1\cup H_1''',\quad
\cF_{2,2}^{(r)}=T_2\cup H_2'''\quad\mbox{and}\quad
\cF_{2,1}^{(r)}=T_2.\label{silly transformer}
\end{align}
(In particular, $T_1$ is a $2$-well separated $(H,H_1''';F)$-transformer in $G$ and $T_2$ is a $2$-well separated $(H_2''',\emptyset;F)$-transformer in $G$.)
Indeed, we clearly have that $\cF_1,\cF_2,\cF_1'',\cF_2''$ are $1$-well separated $F$-packings in $G$, where $\cF_1^{(r)}=H_1-H$, $\cF_2^{(r)}=H_2$, and for $i\in[2]$, $\cF_i''^{(r)}=H_i''-H_i'$. Moreover, by Fact~\ref{fact:ext in complex}, for $i\in[2]$, $\cF_i'$ and $\cF_i'''$ are $1$-well separated $F$-packings in $G$ with $\cF_i'^{(r)}=H_i\cup H_i'$ and $\cF_i'''^{(r)}=H_i''\cup H_i'''$.
Note that
\begin{align*}
T_1\cup H &= H_1 \cup H_1''=(H_1\cup H_1') \cupdot (H_1''-H_1')=\cF_1'^{(r)}\cupdot \cF_1''^{(r)} = \cF_{1,1}^{(r)};\\
T_1\cup H_1''' &=(H_1-H)\cupdot (H_1''\cup H_1''')=\cF_1^{(r)}\cupdot \cF_1'''^{(r)} = \cF_{1,2}^{(r)};\\
T_2\cup H_2''' &=H_2 \cupdot (H_2''\cup H_2''')=\cF_2^{(r)}\cupdot \cF_2'''^{(r)} = \cF_{2,2}^{(r)};\\
T_2  &= H_2\cup H_2''=(H_2\cup H_2') \cupdot (H_2''-H_2')=\cF_2'^{(r)}\cupdot \cF_2''^{(r)} = \cF_{2,1}^{(r)}.
\end{align*}
To check that $\cF_{1,1}$, $\cF_{1,2}$, $\cF_{2,1}$ and $\cF_{2,2}$ are $2$-well separated $F$-packings, by Fact~\ref{fact:ws}\ref{fact:ws:1} it is now enough to show for $i\in[2]$ that $\cF_i'$ and $\cF_i''$ are $(r+1)$-disjoint and that $\cF_i$ and $\cF_i'''$ are $(r+1)$-disjoint.
Note that for all $F'\in \cF_i'$ and $F''\in \cF_i''$, we have $V(F')\In V(H_i')$ and $V(F'')\cap V(H_i')=\emptyset$, thus $V(F')\cap V(F'')=\emptyset$. For all $F'\in \cF_i$ and $F''\in \cF_i'''$, we have $V(F')\In V(H_i)$ and $|V(F'')\cap V(H_i)|\le |V(F'')\cap V(H_i'')|\le r$ by Fact~\ref{fact:ext in complex}, thus $|V(F')\cap V(F'')|\le r$.
This completes the proof of \eqref{silly transformer}.

Let
\begin{align*}
O_{r}&:=H_1\cup H_1''\cup H_2 \cup H_2'';\\
O_{r+1,3}&:=\cF_{1,1}^{\le(r+1)}\cup \cF_{1,2}^{\le(r+1)} \cup \cF_{2,1}^{\le(r+1)} \cup \cF_{2,2}^{\le(r+1)}.
\end{align*}
By Fact~\ref{fact:ws}\ref{fact:ws:maxdeg}, $\Delta(O_{r+1,3})\le 8(f-r)$. Note that $H_1''',M'\In G^{(r)}-(O_{r}\cup H_2''')$. Thus, by Proposition~\ref{prop:noise}\ref{noise:supercomplex} and Lemma~\ref{lem:transformer}, there exists a $(\kappa/3)$-well separated $(H_1''',M';F)$-transformer $T_3$ in $G-(O_{r}\cup H_2''')-O_{r+1,3}$ with $\Delta(T_3)\le \gamma n/3$. Let $\cF_{3,1}$ and $\cF_{3,2}$ be $(\kappa/3)$-well separated $F$-packings in $G-(O_{r}\cup H_2''')-O_{r+1,3}$ such that $\cF_{3,1}^{(r)}=T_3\cup H_1'''$ and $\cF_{3,2}^{(r)}=T_3\cup M'$.

Similarly, let $O_{r+1,4}:=O_{r+1,3}\cup \cF_{3,1}^{\le(r+1)}\cup \cF_{3,2}^{\le(r+1)}$. By Fact~\ref{fact:ws}\ref{fact:ws:maxdeg}, $\Delta(O_{r+1,4})\le (8+2\kappa/3)(f-r)$. Note that $H_2''',M'\In G^{(r)}-(O_r\cup H_1'''\cup T_3)$. Using Proposition~\ref{prop:noise}\ref{noise:supercomplex} and Lemma~\ref{lem:transformer} again, we can find a $(\kappa/3)$-well separated $(H_2''',M';F)$-transformer $T_4$ in $G-(O_r\cup H_1'''\cup T_3)-O_{r+1,4}$ with $\Delta(T_4)\le \gamma n/3$. Let $\cF_{4,1}$ and $\cF_{4,2}$ be $(\kappa/3)$-well separated $F$-packings in $G-(O_r\cup H_1'''\cup T_3)-O_{r+1,4}$  such that of $\cF_{4,1}^{(r)}=T_4\cup H_2'''$ and $\cF_{4,2}^{(r)}=T_4\cup M'$.

Let
\begin{align*}
A&:=T_1\cupdot H_1''' \cupdot T_3 \cupdot M' \cupdot T_4 \cupdot H_2''' \cupdot T_2;\\
\cF_{\circ}&:=\cF_{1,2}\cup\cF_{3,2}\cup\cF_{4,1}\cup\cF_{2,1};\\
\cF_{\bullet}&:=\cF_{1,1}\cup\cF_{3,1}\cup\cF_{4,2}\cup\cF_{2,2}.
\end{align*}
Clearly, $A\In G^{(r)}$, and $\Delta(A)\le \gamma n$. Moreover, $A$ and $H$ are edge-disjoint. Using \eqref{silly transformer}, we can check that
\begin{align*}
\cF_{\circ}^{(r)} &= \cF_{1,2}^{(r)}\cupdot \cF_{3,2}^{(r)}\cupdot \cF_{4,1}^{(r)}\cupdot \cF_{2,1}^{(r)} = (T_1\cup H_1''') \cupdot (T_3 \cup M') \cupdot (T_4 \cup H_2''') \cupdot T_2= A;\\
\cF_{\bullet}^{(r)} &= \cF_{1,1}^{(r)}\cupdot\cF_{3,1}^{(r)}\cupdot\cF_{4,2}^{(r)}\cupdot\cF_{2,2}^{(r)} =(H\cup T_1) \cupdot (H_1''' \cup T_3) \cupdot (M'\cup T_4) \cupdot (H_2'''\cup T_2)= A\cup H.
\end{align*}
By definition of $O_{r+1,3}$ and $O_{r+1,4}$, we have that $\cF_{1,2},\cF_{3,2},\cF_{4,1},\cF_{2,1}$ are $(r+1)$-disjoint. Thus, $\cF_{\circ}$ is a $(2\cdot \kappa/3+4)$-well separated $F$-packing in $G$ by Fact~\ref{fact:ws}\ref{fact:ws:1}. Similarly, $\cF_{\bullet}$ is a $(2\cdot \kappa/3+4)$-well separated $F$-packing in $G$.
So~$A$ is indeed a $\kappa$-well separated $F$-absorber for $H$ in~$G$.
\endproof

\section{Proof of the main theorems}\label{sec:main proofs}

\subsection{Main complex decomposition theorem}\label{subsec:main complex thm}

We can now deduce our main decomposition result for supercomplexes. The main ingredients for the proof of Theorem~\ref{thm:main complex} are Lemma~\ref{lem:get vortex} (to find a vortex), Lemma~\ref{lem:absorbing lemma} (to find absorbers for the possible leftovers in the final vortex set), and Lemma~\ref{lem:almost dec} (to cover all edges outside the final vortex set).

\lateproof{Theorem~\ref{thm:main complex}}
We proceed by induction on $r$. The case $r=1$ forms the base case of the induction and in this case we do not rely on any inductive assumption. Suppose that $r\in \bN$ and that \ind{i} is true for all $i\in[r-1]$.

We may assume that $1/n\ll 1/\kappa\ll \eps$. Choose new constants $\kappa',m'\in \bN$ and $\gamma,\mu>0$ such that $$1/n\ll 1/\kappa \ll \gamma \ll 1/m' \ll 1/\kappa'\ll \eps \ll \mu \ll \xi,1/f$$ and suppose that $F$ is a weakly regular $r$-graph on $f>r$ vertices.

Let $G$ be an $F$-divisible $(\eps,\xi,f,r)$-supercomplex on $n$ vertices. We are to show the existence of a $\kappa$-well separated $F$-decomposition of $G$. By Lemma~\ref{lem:get vortex}, there exists a $(2\sqrt{\eps},\mu,\xi-\eps,f,r,m)$-vortex $U_0,U_1,\dots,U_\ell$ in $G$ for some $\mu m' \le m \le m'$. Let $H_1,\dots,H_s$ be an enumeration of all spanning $F$-divisible subgraphs of $G[U_\ell]^{(r)}$. Clearly, $s\le 2^{\binom{m}{r}}$. We will now find edge-disjoint subgraphs $A_1,\dots,A_s$ of $G^{(r)}$ and $\sqrt{\kappa}$-well separated $F$-packings $\cF_{1,\circ},\cF_{1,\bullet},\dots,\cF_{s,\circ},\cF_{s,\bullet}$ in $G$ such that for all $i\in[s]$ we have that
\begin{enumerate}[label={\rm(A\arabic*)}]
\item $\cF_{i,\circ}^{(r)}=A_i$ and $\cF_{i,\bullet}^{(r)}=A_i\cup H_i$;\label{absorber:property without}
\item $\Delta(A_i)\le \gamma n$;\label{absorber:maxdeg}
\item $A_i[U_1]$ is empty;\label{absorber:outside}
\item $\cF_{i,\bullet}^{\le},G[U_1],\cF_{1,\circ}^{\le},\dots,\cF_{i-1,\circ}^{\le},\cF_{i+1,\circ}^{\le},\dots,\cF_{s,\circ}^{\le}$ are $(r+1)$-disjoint.\label{absorber:disjointness}
\end{enumerate}
Suppose that for some $t\in[s]$, we have already found edge-disjoint $A_1,\dots,A_{t-1}$ together with $\cF_{1,\circ},\cF_{1,\bullet},\dots,\cF_{t-1,\circ},\cF_{t-1,\bullet}$ that satisfy \ref{absorber:property without}--\ref{absorber:disjointness} (with $t-1$ playing the role of $s$).
Let
\begin{align*}
T_t&:=(G^{(r)}[U_1]- H_t) \cup \bigcup_{i\in[t-1]}A_i;\\
T_t'&:=G^{(r+1)}[U_1] \cup \bigcup_{i\in[t-1]}(\cF_{i,\circ}^{\le(r+1)}\cup \cF_{i,\bullet}^{\le(r+1)}).
\end{align*}
Clearly, $\Delta(T_t)\le \mu n + s\gamma n \le 2\mu n$ by \ref{vortex:size} and \ref{absorber:maxdeg}. Also, $\Delta(T_t')\le \mu n+ 2s\sqrt{\kappa} (f-r)\le 2\mu n$ by \ref{vortex:size} and Fact~\ref{fact:ws}\ref{fact:ws:maxdeg}. Thus, applying Proposition~\ref{prop:noise}\ref{noise:supercomplex} twice we see that $G_{abs,t}:=G-T_t-T_t'$ is still a $(\sqrt{\mu},\xi/2,f,r)$-supercomplex. Moreover, $H_t\In G_{abs,t}^{(r)}$ by~\ref{absorber:outside}.
Hence, by Lemma~\ref{lem:absorbing lemma},\COMMENT{with $\sqrt{\mu}$ playing the role of $\eps$} there exists a $\sqrt{\kappa}$-well separated $F$-absorber $A_t$ for $H_t$ in $G_{abs,t}$ with $\Delta(A_t)\le \gamma n$. Let $\cF_{t,\circ}$ and $\cF_{t,\bullet}$ be $\sqrt{\kappa}$-well separated $F$-packings in $G_{abs,t} \In G$ such that $\cF_{t,\circ}^{(r)}=A_t$ and $\cF_{t,\bullet}^{(r)}=A_t\cup H_t$. Clearly, $A_t$ is edge-disjoint from $A_1,\dots,A_{t-1}$. Moreover, \ref{absorber:outside} holds since $G_{abs,t}^{(r)}[U_1]=H_t$ and $A_t$ is edge-disjoint from $H_t$, and \ref{absorber:disjointness} holds with $t$ playing the role of $s$ due to the definition of $T_t'$.

Let $A^\ast:=A_1\cup \dots\cup A_s$ and $T^\ast:=\bigcup_{i\in[s]}(\cF_{i,\circ}^{\le(r+1)}\cup \cF_{i,\bullet}^{\le(r+1)})$. We claim that the following hold:
\begin{enumerate}[label={\rm(A\arabic*$'$)}]
\item for every $F$-divisible subgraph $H^\ast$ of $G[U_\ell]^{(r)}$, $A^\ast\cup H^\ast$ has an $s\sqrt{\kappa}$-well separated $F$-decomposition $\cF^\ast$ with $\cF^{\ast\le}\In G[T^\ast]$;\COMMENT{The last condition is equivalent to $\cF^{\ast\le}\In G$ and $\cF^{\ast\le(r+1)}\In T^\ast$}\label{absorber all:property}
\item $\Delta(A^\ast)\le \eps n$ and $\Delta(T^\ast)\le 2s\sqrt{\kappa}(f-r)\le \eps n$;\label{absorber all:maxdeg}
\item $A^\ast[U_1]$ and $T^\ast[U_1]$ are empty.\label{absorber all:outside}
\end{enumerate}

For \ref{absorber all:property}, we have that $H^\ast=H_t$ for some $t\in[s]$. Then $\cF^\ast:=\cF_{t,\bullet}\cup \bigcup_{i\in[s]\sm\Set{t}}\cF_{i,\circ}$ is an $F$-decomposition of $A^\ast \cup H^\ast=(A_t\cup H_t) \cup \bigcup_{i\in[s]\sm\Set{t}}A_i$ by \ref{absorber:property without} and since $H_t,A_1,\dots,A_s$ are pairwise edge-disjoint. By \ref{absorber:disjointness} and Fact~\ref{fact:ws}\ref{fact:ws:1}, $\cF^\ast$ is $s\sqrt{\kappa}$-well separated. We clearly have $\cF^{\ast\le} \In G$ and $\cF^{\ast\le(r+1)}\In T^\ast$. Thus $\cF^{\ast\le}\In G[T^\ast]$ and so \ref{absorber all:property} holds. It is straightforward to check that \ref{absorber all:maxdeg} follows from \ref{absorber:maxdeg} and Fact~\ref{fact:ws}\ref{fact:ws:maxdeg}, and that \ref{absorber all:outside} follows from \ref{absorber:outside} and \ref{absorber:disjointness}.

Let $G_{almost}:=G-A^\ast-T^\ast$. By~\ref{absorber all:maxdeg} and Proposition~\ref{prop:noise}\ref{noise:supercomplex}, $G_{almost}$ is an $(\sqrt{\eps},\xi/2,f,r)$-supercomplex. Moreover, since $A^\ast$ must be $F$-divisible, we have that $G_{almost}$ is $F$-divisible. By~\ref{absorber all:outside}, $U_1,\dots,U_\ell$ is a $(2\sqrt{\eps},\mu,\xi-\eps,f,r,m)$-vortex in $G_{almost}[U_1]$. Moreover, \ref{absorber all:maxdeg} and Proposition~\ref{prop:random noise} imply that $U_1$ is $(\eps^{1/5},\mu,\xi/2,f,r)$-random in $G_{almost}$ and $U_1\sm U_2$ is $(\eps^{1/5},\mu(1-\mu),\xi/2,f,r)$-random in $G_{almost}$. Hence, $U_0,U_1,\dots,U_\ell$ is still an $(\eps^{1/5},\mu,\xi/2,f,r,m)$-vortex in $G_{almost}$. Thus, by Lemma~\ref{lem:almost dec}, there exists a $4\kappa'$-well separated $F$-packing $\cF_{almost}$ in $G_{almost}$ which covers all edges of $G_{almost}^{(r)}$ except possibly some inside $U_\ell$. Let $H^\ast:=(G_{almost}^{(r)}-\cF_{almost}^{(r)})[U_\ell]$. Since $H^\ast$ is $F$-divisible, $A^\ast\cup H^\ast$ has an $s\sqrt{\kappa}$-well separated $F$-decomposition $\cF^{\ast}$ with $\cF^{\ast\le}\In G[T^\ast]$ by \ref{absorber all:property}. Clearly, $$G^{(r)}=G_{almost}^{(r)} \cupdot A^\ast =\cF_{almost}^{(r)}\cupdot H^\ast \cupdot A^\ast= \cF_{almost}^{(r)} \cupdot \cF^{\ast(r)},$$ and $\cF_{almost}$\COMMENT{`$\In G-T^\ast$'} and $\cF^\ast$ are $(r+1)$-disjoint. Thus, by Fact~\ref{fact:ws}\ref{fact:ws:1}, $\cF_{almost}\cup \cF^\ast$ is a $(4\kappa'+s\sqrt{\kappa})$-well separated $F$-decomposition of $G$, completing the proof.
\endproof

\subsection{Resolvable partite designs}\label{subsec:resolvable}

Perhaps surprisingly, it is much easier to obtain decompositions of complete partite $r$-graphs than of complete (non-partite) $r$-graphs. In fact, we can obtain (explicit)~resolvable decompositions (sometimes referred to as \defn{Kirkman systems} or \emph{large sets of designs}) in the partite setting using basic linear algebra. We believe that this result and the corresponding construction are of independent interest. Here, we will use this result to show that for every $r$-graph $F$, there is a weakly regular $r$-graph $F^\ast$ which is $F$-decomposable (see Lemma~\ref{lem:regularisation}).

Let $G$ be a complex. We say that a $\krq{r}{f}$-decomposition $\cK$ of $G$ is \defn{resolvable} if $\cK$ can be partitioned into $\krq{r-1}{f}$-decompositions of $G$, that is, $\cK^{\le(f)}$ can be partitioned into sets $Y_1,\dots,Y_t$ such that for each $i\in[t]$, $\cK_i:=\set{G^{(r-1)}[Q]}{Q\in Y_i}$ is a $\krq{r-1}{f}$-decomposition of $G$. Clearly, $\cK_1,\dots,\cK_t$ are $r$-disjoint.

Let $K_{n\times k}$ be the complete $k$-partite complex with each vertex class having size $n$. More precisely, $K_{n\times k}$ has vertex set $V_1\cupdot \dots \cupdot V_k$ such that $|V_i|=n$ for all $i\in[k]$ and $e\in K_{n\times k}$ if and only if $e$ is \defn{crossing}, that is, intersects with each $V_i$ in at most one vertex. Since every subset of a crossing set is crossing, this defines a complex.

\begin{theorem}\label{thm:partite designs}
Let $q$ be a prime power and $2f\le q$. Then for every $r\in[f-1]$, $K_{q\times f}$ has a resolvable $\krq{r}{f}$-decomposition.
\end{theorem}

Let us first motivate the proof of Theorem~\ref{thm:partite designs}. Let $\bF$ be the finite field of order $q$. Assume that each class of $K_{q\times f}$ is a copy of $\bF$. Suppose further that we are given a matrix $A\in \bF^{(f-r)\times f}$ with the property that every $(f-r)\times (f-r)$-submatrix is invertible. Identifying $K_{q\times f}^{(f)}$ with $\bF^f$ in the obvious way, we let $\cK$ be the set of all $Q\in K_{q\times f}^{(f)}$ with $AQ=0$. Fixing the entries of $r$ coordinates of $Q$ (which can be viewed as fixing an $r$-set) transforms this into an equation $A'Q'=b'$, where $A'$ is an $(f-r)\times (f-r)$-submatrix of $A$. Thus, there exists a unique solution, which will translate into the fact that every $r$-set of $K_{q\times f}$ is contained in exactly one $f$-set of $\cK$, i.e.~we have a $\krq{r}{f}$-decomposition.

There are several known classes of matrices over finite fields which have the desired property that every square submatrix is invertible. We use so-called Cauchy matrices, introduced by Cauchy~\cite{C}, which are very convenient for our purposes. For an application of Cauchy matrices to coding theory, see e.g.~\cite{BKKKLZ}.

Let $\bF$ be a field and let $x_1,\dots,x_m,y_1,\dots,y_n$ be distinct elements of $\bF$. The \defn{Cauchy matrix generated by $(x_i)_{i\in[m]}$ and $(y_j)_{j\in[n]}$} is the $m\times n$-matrix $A\in \bF^{m\times n}$ defined by $a_{i,j}:=(x_i-y_j)^{-1}$. Obviously, every submatrix of a Cauchy matrix is itself a Cauchy matrix. For $m=n$, it is well known that the Cauchy determinant is given by the following formula (cf.~\cite{Sche}):\COMMENT{only stated for complex numbers, see wikipedia}
\begin{align*}
\det(A)=\frac{\prod_{1\le i<j\le n}(x_j-x_i)(y_i-y_j)}{\prod_{1\le i,j\le n}(x_i-y_j)}.
\end{align*}
In particular, every square Cauchy matrix is invertible.

\lateproof{Theorem~\ref{thm:partite designs}}
Let $\bF$ be the finite field of order $q$. Since $2f\le q$, there exists a Cauchy matrix $A\in \bF^{(f-r+1)\times f}$.\COMMENT{just need $2f-r+1$ distinct elements in $\bF$} Let $\mathbf{\hat{a}}$ be the final row of $A$ and let $A'\in \bF^{(f-r)\times f}$ be obtained from $A$ by deleting $\mathbf{\hat{a}}$.

We assume that the vertex set of $K_{q\times f}$ is $\bF\times [f]$. Hence, for every $e\in K_{q\times f}$, there are unique $1\le i_1< \dots <i_{|e|}\le f$ and $x_1,\dots,x_{|e|}\in \bF$ such that $e=\set{(x_j,i_j)}{j\in[|e|]}$. Let
$$
I_e:=\Set{i_1,\dots,i_{|e|}}\In [f] \quad \mbox{ and } \quad\mathbf{x_e}:=\left(
\begin{array}{c}
x_1\\
\vdots \\
x_{|e|}\\
\end{array}\right)\in \bF^{|e|}.
$$
Clearly, $Q\in K_{q\times f}^{(f)}$ is uniquely determined by $\mathbf{x_{Q}}$.

Define $Y\In K_{q\times f}^{(f)}$ as the set of all $Q\in K_{q\times f}^{(f)}$ which satisfy $A'\cdot \mathbf{x_{Q}} =\mathbf{0}$.
Moreover, for each $x^\ast\in \bF$, define $Y_{x^\ast}\In Y$ as the set of all $Q\in Y$ which satisfy $\mathbf{\hat{a}}\cdot\mathbf{x_{Q}}= x^\ast$. Clearly, $\set{Y_{x^\ast}}{x^\ast \in \bF}$ is a partition of $Y$. Let $\cK:=\set{K_{q\times f}^{(r)}[Q]}{Q\in Y}$ and $\cK_{x^\ast}:=\set{K_{q\times f}^{(r-1)}[Q]}{Q\in Y_{x^\ast}}$ for each $x^\ast\in \bF$. We claim that $\cK$ is a $\krq{r}{f}$-decomposition of $K_{q\times f}$ and that $\cK_{x^\ast}$ is a $\krq{r-1}{f}$-decomposition of $K_{q\times f}$ for each $x^\ast\in \bF$.

For $I\In [f]$, let $A_I$ be the $(f-r+1)\times |I|$-submatrix of $A$ obtained by deleting the columns which are indexed by $[f]\sm I$. Similarly, for $I\In [f]$, let $A_I'$ be the $(f-r)\times |I|$-submatrix of $A'$ obtained by deleting the columns which are indexed by $[f]\sm I$. Finally, for a vector $\mathbf{x}\in \bF^f$ and $I\In [f]$, let $\mathbf{x}_I\in \bF^{|I|}$ be the vector obtained from $\mathbf{x}$ by deleting the coordinates not in $I$.

Observe that for all $e\in K_{q\times f}$ and $Q\in K_{q\times f}^{(f)}$, we have
\begin{align}
e\In Q\mbox{ if and only if } \mathbf{x_{Q}}_{I_e}=\mathbf{x_e}.\label{vector containment}
\end{align}

Consider $e\in K_{q\times f}^{(r)}$. By~\eqref{vector containment}, the number of $Q\in Y$ containing $e$ is equal to the number of $\mathbf{x}\in \bF^{f}$ such that $A'\cdot \mathbf{x} =\mathbf{0}$ and $\mathbf{x}_{I_e}=\mathbf{x_e}$, or equivalently, the number of $\mathbf{x'}\in \bF^{f-r}$ satisfying $A'_{I_e}\cdot \mathbf{x_e}+A'_{[f]\sm I_e}\cdot \mathbf{x'} =\mathbf{0}$. Since $A'_{[f]\sm I_e}$ is an $(f-r)\times(f-r)$-Cauchy matrix, the equation $A'_{[f]\sm I_e}\cdot \mathbf{x'} =-A'_{I_e}\cdot \mathbf{x_e}$ has a unique solution $\mathbf{x'}\in \bF^{f-r}$, i.e.~there is exactly one $Q\in Y$ which contains $e$. Thus, $\cK$ is a $\krq{r}{f}$-decomposition of $K_{q\times f}$.

Now, fix $x^\ast\in \bF$ and $e\in K_{q\times f}^{(r-1)}$. By~\eqref{vector containment}, the number of $Q\in Y_{x^\ast}$ containing $e$ is equal to the number of $\mathbf{x}\in \bF^{f}$ such that $A'\cdot \mathbf{x} =\mathbf{0}$,  $\mathbf{\hat{a}}\cdot\mathbf{x}= x^\ast$ and $\mathbf{x}_{I_e}=\mathbf{x_e}$, or equivalently, the number of $\mathbf{x'}\in \bF^{f-(r-1)}$ satisfying $A_{I_e}\cdot \mathbf{x_e}+A_{[f]\sm I_e}\cdot \mathbf{x'} =\left(
\begin{array}{c}
\mathbf{0}\\
x^\ast\\
\end{array}\right)$. Since $A_{[f]\sm I_e}$ is an $(f-r+1)\times(f-r+1)$-Cauchy matrix, this equation has a unique solution $\mathbf{x'}\in \bF^{f-r+1}$, i.e.~there is exactly one $Q\in Y_{x^\ast}$ which contains $e$. Hence, $\cK_{x^\ast}$ is a $\krq{r-1}{f}$-decomposition of $K_{q\times f}$.
\endproof

Our application of Theorem~\ref{thm:partite designs} is as follows.

\begin{lemma}\label{lem:regularisation}
Let $2\le r<f$. Let $F$ be any $r$-graph on $f$ vertices. There exists a weakly regular $r$-graph~$F^\ast$ on at most $2f\cdot f!$ vertices which has a $1$-well separated $F$-decomposition.
\end{lemma}

\proof
Choose a prime power $q$ with $f! \le q\le 2f!$.\COMMENT{exists e.g.~by Bertrand's postulate} Let $V(F)=\Set{v_1,\dots,v_f}$. By Theorem~\ref{thm:partite designs}, there exists a resolvable $\krq{r}{f}$-decomposition $\cK$ of $K_{q\times f}$.\COMMENT{$f!\ge 2f$.} Let the vertex classes of $K_{q\times f}$ be $V_1,\dots, V_f$. Let $\cK_1,\dots,\cK_q$ be a partition of $\cK$ into $\krq{r-1}{f}$-decompositions of $K_{q\times f}$. (We will only need $\cK_1,\dots,\cK_{f!}$.) We now construct $F^\ast$ with vertex set $V(K_{q\times f})$ as follows: Let $\pi_1,\dots,\pi_{f!}$ be an enumeration of all permutations on $[f]$. For every $i\in[f!]$ and $Q\in \cK_i^{\le(f)}$, let $F_{i,Q}$ be a copy of $F$ with $V(F)=Q$ such that for every $j\in[f]$, the unique vertex in $Q\cap V_{\pi_i(j)}$ plays the role of $v_j$.
Let
\begin{align*}
F^\ast&:=\bigcup_{i\in[f!],Q\in \cK_i^{\le(f)}}F_{i,Q};\\
\cF&:=\set{F_{i,Q}}{i\in[f!],Q\in \cK_i^{\le(f)}}.
\end{align*}
Since $\cK_1,\dots,\cK_{f!}$ are $r$-disjoint, we have $|V(F')\cap V(F'')|<r$ for all distinct $F',F''\in \cF$. Thus, $\cF$ is a $1$-well separated $F$-decomposition of $F^\ast$.

We now show that $F^\ast$ is weakly regular. Let $i\in[r-1]_0$ and $S\in \binom{V(F^\ast)}{i}$. If $S$ is not crossing, then $|F^\ast(S)|=0$, so assume that $S$ is crossing. If $i=r-1$, then $S$ plays the role of every $(r-1)$-subset of $V(F)$ exactly $k$ times, where $k$ is the number of permutations on $[f]$ that map $[r-1]$ to $[r-1]$. Hence, $$|F^\ast(S)|=|F|{r}k=|F|\cdot r!(f-r+1)!=:s_{r-1}.$$ If $i<r-1$, then $S$ is contained in exactly $c_i:=\binom{f-i}{r-1-i}q^{r-1-i}$ crossing $(r-1)$-sets. Thus, $$|F^\ast(S)|=\frac{s_{r-1}c_i}{r-i}=:s_i.$$ Therefore, $F^\ast$ is weakly $(s_0,\dots,s_{r-1})$-regular.
\endproof

\subsection{Proofs of Theorems~\ref{thm:design},~\ref{thm:near optimal} and~\ref{thm:min deg}}\label{subsec:main r graph theorems}

We now prove our main theorems which guarantee $F$-decompositions in $r$-graphs of high minimum degree (for weakly regular $r$-graphs $F$, see Theorem~\ref{thm:min deg}), and $F$-designs in typical $r$-graphs (for arbitrary $r$-graphs $F$, see Theorem~\ref{thm:design}). We will also derive Theorem~\ref{thm:near optimal}.

We first prove the minimum degree version (for weakly regular $r$-graphs $F$).
Instead of directly proving Theorem~\ref{thm:min deg} we actually prove a stronger `local resilience version'. Let $\cH_r(n,p)$ denote the random binomial $r$-graph on $[n]$ whose edges appear independently with probability~$p$.

\begin{theorem}[Resilience version]\label{thm:resilience}
Let $p\in(0,1]$ and $f,r \in \bN$ with $f>r$ and let $$c(f,r,p):=\frac{r!p^{2^r\binom{f+r}{r}}}{3 \cdot 14^r f^{2r}}.$$ Then the following holds \whp for $H\sim \cH_r(n,p)$. For every weakly regular $r$-graph $F$ on $f$ vertices and any $r$-graph $L$ on $[n]$ with $\Delta(L)\le c(f,r,p)n$, $H\bigtriangleup L$ has an $F$-decomposition whenever it is $F$-divisible.
\end{theorem}

The case $p=1$ immediately implies Theorem~\ref{thm:min deg}.

\proof
Choose $n_0\in \bN$ and $\eps>0$ such that $1/n_0\ll \eps \ll p,1/f$ and let $n\ge n_0$, $$c':= \frac{1.1 \cdot 2^{r}\binom{f+r}{r}}{(f-r)!}c(f,r,p),\quad \xi:=0.99/f!, \quad \xi':=0.95\xi p^{2^r\binom{f+r}{r}}, \quad \xi'':=0.9(1/4)^{\binom{f+r}{f}}(\xi'-c').$$
Recall that the complete complex $K_n$ is an $(\eps,\xi,f,r)$-supercomplex (cf. Example~\ref{ex:complete}).\COMMENT{Actually the $\eps$ is zero, but need it for Corollary~\ref{cor:random supercomplex}}
Let $H\sim \cH_r(n,p)$. We can view $H$ as a random subgraph of $K_n^{(r)}$. By Corollary~\ref{cor:random supercomplex}, the following holds \whp for all $L\In K_n^{(r)}$ with $\Delta(L)\le c(f,r,p) n$:
$$K_n[H\bigtriangleup L]\mbox{ is a }(3\eps+c',\xi'-c',f,r)\mbox{-supercomplex.}$$
Note that $c'\le \frac{p^{2^r\binom{f+r}{r}}}{2.7(2\sqrt{e})^r f!}$.\COMMENT{$\binom{f+r}{r}\le (2f)^r/r!$. $c'\le \frac{1.1\cdot 2^{2r}f^r}{r!(f-r)!}\cdot c= \frac{1.1\cdot 2^{2r}f^r}{r!(f-r)!}\cdot \frac{r!p^{2^r\binom{f+r}{r}}}{3\cdot 14^r f^{2r}} \le \frac{p^{2^r\binom{f+r}{r}}}{2.7(2\sqrt{e})^r f!}$}
Thus, $2(2\sqrt{e})^r\cdot (3\eps+c')\le \xi'-c'$.\COMMENT{$2(2\sqrt{e})^r(3\eps+c')+c'\le 2.4(2\sqrt{e})^rc'$ and $2.4/2.7\le 0.95\cdot 0.99$} Lemma~\ref{lem:boost complex} now implies that $K_n[H\bigtriangleup L]$ is an $(\eps,\xi'',f,r)$-supercomplex. Hence, if $H\bigtriangleup L$ is $F$-divisible, it has an $F$-decomposition by Theorem~\ref{thm:main complex}.
\endproof

Next, we derive Theorem~\ref{thm:design}. As indicated previously, we cannot apply Theorem~\ref{thm:main complex} directly, but have to carry out two reductions. As shown in Lemma~\ref{lem:regularisation}, we can `perfectly' pack any given $r$-graph $F$ into a weakly regular $r$-graph $F^\ast$.
We also need the following lemma, which we will prove later in Section~\ref{sec:make divisible}. It allows us to remove a sparse $F$-decomposable subgraph $L$ from an $F$-divisible $r$-graph $G$ to achieve that $G-L$ is $F^\ast$-divisible. Note that we do not need to assume that $F^\ast$ is weakly regular.

\begin{lemma}\label{lem:make divisible}
Let $1/n\ll \gamma \ll \xi,1/f^\ast$ and $r\in[f^\ast-1]$. Let $F$ be an $r$-graph.\COMMENT{We allow $|V(F)|=r$ here} Let $F^\ast$ be an $r$-graph on $f^\ast$ vertices which has a $1$-well separated $F$-decomposition.
Let $G$ be an $r$-graph on $n$ vertices such that for all $A\In\binom{V(G)}{r-1}$ with $|A|\le \binom{f^\ast-1}{r-1}$, we have $|\bigcap_{S\in A}G(S)|\ge \xi n$. Let $O$ be an $(r+1)$-graph on $V(G)$ with $\Delta(O)\le \gamma n$. Then there exists an $F$-divisible subgraph $D\In G$ with $\Delta(D)\le \gamma^{-2}$ such that the following holds: for every $F$-divisible $r$-graph $H$ on $V(G)$ which is edge-disjoint from $D$, there exists a subgraph $D^\ast\In D$ such that $H \cup D^\ast$ is $F^\ast$-divisible and $D-D^\ast$ has a $1$-well separated $F$-decomposition $\cF$ such that $\cF^{\le(r+1)}$ and $O$ are edge-disjoint.
\end{lemma}

In particular, we will apply this lemma when $G$ is $F$-divisible and thus $H:=G-D$ is $F$-divisible. Then $L:=D-D^\ast$ is a subgraph of $G$ with $\Delta(L)\le \gamma^{-2}$ and has a $1$-well separated $F$-decomposition $\cF$ such that $\cF^{\le(r+1)}$ and $O$ are edge-disjoint. Moreover, $G-L=H\cup D^\ast$ is $F^\ast$-divisible.

We can deduce the following corollary from the case $F=\krq{r}{r}$ of Lemma~\ref{lem:make divisible}.

\begin{cor}\label{cor:make divisible typical}
Let $1/n\ll \gamma \ll \xi,1/f$ and $r\in[f-1]$.
Let $F$ be an $r$-graph on $f$ vertices.
Let $G$ be an $r$-graph on $n$ vertices such that for all $A\In\binom{V(G)}{r-1}$ with $|A|\le \binom{f-1}{r-1}$, we have $|\bigcap_{S\in A}G(S)|\ge \xi n$. Then there exists a subgraph $D\In G$ with $\Delta(D)\le \gamma^{-2}$ such that the following holds: for any $r$-graph $H$ on $V(G)$ which is edge-disjoint from $D$, there exists a subgraph $D^\ast\In D$ such that $H\cup D^\ast$ is $F$-divisible.
\end{cor}

In particular, using $H:=G-D$, there exists a subgraph $L:=D-D^\ast\In G$ with $\Delta(L)\le \gamma^{-2}$ such that $G-L=H\cup D^\ast$ is $F$-divisible.

\proof
Apply Lemma~\ref{lem:make divisible} with $F,\krq{r}{r}$ playing the roles of $F^\ast,F$.
\endproof

We now prove the following theorem, which immediately implies the case $\lambda=1$ of Theorem~\ref{thm:design}.

\begin{theorem}\label{thm:typical separated dec}
Let $1/n\ll \gamma, 1/\kappa \ll c, p, 1/f$ and $r\in[f-1]$, and
\begin{align}
c \le p^{h}/(q^r4^q), \mbox{ where }h:=2^r\binom{q+r}{r}\mbox{ and } q:=2f\cdot f!.\label{typical condition}
\end{align}
Let $F$ be any $r$-graph on $f$ vertices. Suppose that $G$ is a $(c,h,p)$-typical $F$-divisible $r$-graph on $n$ vertices. Let $O$ be an $(r+1)$-graph on $V(G)$ with $\Delta(O)\le \gamma n$.
Then $G$ has a $\kappa$-well separated $F$-decomposition $\cF$ such that $\cF^{\le(r+1)}$ and $O$ are edge-disjoint.
\end{theorem}

\proof
By Lemma~\ref{lem:regularisation}, there exists a weakly regular $r$-graph $F^\ast$ on $f^\ast\le q$ vertices which has a $1$-well separated $F$-decomposition.

By Lemma~\ref{lem:make divisible}\COMMENT{$\binom{f^\ast-1}{r-1}\le \binom{f^\ast}{r}\le 2^{r}\binom{2f\cdot f!}{r}\le h$} (with $0.5 p^{\binom{f^\ast-1}{r-1}}$ playing the role of $\xi$), there exists a subgraph $L\In G$ with $\Delta(L)\le \gamma^{-2}$ such that $G-L$ is $F^\ast$-divisible and $L$ has a $1$-well separated $F$-decomposition $\cF_{div}$ such that $\cF_{div}^{\le(r+1)}$ and $O$ are edge-disjoint. By Fact~\ref{fact:ws}\ref{fact:ws:maxdeg}, $\Delta(\cF_{div}^{\le(r+1)})\le f-r$.
Let $$G':=G^{\leftrightarrow}-L-\cF_{div}^{\le(r+1)}-O.$$
By Example~\ref{ex:typical super}, $G^{\leftrightarrow}$ is an $(\eps,\xi,f^\ast,r)$-supercomplex, where $\eps:=2^{f^\ast-r+1}c/(f^\ast-r)!$ and $\xi:=(1-2^{f^\ast+1}c)p^{2^r\binom{f^\ast+r}{r}}/{f^\ast!}$.
Observe that assumption \eqref{typical condition} now guarantees that $2(2\sqrt{\eul})^r \eps \le \xi$.\COMMENT{By \eqref{typical condition} we have $2^{q+r+3}c\le 4^q c\le 1$. In particular, $2^{f^\ast+1}c\le 2^{q+1}c\le 1/2$. It thus suffices to check that $2(2\sqrt{\eul})^r 2^{f^\ast-r+1}c/(f^\ast-r)!\le p^{2^r\binom{f^\ast+r}{r}}/{2f^\ast!}$ which certainly holds if $f^{\ast r} 4(2\sqrt{\eul})^r 2^{f^\ast-r+1}c\le p^{2^r\binom{f^\ast+r}{r}}$. Since $q\ge f^\ast$, this in turn certainly holds if $q^r 4(2\sqrt{\eul})^r 2^{q-r+1}c\le p^{2^r\binom{q+r}{r}}$. This holds by \eqref{typical condition} since $4(2\sqrt{\eul})^r 2^{q-r+1}\le 4^{1+r}2^{q-r+1}\le 2^{q+r+3}\le 4^q$.}
Thus, by Lemma~\ref{lem:boost complex}, $G^{\leftrightarrow}$ is a $(\gamma,\xi',f^\ast,r)$-supercomplex, where $\xi':=0.9(1/4)^{\binom{f^\ast +r}{r}}\xi$.
By Proposition~\ref{prop:noise}\ref{noise:supercomplex}, we have that $G'$ is a $(\sqrt{\gamma},\xi'/2,f^\ast,r)$-supercomplex. Moreover, $G'$ is $F^\ast$-divisible. Thus, by Theorem~\ref{thm:main complex}, $G'$ has a $(\kappa-1)$-well separated $F^\ast$-decomposition $\cF^\ast$. Since $F^\ast$ has a $1$-well separated $F$-decomposition, we can conclude that $G'$ has a $(\kappa-1)$-well separated $F$-decomposition $\cF_{complex}$. Let $\cF:=\cF_{div}\cup \cF_{complex}$. By Fact~\ref{fact:ws}\ref{fact:ws:1}, $\cF$ is a $\kappa$-well separated $F$-decomposition of $G$. Moreover, $\cF^{\le(r+1)}$ and $O$ are edge-disjoint.
\endproof

It remains to derive Theorem~\ref{thm:design} from Theorem~\ref{thm:typical separated dec} and Corollary~\ref{cor:make divisible typical}.

\lateproof{Theorem~\ref{thm:design}}
Choose a new constant $\kappa\in \bN$ such that $$1/n\ll \gamma \ll  1/\kappa \ll c, p,1/f.$$ 

Suppose that $G$ is a $(c,h,p)$-typical $(F,\lambda)$-divisible $r$-graph on $n$ vertices. Split $G$ into two subgraphs $G_1'$ and $G_2'$ which are both $(c+\gamma,h,p/2)$-typical (a standard Chernoff-type bound shows that \whp a random splitting of $G$ yields the desired property).

By Corollary~\ref{cor:make divisible typical} (applied with $G_2',0.5(p/2)^{\binom{f-1}{r-1}}$ playing the roles of $G,\xi$), there exists a subgraph $L^\ast\In G_2'$ with $\Delta(L^\ast)\le \kappa$ such that $G_2:=G_2'-L^\ast$ is $F$-divisible. Let $G_1:=G_1'\cup L^\ast =G-G_2$. Clearly, $G_1$ is still $(F,\lambda)$-divisible. By repeated applications of Corollary~\ref{cor:make divisible typical}, we can find edge-disjoint subgraphs $L_1,\dots,L_\lambda$ of $G_1$ such that $R_i:=G_1-L_i$ is $F$-divisible and $\Delta(L_i)\le \kappa $ for all $i\in[\lambda]$.
Indeed, suppose that we have already found $L_1,\dots,L_{i-1}$. Then $\Delta(L_1\cup \dots \cup L_{i-1})\le \lambda \kappa\le \gamma^{1/2} n$ (recall that $\lambda\le \gamma n$). Thus, by Corollary~\ref{cor:make divisible typical}\COMMENT{(with $G_1'-(L_1\cup \dots \cup L_{i-1})$ and $0.5(p/2)^{\binom{f-1}{r-1}} -\binom{f-1}{r-1}\gamma^{1/2}$ playing the roles of $G$ and $\xi$)}, there exists a subgraph $L_i\In G_1'-(L_1\cup \dots \cup L_{i-1})$ with $\Delta(L_i)\le \kappa $ such that $G_1-L_i$ is $F$-divisible.

Let $G_2'':=G_2 \cup L_1 \cup \dots \cup L_\lambda$. We claim that $G_2''$ is $F$-divisible. Indeed, let $S\In V(G)$ with $|S|\le r-1$. We then have that
$|G_2''(S)| 
= |G_2(S)|+\sum_{i\in[\lambda]}|(G_1-R_i)(S)|
            = |G_2(S)|+\lambda |G_1(S)| - \sum_{i\in[\lambda]}|R_i(S)| \equiv 0 \mod{Deg(F)_{|S|}}$.

Since $G_1'$ and $G_2'$ are both $(c+\gamma,h,p/2)$-typical and $\Delta(L^\ast \cup L_1 \cup \dots \cup L_\lambda)\le 2\gamma^{1/2} n $, we have that each of $G_2$, $G_2''$, $R_1,\dots,R_\lambda$ is $(c+\gamma^{1/3},h,p/2)$-typical (and they are $F$-divisible by construction).\COMMENT{If $H$ is $(c,h,p)$-typical and $\Delta(L)\le \gamma n$ with $V(L)=V(H)$, then $H\bigtriangleup L$ is $(c+hp^{-h}\gamma, h,p)$-typical.}

Using Theorem~\ref{thm:typical separated dec} repeatedly, we can thus find $\kappa$-well separated $F$-decompositions $\cF_1,\dots,\cF_{\lambda-1}$ of $G_2$, a $\kappa$-well separated $F$-decomposition $\cF^\ast$ of $G_2''$, and for each $i\in[\lambda]$, a $\kappa$-well separated $F$-decomposition $\cF_i'$ of $R_i$. Moreover, we can assume that all these decompositions are pairwise $(r+1)$-disjoint. Indeed, this can be achieved by choosing them successively: Let $O$ consist of the $(r+1)$-sets which are covered by the decompositions we have already found. Then by Fact~\ref{fact:ws}\ref{fact:ws:maxdeg} we have that $\Delta(O)\le 2\lambda \cdot \kappa(f-r) \le \gamma^{1/2} n$. Hence, using Theorem~\ref{thm:typical separated dec}, we can find the next $\kappa$-well separated $F$-decomposition which is $(r+1)$-disjoint from the previously chosen ones.

Then $\cF:=\cF^\ast\cup \bigcup_{i\in[\lambda-1]}\cF_i \cup \bigcup_{i\in[\lambda]} \cF_i'$ is the desired $(F,\lambda)$-design. Indeed, every edge of $G_1-(L_1\cup \dots \cup L_\lambda)$ is covered by each of $\cF_1',\dots,\cF_\lambda'$. For each $i\in[\lambda]$, every edge of $L_i$ is covered by $\cF^\ast$ and each of $\cF_1',\dots,\cF_{i-1}',\cF_{i+1}',\dots,\cF_\lambda'$. Finally, every edge of $G_2$ is covered by each of $\cF_1,\dots,\cF_{\lambda-1}$ and $\cF^\ast$.
\endproof

Finally, we also prove Theorem~\ref{thm:near optimal}, which is an immediate consequence of Theorem~\ref{thm:typical separated dec} and Corollary~\ref{cor:make divisible typical}.

\lateproof{Theorem~\ref{thm:near optimal}}
Apply Corollary~\ref{cor:make divisible typical} (with $G,0.5p^{\binom{f-1}{r-1}}$ playing the roles of $G,\xi$) to find a subgraph $L\In G$ with $\Delta(L)\le C$ such that $G-L$ is $F$-divisible. It is easy to see that $G-L$ is $(1.1c,h,p)$-typical. Thus, we can apply Theorem~\ref{thm:typical separated dec} to obtain an $F$-decomposition $\cF$ of $G-L$. 
\endproof

\section{Achieving divisibility}\label{sec:make divisible}

It remains to show that we can turn every $F$-divisible $r$-graph $G$ into an $F^\ast$-divisible $r$-graph $G'$ by removing a sparse $F$-decomposable subgraph of $G$, that is, to prove Lemma~\ref{lem:make divisible}. Note that in Lemma~\ref{lem:make divisible}, we do not need to assume that $F^\ast$ is weakly regular. On the other hand, our argument heavily relies on the assumption that $F^\ast$ is $F$-decomposable.\COMMENT{from an algebraic point of view, $F$-divisibility of $F^\ast$ could also be enough. I was to say that it is easy to see that $F$-divisibility of $F^\ast$ is necessary, but not sure anymore.}

We first sketch the argument. Let $F^\ast$ be $F$-decomposable, let $b_k:=Deg(F^\ast)_k$ and $h_k:=Deg(F)_k$. Clearly, we have $h_k\mid b_k$. First, consider the case $k=0$. Then $b_0=|F^\ast|$ and $h_0=|F|$. We know that $|G|$ is divisible by $h_0$. Let $0\le x<b_0$ be such that $|G|\equiv x \mod{b_0}$. Since $h_0$ divides $|G|$ and $b_0$, it follows that $x=ah_0$ for some $0\le a <b_0/h_0$. Thus, removing $a$ edge-disjoint copies of $F$ from $G$ yields an $r$-graph $G'$ such that $|G'|=|G|-ah_0\equiv 0 \mod{b_0}$, as desired. This will in fact be the first step of our argument.

We then proceed by achieving $Deg(G')_1\equiv 0\mod{b_1}$. Suppose that the vertices of $G'$ are ordered $v_1,\dots,v_n$. We will construct a \defn{degree shifter} which will fix the degree of $v_1$ by allowing the degree of $v_2$ to change, whereas all other degrees are unaffected (modulo $b_1$). Step by step, we will fix all the degrees from $v_1,\dots,v_{n-1}$. Fortunately, the degree of $v_n$ will then automatically be divisible by $b_1$. For $k>1$, we will proceed similarly, but the procedure becomes more intricate. It is in general impossible to shift degree from one $k$-set to another one without affecting the degrees of any other $k$-set. Roughly speaking, the degree shifter will contain a set of $2k$ special `root vertices', and the degrees of precisely $2^k$ $k$-subsets\COMMENT{only the non-degenerate ones} of this root set change, whereas all other $k$-degrees are unaffected (modulo $b_k$). This will allow us to fix all the degrees of $k$-sets in $G'$ except the ones inside some final $(2k-1)$-set, where we use induction on $k$ as well. Fortunately, the remaining $k$-sets will again automatically satisfy the desired divisibility condition (cf.~Lemma~\ref{lem:auto div}).

The proof of Lemma~\ref{lem:make divisible} divides into three parts. In the first subsection, we will construct the degree shifters. In the second subsection, we show on a very abstract level (without considering a particular host graph) how the shifting has to proceed in order to achieve overall divisibility.
Finally, we will prove Lemma~\ref{lem:make divisible} by embedding our constructed shifters (using Lemma~\ref{lem:rooted embedding}) according to the given shifting procedure.

\subsection{Degree shifters}

The aim of this subsection is to show the existence of certain $r$-graphs which we call degree shifters. They allow us to locally `shift' degree among the $k$-sets of some host graph $G$.

\begin{defin}[$\mathbf{x}$-shifter]
Let $1\le k<r$ and let $F,F^\ast$ be $r$-graphs.
Given an $r$-graph $T_k$ and distinct vertices $x^0_1, \dots, x^0_k  , x^1_{1},\dots, x^1_{k}$ of $T_k$, we say that $T_k$ is an \emph{$( x^0_1, \dots, x^0_k ,  x^1_{1},\dots, x^1_{k} )$-shifter with respect to $F,F^\ast$} if the following hold:
\begin{enumerate}[label=\rm{(SH\arabic*)}]
\item $T_k$ has a $1$-well separated $F$-decomposition $\cF$ such that for all $F'\in \cF$ and all $i\in[k]$, $|V(F')\cap\Set{x_i^0,x_i^1}|\le 1$;\label{shifter decomposable}
\item $|T_k(S)|\equiv 0\mod{Deg(F^\ast)_{|S|}}$ for all $S\In V(T_k)$ with $|S|<k$;\label{shifter low degrees}
\item for all $S\in \binom{V(T_k)}{k}$, \begin{align*}
	|T_k(S)| \equiv
	\begin{cases}
		(-1)^{\sum_{i \in [k]}z_i } Deg(F)_{k} \mod{ Deg(F^\ast)_{k} }& \text{if $S = \{ x_i^{z_i} : i \in [k] \}$, }\\
		0 \mod{Deg(F^\ast)_{k} }& \text{otherwise.}
	\end{cases}
\end{align*}\label{shifter degrees}
\end{enumerate}
\end{defin}

We will now show that such shifters exist. Ultimately, we seek to find them as rooted subgraphs in some host graph $G$. Therefore, we impose additional conditions which will allow us to apply Lemma~\ref{lem:rooted embedding}.

\begin{lemma}\label{lem:shifters exist}
Let $1\le k<r$, let $F,F^\ast$ be $r$-graphs and suppose that $F^\ast$ has a $1$-well separated $F$-decomposition $\cF$. Let $f^\ast:=|V(F^\ast)|$. There exists an $(  x^0_1, \dots, x^0_k  ,  x^1_{1},\dots, x^1_{k}  )$-shifter $T_k$ with respect to $F,F^\ast$ such that $T_k[X]$ is empty and $T_k$ has degeneracy at most $\binom{f^\ast-1}{r-1}$ rooted at $X$, where $X:=\Set{x^0_1, \dots, x^0_k ,x^1_{1},\dots, x^1_{k}}$.
\end{lemma}

In order to prove Lemma~\ref{lem:shifters exist}, we will first prove a multigraph version (Lemma~\ref{lem:multishifter}), which is more convenient for our construction. We will then recover the desired (simple) $r$-graph by applying an operation similar to the extension operator $\nabla_{(F,e_0)}$ defined in Section~\ref{subsec:canonical}. The difference is that instead of extending every edge to a copy of $F$, we will consider an $F$-decomposition of the multigraph shifter and then extend every copy of $F$ in this decomposition to a copy of $F^\ast$ (and then delete the original multigraph).

For a word $w=w_1\dots w_{k}\in \Set{0,1}^{k}$, let $|w|_0$ denote the number of $0$'s in $w$ and let $|w|_1$ denote the number of $1$'s in $w$. Let $W_e(k)$ be the set of words $w\in \Set{0,1}^k$ with $|w|_1$ being even, and let $W_o(k)$ be the set of words $w\in \Set{0,1}^k$ with $|w|_1$ being odd.

\begin{fact}\label{fact:words}
For every $k\ge 1$, $|W_e(k)|=|W_o(k)|=2^{k-1}$.\COMMENT{Proof by induction. We have $W_e(1)=\Set{0}$ and $W_o(1)=\Set{1}$. For $k>1$, it is easy to see that $|W_e(k)|=|W_e(k-1)|+|W_o(k-1)|$ (and same for $W_o(k)$).}
\end{fact}

\begin{lemma}\label{lem:multishifter}
Let $1\le k<r$ and let $F,F^\ast$ be $r$-graphs such that $F^\ast$ is $F$-decomposable.
Let $x^0_1, \dots, x^0_k  , x^1_{1},\dots, x^1_{k}$ be distinct vertices. There exists a multi-$r$-graph~$T_k^\ast$ which satisfies \ref{shifter decomposable}--\ref{shifter degrees}, except that $\cF$ does not need to be $1$-well separated. 
\end{lemma}

\proof
Let $\cS_k:=\binom{V(F)}{k}$. For every $S^\ast\in \cS_k$, we will construct a multi-$r$-graph $T_{k,S^\ast}$ such that $x^0_1, \dots, x^0_k  , x^1_{1},\dots, x^1_{k}\in V(T_{k,S^\ast})$ and
\begin{enumerate}[label=\rm{(sh\arabic*)}]
\item $T_{k,S^\ast}$ has an $F$-decomposition $\cF$ such that for all $F'\in \cF$ and all $i\in[k]$, $|V(F')\cap\Set{x_i^0,x_i^1}|\le 1$;\label{shifter dec special}
\item $|T_{k,S^\ast}(S)|\equiv 0\mod{Deg(F^\ast)_{|S|}}$ for all $S\In V(T_{k,S^\ast})$ with $|S|<k$;\label{shifter low degrees special}
\item for all $S\in \binom{V(T_{k,S^\ast})}{k}$, \begin{align*}
	|T_{k,S^\ast}(S)| \equiv
	\begin{cases}
		(-1)^{\sum_{i \in [k]}z_i } |F(S^\ast)| \mod{ Deg(F^\ast)_{k} }& \text{if $S = \{ x_i^{z_i} : i \in [k] \}$, }\\
		0 \mod{Deg(F^\ast)_{k} }& \text{otherwise.}
	\end{cases}
\end{align*}\label{shifter degrees special}
\end{enumerate}
Following from this, it easy to construct $T_k^\ast$ by overlaying the above multi-$r$-graphs $T_{k,S^\ast}$. Indeed, there are integers $(a_{S^\ast}')_{S^\ast\in \cS_k}$ such that $\sum_{S^\ast\in \cS_k}a_{S^\ast}'|F(S^\ast)|=Deg(F)_k$. Hence, there are positive\COMMENT{just add multiple of $Deg(F^\ast)_{k}$ to each $a_{S^\ast}'$ to make it positive} integers $(a_{S^\ast})_{S^\ast\in \cS_k}$ such that
\begin{align}
\sum_{S^\ast\in \cS_k}a_{S^\ast}|F(S^\ast)|\equiv Deg(F)_k \mod{Deg(F^\ast)_{k}}.\label{gcd representation}
\end{align}
Therefore, we take $T_k^\ast$ to be the union of $a_{S^\ast}$ copies of $T_{k,S^\ast}$ for each $S^\ast\in \cS_k$.
Then $T_k^\ast$ has the desired properties.

Let $S^\ast\in \cS_k$. It remains to construct $T_{k,S^\ast}$. Let $X_0:=\Set{x_1^0,\dots,x_k^0}$ and $X_1:=\Set{x_1^1,\dots,x_k^1}$. We may assume that $V(F^\ast)\cap(X_0\cup X_1)=\emptyset$. Let $\cF^\ast$ be an $F$-decomposition of $F^\ast$ and $F'\in \cF^\ast$. Let $X=\Set{x_1,\dots,x_k}\In V(F')$ be the $k$-set which plays the role of $S^\ast$ in $F'$, in particular $|F'(X)|=|F(S^\ast)|$. We first define an auxiliary $r$-graph $T_{1,x_k}$ as follows: Let $F''$ be obtained from $F'$ by replacing $x_k$ with a new vertex $\hat{x}_k$. Then let $$T_{1,x_k}:=(F^\ast-F')\cupdot F''.$$ Clearly, $(\cF^\ast\sm\Set{F'})\cup \Set{F''}$ is an $F$-decomposition of $T_{1,x_k}$. Moreover, observe that for every set $S\In V(T_{1,x_k})$ with $|S|<r$, we have
\begin{align}
|T_{1,x_k}(S)|&=\begin{cases} 0 & \mbox{if }\Set{x_k,\hat{x}_k}\In S;\\
                           |F^\ast(S)| & \mbox{if }\Set{x_k,\hat{x}_k}\cap S=\emptyset;\\
													 |F^\ast(S)|-|F'(S)| & \mbox{if }x_k\in S,\hat{x}_k\notin S;\\
													 |F''(S)|=|F'((S\sm\Set{\hat{x}_k})\cup\Set{x_k})|  & \mbox{if }x_k\notin S,\hat{x}_k\in S.\label{aux shifter degrees}
						\end{cases}
\end{align}
We now overlay copies of $T_{1,x_k}$ in a suitable way in order to obtain the multi-$r$-graph $T_{k,S^\ast}$.
The vertex set of $T_{k,S^\ast}$ will be $$V(T_{k,S^\ast})=(V(F^\ast)\sm X) \cupdot X_0 \cupdot X_1.$$ For every word $w=w_1\dots w_{k-1}\in \Set{0,1}^{k-1}$, let $T_w$ be a copy of $T_{1,x_k}$, where 
\begin{enumerate}[label=(\alph*)]
\item for each $i\in[k-1]$, $x_i^{w_i}$ plays the role of $x_i$ (and $x_i^{1-w_i}\notin V(T_w)$);\label{role playing projected}
\item if $|w|_1$ is odd, then $x_k^0$ plays the role of $x_k$ and $x_k^1$ plays the role of $ \hat{x}_k$, whereas if $|w|_1$ is even, then $x_k^0$ plays the role of $\hat{x}_k$ and $x_k^1$ plays the role of $x_k$;\label{role playing odd even}
\item the vertices in $V(T_{1,x_k})\sm \Set{x_1,\dots,x_{k-1},x_k,\hat{x}_k}$ keep their role.\label{role playing keep isolated}
\end{enumerate}
Let $$T_{k,S^\ast}:=\bigcup_{w\in\Set{0,1}^{k-1}}T_w.$$
(Note that if $k=1$, then $T_{k,S^\ast}$ is just a copy of $T_{1,x_k}$, where $x_1^0$ plays the role of $\hat{x}_1$ and $x_1^1$ plays the role of $x_1$.)
We claim that $T_{k,S^\ast}$ satisfies \ref{shifter dec special}--\ref{shifter degrees special}. Clearly, \ref{shifter dec special} is satisfied because each $T_w$ is a copy of $T_{1,x_k}$ which is $F$-decomposable, and for all $w\in\Set{0,1}^{k-1}$ and all $i\in[k-1]$, $|V(T_w)\cap \Set{x_i^0,x_i^1}|=1$, and since $x_k\notin V(F'')$.

We will now use \eqref{aux shifter degrees} in order to determine an expression for $|T_{k,S^\ast}(S)|$ (see~\eqref{shifter degree formula}) which will imply \ref{shifter low degrees special} and \ref{shifter degrees special}. Call $S\In V(T_{k,S^\ast})$ \defn{degenerate} if $\Set{x_i^{0},x_i^1}\In S$ for some $i\in[k]$. Clearly, if $S$ is degenerate, then $|T_{w}(S)|=0$ for all $w\in\Set{0,1}^{k-1}$.
If $S\In V(T_{k,S^\ast})$ is non-degenerate, define $I(S)$ as the set of all indices $i\in[k]$ such that $|S\cap \Set{x_i^0,x_i^1}|=1$, and define the `projection'
\begin{align*}
\pi(S)&:=(S\sm (X_0\cup X_1))\cup \set{x_i}{i\in I(S)}.
\end{align*}
Clearly, $\pi(S)\In V(F^\ast)$ and $|\pi(S)|=|S|$. Note that if $S\In V(T_w)$ and $k\notin I(S)$, then $S$ plays the role of $\pi(S)\In V(T_{1,x_k})$ in $T_w$ by~\ref{role playing projected}.
For $i\in I(S)$, let $z_i(S)\in\Set{0,1}$ be such that $S\cap \Set{x_i^0,x_i^1}=\Set{x_i^{z_i(S)}}$, and let $z(S):=\sum_{i\in I(S)}z_i(S)$. We claim that the following holds:
\begin{align}
|T_{k,S^\ast}(S)|&\equiv \begin{cases} (-1)^{z(S)}|F'(\pi(S))| \mod{Deg(F^\ast)_{|S|}}  &  \mbox{if }S\mbox{ is non-degenerate}\\
                                            & \mbox{and }|I(S)|=k;\\
                                0 \mod{Deg(F^\ast)_{|S|}} & \mbox{otherwise.}
					\end{cases}\label{shifter degree formula}
\end{align}

As seen above, if $S$ is degenerate, then we have $|T_{k,S^\ast}(S)|=0$. From now on, we assume that $S$ is non-degenerate.
Let $W(S)$ be the set of words $w=w_1\dots w_{k-1}\in \Set{0,1}^{k-1}$ such that $w_i=z_i(S)$ for all $i\in I(S)\sm\Set{k}$.
Clearly, if $w\in\Set{0,1}^{k-1}\sm W(S)$, then $|T_w(S)|=0$ by \ref{role playing projected}.\COMMENT{there exists $i\in I(S)\sm\Set{k}$ with $S\cap \Set{x_i^0,x_i^1}\neq\Set{x_i^{w_i}}$, i.e. $x_i^{1-w_i}\in S$. But $x_i^{1-w_i}$ is not contained in $V(T_w)$.}
Suppose that $w\in W(S)$. If $k\notin I(S)$, then $S$ plays the role of $\pi(S)$ in $T_w$ and hence we have $|T_w(S)|=|T_{1,x_k}(\pi(S))|=|F^{\ast}(\pi(S))|$ by \eqref{aux shifter degrees}. It follows that $|T_{k,S^\ast}(S)|\equiv 0 \mod{Deg(F^\ast)_{|S|}}$, as required.

From now on, suppose that $k\in I(S)$.
Let
\begin{align*}
W_e(S)&:=\set{w\in W(S)}{|w|_1+z_k(S)\mbox{ is even}};\\
W_o(S)&:=\set{w\in W(S)}{|w|_1+z_k(S)\mbox{ is odd}}.
\end{align*}
By \ref{role playing odd even}, we know that $x_k^{z_k(S)}$ plays the role of $x_k$ in $T_w$ if $w\in W_o(S)$ and the role of $\hat{x}_k$ if $w\in W_e(S)$. Hence, if $w\in W_o(S)$ then $S$ plays the role of $\pi(S)$ in $T_w$, and if $w\in W_e(S)$, then $S$ plays the role of $(\pi(S)\sm\Set{x_k})\cup \Set{\hat{x}_k}$ in $T_w$. Thus, we have
\begin{align*}
|T_w(S)|&=\begin{cases} |T_{1,x_k}(\pi(S))|\overset{\eqref{aux shifter degrees}}{=}|F^\ast(\pi(S))|-|F'(\pi(S))|        & \mbox{if }w\in W_o(S);\\
                        |T_{1,x_k}((\pi(S)\sm\Set{x_k})\cup \Set{\hat{x}_k})|\overset{\eqref{aux shifter degrees}}{=}|F'(\pi(S))|  &  \mbox{if }w\in W_e(S);\\
												0 & \mbox{if }w\notin W(S).
					\end{cases}
\end{align*}
It follows that
\begin{align*}
|T_{k,S^\ast}(S)|=\sum_{w\in\Set{0,1}^{k-1}}|T_w(S)|\equiv (|W_e(S)|-|W_o(S)|)|F'(\pi(S))| \mod{Deg(F^\ast)_{|S|}}.
\end{align*}
Observe that
\begin{align*}
|W_e(S)|&=|\set{w'\in\Set{0,1}^{k-|I(S)|}}{|w'|_1+z(S) \mbox{ is even}}|;\\
|W_o(S)|&=|\set{w'\in\Set{0,1}^{k-|I(S)|}}{|w'|_1+z(S) \mbox{ is odd}}|.
\end{align*}
Hence, if $|I(S)|<k$, then by Fact~\ref{fact:words} we have $|W_e(S)|=|W_o(S)|=2^{k-|I(S)|-1}$. If $|I(S)|=k$, then $|W_e(S)|=1$ if $z(S)$ is even and $|W_e(S)|=0$ if $z(S)$ is odd, and for $W_o(S)$, the reverse holds. Altogether, this implies \eqref{shifter degree formula}.

It remains to show that \eqref{shifter degree formula} implies \ref{shifter low degrees special} and \ref{shifter degrees special}. Clearly, \ref{shifter low degrees special} holds. Indeed, if $|S|<k$, then $S$ is degenerate or we have $|I(S)|<k$, and \eqref{shifter degree formula} implies that $|T_{k,S^\ast}(S)|\equiv 0 \mod{Deg(F^\ast)_{|S|}}$.

Finally, consider $S\in \binom{V(T_{k,S^\ast})}{k}$. If $S$ does not have the form $ \{ x_i^{z_i} : i \in [k] \}$ for suitable $z_1,\dots,z_k\in\Set{0,1}$, then $S$ is degenerate or $|I(S)|<k$ and \eqref{shifter degree formula} implies that $|T_{k,S^\ast}(S)|\equiv 0 \mod{Deg(F^\ast)_{k}}$, as required. Assume now that $S= \{ x_i^{z_i} : i \in [k] \}$ for suitable $z_1,\dots,z_k\in\Set{0,1}$. Then $S$ is not degenerate, $I(S)=[k]$, $z(S)=\sum_{i\in[k]}z_i$ and $\pi(S)=\Set{x_1,\dots,x_{k}}=X$, in which case \eqref{shifter degree formula} implies that $$|T_{k,S^\ast}(S)|\equiv (-1)^{z(S)}|F'(X)|=(-1)^{z(S)}|F(S^\ast)| \mod{Deg(F^\ast)_k},$$ as required for \ref{shifter degrees special}.
\endproof

\lateproof{Lemma~\ref{lem:shifters exist}}
By applying Lemma~\ref{lem:multishifter} (with $x_k^0$ and $x_k^1$ swapping their roles), we can see that there exists a multi-$r$-graph $T_k^\ast$ with $x^0_1, \dots, x^0_k  , x^1_{1},\dots, x^1_{k}\in V(T_k^\ast)$ such that the following properties hold:
\begin{enumerate}[label=\rm{\textbullet}]
\item $T_k^\ast$ has an $F$-decomposition $\Set{F_1,\dots,F_m}$ such that for all $j\in[m]$ and all $i\in[k]$, we have $|V(F_j)\cap \Set{x_i^0,x_i^1}|\le 1 $;\label{shifter decomposable new}
\item $|T_k^\ast(S)|\equiv 0\mod{Deg(F^\ast)_{|S|}}$ for all $S\In V(T_k^\ast)$ with $|S|<k$;
\item for all $S\in \binom{V(T_k^\ast)}{k}$, \begin{align*}
	|T_k^\ast(S)| \equiv
	\begin{cases}
		(-1)^{\sum_{i \in [k-1]}z_i +(1-z_k)} Deg(F)_{k} \mod{ Deg(F^\ast)_{k} }& \text{if $S = \{ x_i^{z_i} : i \in [k] \}$, }\\
		0 \mod{Deg(F^\ast)_{k} }& \text{otherwise.}
	\end{cases}
\end{align*}\label{shifter degrees new}
\end{enumerate}

Let $f:=|V(F)|$. For every $j\in[m]$, let $Z_{j}$ be a set of $f^\ast-f$ new vertices, such that $Z_{j}\cap Z_{j'}=\emptyset$ for all distinct $j,j'\in[m]$ and $Z_{j}\cap V(T_k^\ast)=\emptyset$ for all $j\in[m]$. Now, for every $j\in[m]$, let $F^\ast_{j}$ be a copy of $F^\ast$ on vertex set $V(F_j)\cup Z_{j}$ such that $\cF_j\cup \Set{F_j}$ is a $1$-well separated $F$-decomposition of $F^\ast_j$. In particular, we have that
\begin{enumerate}[label=\rm{(\alph*)}]
\item $(F^\ast_{j}-F_j)[V(F_j)]$ is empty;\label{F-extension 1}
\item $\cF_j$ is a $1$-well separated $F$-decomposition of $F^\ast_j-F_j$ such that for all $F'\in \cF_j$, $|V(F')\cap V(F_j)|\le r-1$.\label{F-extension 2}
\end{enumerate}
Let $$T_k:=\mathop{\dot{\bigcup}}_{j\in[m]}(F^\ast_j-F_j).$$

We claim that $T_k$ is the desired shifter. First, observe that $T_k$ is a (simple) $r$-graph since $(F^\ast_{j}-F_j)[V(F_j)]$ is empty for every $j\in[m]$ by \ref{F-extension 1}. Moreover, since $\cF_1,\dots,\cF_m$ are $r$-disjoint by \ref{F-extension 2}, Fact~\ref{fact:ws}\ref{fact:ws:2} implies that $\cF:=\cF_1\cup \dots \cup \cF_m$ is a $1$-well separated $F$-decomposition of $T_k$, and for each $j\in[m]$, all $F'\in \cF_j$ and all $i\in[k]$, we have $|V(F')\cap \Set{x_i^0,x_i^1}|\le |V(F_j)\cap \Set{x_i^0,x_i^1}|\le 1 $.
Thus, \ref{shifter decomposable} holds.

Moreover, note that for every $j\in[m]$, we have $|(F^\ast_j-F_j)(S)|\equiv -|F_j(S)| \mod{Deg(F^\ast)_{|S|}}$ for all $S\In V(T_k)$ with $|S|\le r-1$. Thus, $$|T_k(S)|\equiv \sum_{j\in[m]}-|F_j(S)| = -|T_k^\ast(S)| \mod{Deg(F^\ast)_{|S|}}$$ for all $S\In V(T_k)$ with $|S|\le r-1$. Hence, \ref{shifter low degrees} clearly holds. If $S=\set{x_i^{z_i}}{i\in[k]}$ for suitable $z_1,\dots,z_k\in\Set{0,1}$, then $$|T_k(S)|\equiv -|T_k^\ast(S)| \equiv (-1)^{\sum_{i \in [k]}z_i} Deg(F)_{k} \mod{Deg(F^\ast)_{k}}$$ and \ref{shifter degrees} holds. Thus, $T_k$ is indeed an $(  x^0_1, \dots, x^0_k  ,  x^1_{1},\dots, x^1_{k}  )$-shifter with respect to~$F,F^\ast$.

Finally, to see that $T_k$ has degeneracy at most $\binom{f^\ast-1}{r-1}$ rooted at $X$, consider the vertices of $V(T_k)\sm X$ in an ordering where the vertices of $V(T_k^\ast)\sm X$ precede all the vertices in sets $Z_j$, for $j\in[m]$. Note that $T_k[V(T_k^\ast)]$ is empty by \ref{F-extension 1}, i.e.~a vertex in $V(T_k^\ast)\sm X$ has no `backward' edges. Moreover, if $z\in Z_j$ for some $j\in[m]$, then $|T_k(\Set{z})|=|F^\ast_j(\Set{z})|\le \binom{f^\ast-1}{r-1}$.
\endproof

\subsection{Shifting procedure}
In the previous section, we constructed degree shifters which allow us to locally change the degrees of $k$-sets in some host graph. We will now show how to combine these local shifts in order to transform any given $F$-divisible $r$-graph $G$ into an $F^\ast$-divisible $r$-graph. It turns out to be more convenient to consider the shifting for `$r$-set functions' rather than $r$-graphs. We will then recover the graph theoretical statement by considering a graph as an indicator set function (see below).

Let $\phi: \binom{V}{r} \rightarrow \mathbb{Z}$. (Think of $\phi$ as the multiplicity function of a multi-$r$-graph.)
We extend $\phi$ to $\phi : \bigcup_{ k \in [r]_0 } \binom{V}{k} \rightarrow \mathbb{Z}$ by defining for all $S \subseteq V$ with $|S| = k \le r$,
\begin{align}
	\phi (S) := \sum_{S' \in \binom{V}{r} : S \subseteq S'} \phi (S').\label{function extension to lower sets}
\end{align}
Thus for all $0\le i \le k \le r$ and all $S \in \binom{V}{i}$,
\begin{align}
	\binom{r-i  }{k-i} \phi (S) = \sum_{S' \in \binom{V}{k} : S \subseteq S'} \phi (S').\label{handshaking for functions}
\end{align}

For $k \in [r-1]_0$ and $b_0,\dots,b_k \in \mathbb{N}$, we say that $\phi$ is \defn{$(b_0, \dots, b_k)$-divisible} if $b_{|S|} \mid \phi(S)$ for all $S\In V$ with $|S|\le k$.

If $G$ is an $r$-graph with $V(G)\In V$, we define $\mathds{1}_G\colon \binom{V}{r}\to \bZ$ as
$$\mathds{1}_G(S):=\begin{cases} 1 & \mbox{if }S\in G;\\ 0 & \mbox{if }S\notin G.\end{cases}$$
and extend $\mathds{1}_G$ as in \eqref{function extension to lower sets}.
Hence, for a set $S\In V$ with $|S|<r$, we have $\mathds{1}_G(S)=|G(S)|$. Thus, \eqref{handshaking for functions} corresponds to the handshaking lemma for $r$-graphs (cf.~\eqref{handshaking}).
Clearly, if $G$ and $G'$ are edge-disjoint, then we have $\mathds{1}_G+\mathds{1}_{G'}=\mathds{1}_{G\cup G'}$.
Moreover, for an $r$-graph $F$, $G$ is $F$-divisible if and only if $\mathds{1}_G$ is $(Deg(F)_0,\dots,Deg(F)_{r-1})$-divisible.

As mentioned before, our strategy is to successively fix the degrees of $k$-sets until we have fixed the degrees of all $k$-sets except possibly the degrees of those $k$-sets contained in some final vertex set $K$ which is too small as to continue with the shifting. However, as the following lemma shows, divisibility is then automatically satisfied for all the $k$-sets lying inside $K$. For this to work it is essential that the degrees of all $i$-sets for $i<k$ are already fixed.

\begin{lemma} \label{lem:auto div}
Let $1\le k<r$ and $b_0, \dots, b_k \in \mathbb{N}$ be such that $\binom{r-i}{k-i} b_i \equiv 0 \mod{b_{k}}$ for all $i \in [k]_0$.
Let $\phi: \binom{V}{r} \rightarrow \mathbb{Z}$ be a $(b_0, \dots, b_{k-1})$-divisible function.
Suppose that there exists a subset $K \subseteq V$ of size~$2k-1$ such that if $S \in \binom{V}{k}$ with $\phi(S) \not \equiv 0 \mod{b_k}$, then $S \subseteq K$.
Then $\phi$ is $(b_0, \dots, b_{k})$-divisible.
\end{lemma}

\proof
Let $\cK$ be the set of all subsets $T''$ of $K$ of size less than $k$. We first claim that for all $T''\in \cK$, we have
\begin{align}
\sum_{T'\in \binom{K}{k}\colon T''\In T'}\phi(T')\equiv 0\mod{b_k}.\label{support sets zero}
\end{align}
Indeed, suppose that $|T''|=i<k$, then we have
\begin{align*}
\sum_{T'\in \binom{K}{k}\colon T''\In T'}\phi(T') \equiv \sum_{T'\in \binom{V}{k}\colon T''\In T'}\phi(T')\overset{\eqref{handshaking for functions}}{=}\binom{r-i}{k-i}\phi(T'') \mod{b_k}.
\end{align*}
Since $\phi$ is $(b_0,\dots,b_{k-1})$-divisible, we have $\phi(T'')\equiv 0\mod{b_i}$, and since $\binom{r-i}{k-i} b_i \equiv 0 \mod{b_{k}}$, the claim follows.

Let $T\in \binom{K}{k}$. We need to show that $\phi(T)\equiv 0\mod{b_k}$. To this end, define the function $f\colon \cK\to \bZ$ as $$f(T''):=\begin{cases} (-1)^{|T''|} & \mbox{if }T''\In K\sm T;\\ 0 & \mbox{otherwise.}\end{cases}$$

We claim that for all $T'\in \binom{K}{k}$, we have
\begin{align}
\sum_{T''\subsetneq T'}f(T'')=\begin{cases} 1 & \mbox{if }T'=T;\\ 0 & \mbox{otherwise.}\end{cases}\label{counting function}
\end{align}
Indeed, let $T'\in \binom{K}{k}$, and set $t:=|T'\sm T|$. We then check that (using $|K|<2k$ in the first equality)
$$\sum_{T''\subsetneq T'}f(T'')=\sum_{T''\In (K\sm T)\cap T'}(-1)^{|T''|}= \sum_{j=0}^{t}(-1)^{j}\binom{t}{j} =\begin{cases} 1 & \mbox{if }t=0;\\ 0 & \mbox{if }t>0.\end{cases}$$
\COMMENT{if $t=0$ then $T'=T$ and thus $\sum_{T''\subsetneq T'}f(T'')=f(\emptyset)=1$}
We can now conclude that
\begin{align*}
\phi(T)&\overset{\eqref{counting function}}{=}\sum_{T'\in \binom{K}{k}}\phi(T')\sum_{T''\subsetneq T'}f(T'')= \sum_{T''\in \cK}f(T'')\left(\sum_{T'\in \binom{K}{k}\colon T''\In T'}\phi(T')\right)\overset{\eqref{support sets zero}}{\equiv} 0\mod{b_k},
\end{align*}
as desired.
\endproof

We now define a more abstract version of degree shifters, which we call adapters. They represent the effect of shifters and will finally be replaced by shifters again.

\begin{defin}[$\mathbf{x}$-adapter]\label{def:adapter}
Let $V$ be a vertex set and $k,r,b_0,\dots,b_k,h_k\in \bN$ be such that $k<r$ and $h_k\mid b_k$.
For distinct vertices $x^0_1, \dots, x^0_{k}, x^1_{1},\dots, x^1_{k} $ in $V$, we say that $\tau\colon \binom{V}{r}\to \bZ$ is an \emph{$( x^0_1, \dots, x^0_k  ,  x^1_{1},\dots, x^1_{k} )$-adapter with respect to $(b_0,\dots,b_k;h_k)$} if $\tau$ is $(b_0,\dots,b_{k-1})$-divisible and for all $S \in \binom{V}{k}$,
\begin{align*}
	\tau(S) \equiv
	\begin{cases}
		(-1)^{\sum_{i \in [k]}z_i }h_k \mod{ b_k }& \text{if $S = \{ x_i^{z_i} : i \in [k] \}$, }\\
		0 \mod{ b_k }& \text{otherwise.}
	\end{cases}
\end{align*}
\end{defin}

Note that such an adapter $\tau$ is $(b_0, \dots, b_{k-1},h_k)$-divisible.

\begin{fact}\label{fact:shifter is adapter}
If $T$ is an $\mathbf{x}$-shifter with respect to $F,F^*$, then $\mathds{1}_T$ is an $\mathbf{x}$-adapter with respect to $(Deg(F^\ast)_0,\dots,Deg(F^\ast)_k;Deg(F)_k)$.
\end{fact}

The following definition is crucial for the shifting procedure. Given some function $\phi$, we intend to add adapters in order to obtain a divisible function. Every adapter is characterised by a tuple $\mathbf{x}$ consisting of $2k$ distinct vertices, which tells us where to apply the adapter. All these tuples are contained within a multiset $\Omega$, which we call a balancer. $\Omega$ is capable of dealing with any input function $\phi$ in the sense that there is a multisubset of $\Omega$ which tells us where to apply the adapters in order to make $\phi$ divisible. Moreover, as we finally want to replace the adapters by shifters (and thus embed them into some host graph), there must not be too many of them.

\begin{defin}[balancer]\label{def:balancer}
Let $r,k,b_0,\dots,b_k\in \bN$ with $k<r$ and let $U,V$ be sets with $U\In V$. Let $\Omega_{k}$ be a multiset containing ordered tuples $\mathbf{x}=(x_1,\dots,x_{2k})$, where $x_1,\dots,x_{2k}\in U$ are distinct. We say that $\Omega_k$ is a \defn{$(b_0, \dots, b_{k})$-balancer for $V$ with uniformity $r$ acting on $U$} if for any $h_k\in \bN$ with $h_k\mid b_k$, the following holds: let $\phi\colon \binom{V}{r} \rightarrow \mathbb{Z}$ be any $(b_0, \dots, b_{k-1},h_k)$-divisible function such that $S\In U$ whenever $S\in \binom{V}{k}$ and $\phi(S)\not\equiv 0\mod{b_k}$. There exists a multisubset $\Omega'$ of $\Omega_{k}$ such that $\phi + \tau_{\Omega'}$ is $(b_0, \dots, b_k)$-divisible, where $\tau_{\Omega'} := \sum_{\mathbf{x} \in \Omega'} \tau_{\mathbf{x}}$ and $\tau_{\mathbf{x}}$ is any $\mathbf{x}$-adapter with respect to $(b_0,\dots,b_k;h_k)$.

For a set $S \in \binom{V}{k}$, let $\deg_{\Omega_{k}}(S)$ be the number of $\mathbf{x} = ( x_1, \dots, x_{2k} ) \in \Omega_k$ such that $|S \cap \{x_i,x_{i+k}\}| =1$ for all $i \in [k]$.
Furthermore, we denote $\Delta (\Omega_{k})$ to be the maximum value of $\deg_{\Omega_{k}}(S)$ over all $S \in \binom{V}{k}$.
\end{defin}

The following lemma shows that these balancers exist, i.e.~that the local shifts performed by the degree shifters guaranteed by Lemma~\ref{lem:shifters exist} are sufficient to obtain global divisibility (for which we apply Lemma~\ref{lem:auto div}).

\begin{lemma} \label{lem:balancer}
Let $1 \le k < r $.
Let $b_0, \dots, b_k \in \mathbb{N}$ be such that $\binom{r-s}{k-s} b_s \equiv 0 \mod{b_{k}}$ for all $s \in [k]_0$.
Let $U$ be a set of $n \ge 2k$ vertices and $U\In V$.\COMMENT{$\ge 2k$ technically not needed, if $n$ is smaller, $\Omega_k$ must be empty, but is still a balancer since automatic divisibility guarantees $(b_0,\dots,b_k)$-divisibility}
Then there exists a $(b_0, \dots, b_{k})$-balancer~$\Omega_k$ for~$V$ with uniformity $r$ acting on $U$ such that $\Delta( \Omega_k) \le 2^{k}  (k!)^2b_k$.
\end{lemma}

\begin{proof}
We will proceed by induction on $k$. First, consider the case when $k=1$. Write $U= \{v_1, \dots, v_n\}$.
Define $\Omega_1$ to be the multiset containing precisely $b_1-1$ copies of $(v_j,v_{j+1})$ for all $j \in [n-1]$.
Note that $\Delta (\Omega_1) \le 2 b_1$.

We now show that $\Omega_1$ is a $(b_0,b_1)$-balancer for $V$ with uniformity $r$ acting on $U$.
Let $\phi: \binom{V}{r} \rightarrow \mathbb{Z}$ be $(b_0,h_1)$-divisible for some $h_1\in \bN$ with $h_1\mid b_1$, such that $v\in U$ whenever $v\in V$ and $\phi(\Set{v})\not\equiv 0\mod{b_1}$.
Let $m_0 := 0$.
For each $j \in [n-1]$, let $0 \le m_j < b_1$ be such that $(m_{j-1} - m_j) h_1 \equiv \phi (\Set{v_j})  \mod{ b_1 }$.
Let $\Omega' \subseteq \Omega_1$ consist of precisely $m_j$ copies of $(v_{j}, v_{j+1})$ for all $j \in [n -1]$.
Let $\tau : = \sum_{\mathbf{x} \in \Omega'} \tau_{\mathbf{x}}$, where $\tau_{\mathbf{x}}$ is an $\mathbf{x}$-adapter with respect to $(b_0,b_1;h_1)$, and let $\phi':=\phi+\tau$. Clearly, $\phi'$ is $(b_0)$-divisible.
Note that, for all $j \in [n-1]$,\COMMENT{The case $j=1$ is ok since $m_0=0$ and so we don't actually have to consider $\tau_{(v_{0},v_1)} $}
\begin{align}
		\tau (\Set{v_j}) & \equiv m_{j-1}  \tau_{(v_{j-1},v_j)} (\Set{v_j}) + m_j \tau_{(v_{j},v_{j+1})}(\Set{v_j})\mod{ b_1 }  \nonumber \\
		& \equiv (-m_{j-1} + m_j) h_1
		\equiv  - \phi (\Set{v_j}) \mod{ b_1 }, \label{eqn:phi1}
\end{align}
implying that $\phi'(\Set{v_j})\equiv 0\mod{b_1}$ for all $j\in[n-1]$. Moreover, for all $v\in V\sm U$, we have $\phi(\Set{v})\equiv 0 \mod{b_1}$ by assumption and $\tau(\Set{v})\equiv 0 \mod{b_1}$ since no element of $\Omega_1$ contains $v$. Thus, by Lemma~\ref{lem:auto div} (with $\Set{v_n}$ playing the role of $K$), $\phi'$ is $(b_0,b_1)$-divisible, as required.

We now assume that $k >1$ and that the statement holds for smaller $k$. Again, write $U= \{v_1, \dots, v_n\}$.
For every $\ell \in [n]$, let $U_{\ell} := \{v_j: j \in [\ell]\}$. We construct $\Omega_k$ inductively.
For each $\ell \in \Set{2k,\dots,n}$, we define a multiset $\Omega_{k,\ell}$ as follows.
Let $\Omega_{k-1, \ell-1}$ be a $(b_1, \dots,b_{k} )$-balancer for $V\sm \Set{v_\ell}$ with uniformity $r-1$ acting on $U_{\ell-1}$ and
\begin{align*}
\Delta( \Omega_{k-1, \ell-1}) \le 2^{k-1}  (k-1)!^2 b_k.
\end{align*}
(Indeed, $\Omega_{k-1, \ell-1}$ exists by our induction hypothesis with $r-1,k-1,b_1, \dots,b_{k},U_{\ell-1},V\sm \Set{v_\ell}$ playing the roles of $r,k,b_0, \dots, b_k,U,V$.)
For each $\mathbf{v} = (v_{j_1},\dots,v_{j_{2k-2}}) \in\Omega_{k-1, \ell-1}$, let
\begin{align}
				\mathbf{v'} := (v_{\ell}, v_{j_1}, \dots, v_{j_{k-1}} , v_{j_{\mathbf{v}}}, v_{j_{k}}, \dots, v_{j_{2k-2}} ) \in U_{\ell}\times U_{\ell-1}^{2k-1},\label{shift extension}
\end{align}
such that $j_{\mathbf{v}}\in \Set{\ell-2k+1,\dots,\ell} \sm \Set{\ell,j_{1},\dots,j_{2k-2}}$ (which exists since $\ell\ge 2k$). We let $\Omega_{k, \ell}:=\set{\mathbf{v'}}{\mathbf{v} \in \Omega_{k-1,\ell-1}}$. Now, define $$\Omega_k:=\bigcup_{\ell=2k}^n\Omega_{k, \ell}.$$

\begin{NoHyper}\begin{claim}
$\Delta (\Omega_k) \le  2^{k}  (k!)^2b_k$
\end{claim}\end{NoHyper}

\claimproof
Consider any $S \in \binom{V}{k}$. Clearly, if $S\not\In U$, then $\deg_{\Omega_{k}}(S)=0$, so assume that $S\In U$. Let $i_0$ be the largest~$i \in [n]$ such that $v_i \in S$.

First note that for all $\ell\in \Set{2k,\dots,n}$, we have $$\deg_{\Omega_{k, \ell}}(S)\le \sum_{v\in S} \deg_{\Omega_{k-1,\ell-1}}(S\sm\Set{v}) \le k\Delta (\Omega_{k-1, \ell-1}).$$
On the other hand, we claim that if $\ell<i_0$ or $\ell \ge i_0+2k$, then $\deg_{\Omega_{k, \ell}}(S)=0$. Indeed, in the first case, we have $S\not\In U_\ell$ which clearly implies that $\deg_{\Omega_{k, \ell}}(S)=0$. In the latter case, for any $\mathbf{v} \in\Omega_{k-1, \ell-1}$, we have $j_{\mathbf{v}}\ge \ell-2k+1 >i_0$ and thus $|S\cap \Set{v_\ell,v_{j_{\mathbf{v}}}}|=0$, which also implies $\deg_{\Omega_{k, \ell}}(S)=0$.
Therefore,
\begin{align*}
\deg_{\Omega_{k}}(S)=\sum_{\ell=2k}^n \deg_{\Omega_{k,\ell}}(S) \le 2k^2  \Delta (\Omega_{k-1, \ell-1}) \le 2^{k}  (k!)^2b_k,
\end{align*}
as required.
\endclaimproof

We now show that $\Omega_k$ is indeed a $(b_0, \dots, b_k)$-balancer on $V$ with uniformity $r$ acting on~$U$. The key to this is the following claim, which we will apply repeatedly.

\begin{claim} \label{clm:Omega}
Let $2k \le \ell \le n$.
Let $\phi_{\ell}\colon \binom{V}{r} \rightarrow \mathbb{Z}$ be any $(b_0, \dots, b_{k-1},h_k)$-divisible function for some $h_k \in \bN$ with $h_k\mid b_k$.
Suppose that if $\phi_{\ell}(S) \not\equiv 0 \mod{b_k}$ for some $S\in \binom{V}{k}$, then $S \subseteq U_{\ell}$.
Then there exists $\Omega_{k,\ell}' \subseteq \Omega_{k,\ell}$ such that $\phi_{\ell-1}:=\phi_{\ell} + \tau_{\Omega_{k,\ell}'}$ is $(b_0, \dots, b_{k-1},h_k)$-divisible and if $\phi_{\ell-1}(S) \not\equiv 0 \mod{b_k}$ for some $S\in \binom{V}{k}$, then $S \subseteq U_{\ell-1}$.
\end{claim}

(Here, $ \tau_{\Omega_{k,\ell}'}$ is as in Definition~\ref{def:balancer}, i.e.~$ \tau_{\Omega_{k,\ell}'}:=\sum_{\mathbf{v'}\in \Omega_{k,\ell}'}\tau_{\mathbf{v'}}$ and $\tau_{\mathbf{v'}}$ is an arbitrary $\mathbf{v'}$-adapter with respect to $(b_0,\dots,b_k;h_k)$.)

\claimproof
Define $\rho : \binom{ V\sm\Set{v_\ell}  }{ r - 1 } \rightarrow \mathbb{Z}$ such that for all $S \in \binom{ V\sm\Set{v_\ell}}{r-1}$,
\begin{align*}
\rho (S) := \phi_\ell(S \cup \Set{v_\ell}  ) .
\end{align*}
It is easy to check that this identity transfers to smaller sets $S$, that is, for all $S\In V\sm\Set{v_\ell}$, with $|S|\le r-1$, we have $\rho(S)=\phi_\ell(S \cup \Set{v_\ell}  ) $, where $\rho(S)$ and $\phi_\ell(S \cup \Set{v_\ell}  )$ are as defined in~\eqref{function extension to lower sets}.\COMMENT{$$\rho(S)=\sum_{S' \in \binom{V\sm\Set{v_\ell}}{r-1} : S \subseteq S'} \rho (S')=\sum_{S' \in \binom{V\sm\Set{v_\ell}}{r-1} : S \subseteq S'} \phi_\ell(S' \cup \Set{v_\ell}  )=\sum_{S'' \in \binom{V}{r} : S \cup \Set{v_\ell}\subseteq S''} \phi_\ell(S'')=\phi_\ell(S\cup \Set{v_\ell})$$}

Hence, since $\phi_\ell$ is $(b_0, \dots, b_{k-1},h_k)$-divisible, $\rho$ is $(b_1, \dots, b_{k-1},h_k)$-divisible. Moreover, for all $S\in \binom{V\sm\Set{v_\ell}}{k-1}$ with $\rho(S)\not \equiv 0\mod{b_k}$, we have $S\In U_{\ell-1}$.

Recall that $\Omega_{k-1,\ell-1}$ is a $(b_1, \dots, b_k)$-balancer for~$V\sm\Set{v_\ell}$ with uniformity $r-1$ acting on $U_{\ell-1}$.
Thus, there exists a multiset $\Omega' \subseteq \Omega_{k-1,\ell-1}$ such that
\begin{align}
\rho + \tau_{ \Omega'}\mbox{ is }( b_1, \dots, b_k )\mbox{-divisible}.\label{inductive divisibility}
\end{align}

Let $\Omega'_{k,\ell} \subseteq \Omega_{k,\ell}$ be induced by $\Omega'$, that is, $\Omega'_{k,\ell}:=\set{\mathbf{v'}}{\mathbf{v}\in \Omega'}$ (see~\eqref{shift extension}).
Let $\mathbf{v'}\in \Omega'_{k,\ell}$ and let $\tau_{\mathbf{v'}}$ be any $\mathbf{v'}$-adapter with respect to $(b_0,\dots,b_k;h_k)$. As noted after Definition~\ref{def:adapter}, $\tau_{\mathbf{v'}}$ is $(b_0,\dots,b_{k-1},h_k)$-divisible. Crucially, if $S \in \binom{V}{k}$ and $v_{\ell} \in S$, then $\tau_{\mathbf{v'}}(S)\equiv \tau_{\mathbf{v}} (S \setminus \Set{v_{\ell}})\mod{b_k}$. Indeed, let $x_1^0,\dots,x_{k-1}^0, x_1^1,\dots,x_{k-1}^1$ be such that $\mathbf{v}=(x_1^0,\dots,x_{k-1}^0, x_1^1,\dots,x_{k-1}^1)$ and thus $\mathbf{v'}=(v_\ell,x_1^0,\dots,x_{k-1}^0,v_{j_{\mathbf{v}}}, x_1^1,\dots,x_{k-1}^1)$. Then by Definition~\ref{def:adapter}, as $v_\ell\in S$, we have
\begin{align*}
\tau_{\mathbf{v'}}(S) &\equiv
   \begin{cases}
		(-1)^{0+\sum_{i \in [k-1]}z_i }h_k \mod{ b_k }& \text{if $S\sm\Set{v_\ell} = \{ x_i^{z_i} : i \in [k-1] \}$, }\\
		0 \mod{ b_k }& \text{otherwise,}
	\end{cases}\\
	&\equiv
\tau_{\mathbf{v}} (S \setminus \Set{v_{\ell}})\mod{b_k}.
\end{align*}

Let $\tau_{\Omega'_{k,\ell}}:=\sum_{\mathbf{v'}\in \Omega'_{k,\ell}}\tau_{\mathbf{v'}}$ and $\phi_{\ell-1}:=\phi_\ell+\tau_{\Omega'_{k,\ell}}$.
Note that for all $S \not \subseteq U_{\ell}$, we have $\tau_{\Omega'_{k,\ell}}(S)=0$ by \eqref{shift extension}. Moreover, if $S \in \binom{V}{k}$ and $v_{\ell} \in S$, then $\tau_{\Omega'_{k,\ell}}(S)\equiv \tau_{\Omega'} (S \setminus \Set{v_{\ell}})\mod{b_k}$ by the above.

Clearly, $\phi_{\ell-1}$ is $(b_0, \dots, b_{k-1},h_k)$-divisible. Now, consider any $S \in \binom{V}{k}$ with $S\not\In U_{\ell-1}$.
If $S \not \subseteq U_{\ell}$, then
\begin{align*}
	\phi_{\ell-1}(S) = \phi_{\ell}(S) + \tau_{\Omega'_{k,\ell}} (S) \equiv 0 +0 \equiv 0 \mod{b_k}.
\end{align*}
If $S \subseteq U_{\ell}$, then since $S\not\In U_{\ell-1}$ we must have $v_{\ell} \in S$, and so
\begin{align*}
	\phi_{\ell-1}(S) & = \phi_{\ell}(S) + \tau_{\Omega'_{k,\ell}} (S)
	\equiv \rho(S\sm \Set{v_{\ell}}) + \tau_{\Omega'} (S \sm \Set{v_{\ell}})
	\overset{\eqref{inductive divisibility}}{\equiv} 0 \mod{b_k}.
\end{align*}
This completes the proof of the claim.
\endclaimproof

Now, let $h_k\in \bN$ with $h_k\mid b_k$ and let $\phi\colon \binom{V}{r} \rightarrow \mathbb{Z}$ be any $(b_0, \dots, b_{k-1},h_k)$-divisible function such that $S\In U$ whenever $S\in \binom{V}{k}$ and $\phi(S)\not\equiv 0\mod{b_k}$. Let $\phi_n:=\phi$ and note that $U=U_n$. Thus, by Claim~\ref{clm:Omega}, there exists $\Omega_{k,n}' \subseteq \Omega_{k,n}$ such that $\phi_{n-1}:=\phi_{n} + \tau_{\Omega_{k,n}'}$ is $(b_0, \dots, b_{k-1},h_k)$-divisible and if $\phi_{n-1}(S) \not\equiv 0 \mod{b_k}$ for some $S\in \binom{V}{k}$, then $S \subseteq U_{n-1}$. Repeating this step finally yields some $\Omega_k' \subseteq \Omega_k$ such that $\phi^\ast := \phi + \tau_{\Omega_k'}$ is $(b_0, \dots, b_{k-1},h_k)$-divisible and such that $S\In U_{2k-1}$ whenever $S\in \binom{V}{k}$ and $\phi(S)\not\equiv 0\mod{b_k}$.
 By Lemma~\ref{lem:auto div} (with $U_{2k-1}$ playing the role of $K$), $\phi^\ast$ is then $(b_0, \dots, b_k)$-divisible.
Thus $\Omega_k$ is indeed a $(b_0, \dots, b_k)$-balancer.
\end{proof}

\subsection{Proof of Lemma~\ref{lem:make divisible}}

We now prove Lemma~\ref{lem:make divisible}. For this, we consider the balancers $\Omega_{k}$ guaranteed by Lemma~\ref{lem:balancer}. Recall that these consist of suitable adapters, and that Lemma~\ref{lem:shifters exist} guarantees the existence of shifters corresponding to these adapters. It remains to embed these shifters in a suitable way, which is achieved via Lemma~\ref{lem:rooted embedding}.
The following fact will help us to verify the conditions of Lemma~\ref{lem:balancer}.

\begin{fact}\label{fact:gcd connection}
Let $F$ be an $r$-graph. Then for all $0\le i \le k<r$, we have $\binom{r-i}{k-i}Deg(F)_i\equiv 0\mod{Deg(F)_k}$.
\end{fact}

\proof
Let $S$ be any $i$-set in $V(F)$. By \eqref{handshaking}, we have that $$\binom{r-i}{k-i}|F(S)|=\sum_{T\in \binom{V(F)}{k}\colon S\In T}|F(T)| \equiv 0\mod{Deg(F)_k},$$ and this implies the claim.
\endproof

\lateproof{Lemma~\ref{lem:make divisible}}
Let $x_1^0,\dots,x_{r-1}^0,x_1^1,\dots,x_{r-1}^1$ be distinct vertices (not in $V(G)$). For $k\in[r-1]$, let $X_k:=\Set{x_1^0,\dots,x_{k}^0,x_1^1,\dots,x_{k}^1}$.
By Lemma~\ref{lem:shifters exist}, for every $k\in[r-1]$, there exists an $( x^0_1, \dots, x^0_k , x^1_{1},\dots, x^1_{k} )$-shifter $T_k$ with respect to $F,F^\ast$ such that $T_k[X_k]$ is empty and $T_k$ has degeneracy at most $\binom{f^\ast-1}{r-1}$ rooted at $X_k$. Note that \ref{shifter decomposable} implies that
\begin{align}
|T_k(\Set{x_i^0,x_i^1})|=0\mbox{ for all }i\in[k].\label{no degenerate root}
\end{align}

We may assume that there exists $t\ge \max_{k\in[r-1]}|V(T_k)|$ such that $1/n\ll \gamma \ll 1/t \ll \xi,1/f^\ast$.
Let $Deg(F) = (h_0, h_1, \dots, h_{r-1})$ and let $Deg(F^\ast) = (b_0, b_1, \dots, b_{r-1})$. Since $F^\ast$ is $F$-decomposable and thus $F$-divisible, we have $h_k\mid b_k$ for all $k\in[r-1]_0$.

By Fact~\ref{fact:gcd connection}, we have $\binom{r - i}{k - i} b_{i} \equiv 0 \mod{ b_{k} }$ for all $0 \le i \le k < r$. For each $k\in[r-1]$ with $h_k<b_k$, we apply Lemma~\ref{lem:balancer} to obtain a $(b_0, \dots, b_{k})$-balancer~$\Omega_k$ for~$V(G)$ with uniformity $r$ acting on $V(G)$ such that $\Delta( \Omega_k) \le 2^{k}(k!)^2 b_k$. For values of $k$ for which we have $h_k=b_k$, we let $\Omega_k:=\emptyset$.
For every $k\in[r-1]$ and every $\mathbf{v}=(v_1,\dots,v_{2k})\in \Omega_k$, define the labelling $\Lambda_{\mathbf{v}}\colon X_k\to V(G)$ by setting $\Lambda_{\mathbf{v}}(x_i^0):=v_i$ and $\Lambda_{\mathbf{v}}(x_i^1):=v_{i+k}$ for all $i\in[k]$.

For technical reasons, let $T_0$ be a copy of $F$ and let $X_0:=\emptyset$. Let $\Omega_0$ be the multiset containing $b_0/h_0$ copies of $\emptyset$, and for every $\mathbf{v}\in \Omega_0$, let $\Lambda_{\mathbf{v}}\colon X_0 \to V(G)$ be the trivial $G$-labelling of $(T_0,X_0)$. Note that $T_0$ has degeneracy at most $\binom{f^\ast-1}{r-1}$ rooted at $X_0$. Note also that $\Lambda_{\mathbf{v}}$ does not root any set $S\In V(G)$ with $|S|\in[r-1]$.

We will apply Lemma~\ref{lem:rooted embedding} in order to find faithful embeddings of the $T_k$ into $G$. Let $\Omega:=\bigcup_{k=0}^{r-1}\Omega_k$. Let $\alpha:=\gamma^{-2}/n$.

\begin{NoHyper}
\begin{claim}\label{claim:rooted ok}
For every $k\in[r-1]$ and every $S\In V(G)$ with $|S|\in[r-1]$, we have $|\set{\mathbf{v}\in \Omega_k}{\Lambda_{\mathbf{v}}\mbox{ roots }S}|\le r^{-1}\alpha \gamma n^{r-|S|}$. Moreover, $|\Omega_k|\le r^{-1}\alpha \gamma n^{r}$.
\end{claim}
\end{NoHyper}

\claimproof
Let $k\in[r-1]$ and $S\In V(G)$ with $|S|\in[r-1]$. Consider any $\mathbf{v}=(v_1,\dots,v_{2k})\in \Omega_k$ and suppose that $\Lambda_{\mathbf{v}}$ roots $S$, i.e.~$S\In \Set{v_1,\dots,v_{2k}}$ and $|T_k(\Lambda_{\mathbf{v}}^{-1}(S))|>0$. Note that if we had $\Set{x_i^0,x_i^1}\In \Lambda_{\mathbf{v}}^{-1}(S)$ for some $i\in[k]$ then $|T_k(\Lambda_{\mathbf{v}}^{-1}(S))|=0$ by \eqref{no degenerate root}, a contradiction. We deduce that $|S\cap \Set{v_i,v_{i+k}}|\le 1$ for all $i\in[k]$, in particular $|S|\le k$. Thus there exists $S'\supseteq S$ with $|S'|=k$ and such that $|S'\cap \Set{v_i,v_{i+k}}|=1$ for all $i\in[k]$. However, there are at most $n^{k-|S|}$ sets $S'$ with $|S'|=k$ and $S'\supseteq S$, and for each such $S'$, the number of $\mathbf{v}=(v_1,\dots,v_{2k})\in \Omega_k$ with $|S'\cap \Set{v_i,v_{i+k}}|=1$ for all $i\in[k]$ is at most $\Delta(\Omega_k)$. Thus, $|\set{\mathbf{v}\in \Omega_k}{\Lambda_{\mathbf{v}}\mbox{ roots }S}|\le n^{k-|S|}\Delta(\Omega_k)\le n^{r-1-|S|} 2^{k}(k!)^2 b_k \le r^{-1}\alpha \gamma n^{r-|S|}$. Similarly, we have $|\Omega_k|\le n^k\Delta(\Omega_k)\le r^{-1}\alpha \gamma n^{r}$.
\endclaimproof

Claim~\ref*{claim:rooted ok} implies that for every $S\In V(G)$ with $|S|\in [r-1]$, we have\COMMENT{here we also use that for each $\mathbf{v}\in \Omega_0$, $\Lambda_{\mathbf{v}}$ does not root $S$} $$|\set{\mathbf{v}\in \Omega}{\Lambda_{\mathbf{v}}\mbox{ roots }S}|\le \alpha \gamma n^{r-|S|}-1,$$\COMMENT{$-1$ needed for Lemma~\ref{lem:rooted embedding}, makes the analysis easier there}
and we have $|\Omega|\le b_0/h_0+\sum_{k=1}^{r-1}|\Omega_k|\le \alpha \gamma n^r$.
Therefore, by Lemma~\ref{lem:rooted embedding}, for every $k\in[r-1]_0$ and every $\mathbf{v}\in \Omega_k$, there exists a $\Lambda_{\mathbf{v}}$-faithful embedding $\phi_{\mathbf{v}}$ of $(T_k,X_k)$ into $G$, such that, letting $T_{\mathbf{v}}:=\phi_{\mathbf{v}}(T_k)$, the following hold:

\begin{enumerate}[label=\rm{(\alph*)}]
\item for all distinct $\mathbf{v}_1,\mathbf{v}_2\in \Omega$, the hulls of $(T_{\mathbf{v}_1},\Ima(\Lambda_{\mathbf{v}_1}))$ and $(T_{\mathbf{v}_2},\Ima(\Lambda_{\mathbf{v}_2}))$ are edge-disjoint;\label{rooted embedding:disjoint new}
\item for all $\mathbf{v}\in \Omega$ and $e\in O$ with $e\In V(T_{\mathbf{v}})$, we have $e\In \Ima(\Lambda_{\mathbf{v}})$;\label{rooted embedding:f-disjoint new}
\item $\Delta(\bigcup_{\mathbf{v}\in \Omega}T_{\mathbf{v}})\le \alpha \gamma^{(2^{-r})} n$.\label{rooted embedding:maxdeg new}
\end{enumerate}

Note that by \ref{rooted embedding:disjoint new}, all the graphs $T_{\mathbf{v}}$ are edge-disjoint. Let $$D:=\bigcup_{\mathbf{v}\in \Omega}T_{\mathbf{v}}.$$ By \ref{rooted embedding:maxdeg new}, we have $\Delta(D)\le \gamma^{-2}$. We will now show that $D$ is as desired.

For every $k\in[r-1]$ and $\mathbf{v}\in \Omega_k$, we have that $T_{\mathbf{v}}$ is a $\mathbf{v}$-shifter with respect to $F,F^\ast$ by definition of $\Lambda_{\mathbf{v}}$ and since $\phi_{\mathbf{v}}$ is $\Lambda_{\mathbf{v}}$-faithful. Thus, by Fact~\ref{fact:shifter is adapter},
\begin{align}
\mathds{1}_{T_{\mathbf{v}}}\mbox{ is a }\mathbf{v}\mbox{-adapter with respect to }(b_0,\dots,b_k;h_k).\label{shifter is adapter}
\end{align}

\begin{NoHyper}
\begin{claim}\label{claim:all shifter dec}
For every $\Omega'\In \Omega$, $\bigcup_{\mathbf{v}\in \Omega'}T_{\mathbf{v}}$ has a $1$-well separated $F$-decomposition $\cF$ such that $\cF^{\le(r+1)}$ and $O$ are edge-disjoint.
\end{claim}
\end{NoHyper}

\claimproof
Clearly, for every $\mathbf{v}\in \Omega_0$, $T_{\mathbf{v}}$ is a copy of $F$ and thus has a $1$-well separated $F$-decomposition $\cF_{\mathbf{v}}=\Set{T_{\mathbf{v}}}$. Moreover, for each $k\in[r-1]$ and all $\mathbf{v}=(v_1,\dots,v_{2k})\in \Omega_k$, $T_{\mathbf{v}}$ has a $1$-well separated $F$-decomposition $\cF_{\mathbf{v}}$ by \ref{shifter decomposable} such that for all $F'\in \cF_{\mathbf{v}}$ and all $i\in[k]$, $|V(F')\cap\Set{v_i,v_{i+k}}|\le 1$.

In order to prove the claim, it is thus sufficient to show that for all distinct $\mathbf{v}_1,\mathbf{v}_2\in \Omega$, $\cF_{\mathbf{v}_1}$ and $\cF_{\mathbf{v}_2}$ are $r$-disjoint (implying that $\cF:=\bigcup_{\mathbf{v}\in \Omega'}\cF_{\mathbf{v}}$ is $1$-well separated by Fact~\ref{fact:ws}\ref{fact:ws:2}) and that for every $\mathbf{v}\in \Omega$, $\cF_{\mathbf{v}}^{\le(r+1)}$ and $O$ are edge-disjoint.

To this end, we first show that for every $\mathbf{v}\in \Omega$ and $F'\in \cF_{\mathbf{v}}$, we have that $|V(F')\cap \Ima(\Lambda_{\mathbf{v}})|<r$ and every $e\in \binom{V(F')}{r}$ belongs to the hull of $(T_{\mathbf{v}},\Ima(\Lambda_{\mathbf{v}}))$. If $\mathbf{v}\in \Omega_0$, this is clear since $\Ima(\Lambda_{\mathbf{v}})=\emptyset$ and $F'=T_{\mathbf{v}}$, so suppose that $\mathbf{v}=(v_1,\dots,v_{2k})\in \Omega_k$ for some $k\in[r-1]$. (In particular, $h_k<b_k$.) By the above, we have $|V(F')\cap\Set{v_i,v_{i+k}}|\le 1$ for all $i\in[k]$. In particular, $|V(F')\cap \Ima(\Lambda_{\mathbf{v}})|\le k<r$, as desired. Moreover, suppose that $e\in \binom{V(F')}{r}$. If $e\cap \Ima(\Lambda_{\mathbf{v}})=\emptyset$, then $e$ belongs to the hull of $(T_{\mathbf{v}},\Ima(\Lambda_{\mathbf{v}}))$, so suppose further that $S:=e\cap \Ima(\Lambda_{\mathbf{v}})$ is not empty. Clearly, $|S\cap \Set{v_i,v_{i+k}}|\le |V(F')\cap \Set{v_i,v_{i+k}}|\le 1$ for all $i\in[k]$. Thus, there exists $S'\supseteq S$ with $|S'|=k$ and $|S'\cap \Set{v_i,v_{i+k}}|=1$ for all $i\in[k]$. By \ref{shifter degrees} (and since $h_k<b_k$), we have that $|T_{\mathbf{v}}(S')|>0$, which clearly implies that $|T_{\mathbf{v}}(S)|>0$. Thus, $e\cap \Ima(\Lambda_{\mathbf{v}})=S$ is a root of $(T_{\mathbf{v}},\Ima(\Lambda_{\mathbf{v}}))$ and therefore $e$ belongs to the hull of $(T_{\mathbf{v}},\Ima(\Lambda_{\mathbf{v}}))$.

Now, consider distinct $\mathbf{v}_1,\mathbf{v}_2\in \Omega$ and suppose, for a contradiction, that there is $e\in \binom{V(G)}{r}$ such that $e\In V(F')\cap V(F'')$ for some $F'\in \cF_{\mathbf{v}_1}$ and $F''\in \cF_{\mathbf{v}_2}$. But by the above, $e$ belongs to the hulls of both $(T_{\mathbf{v}_1},\Ima(\Lambda_{\mathbf{v}_1}))$ and $(T_{\mathbf{v}_2},\Ima(\Lambda_{\mathbf{v}_2}))$, a contradiction to \ref{rooted embedding:disjoint new}.

Finally, consider $\mathbf{v}\in \Omega$ and $e\in O$. We claim that $e\notin \cF_{\mathbf{v}}^{\le(r+1)}$. Let $F'\in \cF_{\mathbf{v}}$ and suppose, for a contradiction, that $e\In V(F')$. By \ref{rooted embedding:f-disjoint new}, we have $e\In \Ima(\Lambda_{\mathbf{v}})$. On the other hand, by the above, we have $|V(F')\cap \Ima(\Lambda_{\mathbf{v}})|<r$, a contradiction.
\endclaimproof

Clearly, $D$ is $F$-divisible by Claim~\ref*{claim:all shifter dec}.
We will now show that for every $F$-divisible $r$-graph $H$ on $V(G)$ which is edge-disjoint from $D$, there exists a subgraph $D^\ast\In D$ such that $H \cup D^\ast$ is $F^\ast$-divisible and $D-D^\ast$ has a $1$-well separated $F$-decomposition $\cF$ such that $\cF^{\le(r+1)}$ and $O$ are edge-disjoint.

Let $H$ be any $F$-divisible $r$-graph on $V(G)$ which is edge-disjoint from~$D$. We will inductively prove that the following holds for all $k\in[r-1]_0$:
\begin{itemize}
\item[SHIFT$_{k}$] there exists $\Omega_k^\ast\In \Omega_0\cup \dots \cup \Omega_k$ such that $\mathds{1}_{H\cup D_k^\ast}$ is $(b_0,\dots,b_k)$-divisible, where $D_k^\ast:=\bigcup_{\mathbf{v}\in \Omega_k^\ast}T_{\mathbf{v}}$.
\end{itemize}
We first establish SHIFT$_0$. Since $H$ is $F$-divisible, we have $|H|\equiv 0\mod{h_0}$. Since $h_0\mid b_0$, there exists some $0\le a< b_0/h_0$ such that $|H|\equiv ah_0 \mod{b_0}$.
Let $\Omega_0^\ast$ be the multisubset of $\Omega_0$ consisting of $b_0/h_0-a$ copies of $\emptyset$. Let $D_0^\ast:=\bigcup_{\mathbf{v}\in \Omega_0^\ast}T_{\mathbf{v}}$. Hence, $D_0^\ast$ is the edge-disjoint union of $b_0/h_0-a$ copies of $F$. We thus have $|H\cup D_0^\ast|\equiv ah_0+|F|(b_0/h_0-a)\equiv ah_0+b_0-ah_0\equiv 0\mod{b_0}$. Therefore, $\mathds{1}_{H\cup D_0^\ast}$ is $(b_0)$-divisible, as required.

Suppose now that SHIFT$_{k-1}$ holds for some $k \in [r-1]$, that is, there is $\Omega_{k-1}^\ast\In \Omega_0\cup \dots \cup \Omega_{k-1}$ such that $\mathds{1}_{H\cup D_{k-1}^\ast}$ is $(b_0,\dots,b_{k-1})$-divisible, where $D_{k-1}^\ast:=\bigcup_{\mathbf{v}\in \Omega_{k-1}^\ast}T_{\mathbf{v}}$.
Note that $D_{k-1}^\ast$ is $F$-divisible by Claim~\ref*{claim:all shifter dec}. Thus, since both $H$ and $D_{k-1}^\ast$ are $F$-divisible, we have $\mathds{1}_{H\cup D_{k-1}^\ast}(S)=|(H\cup D_{k-1}^\ast)(S)|\equiv 0\mod{h_k}$ for all $S\in \binom{V(G)}{k}$. Hence, $\mathds{1}_{H\cup D_{k-1}^\ast}$ is in fact $(b_0,\dots,b_{k-1},h_k)$-divisible. Thus, if $h_k=b_k$, then $\mathds{1}_{H\cup D_{k-1}^\ast}$ is $(b_0,\dots,b_k)$-divisible and we let $\Omega_k':=\emptyset$. Now, assume that $h_k<b_k$.
Recall that $\Omega_k$ is a $(b_0, \dots, b_k)$-balancer and that $h_k\mid b_k$.
Thus, there exists a multisubset $\Omega_k'$ of $\Omega_k$ such that the function $\mathds{1}_{H\cup D_{k-1}^\ast} + \sum_{\mathbf{v} \in \Omega_k'} \tau_{\textbf{v}}$ is $( b_0, \dots, b_k)$-divisible, where $\tau_{\textbf{v}}$ is any $\mathbf{v}$-adapter with respect to $(b_0, \dots, b_{k};h_k)$.
Recall that by \eqref{shifter is adapter} we can take $\tau_{\textbf{v}}=\mathds{1}_{T_{\mathbf{v}}}$.
In both cases, let
\begin{align*}
\Omega_k^\ast&:=\Omega_{k-1}^\ast \cup \Omega_k'\In \Omega_0\cup \dots \cup \Omega_k;\\
D_k'&:=\bigcup_{\mathbf{v}\in \Omega_k'}T_{\mathbf{v}};\\
D_k^\ast&:=\bigcup_{\mathbf{v}\in \Omega_k^\ast}T_{\mathbf{v}}=D_{k-1}^\ast \cup D_k'.
\end{align*}
Thus, $\sum_{\mathbf{v} \in \Omega_k'} \tau_{\textbf{v}} = \mathds{1}_{D_k'}$ and hence $\mathds{1}_{H\cup D_{k}^\ast}=\mathds{1}_{H\cup D_{k-1}^\ast}+ \mathds{1}_{D_k'}$ is $( b_0, \dots, b_k)$-divisible, as required.

Finally, SHIFT$_{r-1}$ implies that there exists $\Omega_{r-1}^\ast\In \Omega$ such that $\mathds{1}_{H\cup D^\ast}$ is $(b_0,\dots,b_{r-1})$-divisible, where $D^\ast:=\bigcup_{\mathbf{v}\in \Omega_{r-1}^\ast}T_{\mathbf{v}}$. Clearly, $D^\ast\In D$, and we have that $H \cup D^\ast$ is $F^\ast$-divisible. Finally, by Claim~\ref*{claim:all shifter dec}, $$D-D^\ast=\bigcup_{\mathbf{v}\in \Omega \sm \Omega_{r-1}^\ast}T_{\mathbf{v}}$$ has a $1$-well separated $F$-decomposition $\cF$ such that $\cF^{\le(r+1)}$ and $O$ are edge-disjoint, completing the proof.
\endproof

\vspace{2cm}

\appendix

\section{Additional tools}\label{app:tools}

We gather a few more tools which we use in the appendix.

\begin{prop}[\cite{GKLO}]\label{prop:sparse noise containment}
Let $f,r'\in \bN$ and $r\in \bN_0$ with $f>r$.\COMMENT{Don't require $f>r'$ because in some applications $f-i$ plays the role of $f$ and might be smaller than $r'$.} Let $L$ be an $r'$-graph on $n$ vertices with $\Delta(L)\le \gamma n$. Then every $e\in \binom{V(L)}{r}$ that does not contain any edge of $L$ is contained in at most $2^r \gamma n^{f-r}$ $f$-sets of $V(L)$ that contain an edge of $L$.\COMMENT{Don't use the more complicated bound here}
\end{prop}

\begin{fact}[\cite{GKLO}]\label{fact:connected}
If $G$ is an $(\eps,\xi,f,r)$-supercomplex, then for all distinct $e,e'\in G^{(r)}$, we have $|G^{(f)}(e)\cap G^{(f)}(e')|\ge (\xi-\eps)(n-2r)^{f-r}$.\COMMENT{$h=r$, $F=\Set{e,e'}$. $(G(e)\cap G(e'))[Y]$ is $(\eps,d,f-r,0)$-regular for some $d\ge \xi$ and $Y\In(G(e)\cap G(e'))^{(f-r)}$. In particular, since $\emptyset \in (G(e)\cap G(e'))^{(0)}$ this implies that $|(G(e)\cap G(e'))^{(f-r)}|\ge (d-\eps)(n-|e\cup e'|)^{f-r-0}$.}
\end{fact}

\begin{prop}[\cite{GKLO}]\label{prop:effect of decs}
Let $1/n\ll \eps,\xi,1/f$ and $r\in[f-1]$. Suppose that $G$ is an $(\eps,\xi,f,r)$-supercomplex on $n$ vertices. Let $Y_{used}$ be an $f$-graph on $V(G)$ with $\Delta_r(Y_{used})\le \eps n^{f-r}$. Then $G-Y_{used}$ is a $(2^{r+2}\eps,\xi-2^{2r+1}\eps,f,r)$-supercomplex.
\end{prop}

Using \ref{separatedness:2}, we can deduce that there are many $f$-disjoint $F$-decompositions of a supercomplex. This will be an important tool in the proof of the Cover down lemma (Lemma~\ref{lem:cover down}), where we will find many candidate $F$-decompositions and then pick one at random.

\begin{cor}\label{cor:many decs new}
Let $1/n\ll \eps\ll \xi,1/f$ and $r\in[f-1]$ and assume that \ind{r} is true. Let $F$ be a weakly regular $r$-graph on $f$ vertices. Suppose that $G$ is an $F$-divisible $(\eps,\xi,f,r)$-supercomplex on $n$ vertices. Then the number of pairwise $f$-disjoint $1/\eps$-well separated $F$-decompositions of $G$ is at least $\eps^2 n^{f-r}$.
\end{cor}

\proof
Suppose that $\cF_1,\dots,\cF_t$ are $f$-disjoint $1/\eps$-well separated $F$-decompositions of $G$, where $t\le \eps^2 n^{f-r}$. Let $Y_{used}:=\bigcup_{j\in[t]}\cF_j^{\le(f)}$. By \ref{separatedness:2}, we have $\Delta_r(Y_{used})\le t/\eps \le \eps n^{f-r}$. Thus, by Proposition~\ref{prop:effect of decs}, $G-Y_{used}$ is an $F$-divisible $(2^{r+2}\eps,\xi-2^{2r+1}\eps,f,r)$-supercomplex and thus has a $1/\eps$-well separated $F$-decomposition $\cF_{t+1}$ by~\ind{r}, which is $f$-disjoint from $\cF_1,\dots,\cF_t$.
\endproof

\subsection{Probabilistic tools}

\begin{lemma}[see {\cite[{Corollary~2.3, Corollary~2.4, Remark 2.5 and Theorem 2.8}]{JLR}}] \label{lem:chernoff}
Let $X$ be the sum of $n$ independent Bernoulli random variables. Then the following hold.
\begin{enumerate}[label={\rm(\roman*)}]
\item For all $t\ge 0$, $\prob{|X - \expn{X}| \geq t} \leq 2\eul^{-2t^2/n}$.\label{chernoff t}
\item For all $0\le\eps \le 3/2$, $\prob{|X - \expn{X}| \geq \eps\expn{X} } \leq 2\eul^{-\eps^2\expn{X}/3}$.\label{chernoff eps}
\item If $t\ge 7 \expn{X}$, then $\prob{X\ge t}\le \eul^{-t}$.\label{chernoff crude}
\end{enumerate}
\end{lemma}

\begin{lemma}[\cite{GKLO}] \label{lem:separable chernoff}
Let $1/n\ll p,\alpha,1/a,1/B$. Let $\cI$ be a set of size at least $\alpha n^a$ and let $(X_i)_{i\in \cI}$ be a family of Bernoulli random variables with $\prob{X_i=1}\ge p$. Suppose that $\cI$ can be partitioned into at most $Bn^{a-1}$ sets $\cI_1,\dots,\cI_k$ such that for each $j\in[k]$, the variables $(X_i)_{i\in\cI_j}$ are independent. Let $X:=\sum_{i\in \cI}X_i$. Then we have
\begin{align*}
\prob{|X-\expn{X}|\ge n^{-1/5} \expn{X}} \le \eul^{-n^{1/6}}.
\end{align*}
\end{lemma}

Lemma~\ref{lem:separable chernoff} can be conveniently applied in the following situation: We are given an $r$-graph $H$ on $n$ vertices and $H'$ is a random subgraph of $H$, where every edge of $H$ survives (i.e.~lies in $H'$) with some probability $\ge p$.

\begin{cor}[\cite{GKLO}]\label{cor:graph chernoff}
Let $1/n\ll p,1/r,\alpha$. Let $H$ be an $r$-graph on $n$ vertices with $|H|\ge\alpha n^r$. Let $H'$ be a random subgraph of $H$, where each edge of $H$ survives (i.e.~lies in $H'$) with some probability $\ge p$. Moreover, suppose that for every matching $M$ in $H$, the edges of $M$ survive independently. Then we have $$\prob{||H'|-\expn{|H'|}|\ge n^{-1/5} \expn{|H'|}}\le \eul^{-n^{1/6}}.$$\COMMENT{only used once}
\end{cor}

When we apply Corollary~\ref{cor:graph chernoff}, it will be clear that for every matching $M$ in $H$, the edges of $M$ survive independently, and we will not discuss this explicitly.

We will also use the following simple result.

\begin{prop}[Jain, see {\cite[Lemma 8]{R}}] \label{prop:Jain}
Let $X_1, \ldots, X_n$ be Bernoulli random variables such that, for any $i \in [n]$ and any $x_1, \ldots, x_{i-1}\in \{0,1\}$,
 \begin{align*}
\prob{X_i = 1 \mid X_1 = x_1, \ldots, X_{i-1} = x_{i-1}} \leq p.
\end{align*}
Let $B \sim B(n,p)$ and $X:=X_1+\dots+X_n$.  Then $\prob{X \geq a} \leq \prob{B \geq a}$ for any~$a\ge 0$.
\end{prop}

We now use Lemma~\ref{lem:separable chernoff} to prove Lemma~\ref{lem:random packing}, which states that the random $F$-packing derived from a clique packing yields a quasirandom leftover.

\lateproof{Lemma~\ref{lem:random packing}}
For $e\in G^{(r)}$, we let $K_e$ be the unique element of $\cK^{\le(f)}$ with $e\In K_e$. Let $G_{ind}:=G-\cK^{\le(r+1)}$. $G'^{(r)}$ is a random subgraph of $G_{ind}^{(r)}$, where for any $\cI\In G^{(r)}$, the events $\Set{e\in G'^{(r)}}_{e\in \cI}$ are independent if the sets $\Set{K_{e}}_{e\in \cI}$ are distinct. Since $\Delta(\cK^{\le(r+1)})\le f-r$, Proposition~\ref{prop:noise} implies that $G_{ind}$ is $(1.1\eps,d,f,r)$-regular and $(\xi-\eps,f+r,r)$-dense.

For $e\in G^{(r)}$, let $\cQ_e:= G_{ind}^{(f)}(e)$ and $\tilde{\cQ}_e:= G_{ind}^{(f+r)}(e)$. Thus, $|\cQ_e|=(d\pm 1.1\eps)n^{f-r}$ and $|\tilde{\cQ}_e|\ge 0.95\xi n^f$. Let $\cQ_e'$ be the random subgraph of $\cQ_e$ consisting of all $Q\in \cQ_e$ with $\binom{Q\cup e}{r}\sm \Set{e}\In G'^{(r)}$. Similarly, let $\tilde{\cQ}_e'$ be the random subgraph of $\tilde{\cQ}_e$ consisting of all $Q\in \tilde{\cQ}_e$ with $\binom{Q\cup e}{r}\sm \Set{e}\In G'^{(r)}$. Note that if $e\in G'^{(r)}$, then $\cQ_e'=G'^{(f)}(e)$.
Moreover, note that by definition of $G_{ind}$, we have
\begin{align}
|(e\cup Q)\cap K|\le r \mbox{ for all } Q\in \cQ_e,K\in \cK.\label{intersection for independence}
\end{align}
Consider $Q\in \cQ_e$. By \eqref{intersection for independence}, the $K_{e'}$ with $e'\in \binom{Q\cup e}{r}\sm \Set{e}$ are all distinct, hence we have $\prob{Q\in \cQ_e'}=p^{\binom{f}{r}-1}$. Thus, $\expn{|\cQ_e'|}=p^{\binom{f}{r}-1}|\cQ_e|$.

Define an auxiliary graph $A_e$ on vertex set $\cQ_e$ where $QQ'\in A_e$ if and only if there exists $K\in \cK^{\le(f)}\sm\Set{K_e}$ such that $|(e\cup Q)\cap K|=r$ and $|(e\cup Q')\cap K|=r$. Using \eqref{intersection for independence}, it is easy to see that if $Y$ is an independent set in $A_e$, then the events $\Set{Q\in \cQ_e'}_{Q\in Y}$ are independent.\COMMENT{Let $Y$ be an independent set in $A_e$. For $Q\in Y$, the event $\Set{Q\in \cQ_{e}'}$ depends only on the (elementary) events $\Set{e'\in G'^{(r)}}$, $e'\in \binom{Q\cup e}{r}\sm \Set{e}$. Thus, if the events $\Set{Q\in \cQ_e'}_{Q\in Y}$ were not independent, then there must be distinct $Q,Q'\in Y$ and $e'\in \binom{Q\cup e}{r}\sm \Set{e}$ and $e''\in \binom{Q'\cup e}{r}\sm \Set{e}$ such that $K_{e'}=K_{e''}=:K$. Using \eqref{intersection for independence}, we can see that $K\neq K_e$ (otherwise, $e'\cup e\In (e\cup Q)\cap K$ and thus $|(e\cup Q)\cap K|\ge r+1$). Moreover, $|(e\cup Q)\cap K|=r$ (as $e'\In (e\cup Q)\cap K$) and $|(e\cup Q')\cap K|=r$ (as $e''\In (e\cup Q')\cap K$). Thus, $QQ'\in A_e$, a contradiction to $Y$ being independent in $A_e$. }

\begin{claim}
$\cQ_e$ can be partitioned into $2\binom{f}{r}^2 n^{f-r-1}$ independent sets in $A_e$.
\end{claim}

\claimproof
It is sufficent to prove that $\Delta(A_e)\le \binom{f}{r}^2 n^{f-r-1}$. Fix $Q\in V(A_e)$. There are $\binom{f}{r}-1$ $r$-subsets $e'$ of $e\cup Q$ other than $e$. For each of these, $K_{e'}$ is the unique $K\in \cK^{\le(f)}\sm \Set{K_e}$ which contains $e'$. Each choice of $K_{e'}$ has $\binom{f}{r}$ $r$-subsets $e''$. If we want $e\cup Q'$ to contain $e''$, then since $e''\neq e$, we have $|e\cup e''| \ge r+1$ and thus there are at most $n^{f-r-1}$ possibilities for $Q'$.
\endclaimproof
By Lemma~\ref{lem:separable chernoff}, we thus have $\prob{|\cQ_e'| \neq (1\pm n^{-1/5})\expn{|\cQ_e'|}}\le \eul^{-n^{1/6}}$. We conclude that with probability at least $1-\eul^{-n^{1/6}}$ we have $|\cQ_e'|=(p^{\binom{f}{r}-1}d \pm 2\eps)n^{f-r}$. Together with a union bound, this implies that whp $G'$ is $(2\eps,p^{\binom{f}{r}-1}d,f,r)$-regular, which proves \ref{f nibble trick regular}.

A similar argument shows that whp $G'$ is $(0.9p^{\binom{f+r}{r}-1}\xi,f+r,r)$-dense.

To prove \ref{f nibble trick maxdeg}, let $S\in \binom{V(G)}{r-1}$. Clearly, we have $\expn{|G'^{(r)}(S)|}=p|G^{(r)}(S)|$. If $|G^{(r)}(S)|=0$, then we clearly have $|G^{(r)}(S)|\le 1.1p\Delta(G^{(r)})$, so assume that $S\In e\in G^{(r)}$. Since $e$ is contained in at least $0.5\xi n^{f-r}$ $f$-sets in $G$, and every $r$-set $e'\neq e$ is contained in a most $n^{f-(r+1)}$ of these, we can deduce that $|G^{(r)}(S)|\ge 0.5\xi n$. Define the auxiliary graph $A_S$ with vertex set $G^{(r)}(S)$ such that $e_1e_2\in A_S$ if and only if $K_{S\cup e_1}=K_{S\cup e_2}$. Again, we have $\Delta(A_S)\le f-r$ and thus $G^{(r)}(S)$ can be partitioned into $f-r+1$ sets which are independent in $A_S$. By Lemma~\ref{lem:separable chernoff}, we thus have $\prob{|G'^{(r)}(S)| \neq (1\pm n^{-1/5})p|G^{(r)}(S)|}\le \eul^{-n^{1/6}}$. Using a union bound, we conclude that whp $\Delta(G'^{(r)})\le 1.1p\Delta(G^{(r)})$.
\endproof

Finally, we consider subcomplexes obtained by taking a random subset of the vertex set of~$G$.

\begin{cor}[\cite{GKLO}]\label{cor:random induced subcomplex}
Let $1/n\ll \gamma \ll \mu \ll \eps \ll \xi,1/f$ and $r\in[f-1]$. Let $G$ be an $(\eps,\xi,f,r)$-supercomplex on $n$ vertices. Suppose that $U$ is a random subset of $V(G)$ obtained by including every vertex from $V(G)$ independently with probability $\mu$. Then \whp for any $W\In V(G)$ with $|W|\le \gamma n$, $G[U\bigtriangleup W]$ is a $(2\eps,\xi-\eps,f,r)$-supercomplex.
\end{cor}

\subsection{Greedy coverings and divisibility}\label{subsec:greedy cover}

The following lemma allows us to extend a given collection of $r$-sets into suitable $r$-disjoint $f$-cliques.

\begin{lemma}[\cite{GKLO}]\label{lem:greedy cover uni}
Let $1/n\ll \gamma \ll \alpha,1/s,1/f$ and $r\in[f-1]$. Let $G$ be a complex on $n$ vertices and let $L\In G^{(r)}$ satisfy $\Delta(L)\le \gamma n$. Suppose that $L$ decomposes into $L_1,\dots,L_m$ with $1\le |L_j|\le s$. Suppose that for every $j\in[m]$, we are given some candidate set $\cQ_j\In \bigcap_{e\in L_j}G^{(f)}(e)$ with $|\cQ_j|\ge \alpha n^{f-r}$. Then there exists $Q_j\in \cQ_j$ for each $j\in[m]$ such that, writing $K_j:=(Q_j\uplus L_j)^{\le}$\COMMENT{If $Q\in \bigcap_{e\in L}G^{(f)}(e)$, then automatically $L\In G^{(f)}(Q)$, so $Q\uplus L=\set{Q\cup e}{e\in L}$ is well-defined}, we have that $K_{j}$ and $K_{j'}$ are $r$-disjoint for all distinct $j,j'\in[m]$, and $\Delta(\bigcup_{j\in[m]}K_j^{(r)})\le \sqrt{\gamma} n$.\COMMENT{So for every $j$, we want to find an $(f-r)$-set `in the middle' of the edges of $L_j$, such that the $f-r$-set forms a clique with all the edges of $L_j$.}
\end{lemma}

\begin{cor}\label{cor:greedy cover}
Let $1/n\ll \gamma \ll \alpha,1/f$ and $r\in[f-1]$. Suppose that $F$ is an $r$-graph on $f$ vertices. Let $G$ be a complex on $n$ vertices and let $H\In G^{(r)}$ with $\Delta(H)\le \gamma n$ and $|G^{(f)}(e)|\ge \alpha n^{f-r}$ for all $e\in H$. Then there is a $1$-well separated $F$-packing $\cF$ in $G$ that covers all edges of $H$ and such that $\Delta(\cF^{(r)})\le \sqrt{\gamma} n$.
\end{cor}

\proof
Let $e_1,\dots,e_m$ be an enumeration of $H$. For $j\in[m]$, define $L_j:=\Set{e_j}$ and $\cQ_j:=G^{(f)}(e)$. Apply Lemma~\ref{lem:greedy cover uni} to obtain $K_1,\dots,K_m$. For each $j\in [m]$, let $F_j$ be a copy of $F$ with $V(F_j)=K_j$ and such that $e_j\in F_j$. Then $\cF:=\Set{F_1,\dots,F_m}$ is as desired.
\endproof

We can combine Lemma~\ref{lem:F nibble} and Corollary~\ref{cor:greedy cover} to deduce the following result. It allows us to make an $r$-graph divisible by deleting a small fraction of edges (even if we are forbidden to delete a certain set of edges $H$). Note that the result has a similar flavour as Corollary~\ref{cor:make divisible typical}, but the assumptions are different.

\begin{cor}\label{cor:make divisible}
Let $1/n\ll \gamma,\eps \ll \xi,1/f$ and $r\in[f-1]$. Let $F$ be an $r$-graph on $f$ vertices. Suppose that $G$ is a complex on $n$ vertices which is $(\eps,d,f,r)$-regular for some $d\ge \xi$ and $(\xi,f+r,r)$-dense. Let $H\In G^{(r)}$ satisfy $\Delta(H)\le \eps n$. Then there exists $L\In G^{(r)}-H$ such that $\Delta(L)\le \gamma n$ and $G^{(r)}-L$ is $F$-divisible.
\end{cor}

\proof
We clearly have $|G^{(f)}(e)|\ge 0.5\xi n^{f-r}$ for all $e\in H$. Thus, by Corollary~\ref{cor:greedy cover}, there exists an $F$-packing $\cF_0$ in $G$ which covers all edges of $H$ and satisfies $\Delta(\cF_0^{(r)})\le \sqrt{\eps}n$. By Proposition~\ref{prop:noise}\ref{noise:regular} and~\ref{noise:dense}, $G':=G-\cF_0^{(r)}$ is still $(2^{r+1}\sqrt{\eps},d,f,r)$-regular\COMMENT{$\eps+2^r \sqrt{\eps}$} and $(\xi/2,f+r,r)$-dense. Thus, by Lemma~\ref{lem:F nibble}, there exists an $F$-packing $\cF_{nibble}$ in $G'$ such that $\Delta(L)\le \gamma n $, where $L:=G'^{(r)}-\cF_{nibble}^{(r)}=G^{(r)}-\cF_0^{(r)}-\cF_{nibble}^{(r)}\In G^{(r)}-H$. Clearly, $G^{(r)}-L$ is $F$-divisible (in fact, $F$-decomposable).
\endproof

\section{Vortices}\label{app:vortices}

Here, we prove Lemma~\ref{lem:almost dec}, which states that given a vortex in a supercomplex, we can cover all edges which do not lie in the final vortex set. As sketched in Section~\ref{sec:vortices}, we achieve this by alternately applying the $F$-nibble lemma (Lemma~\ref{lem:F nibble}) and the Cover down lemma (Lemma~\ref{lem:cover down}). Recall that the Cover down down lemma guarantees the existence of a suitable `cleaning graph' or `partial absorber' which allows us to `clean' the leftover of an application of the $F$-nibble lemma in the sense that the new leftover is guaranteed to lie in the next vortex set. For technical reasons, we will in fact find all cleaning graphs first (one for each vortex set) and set them aside even before the first nibble.

\subsection{Existence of cleaners}
The aim of this subsection is to apply the Cover down lemma to each `level' $i$ of the vortex to obtain a `cleaning graph' $H_i$ (playing the role of $H^\ast$) for each $i\in[\ell]$ (see Lemma~\ref{lem:cleaner}).
Let $G$ be a complex and $U_0 \supseteq U_1 \supseteq \dots \supseteq U_\ell$ a vortex in $G$. We say that $H_1,\dots,H_\ell$ is a \defn{$(\gamma,\nu,\kappa,F)$-cleaner (for the said vortex)} if the following hold for all $i\in[\ell]$:
\begin{enumerate}[label={\rm(C\arabic*)}]
\item $H_i\In G^{(r)}[U_{i-1}]-G^{(r)}[U_{i+1}]$, where $U_{\ell+1}:=\emptyset$;\label{cleaner:location}
\item $\Delta(H_i)\le \nu |U_{i-1}|$;\label{cleaner:maxdeg}
\item $H_i$ and $H_{i+1}$ are edge-disjoint, where $H_{\ell+1}:=\emptyset$;\label{cleaner:disjoint}
\item whenever $L\In G^{(r)}[U_{i-1}]$ is such that $\Delta(L)\le \gamma |U_{i-1}|$ and $H_i\cup L$ is $F$-divisible and $O$ is an $(r+1)$-graph on $U_{i-1}$ with $\Delta(O)\le \gamma |U_{i-1}|$, there exists a $\kappa$-well separated $F$-packing $\cF$ in $G[H_i\cup L][U_{i-1}]-O$ which covers all edges of $H_i\cup L$ except possibly some inside $U_i$.\label{cleaner:cover down}
\end{enumerate}

Note that \ref{cleaner:location} and \ref{cleaner:disjoint} together imply that $H_1,\dots,H_\ell$ are edge-disjoint.
The following proposition will be used to ensure~\ref{cleaner:disjoint}.

\begin{prop}[\cite{GKLO}]\label{prop:sprinkling}
Let $1/n\ll \eps \ll \mu,\xi,1/f$ and $r\in[f-1]$. Let $\xi':=\xi(1/2)^{(8^f+1)}$. Let $G$ be a complex on $n$ vertices and let $U\In V(G)$ of size $\mu n$ and $(\eps,\mu,\xi,f,r)$-random in $G$. Suppose that $H$ is a random subgraph of $G^{(r)}$ obtained by including every edge of $G^{(r)}$ independently with probability $1/2$. Then with probability at least $1-\eul^{-n^{1/10}}$,
\begin{enumerate}[label={\rm(\roman*)}]
\item $U$ is $(\sqrt{\eps},\mu,\xi',f,r)$-random in $G[H]$ and\label{random slice:random preserved}
\item $G$ is $(\sqrt{\eps},f,r)$-dense with respect to $H-G^{(r)}[\bar{U}]$, where $\bar{U}:=V(G)\sm U$.\label{random slice:dense wrt}
\end{enumerate}
\end{prop}

The following lemma shows that cleaners exist.

\begin{lemma}\label{lem:cleaner}
Let $1/m\ll 1/\kappa \ll \gamma\ll \eps \ll \nu \ll \mu,\xi,1/f$ be such that $\mu\le 1/2$ and $r\in[f-1]$. Assume that \ind{i} is true for all $i\in[r-1]$ and that $F$ is a weakly regular $r$-graph on $f$ vertices. Let $G$ be a complex and $U_0 \supseteq U_1 \supseteq \dots \supseteq U_\ell$ an $(\eps,\mu,\xi,f,r,m)$-vortex in $G$. Then there exists a $(\gamma,\nu,\kappa,F)$-cleaner.
\end{lemma}

\proof
For $i\in[\ell]$, define $U_i':=U_i\sm U_{i+1}$, where $U_{\ell+1}:=\emptyset$. For $i\in[\ell-1]$, let $\mu_i:=\mu(1-\mu)$, and let $\mu_\ell:=\mu$. By \ref{vortex:untwisted}\COMMENT{if $i=\ell$} and \ref{vortex:twisted}, we have for all $i\in[\ell]$ that $U_{i}'$ is $(\eps,\mu_i,\xi,f,r)$-random in $G[U_{i-1}]$.

Split $G^{(r)}$ randomly into $G_0$ and $G_1$, that is, independently for every edge $e\in G^{(r)}$, put $e$ into $G_{0}$ with probability $1/2$ and into $G_{1}$ otherwise. We claim that with positive probability, the following hold for every $i\in[\ell]$:
\begin{enumerate}[label={\rm(\roman*)}]
\item $U_{i}'$ is $(\sqrt{\eps},\mu_i,\xi(1/2)^{(8^f+1)},f,r)$-random in $G[G_{i\;\mathrm{mod}\;{2}}][U_{i-1}]$;\label{sprinkling:random}
\item $G[U_{i-1}]$ is $(\sqrt{\eps},f,r)$-dense with respect to $G_{i\;\mathrm{mod}\;{2}}[U_{i-1}]- G^{(r)}[U_{i-1}\sm U_i']$.\label{sprinkling:dense}
\end{enumerate}
By Proposition~\ref{prop:sprinkling}, the probability that \ref{sprinkling:random} or \ref{sprinkling:dense} do not hold for $i\in[\ell]$ is at most $\eul^{-|U_{i-1}|^{1/10}}\le |U_{i-1}|^{-2}$.
Since $\sum_{i=1}^{\ell}|U_{i-1}|^{-2}<1,$\COMMENT{$\le \sum_{i=2}^{\infty}\frac{1}{i^2}=\frac{\pi^2}{6}-1$} we deduce that with positive probability, \ref{sprinkling:random} and \ref{sprinkling:dense} hold for all $i\in[\ell]$.

Therefore, there exist $G_0,G_{1}$ satisfying the above properties. For every $i\in[\ell]$, we will find $H_i$ using the Cover down lemma (Lemma~\ref{lem:cover down}). Let $i\in[\ell]$. Apply Lemma~\ref{lem:cover down} with the following objects/parameters:

\smallskip
{\footnotesize
\noindent
{
\begin{tabular}{c|c|c|c|c|c|c|c|c|c|c|c|c|c}
object/parameter & $G[G_{i\;\mathrm{mod}\;{2}}][U_{i-1}]$ & $U_i'$ & $G[U_{i-1}]$ & $F$ & $|U_{i-1}|$ & $\kappa$ & $\gamma$ & $\sqrt{\eps}$ & $\nu$ & $\mu_i$ & $\xi(1/2)^{(8^f+1)}$ & $f$ & $r$\\ \hline
\rule{0pt}{3ex}playing the role of & $G$ & $U$ & $\tilde{G}$ & $F$ & $n$ & $\kappa$ & $\gamma$ & $\eps$ & $\nu$ & $\mu$ & $\xi$ & $f$ & $r$
\end{tabular}
}
}
\newline \vspace{0.2cm}

Hence, there exists $$H_i\In G_{i\;\mathrm{mod}\;{2}}[U_{i-1}] -G_{i\;\mathrm{mod}\;{2}}[U_{i-1}\sm U_{i}'] \In G_{i\;\mathrm{mod}\;{2}}[U_{i-1}]-G^{(r)}[U_{i+1}]$$ with $\Delta(H_i)\le \nu |U_{i-1}|$ and the following `cleaning' property: for all $L\In G^{(r)}[U_{i-1}]$ with $\Delta(L)\le \gamma |U_{i-1}|$ such that $H_i\cup L$ is $F$-divisible and all $(r+1)$-graphs $O$ on $U_{i-1}$ with $\Delta(O)\le \gamma |U_{i-1}|$, there exists a $\kappa$-well separated $F$-packing $\cF$ in $G[H_i\cup L][U_{i-1}]-O$ which covers all edges of $H_i\cup L$ except possibly some inside $U_i'\In U_i$. Thus, \ref{cleaner:location}, \ref{cleaner:maxdeg} and \ref{cleaner:cover down} hold.

Since $G_0$ and $G_1$ are edge-disjoint, \ref{cleaner:disjoint} holds as well. Thus, $H_1,\dots,H_\ell$ is a $(\gamma,\nu,\kappa,F)$-cleaner.
\endproof

\subsection{Obtaining a near-optimal packing}

Recall that Lemma~\ref{lem:almost dec} guarantees an $F$-packing covering all edges except those in the final set $U_\ell$ of a vortex. We prove this by applying successively the $F$-nibble lemma (Lemma~\ref{lem:F nibble}) and the definition of a cleaner to each set $U_i$ in the vortex.

\lateproof{Lemma~\ref{lem:almost dec}}
Choose new constants $\gamma,\nu>0$ such that $$1/m\ll 1/\kappa \ll \gamma \ll \eps \ll \nu \ll \mu \ll \xi,1/f.$$

Apply Lemma~\ref{lem:cleaner} to obtain a $(\gamma,\nu,\kappa,F)$-cleaner $H_1,\dots,H_\ell$. Note that by \ref{vortex:untwisted} and Fact~\ref{fact:random trivial}\ref{fact:random trivial:subcomplex}, $G[U_{i}]$ is an $(\eps,\xi,f,r)$-supercomplex for all $i\in[\ell]$, and the same holds for $i=0$ by assumption.
Let $H_{\ell+1}:=\emptyset$ and $U_{\ell+1}:=\emptyset$.

For $i\in[\ell]_0$ and $\cF_{i}^\ast$, define the following conditions:
\begin{enumerate}[label={\rm(FP\arabic*$^\ast$)}]
\item$\hspace{-5pt}_{i}$ $\cF_{i}^\ast$ is a $4\kappa$-well separated $F$-packing in $G-H_{i+1}- G^{(r)}[U_{i+1}]$;\label{vortex iterative packing:1}
\item$\hspace{-5pt}_{i}$ $\cF_{i}^\ast$ covers all edges of $G^{(r)}$ that are not inside $U_{i}$;\label{vortex iterative packing:2}
\item$\hspace{-5pt}_{i}$ for all $e\in G^{(r)}[U_{i}]$, $|\cF_{i}^{\ast\le(f)}(e)|\le 2\kappa$;\label{vortex iterative packing:3}
\item$\hspace{-5pt}_{i}$ $\Delta(\cF_{i}^{\ast(r)}[U_{i}])\le \mu |U_{i}|$.\label{vortex iterative packing:5}
\end{enumerate}

Note that \ref{vortex iterative packing:1}$_0$--\ref{vortex iterative packing:5}$_0$ hold trivially with $\cF_{0}^\ast:=\emptyset$. We will now proceed inductively until we obtain $\cF^\ast_\ell$ satisfying \ref{vortex iterative packing:1}$_\ell$--\ref{vortex iterative packing:5}$_\ell$. Clearly, taking $\cF:=\cF^\ast_\ell$ completes the proof (using \ref{vortex iterative packing:1}$_\ell$ and \ref{vortex iterative packing:2}$_\ell$).

Suppose that for some $i\in[\ell]$, we have found $\cF_{i-1}^\ast$ such that \ref{vortex iterative packing:1}$_{i-1}$--\ref{vortex iterative packing:5}$_{i-1}$ hold. Let $$G_i:=G[U_{i-1}]-(\cF_{i-1}^{\ast(r)}\cup H_{i+1}\cup G^{(r)}[U_{i+1}])-\cF_{i-1}^{\ast\le(r+1)}.$$
We now intend to find $\cF_i$ such that:
\begin{enumerate}[label={\rm(FP\arabic*)}]
\item $\cF_i$ is a 2$\kappa$-well separated $F$-packing in $G_i$;\label{vortex iterative step:1}
\item $\cF_i$ covers all edges from $G^{(r)}[U_{i-1}]-\cF_{i-1}^{\ast(r)}$ that are not inside $U_i$;\label{vortex iterative step:2}
\item $\Delta(\cF_{i}^{(r)}[U_{i}])\le \mu |U_{i}|$.\label{vortex iterative step:3}
\end{enumerate}
We first observe that this is sufficient for $\cF_{i}^\ast:=\cF_{i-1}^\ast\cup \cF_i$ to satisfy \ref{vortex iterative packing:1}$_{i}$--\ref{vortex iterative packing:5}$_{i}$. Note that $\cF_i^{(r)}$ and $\cF_{i-1}^{\ast(r)}$ are edge-disjoint, and $\cF_i$ and $\cF_{i-1}^\ast$ are $(r+1)$-disjoint by definition of $G_i$. Together with \ref{vortex iterative packing:1}$_{i-1}$ this implies that $\cF_{i}^\ast$ is a well separated $F$-packing in $G-H_{i+1}- G^{(r)}[U_{i+1}]$.\COMMENT{$H_{i+1}\cup G^{(r)}[U_{i+1}]\In G^{(r)}[U_{i}]$ and thus $G-H_{i}- G^{(r)}[U_{i}]\In G-H_{i+1}- G^{(r)}[U_{i+1}]$} Let $e\in G^{(r)}$. If $e\not\In U_{i-1}$, then $|\cF_i^{\le(f)}(e)|=0$ and hence $|\cF_i^{\ast\le(f)}(e)|=|\cF_{i-1}^{\ast\le(f)}(e)|\le 4\kappa$. If $e\In U_{i-1}$, then we have $|\cF_i^{\ast\le(f)}(e)|=|\cF_{i-1}^{\ast\le(f)}(e)|+|\cF_{i}^{\le(f)}(e)|\le 4\kappa$ by \ref{vortex iterative packing:3}$_{i-1}$ and \ref{vortex iterative step:1}. Thus, $\cF_{i}^\ast$ is $4\kappa$-well separated and \ref{vortex iterative packing:1}$_{i}$ holds.

Clearly, \ref{vortex iterative packing:2}$_{i-1}$ and \ref{vortex iterative step:2} imply \ref{vortex iterative packing:2}$_{i}$. Moreover, observe that $\cF_{i-1}^{\ast\le(r)}[U_i]$ is empty by \ref{vortex iterative packing:1}$_{i-1}$.\COMMENT{$\le$ is important here and stronger than just saying that $\cF_{i-1}^{\ast(r)}[U_i]$ is empty} Thus, \ref{vortex iterative packing:3}$_{i}$ holds since $\cF_i$ is $2\kappa$-well separated, and \ref{vortex iterative step:3} implies \ref{vortex iterative packing:5}$_{i}$.

It thus remains to show that $\cF_i$ satisfying \ref{vortex iterative step:1}--\ref{vortex iterative step:3} exists. We will obtain $\cF_i$ as the union of two packings, one obtained from the $F$-nibble lemma (Lemma~\ref{lem:F nibble}) and one using \ref{cleaner:cover down}.
Let $G_{i,nibble}:=G[U_{i-1}]-(\cF_{i-1}^{\ast(r)}\cup H_i \cup G^{(r)}[U_i])-\cF_{i-1}^{\ast\le(r+1)}$. Recall that $G[U_{i-1}]$ is an $(\eps,\xi,f,r)$-supercomplex. In particular, it is $(\eps,d,f,r)$-regular for some $d\ge \xi$, and $(\xi,f+r,r)$-dense. Note that by \ref{vortex iterative packing:5}$_{i-1}$, \ref{cleaner:maxdeg} and \ref{vortex:size} we have $$\Delta(\cF_{i-1}^{\ast(r)}[U_{i-1}]\cup H_i \cup G^{(r)}[U_i])\le \mu|U_{i-1}|+\nu |U_{i-1}|+\mu |U_{i-1}|\le 3\mu |U_{i-1}|.$$ Moreover, $\Delta(\cF_{i-1}^{\ast\le(r+1)})\le 4\kappa (f-r)\le \mu |U_{i-1}|$ by Fact~\ref{fact:ws}\ref{fact:ws:maxdeg}. Thus, Proposition~\ref{prop:noise}\ref{noise:regular} and~\ref{noise:dense} imply that $G_{i,nibble}$ is still $(2^{r+3}\mu,d,f,r)$-regular and $(\xi/2,f+r,r)$-dense.\COMMENT{$\eps+2^r3\mu +2^r\mu \le 2^{r+3}\mu$}
Since $\mu\ll \xi$, we can apply Lemma~\ref{lem:F nibble} to obtain a $\kappa$-well separated $F$-packing $\cF_{i,nibble}$ in $G_{i,nibble}$ such that $\Delta(L_{i,nibble})\le \frac{1}{2}\gamma |U_{i-1}|$, where $L_{i,nibble}:=G_{i,nibble}^{(r)}-\cF^{(r)}_{i,nibble}$.
Since by \ref{vortex iterative packing:2}$_{i-1}$,
\begin{align*}
G^{(r)}-\cF_{i-1}^{\ast(r)}-\cF_{i,nibble}^{(r)} &=G^{(r)}[U_{i-1}]-\cF_{i-1}^{\ast(r)}-\cF_{i,nibble}^{(r)}\\
                                        &=(G_{i,nibble}^{(r)}\cup H_i \cup G^{(r)}[U_i])-\cF_{i,nibble}^{(r)}\\
																				&=H_i\cup G^{(r)}[U_i]\cup L_{i,nibble},
\end{align*}
 we know that $H_i\cup G^{(r)}[U_i]\cup L_{i,nibble}$ is $F$-divisible.
By \ref{cleaner:location} and \ref{cleaner:disjoint}, we know that $H_{i+1}\cup G^{(r)}[U_{i+1}]\In G^{(r)}[U_i]-H_i$. Moreover, by \ref{cleaner:maxdeg} and Proposition~\ref{prop:noise}\ref{noise:supercomplex} we have that $G[U_i]-H_i$ is a $(2\mu,\xi/2,f,r)$-supercomplex. We can thus apply Corollary~\ref{cor:make divisible} (with $G[U_i]-H_i$, $H_{i+1}\cup G^{(r)}[U_{i+1}]$, $2\mu$ playing the roles of $G,H,\eps$) to find an $F$-divisible subgraph $R_{i}$ of $G^{(r)}[U_i]-H_i$ containing $H_{i+1}\cup G^{(r)}[U_{i+1}]$ such that $\Delta(L_{i,res})\le \frac{1}{2}\gamma |U_{i}|$, where $L_{i,res}:=G^{(r)}[U_i]-H_i-R_{i}$.

Let $L_i:=L_{i,nibble}\cup L_{i,res}$. Clearly, $L_i\In G^{(r)}[U_{i-1}]$ and $\Delta(L_i)\le \gamma |U_{i-1}|$. Note that
\begin{align}
H_i\cup L_i &=(H_i\cupdot (G^{(r)}[U_i]-H_i) \cupdot L_{i,nibble})-R_i=G^{(r)}-\cF_{i-1}^{\ast(r)}-\cF_{i,nibble}^{(r)}-R_i\label{clean protect}
\end{align} is $F$-divisible. Moreover, $\Delta(\cF_{i-1}^{\ast\le(r+1)}\cup \cF_{i,nibble}^{\le(r+1)})\le 5\kappa (f-r)$ by Fact~\ref{fact:ws}\ref{fact:ws:maxdeg}. Thus, by \ref{cleaner:cover down} there exists a $\kappa$-well separated $F$-packing $\cF_{i,clean}$ in $$G_{i,clean}:=G[H_i\cup L_i][U_{i-1}]-\cF_{i-1}^{\ast\le(r+1)}-\cF_{i,nibble}^{\le(r+1)}$$ which covers all edges of $H_i\cup L_i$ except possibly some inside $U_i$.

We claim that $\cF_i:=\cF_{i,nibble}\cup \cF_{i,clean}$ is the desired packing. Since $\cF_{i,nibble}^{(r)}$ and $\cF_{i,clean}^{(r)}$ are edge-disjoint and $\cF_{i,nibble}$ and $\cF_{i,clean}$ are $(r+1)$-disjoint, we have that $\cF_i$ is a $2\kappa$-well separated $F$-packing by Fact~\ref{fact:ws}\ref{fact:ws:1}. Moreover, it is easy to see from \ref{cleaner:location} that $G_{i,nibble}\In G_i$. Crucially, since $R_i$ was chosen to contain $H_{i+1}\cup G^{(r)}[U_{i+1}]$, we have from \ref{vortex iterative packing:2}$_{i-1}$ that $$H_i\cup L_i \overset{\eqref{clean protect}}{\In} G^{(r)}[U_{i-1}]-R_i-\cF_{i-1}^{\ast(r)} \In G^{(r)}[U_{i-1}] - (\cF_{i-1}^{\ast(r)}\cup H_{i+1}\cup G^{(r)}[U_{i+1}])$$ and thus $G_{i,clean}\In G_i$ as well. Hence, \ref{vortex iterative step:1} holds.

Clearly, $\cF_i$ covers all edges of $G^{(r)}[U_{i-1}]-\cF_{i-1}^{\ast(r)}$ that are not inside $U_i$, thus \ref{vortex iterative step:2} holds.
Finally, since $\cF_{i,nibble}^{(r)}[U_i]$ is empty, we have $\Delta(\cF_{i}^{(r)}[U_{i}])\le \Delta(H_i\cup L_i) \le \nu |U_{i-1}|+\gamma |U_{i-1}|\le \mu |U_i|$, as needed for \ref{vortex iterative step:3}.\COMMENT{Here we need $\nu\ll \mu$}
\endproof

\section{Transformers}\label{app:transformers}

Here, we prove Lemma~\ref{lem:transformer}, which guarantees a transformer from $H$ to $H'$ if $H\rightsquigarrow H'$.
As indicated in Section~\ref{subsec:transformers}, a key step in the argument is the ability to construct `localised transformers' between graphs $S_1\uplus L$ and $S_2\uplus L$. This is achieved by the following lemma.

\begin{lemma}\label{lem:iterative transforming local}
Let $1/n \ll \gamma' \ll \gamma,1/\kappa,\eps \ll \xi,1/f$ and $1\le i<r<f$. Assume that \ind{r-i} is true. Let $F$ be a weakly regular $r$-graph on $f$ vertices and assume that $S^\ast\in \binom{V(F)}{i}$ is such that $F(S^\ast)$ is non-empty. Let $G$ be an $(\eps,\xi,f,r)$-supercomplex on $n$ vertices, let $S_1,S_2\in G^{(i)}$ with $S_1\cap S_2=\emptyset$, and let $\phi\colon S_1\to S_2$ be a bijection.  Moreover, suppose that $L$ is an $F(S^\ast)$-divisible subgraph of $G(S_1)^{(r-i)}\cap G(S_2)^{(r-i)}$ with $|V(L)|\le\gamma' n $.

Then there exist $T,R\In G^{(r)}$ such that the following hold:
\begin{enumerate}[label={\rm(TR\arabic*)}]
\item $V(R)\In V(G)\sm S_2$ and $|e\cap S_1|\in[i-1]$ for all $e\in R$ (so if $i=1$, then $R$ must be empty since $[0]=\emptyset$);\label{lemst:iterative transformer local:location}
\item $T$ is a $(\kappa+1)$-well separated $((S_1\uplus L)\cup \phi(R),(S_2\uplus L)\cup R;F)$-transformer in $G$;\label{lemst:iterative transformer local:transform}
\item $|V(T\cup R)|\le \gamma n$.\label{lemst:iterativ transformer local:maxdeg}
\end{enumerate}
\end{lemma}

\proof
We may assume that $\gamma' \ll \gamma\ll 1/\kappa,\eps$. Choose $\mu>0$ with $\gamma'\ll  \mu \ll \gamma \ll 1/\kappa,\eps$. We split the argument into two parts. First, we will establish the following claim, which is the essential part and relies on \ind{r-i}.

\begin{NoHyper}\begin{claim}\label{claim:crucial transformer bit}
There exist $\hat{T},R_{1,A},R_{1,A\cup L}\In G^{(r)}$ and $\kappa$-well separated $F$-packings $\hat{\cF}_1,\hat{\cF}_2$ in~$G$
such that the following hold:
\begin{enumerate}[label={\rm(tr\arabic*)}]
\item $V(R_{1,A}\cup R_{1,A\cup L})\In V(G)\sm S_2$ and $|e\cap S_1|\in[i-1]$ for all $e\in R_{1,A} \cup R_{1,A\cup L}$;\label{lemst:iterative transformer local:location 2}
\item $\hat{T}$, $S_1\uplus L$, $S_2\uplus L$, $R_{1,A}$, $\phi(R_{1,A})$, $R_{1,A\cup L}$, $\phi(R_{1,A\cup L})$ are pairwise edge-disjoint subgraphs of $G^{(r)}$;\label{lemst:iterative transformer local:disjoint 2}
\item $\hat{\cF}_1^{(r)}=\hat{T}\cup (S_1\uplus L)\cup R_{1,A\cup L} \cup \phi(R_{1,A})$ and $\hat{\cF}_2^{(r)}=\hat{T}\cup (S_2\uplus L)\cup R_{1,A} \cup \phi(R_{1,A\cup L})$;\label{lemst:iterative transformer local:transform 2}
\item $|V(\hat{T}\cup R_{1,A} \cup R_{1,A\cup L})|\le 2\mu n$.\label{lemst:iterative transformer local:maxdeg 2}
\end{enumerate}
\end{claim}\end{NoHyper}

\claimproof By Corollary~\ref{cor:random induced subcomplex} and Lemma~\ref{lem:chernoff}\ref{chernoff t}, there exists a subset $U\In V(G)$ with $0.9\mu n \le |U| \le 1.1 \mu n$ such that $G':=G[U\cup S_1 \cup S_2 \cup V(L)]$ is a $(2\eps,\xi-\eps,f,r)$-supercomplex.
By Proposition~\ref{prop:hereditary}, $G'':=G'(S_1)\cap G'(S_2)$ is a $(2\eps,\xi-\eps,f-i,r-i)$-supercomplex. Clearly, $L\In G''^{(r-i)}$ and $\Delta(L)\le \gamma' n\le \sqrt{\gamma'} |U|$.
Thus, by Proposition~\ref{prop:noise}\ref{noise:supercomplex}, $G''-L$ is a $(3\eps,\xi-2\eps,f-i,r-i)$-supercomplex. By Corollary~\ref{cor:make divisible},\COMMENT{With $H,G''-L,\emptyset,F(S^\ast)$ playing role of $L,G,H,F$} there exists $H\In G''^{(r-i)}-L$ such that $A:=G''^{(r-i)}-L-H$ is $F(S^\ast)$-divisible and $\Delta(H)\le \gamma' n$. In particular, by Proposition~\ref{prop:noise}\ref{noise:supercomplex} we have that
\begin{enumerate}[label={\rm(\roman*)}]
\item $G''[A]$ is an $F(S^\ast)$-divisible $(3\eps,\xi/2,f-i,r-i)$-supercomplex;\label{absorber dec 1}
\item $G''[A\cup L]$ is an $F(S^\ast)$-divisible $(3\eps,\xi/2,f-i,r-i)$-supercomplex.\label{absorber dec 2}
\end{enumerate}
Recall that $F$ being weakly regular implies that $F(S^\ast)$ is weakly regular as well (see Proposition~\ref{prop:link divisibility}). By~\ref{absorber dec 1} and \ind{r-i}, there exists a $\kappa$-well separated $F(S^\ast)$-decomposition $\cF_A$ of $G''[A]$. By Fact~\ref{fact:ws}\ref{fact:ws:maxdeg}, $\Delta(\cF_A^{\le (r-i+1)})\le \kappa f$. Thus, by~\ref{absorber dec 2}, Proposition~\ref{prop:noise}\ref{noise:supercomplex} and \ind{r-i}, there also exists a $\kappa$-well separated $F(S^\ast)$-decomposition $\cF_{A\cup L}$ of $G''[A\cup L]-\cF_A^{\le(r-i+1)}$. In particular, $\cF_A$ and $\cF_{A\cup L}$ are $(r-i+1)$-disjoint.

We define
\begin{align*}
(\cF_{1,A},\cF_{2,A}) &:= S_1 \triangleleft \cF_A \triangleright S_2, \\
(\cF_{1,A\cup L},\cF_{2,A\cup L}) &:= S_1 \triangleleft \cF_{A\cup L} \triangleright S_2.
\end{align*}
By Proposition~\ref{prop:S cover two sided}\ref{prop:S cover two sided:packing}, for $j\in[2]$, $\cF_{j,A}$ is a $\kappa$-well separated $F$-packing in $G'\In G$ with $\set{e\in \cF_{j,A}^{(r)}}{S_j\In e}=S_j\uplus A$ and $\cF_{j,A\cup L}$ is a $\kappa$-well separated $F$-packing in $G'\In G$ with $\set{e\in \cF_{j,A\cup L}^{(r)}}{S_j\In e}=S_j\uplus (A\cup L)$.

For $j\in[2]$, let
\begin{align*}
T_{j,A}&:=\set{e\in \cF_{j,A}^{(r)}}{|e\cap S_j|=0}, \\
T_{j,A\cup L}&:=\set{e\in \cF_{j,A\cup L}^{(r)}}{|e\cap S_j|=0}, \\
R_{j,A}&:=\set{e\in \cF_{j,A}^{(r)}}{|e\cap S_j|\in[i-1]}, \\
R_{j,A\cup L}&:=\set{e\in \cF_{j,A\cup L}^{(r)}}{|e\cap S_j|\in[i-1]}.
\end{align*}

By Definition~\ref{def:link extension two sided}, we have that $T_{1,A}=T_{2,A}$ and $T_{1,A\cup L}=T_{2,A\cup L}$. We thus set
\begin{align*}
T_A:=T_{1,A}=T_{2,A} \quad \mbox{and}\quad T_{A\cup L}&:=T_{1,A\cup L}=T_{2,A\cup L}.
\end{align*} Moreover, we have
\begin{align}
\phi(R_{1,A})=R_{2,A} \quad \mbox{and}\quad \phi(R_{1,A\cup L})=R_{2,A\cup L}.\label{mirror}
\end{align}

Note that $R_{1,A},R_{2,A},R_{1,A\cup L},R_{2,A\cup L}$ are empty if $i=1$. Crucially, since $\cF_A$ and $\cF_{A\cup L}$ are $(r-i+1)$-disjoint, it is easy to see (by contradiction) that $T_A$ and $T_{A\cup L}$ are edge-disjoint, and that for $j\in[2]$, $R_{j,A}$ and $R_{j,A\cup L}$ are edge-disjoint.\COMMENT{Suppose to the contrary that there exists $e\in \cF_{j,A}^{(r)}\cap \cF_{j,A\cup L}^{(r)}$ with $|e\cap S_j|<i$. Let $e\in F'_{j,A}\in \cF_{j,A}$ and $e\in F'_{j,A\cup L}\in \cF_{j,A\cup L}$. Thus, $e\In S_j\cup V(F'_{A})$ and $e\In S_j\cup V(F'_{A\cup L})$ for $F'_A\in \cF_A$ and $F'_{A\cup L}\in \cF_{A\cup L}$. But then $|V(F'_A)\cap V(F'_{A\cup L})|\ge |e\sm S_j|>r-i$, a contraction to $\cF_A$ and $\cF_{A\cup L}$ being $(r-i+1)$-disjoint} Further, since $A$ and $L$ are edge-disjoint, we clearly have for $j\in[2]$ that $S_j\uplus L$ and $S_j \uplus A$ are edge-disjoint. Using this, it is straightforward to see that
\begin{itemize}
\item[($\dagger$)] $S_1\uplus L$, $S_2\uplus L$, $S_1\uplus A$, $S_2\uplus A$, $T_A$, $T_{A\cup L}$, $R_{1,A}$, $R_{2,A}$, $R_{1,A\cup L}$, $R_{2,A\cup L}$ are pairwise edge-disjoint subgraphs of $G^{(r)}$.
\end{itemize}
Observe that for $j\in[2]$, we have
\begin{align}
\cF_{j,A}^{(r)} &= (S_j\uplus A) \cupdot R_{j,A} \cupdot T_A\label{link dec:1};\\
\cF_{j,A\cup L}^{(r)} &= (S_j\uplus (A\cup L)) \cupdot R_{j,A\cup L} \cupdot T_{A\cup L}.\label{link dec:2}
\end{align}
Define
\begin{align*}
\hat{T}&:=(S_1\uplus A) \cup (S_2\uplus A) \cup T_A \cup T_{A\cup L}; \\
\hat{\cF}_1&:=\cF_{1,A\cup L} \cup \cF_{2,A};\\
\hat{\cF}_2&:=\cF_{1,A} \cup \cF_{2,A\cup L}.
\end{align*}
We now check that \ref*{lemst:iterative transformer local:location 2}--\ref*{lemst:iterative transformer local:maxdeg 2} hold. First note that by ($\dagger$) we clearly have $\hat{T},R_{1,A},R_{1,A\cup L}\In G^{(r)}$. Moreover, since $\cF_A$ and $\cF_{A\cup L}$ are $(r-i+1)$-disjoint, we have that $\cF_{1,A\cup L}$ and $\cF_{2,A}$ are $r$-disjoint\COMMENT{in fact: $(r-i+1)$-disjoint, $i\ge 1$} and thus $\hat{\cF}_1$ is a $\kappa$-well separated $F$-packing in $G$ by Fact~\ref{fact:ws}\ref{fact:ws:2}. Similarly, $\hat{\cF}_2$ is a $\kappa$-well separated $F$-packing in $G$.

To check \ref*{lemst:iterative transformer local:location 2}, note that $V(R_{1,A})\In V(\cF_{1,A}^{(r)}) \In V(G)\sm S_2$ and $V( R_{1,A\cup L})\In V(\cF_{1,A\cup L}^{(r)}) \In V(G)\sm S_2$ by Proposition~\ref{prop:S cover two sided}\ref{prop:S cover two sided:mirror}. Moreover, for all $e\in R_{1,A}\cup R_{1,A\cup L}$, we have $|e\cap S_1|\in[i-1]$ by definition. Hence, \ref*{lemst:iterative transformer local:location 2} holds. Clearly, \eqref{mirror} and ($\dagger$) imply \ref*{lemst:iterative transformer local:disjoint 2}. Crucially, by~\eqref{mirror}--\eqref{link dec:2} we have that
\begin{align*}
\hat{\cF}_1^{(r)} &= \cF_{1,A\cup L}^{(r)} \cupdot \cF_{2,A}^{(r)}= \hat{T} \cup (S_1\uplus L)\cup R_{1,A\cup L} \cup \phi(R_{1,A});\\
\hat{\cF}_2^{(r)} &= \cF_{1,A}^{(r)} \cupdot \cF_{2,A\cup L}^{(r)}= \hat{T} \cup (S_2\uplus L)\cup R_{1,A} \cup \phi(R_{1,A\cup L}).
\end{align*}
Thus, \ref*{lemst:iterative transformer local:transform 2} is satisfied.
Finally, $|V(\hat{T}\cup R_{1,A} \cup R_{1,A\cup L})|\le |V(G')| \le 2\mu n$, proving the claim.
\endclaimproof

The transformer $\hat{T}$ almost has the required properties, except that to satisfy \ref{lemst:iterative transformer local:transform} we would have needed $R_{1,A\cup L}$ and $\phi(R_{1,A\cup L})$ to be on the `other side' of the transformation. In order to resolve this, we carry out an additional transformation step. (Since $R_{1,A}$ and $R_{1,A\cup L}$ are empty if $i=1$, this additional step is vacuous in this case.)

\begin{NoHyper}\begin{claim}\label{claim:non crucial transformer bit}
There exist $T',R'\In G^{(r)}$ and $1$-well separated $F$-packings $\cF'_1,\cF'_2$ in $G-\hat{\cF}_1^{\le(r+1)}-\hat{\cF}_2^{\le(r+1)}$ such that the following hold:
\begin{enumerate}[label={\rm(tr\arabic*$'$)}]
\item $V(R')\In V(G)\sm S_2$ and $|e\cap S_1|\in[i-1]$ for all $e\in R'$; \label{lemst:iterative transformer local:location 3}
\item $T'$, $R'$, $\phi(R')$, $\hat{T}$, $S_1\uplus L$, $S_2\uplus L$, $R_{1,A}$, $\phi(R_{1,A})$, $R_{1,A\cup L}$, $\phi(R_{1,A\cup L})$ are pairwise edge-disjoint $r$-graphs;\label{lemst:iterative transformer local:disjoint 3}
\item $\cF_1'^{(r)}=T'\cup R_{1,A\cup L}\cup R'$ and $\cF_2'^{(r)}=T'\cup \phi(R_{1,A\cup L})\cup \phi(R')$;\label{lemst:iterative transformer local:transform 3}
\item $|V(T'\cup R')|\le 0.7\gamma n$.\label{lemst:iterative transformer local:maxdeg 3}
\end{enumerate}
\end{claim}\end{NoHyper}

\claimproof
Let $H':=\hat{T}\cup R_{1,A} \cup \phi(R_{1,A}) \cup (S_1\uplus L) \cup (S_2\uplus L)$. Clearly, $\Delta(H')\le 5\mu n$.\COMMENT{$|V(\phi(R_{1,A}))|\le |V(R_{1,A})|$, $|V((S_1\uplus L) \cup (S_2\uplus L))|\le 2\gamma' n$}

Let $W:=V(R_{1,A\cup L})\cup V(\phi(R_{1,A\cup L}))$. By~\ref*{lemst:iterative transformer local:maxdeg 2}, we have that $|W|\le 4\mu n$.
Similarly to the beginning of the proof of Claim~\ref*{claim:crucial transformer bit}, by Corollary~\ref{cor:random induced subcomplex} and Lemma~\ref{lem:chernoff}\ref{chernoff t}, there exists a subset $U'\In V(G)$ with $0.4\gamma n \le |U'| \le 0.6 \gamma n$ such that $G''':=G[U'\cup W]$ is a $(2\eps,\xi-\eps,f,r)$-supercomplex. Let $\tilde{n}:=|U'\cup W|$. Note that $$\Delta(H')\le 5\mu n \le \sqrt{\mu}\tilde{n}\quad \mbox{and}\quad \Delta(\hat{\cF}_j^{\le(r+1)})\le \kappa (f-r)$$ for $j\in[2]$ by Fact~\ref{fact:ws}\ref{fact:ws:maxdeg}. Thus, by Proposition~\ref{prop:noise}\ref{noise:supercomplex}, $$\tilde{G}:=G'''-H'-\hat{\cF}_1^{\le(r+1)}-\hat{\cF}_2^{\le(r+1)}$$ is still a $(3\eps,\xi-2\eps,f,r)$-supercomplex.
For every $e\in R_{1,A\cup L}$, let $$\cQ_e:=\set{Q\in \tilde{G}^{(f)}(e)\cap \tilde{G}^{(f)}(\phi(e))}{Q\cap (S_1\cup S_2)=\emptyset}.$$
By Fact~\ref{fact:connected}, for every $e\in R_{1,A\cup L}\In \tilde{G}^{(r)}$, we have that $|\tilde{G}^{(f)}(e)\cap \tilde{G}^{(f)}(\phi(e))|\ge 0.5\xi \tilde{n}^{f-r}$. Thus, we have that $|\cQ_e|\ge 0.4\xi \tilde{n}^{f-r}$.\COMMENT{There are at most $2i n^{f-r-1}$ $(f-r)$-sets that intersect $S_1\cup S_2$}
Since $\Delta(R_{1,A\cup L}\cup \phi(R_{1,A\cup L}))\le 4\mu n \le \sqrt{\mu}\tilde{n}$, we can apply Lemma~\ref{lem:greedy cover uni} (with $|R_{1,A\cup L}|,2,\Set{e,\phi(e)},\cQ_e$ playing the roles of $m,s,L_j,\cQ_j$)
to find for every $e\in R_{1,A\cup L}$ some $Q_e\in \cQ_e$ such that, writing $K_e:=(Q_e\uplus \Set{e,\phi(e)})^\le$, we have that
\begin{align}
K_e\mbox{ and }K_{e'}\mbox{ are }r\mbox{-disjoint for distinct }e,e'\in R_{1,A\cup L}.\label{disjoint extensions}
\end{align}
For each $e\in R_{1,A\cup L}$, let $\tilde{F}_{e,1}$ and $\tilde{F}_{e,2}$ be copies of $F$ with $V(\tilde{F}_{e,1})=e\cup Q_e$ and $V(\tilde{F}_{e,2})=\phi(e)\cup Q_e$ and such that $e\in \tilde{F}_{e,1}$ and $\phi(\tilde{F}_{e,1})=\tilde{F}_{e,2}$. Clearly, we have that $\phi(e)\in \tilde{F}_{e,2}$. Moreover, since $e\In V( R_{1,A\cup L})\In V(G)\sm S_2$ by \ref*{lemst:iterative transformer local:location 2} and $Q_e\cap (S_1\cup S_2)=\emptyset$, we have $V(\tilde{F}_{e,1})\In V(G)\sm S_2$. Let
\begin{align}
\cF'_1&:=\set{\tilde{F}_{e,1}}{e\in R_{1,A\cup L}};\label{extra transformer 1}\\
\cF'_2&:=\set{\tilde{F}_{e,2}}{e\in R_{1,A\cup L}}.\label{extra transformer 2}
\end{align}
By~\eqref{disjoint extensions}, $\cF'_1$ and $\cF'_2$ are both $1$-well separated $F$-packings in $\tilde{G}\In G-\hat{\cF}_1^{\le(r+1)}-\hat{\cF}_2^{\le(r+1)}$. Moreover, $V(\cF_1'^{(r)}) \In V(G)\sm S_2$ and $\phi(\cF_1'^{(r)})=\cF_2'^{(r)}$. Let
\begin{align}
T'&:=\cF_1'^{(r)}\cap \cF_2'^{(r)};\\
R'&:=\cF_1'^{(r)}-T'-R_{1,A\cup L}.\label{break down 2}
\end{align}

We clearly have $T',R'\In G^{(r)}$ and now check \ref*{lemst:iterative transformer local:location 3}--\ref*{lemst:iterative transformer local:maxdeg 3}. Note that no edge of $T'$ intersects $S_1\cup S_2$. For \ref*{lemst:iterative transformer local:location 3}, we first have that $V(R')\In V(\cF_1'^{(r)})\In V(G)\sm S_2$. Now, consider $e'\in R'$. There exists $e\in R_{1,A\cup L}$ with $e'\in \tilde{F}_{e,1}$ and thus $e'\In e\cup Q_e$. If we had $e'\cap S_1=\emptyset$, then $e'\In (e\sm S_1)\cup Q_e$. Since $\phi(\tilde{F}_{e,1})=\tilde{F}_{e,2}$, it follows that $e'\in T'$, a contradiction to \eqref{break down 2}. Hence, $|e'\cap S_1|>0$. Moreover, by \ref*{lemst:iterative transformer local:location 2} we have $|e'\cap S_1|\le |(e\cup Q_e)\cap S_1| =|e\cap S_1|\le i-1$. Therefore, $|e'\cap S_1|\in [i-1]$ and \ref*{lemst:iterative transformer local:location 3} holds.

In order to check \ref*{lemst:iterative transformer local:transform 3}, observe first that by \eqref{break down 2} and \eqref{extra transformer 1},\COMMENT{\eqref{extra transformer 1} implies that $R_{1,A\cup L}\In \cF_1'^{(r)}$} we have $\cF_1'^{(r)}=T'\cupdot R_{1,A\cup L} \cupdot R'$. Hence, by Fact~\ref{fact:simple proj}\ref{simple proj:union}, we have 
\begin{align}
\cF_2'^{(r)} &= \phi(\cF_1'^{(r)})=\phi(T')\cupdot \phi(R_{1,A\cup L}) \cupdot \phi(R')=T'\cupdot \phi(R_{1,A\cup L}) \cupdot \phi(R'),\label{break down 2 mirror}
\end{align}
so \ref*{lemst:iterative transformer local:transform 3} is satisfied.

We now check \ref*{lemst:iterative transformer local:disjoint 3}.
Note that $T',R',\phi(R')\In \tilde{G}^{(r)}\In G^{(r)}-H'$. Thus, by \ref*{lemst:iterative transformer local:disjoint 2}, it is enough to check that $T',R',\phi(R'),R_{1,A\cup L},\phi(R_{1,A\cup L})$ are pairwise edge-disjoint. Recall that no edge of $T'$ intersects $S_1\cup S_2$. Moreover, for every $e\in R'\cup R_{1,A\cup L}$, we have $|e\cap S_1|\in [i-1]$ and $e\cap S_2=\emptyset$, and for every $e\in \phi(R')\cup \phi(R_{1,A\cup L})$, we have $|e\cap S_2|\in [i-1]$ and $e\cap S_1=\emptyset$. Since $R'$ and $R_{1,A\cup L}$ are edge-disjoint by \eqref{break down 2} and $\phi(R')$ and $\phi(R_{1,A\cup L})$ are edge-disjoint by \eqref{break down 2 mirror}, this implies that $T',R',\phi(R'),R_{1,A\cup L},\phi(R_{1,A\cup L})$ are indeed pairwise edge-disjoint, proving \ref*{lemst:iterative transformer local:disjoint 3}.

Finally, we can easily check that $|V(T'\cup R')|\le \tilde{n}\le 0.7\gamma n$.
\endclaimproof

We now combine the results of Claims~\ref{claim:crucial transformer bit} and~\ref{claim:non crucial transformer bit}. Let
\begin{align*}
T&:=\hat{T}\cup R_{1,A\cup L} \cup \phi(R_{1,A\cup L}) \cup T';\\
R&:=R_{1,A}\cup R';\\
\cF_1&:= \hat{\cF}_1\cup \cF'_2; \\
\cF_2&:=\hat{\cF}_2\cup \cF'_1.
\end{align*}
Clearly, \ref*{lemst:iterative transformer local:location 2} and~\ref*{lemst:iterative transformer local:location 3} imply that \ref{lemst:iterative transformer local:location} holds. Moreover, \ref*{lemst:iterative transformer local:disjoint 3} implies that $T$ is edge-disjoint from both $(S_1\uplus L)\cup \phi(R)$ and $(S_2\uplus L)\cup R$. Using~\ref*{lemst:iterative transformer local:transform 2} and \ref*{lemst:iterative transformer local:transform 3}, observe that
\begin{align*}
T\cup (S_1\uplus L)\cup \phi(R) &= \hat{T}\cup R_{1,A\cup L} \cup \phi(R_{1,A\cup L}) \cup T' \cup (S_1\uplus L) \cup \phi(R_{1,A}) \cup \phi(R')\\
                                &= (\hat{T} \cup (S_1\uplus L)\cup R_{1,A\cup L} \cup \phi(R_{1,A})) \cupdot ( T' \cup \phi(R_{1,A\cup L})\cup \phi(R'))\\
																&=\hat{\cF}_1^{(r)} \cupdot \cF_2'^{(r)}=\cF_1^{(r)}.
\end{align*}
Similarly, $\cF_2^{(r)}=\hat{\cF}_2^{(r)} \cupdot \cF_1'^{(r)}=T\cup (S_2\uplus L)\cup R$. In particular, by Fact~\ref{fact:ws}\ref{fact:ws:1} we can see that $\cF_1$ and $\cF_2$ are $(\kappa+1)$-well separated $F$-packings in $G$. Thus, $T$ is a $(\kappa+1)$-well separated $((S_1\uplus L)\cup \phi(R),(S_2\uplus L)\cup R;F)$-transformer in $G$, so~\ref{lemst:iterative transformer local:transform} holds.
Finally, we have $|V(T\cup R)|\le 4\mu n +0.7\gamma n \le \gamma n$ by \ref*{lemst:iterative transformer local:maxdeg 2} and~\ref*{lemst:iterative transformer local:maxdeg 3}.
\endproof

So far, our maps $\phi\colon S_1\to S_2$ were bijections. When $\phi$ is an edge-bijective homomorphism from $H$ to $H'$, $\phi$ is in general not injective. In order to still have a meaningful notion of `mirroring' as before, we introduce the following notation.

\begin{defin}\label{def:projectable}
Let $V$ be a set and let $V_1,V_2$ be disjoint subsets of $V$, and let $\phi\colon V_1\to V_2$ be a map. For a set $S\In V\sm V_2$, define $\phi(S):=(S\sm V_1)\cup \phi(S\cap V_1)$. Let $r\in \bN$ and suppose that $R$ is an $r$-graph with $V(R)\In V$ and $i\in[r]_0$. We say that $R$ is \defn{$(\phi,V,V_1,V_2,i)$-projectable} if the following hold:
\begin{enumerate}[label={\rm(Y\arabic*)}]
\item for every $e\in R$, we have that $e\cap V_2=\emptyset$ and $|e\cap V_1|\in [i]$ (so if $i=0$, then $R$ must be empty since $[0]=\emptyset$);\label{projectable:intersection}
\item for every $e\in R$, we have $|{\phi}(e)|=r$;\label{projectable:size}
\item for every two distinct edges $e,e'\in R$, we have ${\phi}(e)\neq {\phi}(e')$.\label{projectable:no multi}
\end{enumerate}
Note that if $\phi$ is injective and $e\cap V_2=\emptyset$ for all $e\in R$, then \ref{projectable:size} and~\ref{projectable:no multi} always hold. If $R$ is $(\phi,V,V_1,V_2,i)$-projectable, then let $\phi(R)$ be the $r$-graph on ${\phi}(V(R)\sm V_2)$ with edge set $\set{{\phi}(e)}{e\in R}$.
For an $r$-graph $P$ with $V(P)\In V\sm V_2$ that satisfies \ref{projectable:size}, let $P^\phi$ be the $r$-graph on $V(P)\cup V_1$ that consists of all $e\in\binom{V\sm V_2}{r}$ such that ${\phi}(e)={\phi}(e')$ for some $e'\in P$.\COMMENT{Let $e$ be an $r$-set such that ${\phi}(e)={\phi}(e')$ for some $e'\in P$. We must then have $e\sm V_1=e'\sm V_1\In e'\In V(P)$, hence $e\In V(P)\cup V_1$.}
\end{defin}

The following facts are easy to see.

\begin{prop}[\cite{GKLO}]\label{prop:projectable facts}
Let $V,V_1,V_2,\phi,R,r,i$ be as above and assume that $R$ is $(\phi,V,V_1,V_2,i)$-projectable. Then the following hold:
\begin{enumerate}[label={\rm(\roman*)}]
\item $R\rightsquigarrow \phi(R)$;\COMMENT{Don't really apply this}
\item every subgraph of $R$ is $(\phi,V,V_1,V_2,i)$-projectable;\label{prop:projectable facts:subgraph}
\item for all $e'\in \phi(R)$, we have $e'\cap V_1=\emptyset$ and $|e'\cap V_2|\in [i]$;\label{prop:projectable facts:projected intersection}
\item assume that for all $e\in R$, we have $|e\cap V_1|=i$, and let $\cS$ contain all $S\in \binom{V_1}{i}$ such that $S$ is contained in some edge of $R$, then $$R=\mathop{\dot{\bigcup}}_{S\in \cS} (S\uplus R(S))\quad\mbox{ and }\quad\phi(R)=\mathop{\dot{\bigcup}}_{S\in \cS} (\phi(S)\uplus R(S)).$$\label{prop:projectable facts:splitting}
\end{enumerate}
\end{prop}

We can now prove the Transforming lemma by combining many localised transformers.

\lateproof{Lemma~\ref{lem:transformer}}
We can assume that $1/\kappa \ll \gamma \ll 1/h,\eps$. Choose new constants $\kappa'\in \bN$ and $\gamma_2,\dots,\gamma_r,\gamma_2'\dots,\gamma_r'>0$ such that $$1/n\ll 1/\kappa \ll \gamma_{r} \ll \gamma_r' \ll \gamma_{r-1} \ll \gamma_{r-1}' \ll \dots \ll \gamma_2 \ll \gamma_2' \ll \gamma \ll 1/\kappa', 1/h, \eps \ll \xi,1/f.$$

Let $\phi\colon V(H)\to V(H')$ be an edge-bijective homomorphism from $H$ to $H'$. Extend $\phi$ as in Definition~\ref{def:projectable} with $V(H),V(H')$ playing the roles of $V_1,V_2$. Since $\phi$ is edge-bijective, we have that
\begin{align}
\phi{\restriction_{S}}\mbox{ is injective whenever }S\In e\mbox{ for some }e\in H.\label{projectable:restriction injective}
\end{align}
For every $e\in H$, we have $|G^{(f)}(e)\cap G^{(f)}(\phi(e))|\ge 0.5\xi n^{f-r}$ by Fact~\ref{fact:connected}. It is thus easy\COMMENT{only constantly many, no finding lemma needed} to find for each $e\in H$ some $Q_e\in G^{(f)}(e)\cap G^{(f)}(\phi(e))$ with $Q_e\cap (V(H)\cup V(H'))=\emptyset$ such that $Q_e\cap Q_{e'}=\emptyset$ for all distinct $e,e'\in H$. For each $e\in H$, let $\tilde{F}_{e,1}$ and $\tilde{F}_{e,2}$ be copies of $F$ with $V(\tilde{F}_{e,1})=e\cup Q_e$ and $V(\tilde{F}_{e,2})=\phi(e)\cup Q_e$ and such that $e\in \tilde{F}_{e,1}$ and $\phi(\tilde{F}_{e,1})=\tilde{F}_{e,2}$. Clearly, we have that $\phi(e)\in \tilde{F}_{e,2}$. For $j\in[2]$, define $\cF_{r,j}^\ast:=\set{\tilde{F}_{e,j}}{e\in H}$. Clearly, $\cF_{r,1}^\ast$ and $\cF_{r,2}^\ast$ are both $1$-well separated $F$-packings in $G$. Define
\begin{align}
T_r^\ast&:=\cF_{r,1}^{\ast(r)}\cap \cF_{r,2}^{\ast(r)},\label{transformer middle}\\
R_r^\ast&:= \cF_{r,1}^{\ast(r)}-T_r^\ast-H.\nonumber
\end{align}

Let $\gamma_1:=\gamma$. Furthermore, let $\kappa_r:=1$ and recursively define $\kappa_i:=\kappa_{i+1}+\binom{h}{i}\kappa'$ for all $i\in[r-1]$.

Given $i\in[r-1]_0$ and $T_{i+1}^\ast,R_{i+1}^\ast,\cF_{i+1,1}^\ast,\cF_{i+1,2}^\ast$, we define the following conditions:
\begin{enumerate}[label={\rm(TR\arabic*$^\ast$)}]
\item$\hspace{-5pt}_{i}$ $R_{i+1}^\ast$ is $(\phi,V(G),V(H),V(H'),i)$-projectable;\label{iterative transformer:location}
\item$\hspace{-5pt}_{i}$ $T_{i+1}^\ast,R_{i+1}^\ast,\phi(R_{i+1}^\ast),H,H'$ are edge-disjoint subgraphs of $G^{(r)}$;\label{iterative transformer:disjoint}
\item$\hspace{-5pt}_{i}$ $\cF_{i+1,1}^\ast$ and $\cF_{i+1,2}^\ast$ are $\kappa_{i+1}$-well separated $F$-packings in $G$ with $\cF_{i+1,1}^{\ast(r)}=T_{i+1}^\ast\cup H\cup R_{i+1}^\ast$ and $\cF_{i+1,2}^{\ast(r)}=T_{i+1}^\ast\cup H'\cup \phi(R_{i+1}^\ast)$;\label{iterative transformer:transform}
\item$\hspace{-5pt}_{i}$ $|V(T_{i+1}^\ast\cup R_{i+1}^\ast)|\le \gamma_{i+1} n$.\label{iterative transformer:maxdeg}
\end{enumerate}
We will first show that the above choices of $T_r^\ast,R_r^\ast,\cF_{r,1}^\ast,\cF_{r,2}^\ast$ satisfy \ref{iterative transformer:location}$_{r-1}$--\ref{iterative transformer:maxdeg}$_{r-1}$. We will then proceed inductively until we obtain $T_1^\ast,R_1^\ast,\cF_{1,1}^\ast,\cF_{1,2}^\ast$ satisfying \ref{iterative transformer:location}$_{0}$--\ref{iterative transformer:maxdeg}$_{0}$, which will then easily complete the proof.

\begin{NoHyper}\begin{claim}
$T_r^\ast,R_r^\ast,\cF_{r,1}^\ast,\cF_{r,2}^\ast$ satisfy \ref{iterative transformer:location}$_{r-1}$--\ref{iterative transformer:maxdeg}$_{r-1}$.
\end{claim}\end{NoHyper}

\claimproof \ref{iterative transformer:maxdeg}$_{r-1}$ clearly holds. To see~\ref{iterative transformer:location}$_{r-1}$, consider any $e'\in R_r^\ast$. There exists $e\in H$ such that $e'\in \tilde{F}_{e,1}$. In particular, $e'\In e\cup Q_e$. If $e'\In V(H)$, then $e'=e\in H$, and if $e'\cap V(H)=\emptyset$, then $e'\in \tilde{F}_{e,2}$ since $\phi(\tilde{F}_{e,1})=\tilde{F}_{e,2}$ and thus $e'\in T_r^\ast$. Hence, by definition of $R_r^\ast$, we must have $|e'\cap V(H)|\in[r-1]$. Clearly, $e'\cap V(H')\In (e\cup Q_e)\cap V(H')=\emptyset$, so~\ref{projectable:intersection} holds. Moreover, $e'\cap V(H)\In e$, so $\phi{\restriction_{e'\cap V(H)}}$ is injective by~\eqref{projectable:restriction injective}, and~\ref{projectable:size} holds. Let $e',e''\in R_r^\ast$ and suppose that $\phi(e')=\phi(e'')$. We thus have $e'\sm V(H)=e''\sm V(H)\neq \emptyset$. Since the $Q_e$'s were chosen to be vertex-disjoint, we must have $e',e''\In e\cup Q_e$ for some $e\in H$. Hence, $(e'\cup e'')\cap V(H)\In e$ and so $\phi{\restriction_{(e'\cup e'')\cap V(H)}}$ is injective by~\eqref{projectable:restriction injective}. Since $\phi(e'\cap V(H))=\phi(e''\cap V(H))$ by assumption, we have $e'\cap V(H)=e''\cap V(H)$, and thus $e'=e''$. Altogether, \ref{projectable:no multi} holds, so \ref{iterative transformer:location}$_{r-1}$ is satisfied. In particular, $\phi(R_r^\ast)$ is well-defined. Observe that $$\phi(R_r^\ast)=\cF_{r,2}^{\ast(r)}-T_r^\ast-H'.$$

Clearly, $T_{r}^\ast,R_{r}^\ast,\phi(R_{r}^\ast),H,H'$ are subgraphs of $G^{(r)}$. Using Proposition~\ref{prop:projectable facts}\ref{prop:projectable facts:projected intersection}, it is easy to see that they are indeed edge-disjoint, so \ref{iterative transformer:disjoint} holds. Moreover, note that $\cF_{r,1}^\ast$ and $\cF_{r,2}^\ast$ are $1$-well separated $F$-packings in $G$ with $\cF_{r,1}^{\ast(r)}=T_r^\ast \cup H\cup R_{r}^\ast$ and $\cF_{r,2}^{\ast(r)}=T_r^\ast \cup H'\cup \phi(R_{r}^\ast)$, so $T_r^\ast$ satisfies \ref{iterative transformer:transform}$_{r-1}$.
\endclaimproof

Suppose that for some $i\in[r-1]$,\COMMENT{We assume $r\ge 2$ throughout the section} we have already found $T_{i+1}^\ast,R_{i+1}^\ast,\cF_{{i+1},1}^\ast,\cF_{{i+1},2}^\ast$ such that \ref{iterative transformer:location}$_{i}$--\ref{iterative transformer:maxdeg}$_{i}$ hold.
We will now find $T_{i}^\ast,R_{i}^\ast,\cF_{i,1}^\ast,\cF_{i,2}^\ast$ such that \ref{iterative transformer:location}$_{i-1}$--\ref{iterative transformer:maxdeg}$_{i-1}$ hold.
To this end, let $$R_i:=\set{e\in R_{i+1}^\ast}{|e\cap V(H)|=i}.$$ By Proposition~\ref{prop:projectable facts}\ref{prop:projectable facts:subgraph}, $R_i$ is $(\phi,V(G),V(H),V(H'),i)$-projectable. Let $\cS_i$ be the set of all $S\in \binom{V(H)}{i}$ such that $S$ is contained in some edge of $R_i$. For each $S\in \cS_i$, let $L_S:=R_i(S)$. By Proposition~\ref{prop:projectable facts}\ref{prop:projectable facts:splitting}, we have that
\begin{align}
R_i=\mathop{\dot{\bigcup}}_{S\in \cS_i} (S\uplus L_S)\quad\mbox{ and }\quad\phi(R_i)=\mathop{\dot{\bigcup}}_{S\in \cS_i} (\phi(S)\uplus L_S).\label{break down}
\end{align}
We intend to apply Lemma~\ref{lem:iterative transforming local} to each pair $S,\phi(S)$ with $S\in \cS_i$ individually. For each $S\in \cS_i$, define $$V_S:=(V(G)\sm (V(H)\cup V(H')))\cup S \cup \phi(S).$$

\begin{NoHyper}\begin{claim}\label{claim:transformer building link graph}
For every $S\in \cS_i$, $L_S\In G[V_S](S)^{(r-i)}\cap G[V_S](\phi(S))^{(r-i)}$ and $|V(L_S)|\le 1.1\gamma_{i+1} |V_S|$.
\end{claim}\end{NoHyper}

\claimproof
The second assertion clearly holds by~\ref{iterative transformer:maxdeg}$_{i}$. To see the first one, let $e'\in L_S=R_i(S)$. Since $R_i\In R_{i+1}^\ast\In G^{(r)}$, we have $e'\in G(S)^{(r-i)}$. Moreover, $\phi(S)\cup e'\in \phi(R_i)\In  \phi(R_{i+1}^\ast) \In G^{(r)}$ by \eqref{break down}. Since $R_{i+1}^\ast$ is $(\phi,V(G),V(H),V(H'),i)$-projectable, we have that $e'\cap (V(H)\cup V(H'))=\emptyset$. Thus, $S\cup e'\In V_S$ and $\phi(S)\cup e'\In V_S$.
\endclaimproof

Let $S^\ast\in \binom{V(F)}{i}$ be such that $F(S^\ast)$ is non-empty.

\begin{NoHyper}\begin{claim}\label{claim:transformer building link graph divisible}
For every $S\in \cS_i$, $L_S$ is $F(S^\ast)$-divisible.
\end{claim}\end{NoHyper}

\claimproof
Consider $b\In V(L_S)$ with $|b|< r-i$. We have to check that $Deg(F(S^\ast))_{|b|}\mid |L_S(b)|$. By \ref{iterative transformer:transform}$_{i}$, both $T_{i+1}^\ast\cup H\cup R_{i+1}^\ast$ and $T_{i+1}^\ast\cup H'\cup \phi(R_{i+1}^\ast)$ are necessarily $F$-divisible. Clearly, $H'$ does not contain an edge that contains $S$. Note that by \ref{iterative transformer:location}$_{i}$ and Proposition~\ref{prop:projectable facts}\ref{prop:projectable facts:projected intersection}, $\phi(R_{i+1}^\ast)$ does not contain an edge that contains $S$ either, hence $|T_{i+1}^\ast(S\cup b)|=|(T_{i+1}^\ast\cup H' \cup \phi(R_{i+1}^\ast))(S\cup b)|\equiv 0\mod{Deg(F)_{|S\cup b|}}$. Moreover, since $H$ is $F$-divisible, we have $|(T_{i+1}^\ast \cup R_{i+1}^\ast)(S\cup b)|\equiv |(T_{i+1}^\ast \cup H \cup R_{i+1}^\ast)(S\cup b)|\equiv 0 \mod{Deg(F)_{|S\cup b|}}$. Thus, we have $Deg(F)_{|S\cup b|}\mid |R_{i+1}^\ast(S\cup b)|$. Moreover, $|R_{i+1}^\ast(S\cup b)|=|R_{i}(S\cup b)|=|L_S(b)|$. Hence, $Deg(F)_{|S\cup b|}\mid |L_S(b)|$, which proves the claim as $Deg(F)_{|S\cup b|}=Deg(F(S^\ast))_{|b|}$ by Proposition~\ref{prop:link divisibility}.
\endclaimproof

We now intend to apply Lemma~\ref{lem:iterative transforming local} for every $S\in \cS_i$ in order to define $T_S,R_S\In G^{(r)}$ and $\kappa'$-well separated $F$-packings $\cF_{S,1},\cF_{S,2}$ in $G$ such that the following hold:
\begin{enumerate}[label={\rm(TR\arabic*$'$)}]
\item $R_S$ is $(\phi,V(G),V(H),V(H'),i-1)$-projectable;\label{iterative transformer local:location}
\item $T_S,R_S,\phi(R_S),S\uplus L_S, \phi(S)\uplus L_S$ are edge-disjoint;\label{iterative transformer local:disjoint}
\item $\cF_{S,1}^{(r)}=T_S \cup (S\uplus L_S)\cup \phi(R_S)$ and $\cF_{S,2}^{(r)}=T_S \cup (\phi(S)\uplus L_S)\cup R_S$;\label{iterative transformer local:transform}
\item $|V(T_S\cup R_S)|\le \gamma_{i+1}' n$.\label{iterative transformer local:maxdeg}
\end{enumerate}
We also need to ensure that all these graphs and packings satisfy several `disjointness properties' (see \ref{transformer building blocks:edge-disjoint}--\ref{transformer building blocks:decs disjoint}), and we will therefore choose them successively. Recall that $P^{\phi}$ (for a given $r$-graph $P$) was defined in Definition~\ref{def:projectable}.
Let $\cS'\In \cS_i$ be the set of all $S'\in \cS_i$ for which $T_{S'},R_{S'}$ and $\cF_{S',1},\cF_{S',2}$ have already been defined such that \ref{iterative transformer local:location}--\ref{iterative transformer local:maxdeg} hold. Suppose that next we want to find $T_S$, $R_S$, $\cF_{S,1}$ and $\cF_{S,2}$. Let
\begin{align*}
P_{S}&:=R_{i+1}^\ast \cup \bigcup_{S'\in \cS'}R_{S'},\\
M_S&:=T_{i+1}^\ast \cup R_{i+1}^\ast \cup \phi(R_{i+1}^\ast) \cup \bigcup_{S'\in \cS'}(T_{S'}\cup R_{S'} \cup \phi(R_{S'})),\\
O_S&:=\cF_{{i+1},1}^{\ast\le(r+1)}\cup \cF_{{i+1},2}^{\ast\le(r+1)} \cup \bigcup_{S'\in \cS'}\cF_{S',1}^{\le(r+1)}\cup \cF_{S',2}^{\le(r+1)},\\
G_S&:=G[V_S]-((M_S\cup P_S^\phi) -((S\uplus L_S) \cup (\phi(S)\uplus L_S)))-O_S.
\end{align*}
Observe that \ref{iterative transformer:maxdeg}$_{i}$ and \ref{iterative transformer local:maxdeg} imply that
\begin{align*}
|V(M_S\cup P_S)| &\le |V(T_{i+1}^\ast\cup R_{i+1}^\ast \cup \phi(R_{i+1}^\ast))| +\sum_{S'\in \cS'}|V(T_{S'}\cup R_{S'} \cup \phi(R_{S'}))|\\
                                                       & \le 2\gamma_{i+1} n + 2\binom{h}{i}\gamma_{i+1}' n \le \gamma_i n.
\end{align*}
In particular, $|V(P_S^\phi)|\le |V(P_S)\cup V(H)| \le \gamma_i n +h$. Moreover, by Fact~\ref{fact:ws}\ref{fact:ws:maxdeg}, \ref{iterative transformer:transform}$_{i}$ and \ref{iterative transformer local:transform}, we have that $\Delta(O_S)\le (2\kappa_{i+1}+2\binom{h}{i}\kappa')(f-r)$.
Thus, by Proposition~\ref{prop:noise}\ref{noise:supercomplex} $G_S$ is still a $(2\eps,\xi/2,f,r)$-supercomplex. Moreover, note that $L_S\In G_S(S)^{(r-i)}\cap G_S(\phi(S))^{(r-i)}$ and $|V(L_S)|\le 1.1\gamma_{i+1}|V_S|$ by Claim~\ref*{claim:transformer building link graph} and that $L_S$ is $F(S^\ast)$-divisible by Claim~\ref*{claim:transformer building link graph divisible}.

Finally, by definition of $\cS_i$, $S$ is contained in some $e\in R_i$. Since $R_i$ satisfies \ref{projectable:size} by \ref{iterative transformer:location}$_{i}$, we know that $\phi{\restriction_{e}}$ is injective. Thus, $\phi{\restriction_{S}}\colon S\to \phi(S)$ is a bijection. We can thus apply Lemma~\ref{lem:iterative transforming local} with the following objects/parameters:

\smallskip
{\footnotesize
\noindent
{
\begin{tabular}{c|c|c|c|c|c|c|c|c|c|c|c|c|c|c|c|c}
object/parameter & $G_S$ & $i$ & $S$ & $\phi(S)$ & $\phi{\restriction_{S}}$ & $L_S$ & $1.1\gamma_{i+1}$ & $\gamma_{i+1}'$ & $2\eps$ & $|V_S|$ & $\xi/2$ & $f$ & $r$ & $F$ & $S^\ast$ & $\kappa'/2$\\ \hline
playing the role of & $G$ & $i$ & $S_1$ & $S_2$ & $\phi$ & $L$ & $\gamma'$& $\gamma$ & $\eps$ & $n$ & $\xi$ & $f$ & $r$ & $F$ & $S^\ast$ & $\kappa$
\end{tabular}
}
}
\newline \vspace{0.2cm}

This yields $T_S,R_S\In G_S^{(r)}$ and $\kappa'/2$-well separated $F$-packings $\cF_{S,1},\cF_{S,2}$ such that \ref{iterative transformer local:disjoint}--\ref{iterative transformer local:maxdeg} hold, $V(R_S)\In V(G_S) \sm \phi(S)$ and $|e\cap S|\in[i-1]$ for all $e\in R_S$. Note that the latter implies that $R_S$ is $(\phi,V(G),V(H),V(H'),i-1)$-projectable as $V(H)\cap V(G_S)=S$ and $V(H')\cap V(G_S)=\phi(S)$, so~\ref{iterative transformer local:location} holds as well. Moreover, using \ref{iterative transformer:disjoint}$_{i}$ and \ref{iterative transformer local:disjoint} it is easy to see that our construction ensures that
\begin{enumerate}[label={\rm(\alph*)}]
\item $H,H',T_{i+1}^\ast,R_{i+1}^\ast,\phi(R_{i+1}^\ast),(T_S)_{S\in \cS_i},(R_S)_{S\in \cS_i},(\phi(R_S))_{S\in \cS_i}$ are pairwise edge-disjoint;\label{transformer building blocks:edge-disjoint}
\item for all distinct $S,S'\in \cS_i$ and all $e\in R_S$, $e'\in R_{S'}$, $e''\in R_{i+1}^\ast-R_i$ we have that $\phi(e)$, $\phi(e')$ and $\phi(e'')$ are pairwise distinct;\label{transformer building blocks:no multi}
\item for any $j,j'\in[2]$ and all distinct $S,S'\in \cS_i$, $\cF_{S,j}$ is $(r+1)$-disjoint from $\cF^\ast_{{i+1},j'}$ and from $\cF_{S',j'}$.\label{transformer building blocks:decs disjoint}
\end{enumerate}
Indeed, \ref{transformer building blocks:edge-disjoint} holds by the choice of $M_S$,\COMMENT{When we choose $T_S$ and $R_S$, we exclude all forbidden edges except $(S\uplus L_S)\cup (\phi(S)\uplus L_S)$. But $T_S,R_S,\phi(R_S)$ do automatically not contain any edges from there, and the three picked graphs are automatically edge-disjoint from each other.} \ref{transformer building blocks:no multi} holds by definition of $P_S^\phi$,\COMMENT{Proof of \ref{transformer building blocks:no multi}: Assume $R_S$ was chosen after $R_{S'}$, so $e',e''\in P_S$. Now, if $\phi(e)=\phi(e')$ or $\phi(e)=\phi(e'')$, then $e\in P_S^{\phi}$. Since $e\in R_S\In G_S^{(r)}$, we must have $e\in S\uplus L_S$, a contradiction since $R_S$ is edge-disjoint from $S\uplus L_S$.} and \ref{transformer building blocks:decs disjoint} holds by definition of $O_S$.
Let
\begin{align*}
T_{i}^\ast &:=T_{i+1}^\ast \cup R_i \cup \phi(R_i) \cup \bigcup_{S\in \cS_i} T_S;\\
R_{i}^\ast &:=(R_{i+1}^\ast- R_i) \cup \bigcup_{S\in \cS_i} R_S;\\
\cF_{i,1}^\ast &:= \cF_{i+1,1}^\ast \cup \bigcup_{S\in \cS_i}\cF_{S,2};\\
\cF_{i,2}^\ast &:= \cF_{i+1,2}^\ast \cup \bigcup_{S\in \cS_i}\cF_{S,1}.
\end{align*}
Using \ref{iterative transformer:transform}$_{i}$, \ref{iterative transformer local:transform}, \ref{transformer building blocks:edge-disjoint} and \eqref{break down}, it is easy to check that both $\cF_{i,1}^\ast$ and $\cF_{i,2}^\ast$ are $F$-packings in $G$.
We check that \ref{iterative transformer:location}$_{i-1}$--\ref{iterative transformer:maxdeg}$_{i-1}$ hold.
Using \ref{iterative transformer:maxdeg}$_{i}$ and \ref{iterative transformer local:maxdeg}, we can confirm that
\begin{align*}
|V(T_{i}^\ast\cup R_{i}^\ast)| &\le |V(T_{i+1}^\ast\cup R_{i+1}^\ast \cup \phi(R_{i+1}^\ast))| +\sum_{S\in \cS_i}|V(T_S\cup R_S)|\\
                                                       & \le 2\gamma_{i+1} n + \binom{h}{i}\gamma_{i+1}' n \le \gamma_i n,
\end{align*}
so \ref{iterative transformer:maxdeg}$_{i-1}$ holds.

In order to check \ref{iterative transformer:location}$_{i-1}$, i.e.~that $R_i^\ast$ is $(\phi,V(G),V(H),V(H'),i-1)$-projectable, note that \ref{projectable:intersection} and~\ref{projectable:size} hold by~\ref{iterative transformer:location}$_{i}$, the definition of $R_i$ and \ref{iterative transformer local:location}. Moreover, \ref{projectable:no multi} is implied by~\ref{iterative transformer:location}$_{i}$,\COMMENT{if $e,e'\in R_{i+1}^\ast-R_i$} \ref{iterative transformer local:location}\COMMENT{if $e,e'\in R_S$} and~\ref{transformer building blocks:no multi}.\COMMENT{if $e\in R_S$ and ($e'\in R_{S'}$ or $e'\in R_{i+1}^\ast-R_i$)}

Moreover, \ref{iterative transformer:disjoint}$_{i-1}$ follows from \ref{transformer building blocks:edge-disjoint}.
Finally, we check \ref{iterative transformer:transform}$_{i-1}$. Observe that
\begin{eqnarray*}
T_{i}^\ast \cup H \cup R_i^\ast &=&T_{i+1}^\ast \cup R_i \cup \phi(R_i) \cup \bigcup_{S\in \cS_i} T_S \cup H \cup (R_{i+1}^\ast- R_i) \cup \bigcup_{S\in \cS_i} R_S\\
                                &\overset{\eqref{break down}}{=}&(T_{i+1}^\ast \cup H \cup R_{i+1}^\ast) \cup \bigcup_{S\in \cS_i}(T_S\cup (\phi(S)\uplus L_S) \cup R_S),\\								
T_{i}^\ast \cup H' \cup \phi(R_i^\ast) &=&T_{i+1}^\ast \cup R_i \cup \phi(R_i) \cup \bigcup_{S\in \cS_i} T_S \cup H' \cup (\phi(R_{i+1}^\ast)- \phi(R_i)) \cup \bigcup_{S\in \cS_i} \phi(R_S)\\
                                &\overset{\eqref{break down}}{=}&(T_{i+1}^\ast \cup H' \cup \phi(R_{i+1}^\ast)) \cup \bigcup_{S\in \cS_i}(T_S\cup (S\uplus L_S) \cup \phi(R_S)).
\end{eqnarray*}
Thus, by \ref{iterative transformer:transform}$_{i}$ and \ref{iterative transformer local:transform}, $\cF_{i,1}^\ast$ is an $F$-decomposition of $T_{i}^\ast \cup H \cup R_i^\ast$ and $\cF_{i,2}^\ast$ is an $F$-decomposition of $T_{i}^\ast \cup H' \cup \phi(R_i^\ast)$. Moreover, by \ref{transformer building blocks:decs disjoint} and Fact~\ref{fact:ws}\ref{fact:ws:1}, $\cF_{i,1}^\ast$ and $\cF_{i,2}^\ast$ are both $(\kappa_{i+1}+\binom{h}{i}\kappa')$-well separated in $G$. Since $\kappa_{i+1}+\binom{h}{i}\kappa'=\kappa_{i}$, this establishes \ref{iterative transformer:transform}$_{i-1}$.

Finally, let $T_1^\ast,R_1^\ast,\cF_{1,1}^\ast,\cF_{1,2}^\ast$ satisfy \ref{iterative transformer:location}$_{0}$--\ref{iterative transformer:maxdeg}$_{0}$. Note that $R_1^\ast$ is empty by \ref{iterative transformer:location}$_{0}$ and~\ref{projectable:intersection}. Moreover, $T_1^\ast\In G^{(r)}$ is edge-disjoint from $H$ and $H'$ by \ref{iterative transformer:disjoint}$_{0}$ and $\Delta(T_1^\ast)\le \gamma_1 n$ by \ref{iterative transformer:maxdeg}$_{0}$. Most importantly, $\cF_{1,1}^\ast$ and $\cF_{1,2}^\ast$ are $\kappa_{1}$-well separated $F$-packings in $G$ with $\cF_{1,1}^{\ast(r)}=T_{1}^\ast\cup H$ and $\cF_{1,2}^{\ast(r)}=T_{1}^\ast\cup H'$ by~\ref{iterative transformer:transform}$_{0}$. Therefore, $T_1^\ast$ is a $\kappa_1$-well separated $(H,H';F)$-transformer in $G$ with $\Delta(T_1^\ast)\le \gamma_1 n$. Recall that $\gamma_1=\gamma$ and note that $\kappa_1\le 2^h \kappa' \le \kappa$. Thus, $T_1^\ast$ is the desired transformer.
\endproof

\section{Covering down}\label{app:covering down}

Here, we prove the Cover down lemma (Lemma~\ref{lem:cover down}). Our proof proceeds by induction on the `type' of the edges to be covered. In order to carry out the induction step we will actually prove a significantly stronger result, the `Cover down lemma for setups' (Lemma~\ref{lem:horrible}), from which Lemma~\ref{lem:cover down} immediately follows. In each step, some edges will be covered via an inductive argument, and the remaining ones via the `Localised cover down lemma' (Lemma~\ref{lem:dense jain}).

\subsection{Systems and focuses}

\begin{defin}
Given $i \in \bN_0$, an \defn{$i$-system in a set $V$} is a collection $\cS$ of distinct subsets of $V$ of size~$i$. A subset of $V$ is called \defn{$\cS$-important} if it contains some $S\in \cS$, otherwise we call it \defn{$\cS$-unimportant}. We say that $\cU=(U_S)_{S\in \cS}$ is a \defn{focus for $\cS$} if for each $S\in \cS$, $U_S$ is a subset of $V\sm S$.
\end{defin}

\begin{defin}
Let $G$ be a complex and $\cS$ an $i$-system in $V(G)$. We call $G$ \defn{$r$-exclusive with respect to $\cS$} if every $e\in G$ with $|e|\ge r$ contains at most one element of $\cS$. Let $\cU$ be a focus for $\cS$. If $G$ is $r$-exclusive with respect to $\cS$, the following functions are well-defined: For $r'\ge r$, let $\cE_{r'}$ denote the set of $\cS$-important $r'$-sets in $G$. Define $\tau_{r'}\colon \cE_{r'}\to [r'-i]_0$ as $\tau_{r'}(e):=|e\cap U_S|$, where $S$ is the unique $S\in \cS$ contained in $e$. We call $\tau_{r'}$ the \defn{type function of $G^{(r')}$, $\cS$, $\cU$}.
\end{defin}

\begin{defin}
Let $G$ be a complex and $\cS$ an $i$-system in $V(G)$. Let $\cU$ be a focus for $\cS$ and suppose that $G$ is $r$-exclusive with respect to $\cS$. For $i'\in\Set{i+1,\dots,r-1}$, we define $\cT$ as the set of all $i'$-subsets $T$ of $V(G)$ which satisfy $S\In T\In e\sm U_S$ for some $S\in \cS$ and $e\in G^{(r)}$. We call $\cT$ the \defn{$i'$-extension of $\cS$ in $G$ around $\cU$}.

Clearly, $\cT$ is an $i'$-system in $V(G)$. Moreover, note that for every $T\in \cT$, there is a unique $S\in \cS$ with $S\In T$ because $G$ is $r$-exclusive with respect to $\cS$.\COMMENT{Otherwise $e\in G^{(r)}$ with $T\In e$ would contain two elements of $\cS$.} We let $T{\restriction_{\cS}}:=S$ denote this element. (On the other hand, we may have $|\cT|<|\cS|$.) Note that $\cU':=\set{U_{T{\restriction_{\cS}}}}{T\in \cT}$ is a focus for $\cT$ as $T\cap U_{T{\restriction_{\cS}}}=\emptyset$ for all $T\in \cT$.
\end{defin}

The following proposition contains some basic properties of $i'$-extensions, which are straightforwardly checked using the definitions (see also Step~1 in the proof of Lemma~10.22 in \cite{GKLO}).

\begin{prop}\label{prop:extension types}
Let $0\le i < i'<r$. Let $G$ be a complex and $\cS$ an $i$-system in $V(G)$. Let $\cU$ be a focus for $\cS$ and suppose that $G$ is $r$-exclusive with respect to $\cS$. Let $\cT$ be the $i'$-extension of $\cS$ in $G$ around $\cU$. For $r'\ge r$, let $\tau_{r'}$ be the type function of $G^{(r')}$, $\cS$, $\cU$. Then the following hold for $$G':=G-\set{e\in G^{(r)}}{e\mbox{ is }\cS\mbox{-important and }\tau_r(e)<r-i'}:$$
\begin{enumerate}[label=\rm{(\roman*)}]
\item $G'$ is $r$-exclusive with respect to $\cT$;\label{prop:extension types 1}
\item for all $e\in G$ with $|e|\ge r$, we have
\begin{align*}
e\notin G' \quad \Leftrightarrow \quad e\mbox{ is }\cS\mbox{-important and }\tau_{|e|}(e)< |e|-i';
\end{align*}\label{prop:extension types 2}
\item for $r'\ge r$, the $\cT$-important elements of $G'^{(r')}$ are precisely the elements of $\tau_{r'}^{-1}(r'-i')$.\label{prop:extension types 3}
\end{enumerate}
\end{prop}

Let $\cZ_{r,i}$ be the set of all quadruples $(z_0,z_1,z_2,z_3)\in\bN_0^4$ such that $z_0+z_1<i$, $z_0+z_3<i$ and $z_0+z_1+z_2+z_3=r$. Clearly, $|\cZ_{r,i}|\le (r+1)^3$, and $\cZ_{r,i}=\emptyset$ if $i=0$.

\begin{defin}
Let $V$ be a set of size $n$, let $\cS$ be an $i$-system in $V$ and let $\cU$ be a focus for $\cS$. We say that $\cU$ is a \defn{$\mu$-focus for $\cS$} if each $U_S\in \cU$ has size $\mu n\pm n^{2/3}$.
For all $S\in\cS$, $z=(z_0,z_1,z_2,z_3)\in \cZ_{r,i}$ and all $(z_1+z_2-1)$-sets $b\In V\sm S$,\COMMENT{In appl., we have $b\In U_S$.} define
\begin{align*}
\cJ^b_{S,z} &:=\set{S'\in\cS}{|S\cap S'|=z_0,b\In S'\cup U_{S'},|U_{S'}\cap S|\ge z_3},\\
\cJ^b_{S,z,1} &:=\set{S'\in\cJ_{S,z}^b}{|b\cap S'|=z_1},\\
\cJ^b_{S,z,2} &:=\set{S'\in\cJ_{S,z}^b}{|b\cap S'|=z_1-1,|U_S\cap (S'\sm b)|\ge 1}.
\end{align*}
We say that $\cU$ is a \defn{$(\rho_{size},\rho,r)$-focus for $\cS$} if
\begin{enumerate}[label={\rm(F\arabic*)}]
\item each $U_S$ has size $\rho_{size}\rho n\pm n^{2/3}$;\label{def:focus:size}
\item $|U_S\cap U_{S'}|\le 2\rho^2 n$ for distinct $S,S'\in\cS$;\label{def:focus:intersection}
\item for all $S\in\cS$, $z=(z_0,z_1,z_2,z_3)\in \cZ_{r,i}$ and $(z_1+z_2-1)$-sets $b\In V\sm S$, we have
\begin{align*}
|\cJ^b_{S,z,1}| &\le 2^{6r}\rho^{z_2+z_3-1}n^{i-z_0-z_1},\\
|\cJ^b_{S,z,2}| &\le 2^{9r}\rho^{z_2+z_3+1}n^{i-z_0-z_1+1}.
\end{align*}\label{def:focus:Js}\COMMENT{If $i=0$ then \ref{def:focus:intersection} and \ref{def:focus:Js} are vacuously true. The definition is only needed for $i>0$.}
\end{enumerate}
\end{defin}

The sets $S'$ in $\cJ^b_{S,z,1}$ and $\cJ^b_{S,z,2}$ are those which may give rise to interference when covering the edges containing $S$. \ref{def:focus:Js} ensures that there are not too many of them.
The next lemma states that a suitable random choice of the $U_S$ yields a $(\rho_{size},\rho,r)$-focus.

\begin{lemma}[\cite{GKLO}]\label{lem:focus}
Let $1/n\ll \rho \ll \rho_{size},1/r$ and $i\in[r-1]$. Let $V$ be a set of size $n$, let $\cS$ be an $i$-system in $V$ and let $\cU'=(U_S')_{S\in \cS}$ be a $\rho_{size}$-focus for $\cS$. Let $\cU=(U_S)_{S\in \cS}$ be a random focus obtained as follows: independently for all pairs $S\in \cS$ and $x\in U_S'$, retain $x$ in $U_S$ with probability $\rho$.
Then \whp $\cU$ is a $(\rho_{size},\rho,r)$-focus for $\cS$.
\end{lemma}

The following `Localised cover down lemma' allows us to simultaneously cover all $\cS$-important edges of an $i$-system $\cS$ provided that the associated focus $\cU$ satisfies \ref{def:focus:size}--\ref{def:focus:Js} and all $\cS$-important edges are `localised' in the sense that their links are contained in the respective focus set (or, equivalently, their type is maximal).

\begin{lemma}[Localised cover down lemma]\label{lem:dense jain}
Let $1/n\ll \rho \ll \rho_{size},\xi,1/f$ and $1\le i<r<f$. Assume that \ind{r-i} is true. Let $F$ be a weakly regular $r$-graph on $f$ vertices and $S^\ast\in \binom{V(F)}{i}$ such that $F(S^\ast)$ is non-empty. Let $G$ be a complex on $n$ vertices and let $\cS=\Set{S_1,\dots,S_p}$ be an $i$-system in $G$ such that $G$ is $r$-exclusive with respect to $\cS$. Let $\cU=\Set{U_1,\dots,U_p}$ be a $(\rho_{size},\rho,r)$-focus for $\cS$. Suppose further that whenever $S_j\In e\in G^{(r)}$, we have $e\sm S_j\In U_j$. Finally, assume that
$G(S_j)[U_j]$ is an $F(S^\ast)$-divisible $(\rho,\xi,f-i,r-i)$-supercomplex for all $j\in[p]$.

Then there exists a $\rho^{-1/12}$-well separated $F$-packing $\cF$ in $G$ covering all $\cS$-important $r$-edges.
\end{lemma}

\proof
Recall that by Proposition~\ref{prop:link divisibility}, $F(S^\ast)$ is a weakly regular $(r-i)$-graph. We will use \ind{r-i} together with Corollary~\ref{cor:many decs new} in order to find many $F(S^\ast)$-decompositions of $G(S_j)[U_j]$ and then pick one of these at random. Let $t:=\rho^{1/6}(0.5\rho\rho_{size} n)^{f-r}$ and $\kappa:=\rho^{-1/12}$. For all $j\in[p]$, define $G_j:=G(S_j)[U_j]$. Consider Algorithm~\ref{alg:randomgreedy} which, if successful, outputs a $\kappa$-well separated $F(S^\ast)$-decomposition $\cF_j$ of $G_j$ for every $j\in[p]$.

\begin{algorithm}
\caption{}
\label{alg:randomgreedy}
\begin{algorithmic}
\For{$j$ from $1$ to $p$}
\State{for all $z=(z_0,z_1,z_2,z_3)\in \cZ_{r,i}$, define $T^j_{z}$ as the $(z_1+z_2)$-graph on $U_j$ containing all $Z_1\cupdot Z_2 \In U_j$ with $|Z_1|=z_1,|Z_2|=z_2$ such that for some $j'\in[j-1]$ with $|S_j\cap S_{j'}|=z_0$ and some $K'\in\cF_{j'}^{\le(f-i)}$, we have $Z_1\In S_{j'}$, $Z_2\In K'$ and $|K'\cap S_j|= z_3$}
\If {there exist $\kappa$-well separated $F(S^\ast)$-decompositions $\cF_{j,1},\dots,\cF_{j,t}$ of $G_{j}-\bigcup_{z\in \cZ_{r,i}} T^j_{z}$ which are pairwise $(f-i)$-disjoint}
    \State pick $s\in[t]$ uniformly at random and let $\cF_j:=\cF_{j,s}$
\Else
    \State \Return `unsuccessful'
\EndIf
\EndFor
\end{algorithmic}
\end{algorithm}

\begin{NoHyper}\begin{claim}
If Algorithm~\ref{alg:randomgreedy} outputs $\cF_{1},\dots,\cF_{p}$, then $\cF:=\bigcup_{j\in [p]}\tilde{\cF}_j$ is a packing as desired, where $\tilde{\cF}_j:=S_j\triangleleft \cF_j$.
\end{claim}\end{NoHyper}

\claimproof Since $z_1+z_2>r-i$, we have $G_j^{(r-i)}=(G_{j}-\bigcup_{z\in \cZ_{r,i}} T^j_{z})^{(r-i)}$. Hence, $\cF_{j}$ is indeed an $F(S^\ast)$-decomposition of $G_j$. Thus, by Proposition~\ref{prop:S cover}, $\tilde{\cF}_j$ is a $\kappa$-well separated $F$-packing in $G$ covering all $r$-edges containing $S_j$. Therefore, $\cF$ covers all $\cS$-important $r$-edges of $G$. By Fact~\ref{fact:ws}\ref{fact:ws:2} it suffices to show that $\tilde{\cF}_1,\dots,\tilde{\cF}_p$ are $r$-disjoint.

To this end, let $j'<j$ and suppose, for a contradiction, that there exist $\tilde{K}\in \tilde{\cF}_j^{\le(f)}$ and $\tilde{K}'\in \tilde{\cF}_{j'}^{\le(f)}$ such that $|\tilde{K}\cap \tilde{K}'|\ge r$. Let $K:=\tilde{K}\sm S_j$ and $K':=\tilde{K}'\sm S_{j'}$. Then $K\in\cF_j^{\le(f-i)}$ and $K'\in\cF_{j'}^{\le(f-i)}$ and $|(S_j\cup K)\cap (S_{j'}\cup K')|\ge r$. Let $z_0:=|S_j\cap S_{j'}|$ and $z_3:=|S_j\cap K'|$. Hence, we have $|K\cap (S_{j'}\cup K')|\ge r-z_0-z_3$. Choose $X\In K$ such that $|X\cap (S_{j'}\cup K')|= r-z_0-z_3$ and let $Z_1:=X\cap S_{j'}$ and $Z_2:=X\cap K'$. We claim that $z:=(z_0,|Z_1|,|Z_2|,z_3)\in \cZ_{r,i}$.
Clearly, we have $z_0+|Z_1|+|Z_2|+z_3=r$.
Furthermore, note that $z_0+z_3<i$. Indeed, we clearly have $z_0+z_3= |S_j\cap (S_{j'}\cup K')|\le |S_j|=i$, and equality can only hold if $S_j\In S_{j'}\cup K'=\tilde{K}'$, which is impossible since $G$ is $r$-exclusive. Similarly, we have $z_0+|Z_1|<i$.
Thus, $z\in \cZ_{r,i}$. But this implies that $Z_1\cup Z_2\in T^j_z$, in contradiction to $Z_1\cup Z_2\In K$.
\endclaimproof

In order to prove the lemma, it is thus sufficient to prove that with positive probability, $\Delta(T^j_{z})\le 2^{2r} f\kappa \rho^{1/2}|U_j|$ for all $j\in[p]$ and $z\in \cZ_{r,i}$. Indeed, this would imply that $\Delta(\bigcup_{z\in \cZ_{r,i}} T^j_{z})\le (r+1)^3 2^{2r} f\rho^{1/2-1/12}|U_j|$, and by Proposition~\ref{prop:noise}\ref{noise:supercomplex}, $G_{j}-\bigcup_{z\in \cZ_{r,i}} T^j_{z}$ would be a $(\rho^{1/12},\xi/2,f-i,r-i)$-supercomplex. By Corollary~\ref{cor:many decs new} and since $|U_j|\ge 0.5\rho\rho_{size} n$, the number of pairwise $(f-i)$-disjoint $\kappa$-well separated $F(S^\ast)$-decompositions in $G_{j}-\bigcup_{z\in \cZ_{r,i}} T^j_{z}$ is at least $\rho^{2/12}|U_j|^{(f-i)-(r-i)}\ge t$, so the algorithm would succeed.

In order to analyse $\Delta(T^j_{z})$, we define the following variables.
Suppose that $1\le j'<j\le p$, that $z=(z_0,z_1,z_2,z_3)\in \cZ_{r,i}$ and $b\In U_j$ is a $(z_1+z_2-1)$-set. Let $Y^{b,j'}_{j,z}$ denote the random indicator variable of the event that each of the following holds:
\begin{enumerate}[label={\rm(\alph*)}]
\item there exists some $K'\in \cF_{j'}^{\le(f-i)}$ with $|K'\cap S_j|=z_3$;\label{dense jain:max deg analysis 1}
\item there exist $Z_1\In S_{j'}$, $Z_2\In K'$ with $|Z_1|=z_1$, $|Z_2|=z_2$ such that $b\In Z_1\cup Z_2\In U_j$;\label{dense jain:max deg analysis 2}
\item $|S_j\cap S_{j'}|=z_0$.\label{dense jain:max deg analysis 3}
\end{enumerate}
We say that $v\in \binom{U_j\sm b}{1}$ is a \defn{witness for $j'$} if \ref{dense jain:max deg analysis 1}--\ref{dense jain:max deg analysis 3} hold with $Z_1\cupdot Z_2=b\cupdot v$.\COMMENT{$T^{j}_z(b)$ is a $1$-graph, so $v=\Set{v'}$ for some vertex $v'$}
For all $j\in[p]$, $z=(z_0,z_1,z_2,z_3)\in \cZ_{r,i}$ and $(z_1+z_2-1)$-sets $b\In U_j$, let $X^{b}_{j,z}:=\sum_{j'=1}^{j-1}Y^{b,j'}_{j,z}$.

\begin{NoHyper}\begin{claim}
For all $j\in[p]$, $z=(z_0,z_1,z_2,z_3)\in \cZ_{r,i}$ and $(z_1+z_2-1)$-sets $b\In U_j$, we have $|T^j_{z}(b)|\le 2^{2r} f\kappa X^{b}_{j,z}$.
\end{claim}\end{NoHyper}

\claimproof
Let $j,z$ and $b$ be fixed. Clearly, if $v\in T^{j}_z(b)$, then by Algorithm~\ref{alg:randomgreedy}, $v$ is a witness for some $j'\in[j-1]$. Conversely, we claim that for each $j'\in[j-1]$, there are at most $2^{2r} f\kappa$ witnesses for $j'$. Clearly, this would imply that $|T^j_{z}(b)|\le 2^{2r} f\kappa|\set{j'\in[j-1]}{Y^{b,j'}_{j,z}=1}|= 2^{2r} f\kappa X^{b}_{j,z}$.\COMMENT{E.g. consider the map that maps every $v\in T^{j}_z(b)$ to the minimum $j'$ that $v$ is a witness for. This map is well-defined by the first sentence. By the second claim, every element of the image has at most $2^{2r} f\kappa$ preimages, thus the domain has cardinality at most $2^{2r} f\kappa$ times that of the image.}

Fix $j'\in[j-1]$. If $v$ is a witness for $j'$, then there exists $K_v\in \cF_{j'}^{\le(f-i)}$ such that \ref{dense jain:max deg analysis 1}--\ref{dense jain:max deg analysis 3} hold with $Z_1\cupdot Z_2=b\cupdot v$ and $K_v$ playing the role of $K'$. By~\ref{dense jain:max deg analysis 2} we must have $v\In Z_1\cup Z_2 \In S_{j'}\cup K_v$. Since $|S_{j'}\cup K_v|=f$, there are at most $f$ witnesses $v'$ for $j'$ such that $K_v$ can play the role of $K_{v'}$. It is thus sufficient to show that there are at most $2^{2r}\kappa$ $K'\in \cF_{j'}^{\le(f-i)}$ such that \ref{dense jain:max deg analysis 1}--\ref{dense jain:max deg analysis 3} hold.

Note that for any possible choice of $Z_1,Z_2,K'$, we must have $|b\cap Z_2|\in \Set{z_2,z_2-1}$ and $b\cap Z_2\In Z_2\In K'$ by~\ref{dense jain:max deg analysis 2}.
For any $Z_2'\In b$ with $|Z_{2}'|\in\Set{z_2,z_2-1}$ and any $Z_3\in \binom{S_j}{z_3}$, there can be at most $\kappa$ $K'\in \cF_{j'}^{\le(f-i)}$ with $Z_2'\In K'$ and $K'\cap S_j=Z_3$. This is because $\cF_{j'}$ is a $\kappa$-well separated $F(S^\ast)$-decomposition and $|Z_2'\cup Z_3|\ge z_2-1+z_3\ge r-i$. Hence, there can be at most $2^{|b|}\binom{i}{z_3}\kappa\le 2^{2r}\kappa$ possible choices for $K'$.\COMMENT{at most $2^{|b|}$ choices for $Z_2'$ and $\binom{i}{z_3}$ choices for $Z_3$}
\endclaimproof

The following claim thus implies the lemma.

\begin{NoHyper}\begin{claim}\label{claim:jain analysis}
With positive probability, we have $X^{b}_{j,z}\le \rho^{1/2}|U_j|$ for all $j\in[p]$, $z=(z_0,z_1,z_2,z_3)\in \cZ_{r,i}$ and $(z_1+z_2-1)$-sets $b\In U_j$.
\end{claim}\end{NoHyper}

\claimproof
Fix $j,z,b$ as above. We split $X^{b}_{j,z}$ into two sums. For this, let
\begin{align*}
\cJ^b_{j,z} &:=\set{j'\in[j-1]}{|S_j\cap S_{j'}|=z_0,b\sm S_{j'}\In U_{j'},|U_{j'}\cap S_j|\ge z_3},\\
\cJ^b_{j,z,1} &:=\set{j'\in\cJ^b_{j,z}}{|b\cap S_{j'}|=z_1},\\
\cJ^b_{j,z,2} &:=\set{j'\in\cJ^b_{j,z}}{|b\cap S_{j'}|=z_1-1,|U_j\cap (S_{j'}\sm b)|\ge 1}.
\end{align*}
Since $\cU$ is a $(\rho_{size},\rho,r)$-focus for $\cS$, \ref{def:focus:Js} implies that
\begin{align}
|\cJ^b_{j,z,1}|&\le 2^{6r}\rho^{z_2+z_3-1}n^{i-z_0-z_1},\label{Js:1}\\
|\cJ^b_{j,z,2}|&\le 2^{9r}\rho^{z_2+z_3+1}n^{i-z_0-z_1+1}.\label{Js:2}
\end{align}

Note that if $Y^{b,j'}_{j,z}=1$, then $j'\in \cJ^b_{j,z,1}\cup\cJ^b_{j,z,2}$.\COMMENT{Let $K',Z_1,Z_2$ be such that \ref{dense jain:max deg analysis 1}--\ref{dense jain:max deg analysis 3} hold. Clearly, we have $|S_j\cap S_{j'}|=z_0$ by \ref{dense jain:max deg analysis 3}. Moreover, $b\In Z_1\cup Z_2 \In S_{j'}\cup K'$ and hence $b\sm S_{j'}\In K'\In U_{j'}$. Moreover, $|U_{j'}\cap S_j|\ge |K'\cap S_j|=z_3$. Hence, $j'\in \cJ^b_{j,z}$. We also have $|b\cap S_{j'}|\in\Set{z_1-1,z_1}$, and if $|b\cap S_{j'}|=z_1-1$, then $Z_1\sm b$ is non-empty, implying that $|U_j\cap (S_{j'}\sm b)|\ge |Z_1\cap (Z_1\sm b)|\ge 1$.} Hence, we have $X^{b}_{j,z}=X^{b}_{j,z,1}+X^{b}_{j,z,2}$, where $X^{b}_{j,z,1}:=\sum_{j'\in \cJ^b_{j,z,1}}Y^{b,j'}_{j,z}$ and $X^{b}_{j,z,2}:=\sum_{j'\in \cJ^b_{j,z,2}}Y^{b,j'}_{j,z}$.
We will bound $X^{b}_{j,z,1}$ and $X^{b}_{j,z,2}$ separately.

For $j'\in \cJ^b_{j,z,1}\cup \cJ^b_{j,z,2}$, define
\begin{align}
\cK^{b,j'}_{j,z} := \set{K'\in \binom{U_{j'}}{f-i}}{b\In S_{j'}\cup K',|K'\cap U_j|\ge z_2,|K'\cap S_j|= z_3}.\label{jain:candidates}
\end{align}
Note that if $Y^{b,j'}_{j,z}=1$, then $\cF_{j',k}^{\le(f-i)}\cap \cK^{b,j'}_{j,z} \neq \emptyset$. Recall that the candidates $\cF_{j',1},\dots,\cF_{j',t}$ in Algorithm~\ref{alg:randomgreedy} from which $\cF_{j'}$ was chosen at random are $(f-i)$-disjoint. We thus have $$\prob{Y^{b,j'}_{j,z}=1}\le \frac{|\set{k\in [t]}{\cF_{j',k}^{\le(f-i)}\cap \cK^{b,j'}_{j,z} \neq \emptyset}|}{t} \le \frac{|\cK^{b,j'}_{j,z}|}{t}.$$ This upper bound still holds if we condition on variables $Y^{b,j''}_{j,z}$, $j''\neq j'$. We thus need to bound $|\cK^{b,j'}_{j,z}|$ in order to bound $X^{b}_{j,z,1}$ and $X^{b}_{j,z,2}$.

\medskip

{\em Step~1: Estimating $X^{b}_{j,z,1}$}

\smallskip

Consider $j'\in \cJ^b_{j,z,1}$.
For all $K'\in \cK^{b,j'}_{j,z}$, we have $b\sm S_{j'}\In K'$ and $|b\cap K'|=|b|-|b\cap S_{j'}|=z_2-1$, and the sets $b\cap K'$, $K'\cap S_j$, $(K'\sm b)\cap(U_j\cap U_{j'})$ are disjoint.\COMMENT{To see that $b\cap K'$ and $K'\cap S_j$ are disjoint we use that $b\In U_j$} Moreover,  we have $|(K'\sm b)\cap(U_j\cap U_{j'})|=|(K'\sm b)\cap U_j|\ge |K'\cap U_j|-|b\cap K'|\ge 1$. We can thus count
\begin{align*}
|\cK^{b,j'}_{j,z}| &\le \binom{|S_j|}{z_3}\cdot |U_j\cap U_{j'}|\cdot |U_{j'}|^{f-i-(z_2-1)-1-z_3} \le 2^{i}\cdot 2\rho^2 n\cdot (2\rho\rho_{size}n)^{f-i-z_2-z_3}.
\end{align*}
Let $\tilde{\rho}_1:=\rho^{z_0+z_1-i+5/3}\rho_{size}n^{1+z_0+z_1-i}\in [0,1]$.\COMMENT{We always have $1+z_0+z_1-i\le 0$. If $1+z_0+z_1-i<0$, then obvious. If $1+z_0+z_1-i=0$, we have $\rho^{2/3}\rho_{size}\le 1$.} In order to apply Proposition~\ref{prop:Jain}, let $j_1,\dots,j_m$ be an enumeration of $\cJ^b_{j,z,1}$. We then have for all $k\in[m]$ and all $y_1,\dots,y_{k-1}\in\Set{0,1}$ that
\begin{align*}
\prob{Y^{b,j_k}_{j,z}=1 \mid Y^{b,j_1}_{j,z}=y_1,\dots,Y^{b,j_{k-1}}_{j,z}=y_{k-1}} &\le \frac{|\cK^{b,j_k}_{j,z}|}{t} \le \frac{2^{i}\cdot 2\rho^2 n\cdot (2\rho\rho_{size}n)^{f-i-z_2-z_3}}{\rho^{1/6} (0.5\rho\rho_{size} n)^{f-r}}\\
										&= 2^{2f-r+1-z_2-z_3} \rho^{11/6}(\rho\rho_{size})^{z_0+z_1-i}n^{1+z_0+z_1-i}\\
										&\le \tilde{\rho}_1.
\end{align*}
Let $B_1\sim Bin(|\cJ^b_{j,z,1}|,\tilde{\rho}_1)$ and observe that
\begin{align*}
7\expn{B_1} &=7|\cJ^b_{j,z,1}|\tilde{\rho}_1 \overset{\eqref{Js:1}}{\le} 7\cdot 2^{6r}\rho^{z_2+z_3-1}n^{i-z_0-z_1} \cdot \rho^{z_0+z_1-i+5/3}\rho_{size}n^{1+z_0+z_1-i} \\
             &=7\cdot 2^{6r}\rho^{r-i+2/3}\rho_{size}n       \le 0.5\rho^{1/2}|U_j|.
\end{align*}
Thus,
\begin{align*}
\prob{X^{b}_{j,z,1}\ge 0.5\rho^{1/2}|U_j|} \overset{\mbox{\tiny Proposition~\ref{prop:Jain}}}{\le} \prob{B_1\ge 0.5\rho^{1/2}|U_j|} \overset{\mbox{\tiny Lemma~\ref{lem:chernoff}\ref{chernoff crude}}}{\le} \eul^{-0.5\rho^{1/2}|U_j|}.
\end{align*}

\medskip

{\em Step~2: Estimating $X^{b}_{j,z,2}$}

\smallskip

Consider $j'\in \cJ^b_{j,z,2}$. This time, since $|b\cap S_{j'}|=z_1-1$, we have $|K'\cap b|=|b\sm S_{j'}|=z_2$ for all $K'\in \cK^{b,j'}_{j,z}$. Thus, we count $$|\cK^{b,j'}_{j,z}|\le \binom{|S_j|}{z_3}\cdot |U_{j'}|^{f-i-z_2-z_3}\le 2^i\cdot (2\rho\rho_{size}n)^{f-i-z_2-z_3}.$$
Let $\tilde{\rho}_2:=\rho^{z_0+z_1-i-1/5}\rho_{size}n^{z_0+z_1-i}\in [0,1]$. In order to apply Proposition~\ref{prop:Jain}, let $j_1,\dots,j_m$ be an enumeration of $\cJ^f_{j,z,2}$. We then have for all $k\in[m]$ and all $y_1,\dots,y_{k-1}\in\Set{0,1}$ that
\begin{align*}
\prob{Y^{b,j_k}_{j,z}=1 \mid Y^{b,j_1}_{j,z}=y_1,\dots,Y^{b,j_{k-1}}_{j,z}=y_{k-1}} &\le \frac{|\cK^{b,j_k}_{j,z}|}{t}\le \frac{2^i\cdot (2\rho\rho_{size}n)^{f-i-z_2-z_3}}{\rho^{1/6} (0.5\rho\rho_{size} n)^{f-r}}\\
																	&= 2^{2f-r-z_2-z_3} \rho^{-1/6} (\rho\rho_{size}n)^{z_0+z_1-i} \\
																	&\le \tilde{\rho}_2.
\end{align*}
Let $B_2\sim Bin(|\cJ^b_{j,z,2}|,\tilde{\rho}_2)$ and observe that
\begin{align*}
7\expn{B_2}&=7|\cJ^b_{j,z,2}|\tilde{\rho}_2 \overset{\eqref{Js:2}}{\le} 7 \cdot 2^{9r}\rho^{z_2+z_3+1}n^{i-z_0-z_1+1} \cdot \rho^{z_0+z_1-i-1/5}\rho_{size}n^{z_0+z_1-i}\\
           &= 7 \cdot 2^{9r}\rho^{r-i+4/5}\rho_{size}n    \le 0.5\rho^{1/2}|U_j|.
\end{align*}
Thus,
\begin{align*}
\prob{X^{b}_{j,z,2}\ge 0.5\rho^{1/2}|U_j|} \overset{\mbox{\tiny Proposition~\ref{prop:Jain}}}{\le} \prob{B_2\ge 0.5\rho^{1/2}|U_j|} \overset{\mbox{\tiny Lemma~\ref{lem:chernoff}\ref{chernoff crude}}}{\le} \eul^{-0.5\rho^{1/2}|U_j|}.
\end{align*}
Hence, $$\prob{X^{b}_{j,z}\ge \rho^{1/2}|U_j|}\le \prob{X^{b}_{j,z,1}\ge 0.5\rho^{1/2}|U_j|} +\prob{X^{b}_{j,z,2}\ge 0.5\rho^{1/2}|U_j|} \le 2\eul^{-0.5\rho^{1/2}|U_j|}.$$ Since $p=|\cS|\le n^i$, a union bound easily implies Claim~\ref*{claim:jain analysis}.
\endclaimproof
This completes the proof of Lemma~\ref{lem:dense jain}.
\endproof

\subsection{Partition pairs}
We now develop the appropriate framework to be able to state the Cover down lemma for setups (Lemma~\ref{lem:horrible}). Recall that we will consider (and cover) $r$-sets separately according to their type. The type of an $r$-set $e$ naturally imposes constraints on the type of an $f$-set which covers $e$. We will need to track and adjust the densities of $r$-sets with respect to $f$-sets for each pair of types separately. This gives rise to the following concepts of partition pairs and partition regularity (see Section~\ref{subsec:partition reg}).

Let $X$ be a set. We say that $\cP=(X_1,\dots,X_a)$ is an ordered partition of $X$ if the $X_i$ are disjoint subsets of $X$ whose union is $X$. We let $\cP(i):=X_i$ and $\cP([i]):=(X_1,\dots,X_i)$. If $\cP=(X_1,\dots,X_a)$ is an ordered partition of $X$ and $X'\In X$, we let $\cP[X']$ denote the ordered partition $(X_1\cap X',\dots,X_a\cap X')$ of $X'$.
If $\Set{X',X''}$ is a partition of $X$, $\cP'=(X_1',\dots,X_a')$ is an ordered partition of $X'$ and $\cP''=(X_1'',\dots,X_b'')$ is an ordered partition of $X''$, we let $$\cP'\sqcup \cP'':=(X_1',\dots,X_a',X_1'',\dots,X_b'').$$

\begin{defin}\label{def:partition pair}
Let $G$ be a complex and let $f>r\ge 1$.
An \defn{$(r,f)$-partition pair of $G$} is a pair $(\cP_{r},\cP_{f})$, where $\cP_{r}$ is an ordered partition of $G^{(r)}$ and $\cP_{f}$ is an ordered partition of $G^{(f)}$, such that for all $\cE\in\cP_{r}$ and $\cQ\in\cP_{f}$, every $Q\in\cQ$ contains the same number $C(\cE,\cQ)$ of elements from $\cE$. We call $C\colon \cP_{r}\times \cP_{f} \to [\binom{f}{r}]_0$ the \defn{containment function} of the partition pair. We say that $(\cP_{r},\cP_{f})$ is \defn{upper-triangular} if $C(\cP_{r}(\ell),\cP_{f}(k))=0$ whenever $\ell>k$.
\end{defin}

Clearly, for every $\cQ\in \cP_{f}$, $\sum_{\cE\in\cP_{r}}C(\cE,\cQ)=\binom{f}{r}$. If $(\cP_{r},\cP_{f})$ is an $(r,f)$-partition pair of $G$ and $G'\In G$ is a subcomplex, we define $$(\cP_{r},\cP_{f})[G']:=(\cP_{r}[G'^{(r)}],\cP_{f}[G'^{(f)}]).$$ Clearly, $(\cP_{r},\cP_{f})[G']$ is an $(r,f)$-partition pair of $G'$.\COMMENT{Containment function is restricted to new domain, i.e.~for all $\cE\in \cP_{r}$ and $\cQ\in \cP_f$ we have $C'(\cE\cap G'^{(r)},\cQ\cap G'^{(f)}):=C(\cE,\cQ)$.}

\begin{example} \label{ex:partition pair}
Suppose that $G$ is a complex and $U\In V(G)$. For $\ell\in[r]_0$, define $\cE_\ell:=\set{e\in G^{(r)}}{|e\cap U|=\ell}$. For $k\in[f]_0$, define $\cQ_k:=\set{Q\in G^{(f)}}{|Q\cap U|=k}$. Let $\cP_{r}:=(\cE_0,\dots,\cE_r)$ and $\cP_{f}:=(\cQ_0,\dots,\cQ_f)$. Then clearly $(\cP_{r},\cP_{f})$ is an $(r,f)$-partition pair of $G$, where the containment function is given by $C(\cE_\ell,\cQ_k)=\binom{k}{\ell}\binom{f-k}{r-\ell}$. In particular, $C(\cE_\ell,\cQ_k)=0$ whenever $\ell>k$ or $k>f-r+\ell$. We say that $(\cP_{r},\cP_{f})$ is the \defn{$(r,f)$-partition pair of $G$, $U$.}
\end{example}

The partition pairs we use are generalisations of the above example. More precisely, suppose that $G$ is a complex, $\cS$ is an $i$-system in $V(G)$ and $\cU$ is a focus for $\cS$. Moreover, assume that $G$ is $r$-exclusive with respect to $\cS$.
For $r'\ge r$, let $\tau_{r'}$ denote the type function of $G^{(r')}$, $\cS$, $\cU$. As in the above example, if $\cE_\ell:=\tau_r^{-1}(\ell)$ for all $\ell\in[r-i]_0$ and $\cQ_k:=\tau_f^{-1}(k)$ for all $k\in[f-i]_0$, then every $Q\in \cQ_k$ contains exactly $\binom{k}{\ell}\binom{f-i-k}{r-i-\ell}$ elements from $\cE_\ell$. However, we also have to consider $\cS$-unimportant edges and cliques. It turns out that it is useful to assume that the unimportant edges and cliques are partitioned into $i$ parts each, in an upper-triangular fashion.

More formally, for $r'\ge r$, let $\cD_{r'}$ denote the set of $\cS$-unimportant $r'$-sets of $G$ and assume that $\cP_{r}^\ast$ is an ordered partition of $\cD_{r}$ and $\cP_{f}^\ast$ is an ordered partition of $\cD_{f}$. We say that $(\cP_{r}^\ast,\cP_{f}^\ast)$ is \defn{admissible with respect to $G$, $\cS$, $\cU$} if the following hold:
\begin{enumerate}[label={\rm(P\arabic*)}]
\item $|\cP_{r}^\ast|=|\cP_{f}^\ast|=i$;\label{admissible length}
\item for all $S\in \cS$, $h\in [r-i]_0$ and $B\In G(S)^{(h)}$ with $1\le|B|\le 2^h$ and all $\ell\in[i]$, there exists $D(S,B,\ell)\in\bN_0$ such that for all $Q\in \bigcap_{b\in B}G(S\cup b)[U_S]^{(f-i-h)}$, we have that $$|\set{e\in \cP_{r}^\ast(\ell)}{\exists b\in B\colon e\In S\cup b \cup Q}|=D(S,B,\ell);$$\label{new containment condition}
\item $(\cP_{r}^\ast\sqcup \Set{G^{(r)}\sm \cD_r},\cP_{f}^\ast\sqcup \Set{G^{(f)}\sm \cD_f})$ is an upper-triangular $(r,f)$-partition pair of~$G$.\label{containment}
\end{enumerate}

Note that for $i=0$, $\cS=\Set{\emptyset}$ and $\cU=\Set{U}$ for some $U\In V(G)$, the pair $(\emptyset,\emptyset)$ trivially satisfies these conditions. Also note that \ref{new containment condition} can be viewed as an analogue of the containment function (from Definition~\ref{def:partition pair}) which is suitable for dealing with supercomplexes.

Assume that $(\cP_{r}^\ast,\cP_{f}^\ast)$ is admissible with respect to $G$, $\cS$, $\cU$. Define
\begin{align*}
\cP_{r} &:= \cP_{r}^\ast \sqcup (\tau_{r}^{-1}(0),\dots,\tau_{r}^{-1}(r-i)), \\
\cP_{f} &:= \cP_{f}^\ast \sqcup (\tau_{f}^{-1}(0),\dots,\tau_{f}^{-1}(f-i)).
\end{align*}

It is not too hard to see that $(\cP_{r},\cP_{f})$ is an $(r,f)$-partition pair of $G$. Indeed, $\cP_{r}$ clearly is a partition of $G^{(r)}$ and $\cP_{f}$ is a partition of $G^{(f)}$. Suppose that $C$ is the containment function of $(\cP_{r}^\ast\sqcup \Set{G^{(r)}\sm \cD_r},\cP_{f}^\ast\sqcup \Set{G^{(f)}\sm \cD_f})$. Then $C'$ as defined below is the containment function of $(\cP_{r},\cP_{f})$:
\begin{itemize}
\item For all $\cE\in \cP_{r}^{\ast}$ and $\cQ\in \cP_{f}^{\ast}$, let $C'(\cE,\cQ):=C(\cE,\cQ)$.
\item For all $\ell\in [r-i]_0$ and $\cQ\in \cP_{f}^{\ast}$, let $C'(\tau_{r}^{-1}(\ell),\cQ):=0$.
\item For all $\cE\in \cP_{r}^{\ast}$ and $k\in[f-i]_0$, define $C'(\cE,\tau_{f}^{-1}(k)):=C(\cE,\Set{G^{(f)}\sm \cD_f}).$
\item For all $\ell\in [r-i]_0$, $k\in[f-i]_0$, let
\begin{align}
C'(\tau_{r}^{-1}(\ell),\tau_{f}^{-1}(k)):=\binom{k}{\ell}\binom{f-i-k}{r-i-\ell}.\label{containment types}
\end{align}
\end{itemize}

We say that $(\cP_{r},\cP_{f})$ as defined above is \defn{induced by $(\cP_{r}^\ast,\cP_{f}^\ast)$ and $\cU$}.

Finally, we say that $(\cP_{r},\cP_{f})$ is an \defn{$(r,f)$-partition pair of $G$, $\cS$, $\cU$}, if
\begin{itemize}
\item $(\cP_{r}([i]),\cP_{f}([i]))$ is admissible with respect to $G$, $\cS$, $\cU$;
\item $(\cP_{r},\cP_{f})$ is induced by $(\cP_{r}([i]),\cP_{f}([i]))$ and $\cU$.
\end{itemize}

\begin{figure}
\footnotesize
\begin{center}
    \resizebox{\textwidth}{!}{\begin{tabular}{*{12}{c|}}
																& $\cP_{f}^\ast(1)$			 	& $\dots$ 								& $\cP_{f}^\ast(i)$ 			& $\tau_{f}^{-1}(0)$ 			& $\tau_{f}^{-1}(1)$ 			& $\dots$ 									& $\dots$ 								& $\tau_{f}^{-1}(f-r)$ 		& $\dots$ 																		& $\dots$ & $\tau_{f}^{-1}(f-i)$ \\ \hline
    $\cP_{r}^\ast(1)$ 					& \textcolor{red}{$\ast$} 	&  											&  												&  												&  												&  													&  												&  												&  																						&         &   \\ \hline
		$\dots$ 										& $0$ 											& \textcolor{red}{$\ast$} &  											&  												&  												&  													& 												&  												&  																						&         &    \\ \hline
		$\cP_{r}^\ast(i)$ 					& $0$ 											&	 $0$ 										& \textcolor{red}{$\ast$} &  											&  												&  													&  												&  												&  																						&         &     \\ \hline
		$\tau_{r}^{-1}(0)$ 					& $0$ 											& $0$ 										& $0$ 										& \textcolor{red}{$\ast$} &  											&  												&  												& \textcolor{red}{$\ast$} &  $0$ 																				&  $0$       & $0$ \\ \hline
		$\dots$ 										& $0$ 											& $0$ 										& $0$ 										& $0$ 										& \textcolor{red}{$\ast$} &  												&  												&  											& {\cellcolor{gray!25}\textcolor{red}{$\ast$}} 	& {\cellcolor{gray!25}$0$}        & {\cellcolor{gray!25}$0$}  \\ \hline
		$\dots$ 										& $0$ 											& $0$ 										& $0$ 										& $0$ 										& $0$										  & \textcolor{red}{$\ast$}		&                       	&  											& {\cellcolor{gray!25}} 	                      & {\cellcolor{gray!25}\textcolor{red}{$\ast$}}        & {\cellcolor{gray!25}$0$}  \\ \hline
		$\tau_{r}^{-1}(r-i)$ 				& $0$ 											& $0$ 										& $0$ 										& $0$		 									& $0$ 										& $0$											 	& \textcolor{red}{$\ast$} &  											& {\cellcolor{gray!25}}  											&  {\cellcolor{gray!25}}       & {\cellcolor{gray!25}\textcolor{red}{$\ast$}}  \\ \hline
		\end{tabular}}
\end{center}
\caption{The above table sketches the containment function of an $(r,f)$-partition pair induced by $(\cP_{r}^\ast,\cP_{f}^\ast)$ and $\cU$. The cells marked with \textcolor{red}{$\ast$} and the shaded subtable will play an important role later on.}
\label{fig:containment table}
\end{figure}

The next proposition summarises basic properties of an $(r,f)$-partition pair of $G$,~$\cS$,~$\cU$.

\begin{prop}[\cite{GKLO}\COMMENT{Proposition 10.11}]\label{prop:facts about partition pairs}
Let $0\le i <r <f$ and suppose that $G$ is a complex, $\cS$ is an $i$-system in $V(G)$ and $\cU$ is a focus for $\cS$. Moreover, assume that $G$ is $r$-exclusive with respect to~$\cS$. Let $(\cP_{r},\cP_{f})$ be an $(r,f)$-partition pair of $G$, $\cS$, $\cU$ with containment function~$C$. Then the following hold:
\begin{enumerate}[label={\rm(P\arabic*$'$)}]
\item $|\cP_{r}|=r+1$ and $|\cP_{f}|=f+1$;\label{partition pair facts length}
\item for $i<\ell\le r+1$, $\cP_{r}(\ell)=\tau_{r}^{-1}(\ell-i-1)$, and for $i<k\le f+1$, $\cP_{f}(k)=\tau_{f}^{-1}(k-i-1)$;\label{prop:facts about partition pairs:index shift}
\item $(\cP_{r},\cP_{f})$ is upper-triangular;\label{partition pair facts triangular}
\item $C(\cP_{r}(\ell),\cP_{f}(k))=0$ whenever both $\ell>i$ and $k>f-r+\ell$;\label{partition pair facts zeros}
\item \ref{new containment condition} holds for all $\ell\in[r+1]$, with $\cP_{r}$ playing the role of $\cP_{r}^\ast$.\label{new containment condition full}
\item if $i=0$, $\cS=\Set{\emptyset}$ and $\cU=\Set{U}$ for some $U\In V(G)$, then the (unique) $(r,f)$-partition pair of $G$, $\cS$, $\cU$ is the $(r,f)$-partition pair of $G$, $U$ (cf.~Example~\ref{ex:partition pair});\label{partition pair facts base case}
\item for every subcomplex $G'\In G$, $(\cP_{r},\cP_{f})[G']$ is an $(r,f)$-partition pair of $G'$, $\cS$, $\cU$.\label{partition pair facts subgraph}
\end{enumerate}
\end{prop}

\subsection{Partition regularity}\label{subsec:partition reg}

\begin{defin}
Let $G$ be a complex on $n$ vertices and $(\cP_{r},\cP_{f})$ an $(r,f)$-partition pair of $G$ with $a:=|\cP_{r}|$ and $b:=|\cP_{f}|$. Let $A=(a_{\ell,k})\in[0,1]^{a\times b}$. We say that $G$ is \defn{$(\eps,A,f,r)$-regular with respect to $(\cP_{r},\cP_{f})$} if for all $\ell\in[a]$, $k\in[b]$ and $e\in \cP_{r}(\ell)$, we have
\begin{align}
|(\cP_{f}(k))(e)|=(a_{\ell,k}\pm \eps)n^{f-r},\label{def matrix regular}
\end{align}
where we view $\cP_{f}(k)$ as a subgraph of $G^{(f)}$.
If $\cE\In \cP_{r}(\ell)$ and $\cQ\In \cP_{f}(k)$, we will often write $A(\cE,\cQ)$ instead of $a_{\ell,k}$.
\end{defin}

For $A\in[0,1]^{a\times b}$ with $1\le t \le a\le b$, we define
\begin{itemize}
\item $\min^{\backslash}(A):=\min\set{a_{j,j}}{j\in [a]}$ as the minimum value on the diagonal,
\item $\min^{\backslash t}(A):=\min\set{a_{j,j+b-a}}{j\in\Set{a-t+1,\dots,a}}$ and
\item $\min^{\backslash\backslash t}(A):=\min\Set{\min^{\backslash}(A),\min^{\backslash t}(A)}$.
\end{itemize}
Note that $\min^{\backslash\backslash r-i+1}(A)$ is the minimum value of the entries in $A$ that correspond to the entries marked with \textcolor{red}{$\ast$} in Figure~\ref{fig:containment table}.

\begin{defin}\label{def:diagonal}
We say that $A\in[0,1]^{a\times b}$ is \defn{diagonal-dominant} if $a_{\ell,k}\le a_{k,k}/2(a-\ell)$ for all $1\le \ell<k\le \min\Set{a,b}$.
\end{defin}

\begin{lemma}[\cite{GKLO}]\label{lem:dominant2regular}
Let $1/n\ll \eps \ll \xi,1/f$ and $r\in[f-1]$. Suppose that $G$ is a complex on $n$ vertices and $(\cP_{r},\cP_{f})$ is an upper-triangular $(r,f)$-partition pair of $G$ with $|\cP_{r}|\le|\cP_{f}|\le f+1$. Let $A\in[0,1]^{|\cP_{r}|\times |\cP_{f}|}$ be diagonal-dominant with $d:=\min^{\backslash}(A)\ge \xi$. Suppose that $G$ is $(\eps,A,f,r)$-regular with respect to $(\cP_{r},\cP_{f})$ and $(\xi,f+r,r)$-dense. Then there exists $Y^\ast\In G^{(f)}$ such that $G[Y^\ast]$ is $(2f\eps,d,f,r)$-regular and $(0.9\xi(\xi/4(f+1))^{\binom{f+r}{f}},f+r,r)$-dense.
\end{lemma}

The following concept of a setup turns out to be the appropriate generalisation of Definition~\ref{def:regular subset} to $i$-systems and partition pairs.

\begin{defin}[Setup]\label{def:regular focus}
Let $G$ be a complex on $n$ vertices and $0\le i<r<f$.
We say that $\cS,\cU,(\cP_{r},\cP_{f})$ form an \defn{$(\eps,\mu,\xi,f,r,i)$-setup} for $G$ if there exists an $f$-graph $Y$ on $V(G)$ such that the following hold:
\begin{enumerate}[label={\rm(S\arabic*)}]
\item $\cS$ is an $i$-system in $V(G)$ such that $G$ is $r$-exclusive with respect to $\cS$; $\cU$ is a $\mu$-focus for $\cS$ and $(\cP_{r},\cP_{f})$ is an $(r,f)$-partition pair of $G$, $\cS$, $\cU$;\label{setup:objects}
\item there exists a matrix $A\in[0,1]^{(r+1)\times(f+1)}$ with $\min^{\backslash\backslash r-i+1}(A)\ge \xi$ such that $G[Y]$ is $(\eps,A,f,r)$-regular with respect to $(\cP_{r},\cP_{f})[G[Y]]=(\cP_{r},\cP_{f}[Y])$;\label{setup:matrix regular}
\item every $\cS$-unimportant $e\in G^{(r)}$ is contained in at least $\xi (\mu n)^{f}$ $\cS$-unimportant $Q\in G[Y]^{(f+r)}$, and for every $\cS$-important $e\in G^{(r)}$ with $e\supseteq S\in \cS$, we have $|G[Y]^{(f+r)}(e)[U_S]|\ge \xi (\mu n)^{f}$;\label{setup:dense}
\item for all $S\in \cS$, $h\in[r-i]_0$ and all $B\In G(S)^{(h)}$ with $1\le|B|\le 2^{h}$ we have that $\bigcap_{b\in B}G(S\cup b)[U_S]$ is an $(\eps,\xi,f-i-h,r-i-h)$-complex.\label{setup:intersections}
\end{enumerate}
Moreover, if \ref{setup:objects}--\ref{setup:intersections} are true and $A$ is diagonal-dominant, then we say that $\cS,\cU,(\cP_{r},\cP_{f})$ form a \defn{diagonal-dominant $(\eps,\mu,\xi,f,r,i)$-setup} for $G$.
\end{defin}

Note that \ref{setup:intersections} implies that $G(S)[U_S]$ is an $(\eps,\xi,f-i,r-i)$-supercomplex for every $S\in \cS$, but is stronger in the sense that $B$ is not restricted to $U_S$.
The following observation shows that Definition~\ref{def:regular focus} does indeed generalise Definition~\ref{def:regular subset}. (Recall that the partition pair of $G,U$ was defined in Example~\ref{ex:partition pair}.) We will use it to derive the Cover down lemma from the more general Cover down lemma for setups.

\begin{prop}[\cite{GKLO}]\label{prop:random generalisation}
Let $G$ be a complex on $n$ vertices and suppose that $U\In V(G)$ is $(\eps,\mu,\xi,f,r)$-random in $G$. Let $(\cP_{r},\cP_{f})$ be the $(r,f)$-partition pair of $G,U$. Then $\Set{\emptyset},\Set{U},(\cP_{r},\cP_{f})$ form an $(\eps,\mu,\tilde{\mu}\xi,f,r,0)$-setup for $G$, where $\tilde{\mu}:=(\min{\Set{\mu,1-\mu}})^{f-r}$.
\end{prop}

The following lemma shows that we can (probabilistically) sparsify a given setup so that the resulting setup is diagonal-dominant.
\begin{lemma}[\cite{GKLO}\COMMENT{Lemma 10.19}]\label{lem:make dominant}
Let $1/n\ll \eps \ll \nu \ll \mu,\xi,1/f$ and $0\le i<r<f$.\COMMENT{Will never be applied with $i=0$} Let $\xi':=\nu^{8^f \cdot f+1}$. Let $G$ be a complex on $n$ vertices and suppose that
\begin{align*}
\cS,\cU,(\cP_{r},\cP_{f})\mbox{ form an }(\eps,\mu,\xi,f,r,i)\mbox{-setup for }G.
\end{align*}
Then there exists a subgraph $H\In G^{(r)}$ with $\Delta(H)\le 1.1\nu n$ and the following property: for all $L\In G^{(r)}$ with $\Delta(L)\le \eps n$ and all $(r+1)$-graphs $O$ on $V(G)$ with $\Delta(O)\le \eps n$, the following holds for $G':=G[H\bigtriangleup L]-O$:
\begin{align*}
\cS,\cU,(\cP_{r},\cP_{f})[G']\mbox{ form a diagonal-dominant }(\sqrt{\eps},\mu,\xi',f,r,i)\mbox{-setup for }G'.
\end{align*}
\end{lemma}

We also need a similar result which `sparsifies' the neighbourhood complexes of an $i$-system.

\begin{lemma}[\cite{GKLO}\COMMENT{Lemma 10.20}]\label{lem:local sparsifier}
Let $1/n\ll \eps \ll \mu,\beta,\xi,1/f$ and $1\le i<r<f$. Let $\xi':=0.9\xi\beta^{(8^f)}$.
Let $G$ be a complex on $n$ vertices and let $\cS$ be an $i$-system in $G$ such that $G$ is $r$-exclusive with respect to $\cS$. Let $\cU$ be a $\mu$-focus for $\cS$. Suppose that
\begin{align*}
G(S)[U_S]\mbox{ is an }(\eps,\xi,f-i,r-i)\mbox{-supercomplex for every }S\in\cS.
\end{align*}
Then there exists a subgraph $H\In G^{(r)}$ with $\Delta(H)\le 1.1\beta n$ and the following property: for all $L\In G^{(r)}$ with $\Delta(L)\le \eps n$ and all $(r+1)$-graphs $O$ on $V(G)$ with $\Delta(O)\le \eps n$, the following holds for $G':=G[H\bigtriangleup L]-O$:
\begin{align*}
G'(S)[U_S]\mbox{ is a }(\sqrt{\eps},\xi',f-i,r-i)\mbox{-supercomplex for every }S\in\cS.
\end{align*}
\end{lemma}

The final tool that we need is the following lemma. Given a setup in a supercomplex $G$ and an $i'$-extension $\cT$ of the respective $i$-system $\cS$, it allows us to find a new focus $\cU'$ for $\cT$ and a suitable partition pair which together form a new setup in the complex $G'$ (which is the complex we look at after all edges with type less than $r-i'$ have been covered).

The idea is to choose $\cU'$ randomly. The proof of Lemma~\ref{lem:setup refinement} follows the lines of Claim~4 and Step~2 in the proof of Lemma~10.22 in \cite{GKLO}. In particular, \ref{twisted focus setup} and \ref{after focus super} below correspond to (10.11) and (10.12), respectively, and \ref{twisted focus intersections} holds using Lemma~\ref{lem:focus}.

\begin{lemma}\label{lem:setup refinement}
Let $1/n\ll \eps \ll \rho \ll \mu,\xi,1/f$ and $0\le i<i'<r<f$. Let $G$ be a complex on $n$ vertices and suppose that $\cS,\cU,(\cP_{r},\cP_{f})$ form an $(\eps,\mu,\xi,f,r,i)$-setup for $G$. For $r'\ge r$, let $\tau_{r'}$ be the type function of $G^{(r')}$, $\cS$, $\cU$. Let $\cT$ be the $i'$-extension of $\cS$ in $G$ around $\cU$, and let $$G':=G-\set{e\in G^{(r)}}{e\mbox{ is }\cS\mbox{-important and }\tau_r(e)<r-i'}.$$ Then there exist $\cU',\cP_{r}',\cP_{f}'$ with the following properties:
\begin{enumerate}[label=\rm{(\roman*)}]
\item $\cU'$ is a $(\mu,\rho,r)$-focus for $\cT$ such that $U_T\In U_{T{\restriction_{\cS}}}$ for all $T\in \cT$;\label{twisted focus intersections}
\item $\cT,\cU',(\cP_{r}',\cP_{f}')$ form a $(1.1\eps,\rho\mu,\rho^{f-r}\xi,f,r,i')$-setup for $G'$;\label{twisted focus setup}
\item $G'(T)[U_T]$ is a $(1.1\eps,0.9\xi,f-i',r-i')$-supercomplex for every $T\in\cT$.\label{after focus super}
\end{enumerate}
\end{lemma}

\subsection{Proof of the Cover down lemma}
In this subsection, we state and prove the Cover down lemma for setups and deduce the Cover down lemma (Lemma~\ref{lem:cover down}).

\begin{defin}\label{def:qr divisible}
Let $F$ and $G$ be $r$-graphs, let $\cS$ be an $i$-system in $V(G)$, and let $\cU$ be a focus for $\cS$. We say that $G$ is \defn{$F$-divisible with respect to $\cS,\cU$}, if for all $S\in\cS$ and all $T\In V(G)\sm S$ with $|T|\le r-i-1$ and $|T\sm U_S|\ge 1$, we have $Deg(F)_{i+|T|}\mid |G(S\cup T)|$.
\end{defin}

Note that if $G$ is $F$-divisible, then it is $F$-divisible with respect to any $i$-system and any associated focus.

Recall that a setup for $G$ was defined in Definition~\ref{def:regular focus}, and $G$ being $(\xi,f,r)$-dense with respect to $H\In G^{(r)}$ in Definition~\ref{def:dense wrt}.
We will prove the Cover down lemma for setups by induction on $r-i$. We will deduce the Cover down lemma by applying this lemma with $i=0$.

\begin{lemma}[Cover down lemma for setups]\label{lem:horrible}
Let $1/n\ll 1/\kappa \ll \gamma \ll \eps \ll \nu \ll \mu,\xi,1/f$ and $0\le i<r<f$. Let $F$ be a weakly regular $r$-graph on $f$ vertices. Assume that \ind{\ell} is true for all $\ell\in[r-i-1]$.\COMMENT{So if $i=r-1$, no assumption is needed, and if $i=0$, then we assume \ind{1}--\ind{r-1}} Let $G$ be a complex on $n$ vertices and suppose that $\cS,\cU,(\cP_{r},\cP_{f})$ form an $(\eps,\mu,\xi,f,r,i)$-setup for $G$.
 For $r'\ge r$, let $\tau_{r'}$ denote the type function of $G^{(r')}$, $\cS$, $\cU$.
Then the following hold.
\begin{enumerate}[label={\rm(\roman*)}]
\item Let $\tilde{G}$ be a complex on $V(G)$ with $G\In \tilde{G}$ such that $\tilde{G}$ is $(\eps,f,r)$-dense with respect to $G^{(r)}- \tau_{r}^{-1}(0)$. Then there exists a subgraph $H^{\ast}\In G^{(r)}- \tau_{r}^{-1}(0)$ with $\Delta(H^\ast)\le \nu n$ such that for any $L^\ast \In \tilde{G}^{(r)}$ with $\Delta(L^\ast)\le \gamma n$ and $H^\ast \cup L^\ast$ being $F$-divisible with respect to $\cS,\cU$ and any $(r+1)$-graph $O^\ast$ on $V(G)$ with $\Delta(O^\ast)\le \gamma n$, there exists a $\kappa$-well separated $F$-packing in $\tilde{G}[H^\ast \cup L^\ast]-O^\ast$ which covers all edges of $L^\ast$, and all $\cS$-important edges of $H^\ast$ except possibly some from $\tau_{r}^{-1}(r-i)$.\label{lem:horrible:specific}
\item If $G^{(r)}$ is $F$-divisible with respect to $\cS,\cU$ and the setup is diagonal-dominant, then there exists a $2\kappa$-well separated $F$-packing in $G$ which covers all $\cS$-important $r$-edges except possibly some from $\tau_{r}^{-1}(r-i)$.\label{lem:horrible:all}
\end{enumerate}
\end{lemma}

Before proving Lemma~\ref{lem:horrible}, we show how it implies the Cover down lemma (Lemma~\ref{lem:cover down}). Note that we only need part \ref{lem:horrible:specific} of Lemma~\ref{lem:horrible} to prove Lemma~\ref{lem:cover down}. \ref{lem:horrible:all} is used in the inductive proof of Lemma~\ref{lem:horrible} itself.

\lateproof{Lemma~\ref{lem:cover down}}
Let $\cS:=\Set{\emptyset}$, $\cU:=\Set{U}$ and let $(\cP_{r},\cP_{f})$ be the $(r,f)$-partition pair of $G,U$. By Proposition~\ref{prop:random generalisation}, $\cS,\cU,(\cP_{r},\cP_{f})$ form a $(\eps,\mu,\mu^{f-r}\xi,f,r,0)$-setup for $G$. We can thus apply Lemma~\ref{lem:horrible}\ref{lem:horrible:specific} with $\mu^{f-r}\xi$ playing the role of $\xi$. Recall that all $r$-edges of $G$ are $\cS$-important. Moreover, let $\tau_{r}$ denote the type function of $G^{(r)}$, $\cS$, $\cU$. We then have $\tau_{r}^{-1}(0)=G^{(r)}[\bar{U}]$ and $\tau_{r}^{-1}(r)=G^{(r)}[U]$, where $\bar{U}:=V(G)\sm U$.
\endproof

\lateproof{Lemma~\ref{lem:horrible}}
The proof is by induction on $r-i$. For $i=r-1$, we will prove the statement directly. For $i<r-1$, we assume that the statement is true for all $i'\in\Set{i+1,\dots,r-1}$. We will first prove \ref{lem:horrible:specific} using \ref{lem:horrible:all} inductively, and then derive \ref{lem:horrible:all} from \ref{lem:horrible:specific} (for the same value of $r-i$).

\lateproof{\ref{lem:horrible:specific}}

If $i<r-1$, choose new constants $\nu_1,\rho_1,\beta_1,\dots,\nu_{r-i-1},\rho_{r-i-1},\beta_{r-i-1}$ such that $$1/n\ll  1/\kappa \ll \gamma \ll \eps \ll \nu_1 \ll \rho_1 \ll \beta_1 \ll \dots \ll \nu_{r-i-1} \ll \rho_{r-i-1} \ll \beta_{r-i-1} \ll \nu \ll \mu,\xi,1/f.$$

For every $\ell\in[r-i-1]$, let
\begin{align}
G_{\ell}:=G-\set{e\in G^{(r)}}{e\mbox{ is }\cS\mbox{-important and }\tau_r(e)<\ell}.\label{def G l}
\end{align}
For every $i'\in \Set{i+1,\dots,r-1}$, let $\cT^{i'}$ be the $i'$-extension of $\cS$ in $G$ around $\cU$. By Proposition~\ref{prop:extension types}, the following hold for all $i'\in \Set{i+1,\dots,r-1}$:
\begin{enumerate}[label=\rm{(\Roman*)}]
\item $G_{r-i'}$ is $r$-exclusive with respect to $\cT^{i'}$;\label{prop:extension types 1 newer}
\item the elements of $\tau_{r}^{-1}(r-i')$ are precisely the $\cT^{i'}$-important elements of $G_{r-i'}^{(r)}$.\label{prop:extension types 3 newer}
\end{enumerate}

By Lemma~\ref{lem:setup refinement}, for every $i'\in \Set{i+1,\dots,r-1}$, there exist $\cU^{i'}$, $\cP_{r}^{i'}$, $\cP_{f}^{i'}$ such that the following hold:
\begin{enumerate}[label=\rm{(\alph*)}]
\item $\cU^{i'}$ is a $(\mu,\rho_{r-i'},r)$-focus for $\cT^{i'}$ such that $U_T\In U_{T{\restriction_{\cS}}}$ for all $T\in \cT^{i'}$;\label{refined focus random}
\item $\cT^{i'},\cU^{i'},(\cP_{r}^{i'},\cP_{f}^{i'})$ form a $(1.1\eps,\rho_{r-i'}\mu,\rho_{r-i'}^{f-r}\xi,f,r,i')$-setup for $G_{r-i'}$;\label{twisted focus setup new}
\item $G_{r-i'}(T)[U_T]$ is a $(1.1\eps,0.9\xi,f-i',r-i')$-supercomplex for every $T\in\cT^{i'}$.\label{after focus super new}
\end{enumerate}

\ref{prop:extension types 1 newer} allows us to consider the type function $\tau_{r-i',r}$ of $G_{r-i'}^{(r)},\cT^{i'},\cU^{i'}$.

\medskip

{\bf Step~1: Reserving subgraphs}

\smallskip

In this step, we will find a number of subgraphs of $G^{(r)}- \tau_{r}^{-1}(0)$ whose union will be the $r$-graph $H^\ast$ we seek in \ref{lem:horrible:specific}. Let $\tilde{G}$ be a complex as specified in \ref{lem:horrible:specific}.
Let $\beta_0:=\eps$. Let $H_0$ be a subgraph of $G^{(r)}- \tau_{r}^{-1}(0)$ with $\Delta(H_0)\le 1.1\beta_0 n$ such that for all $e\in \tilde{G}^{(r)}$, we have
\begin{align}
|\tilde{G}[H_0\cup \Set{e}]^{(f)}(e)|\ge  0.9 \beta_0^{\binom{f}{r}}n^{f-r}.\label{greedy assumption}
\end{align}
($H_0$ will be used to greedily cover $L^\ast$.)
That such a subgraph exists can be seen by a probabilistic argument: let $H_0$ be obtained by including every edge of $G^{(r)}- \tau_{r}^{-1}(0)$ with probability $\beta_0$. Clearly, \whp $\Delta(H_0)\le 1.1\beta_0 n$. Also, since $\tilde{G}$ is $(\eps,f,r)$-dense with respect to $G^{(r)}- \tau_{r}^{-1}(0)$ by assumption, we have for all $e\in \tilde{G}^{(r)}$ that
\begin{align*}
\expn{|\tilde{G}[H_0\cup \Set{e}]^{(f)}(e)|}=\beta_0^{\binom{f}{r}-1}|\tilde{G}[(G^{(r)}- \tau_{r}^{-1}(0))\cup \Set{e}]^{(f)}(e)| \ge \beta_0^{\binom{f}{r}-1}\eps n^{f-r}.
\end{align*}
Using Corollary~\ref{cor:graph chernoff} and a union bound, it is then easy to see that \whp $H_0$ satisfies~\eqref{greedy assumption} for all $e\in \tilde{G}^{(r)}$.

\medskip

{\em Step~1.1: Defining `sparse' induction graphs $H_\ell$.}

\smallskip

Consider $\ell\in[r-i-1]$ and let $i':=r-\ell$. Let $\xi_\ell:=\nu_\ell^{8^f\cdot f+1}$.
By \ref{twisted focus setup new} and Lemma~\ref{lem:make dominant} (with $G_\ell,3\beta_{\ell-1},\nu_\ell,\rho_\ell\mu,\rho_\ell^{f-r}\xi,i'$ playing the roles of $G,\eps,\nu,\mu,\xi,i$), there exists a subgraph $H_\ell\In G_\ell^{(r)}$ with $\Delta(H_\ell)\le 1.1\nu_\ell n$ and the following property: for all $L\In G_\ell^{(r)}$ with $\Delta(L)\le 3\beta_{\ell-1} n$ and every $(r+1)$-graph $O$ on $V(G_\ell)$ with $\Delta(O)\le 3\beta_{\ell-1} n$, the following holds for $G':=G_\ell[H_\ell \bigtriangleup L]-O$:
\begin{align}
&\cT^{i'},\cU^{i'},(\cP^{i'}_{r},\cP^{i'}_{f})[G']\mbox{ form a diagonal-dominant }\label{sparse induction graph}\\ &(\sqrt{3\beta_{\ell-1}},\rho_\ell\mu,\xi_\ell,f,r,i')\mbox{-setup for }G'.\nonumber
\end{align}

\medskip

{\em Step~1.2: Defining `localised' cleaning graphs $J_\ell$.}

\smallskip

Again, consider $\ell\in[r-i-1]$ and let $i':=r-\ell$.
Let
\begin{align}
G_\ell^\ast:=G_\ell-\set{e\in G_\ell^{(r)}}{e\mbox{ is }\cT^{i'}\mbox{-important and }\tau_{\ell,r}(e)< \ell}.\label{def G l star}
\end{align}
We claim that $G_\ell^\ast(T)[U_T]=G_\ell(T)[U_T]$ for every $T\in\cT^{i'}$.
Indeed, consider any $T\in\cT^{i'}$ and $e\in G_\ell(T)[U_T]$. Hence, $e\In U_T$ and $e\cup T\in G_\ell$. We need to show that $e\cup T\in G_\ell^\ast$, i.e.~that there is no $\cT^{i'}$-important $r$-subset $e'$ of $e\cup T$ with $\tau_{\ell,r}(e')< \ell$. However, if $e'\in \binom{e\cup T}{r}$ is $\cT^{i'}$-important, then $|e\cup T|\ge |e'|=r$ and since $G_\ell$ is $r$-exclusive with respect to $\cT^{i'}$ by \ref{prop:extension types 1 newer}, we must have $T\In e'$. As $e'\sm T\In e\In U_T$, we deduce that $\tau_{\ell,r}(e')= |e'\cap U_T|=|e'\sm T|=r-i'=\ell$.

Hence, by~\ref{after focus super new}, for every $T\in\cT^{i'}$, $G_\ell^\ast(T)[U_T]$ is a $(1.1\eps,0.9\xi,f-i',r-i')$-supercomplex.
Thus, by Lemma~\ref{lem:local sparsifier} (with $G_\ell^\ast,3\nu_\ell,\rho_\ell\mu,\beta_\ell,0.9\xi$ playing the roles of $G,\eps,\mu,\beta,\xi$), there exists a subgraph $J_\ell\In G_\ell^{\ast(r)}$ with $\Delta(J_\ell)\le 1.1\beta_\ell n$ and the following property: for all $L\In G_\ell^{\ast(r)}$ with $\Delta(L)\le 3\nu_\ell n$ and every $(r+1)$-graph $O$ on $V(G_\ell^\ast)$ with $\Delta(O)\le 3\nu_\ell n$, the following holds for $G^\ast:=G_\ell^\ast[J_\ell\bigtriangleup L]-O$:
\begin{align}
G^\ast(T)[U_T]\mbox{ is a }(\sqrt{3\nu_\ell},0.81\xi\beta_\ell^{(8^f)},f-i',r-i')\mbox{-supercomplex for every }T\in\cT^{i'}.\label{robust dense}
\end{align}

\bigskip

We have defined subgraphs $H_0,H_1,\dots,H_{r-i-1},J_1,\dots,J_{r-i-1}$ of $G^{(r)}- \tau_r^{-1}(0)$. Note that they are not necessarily edge-disjoint. Let $H_0^\ast:=H_0$ and for all $\ell\in[r-i-1]$ define inductively
\begin{align*}
H_\ell' &:=H_{\ell-1}^\ast\cup H_\ell,\\
H_\ell^\ast &:=H_{\ell-1}^\ast\cup H_\ell \cup J_\ell=H_\ell' \cup J_\ell,\\
H^\ast &:=H_{r-i-1}^\ast.
\end{align*}
Clearly, $\Delta(H_{\ell}^\ast)\le 2\beta_{\ell} n$ for all $\ell\in[r-i-1]_0$ and $\Delta(H_{\ell}')\le 2\nu_{\ell} n$ for all $\ell\in[r-i-1]$. In particular, $\Delta(H^\ast)\le 2\beta_{r-i-1}n\le \nu n$, as desired.

\medskip

{\bf Step~2: Covering down}

\smallskip

Let $L^\ast$ be any subgraph of $\tilde{G}^{(r)}$ with $\Delta(L^\ast)\le \gamma n$ such that $H^\ast \cup L^\ast$ is $F$-divisible with respect to $\cS,\cU$, and let $O^\ast \In \tilde{G}^{(r+1)}$ with $\Delta(O^\ast)\le \gamma n$. We need to find a $\kappa$-well separated $F$-packing $\cF$ in $\tilde{G}[H^\ast \cup L^\ast]-O^\ast$ which covers all edges of $L^\ast$, and covers all $\cS$-important edges of $H^\ast$ except possibly some from $\tau_{r}^{-1}(r-i)$.
We will do so by inductively showing that the following holds for all $\ell\in[r-i]$.
\begin{itemize}
\item[\indcov{\ell}] There exists a $(3\ell\sqrt{\kappa})$-well separated $F$-packing $\cF_{\ell-1}^\ast$ in $\tilde{G}[H_{\ell-1}^\ast \cup L^\ast]-O^\ast$ covering all edges of $L^\ast$, and all $\cS$-important $e\in H_{\ell-1}^\ast$ with $\tau_r(e)<\ell$.
\end{itemize}
Clearly, \indcov{r-i} establishes~\ref{lem:horrible:specific}.\COMMENT{$H_{r-i-1}^\ast=H^\ast$, $3(r-i)\sqrt{\kappa}\le \kappa$, $Im(\tau_r)=[r-i]_0$, thus if all edges $e$ with $\tau_r(e)<r-i$ are covered, then the only remaining ones are those of type $r-i$}

\begin{NoHyper}\begin{claim}
\indcov{1} is true.
\end{claim}\end{NoHyper}

\claimproof
Let $H_0':=H_0\cup L^\ast=H_0^\ast\cup L^\ast$. By \eqref{greedy assumption} and Proposition~\ref{prop:sparse noise containment}, for all $e\in L^\ast$ we have that $$|(\tilde{G}[H_0']-O^\ast)^{(f)}(e)|\ge |\tilde{G}[H_0\cup e]^{(f)}(e)| -2^r\gamma n^{f-r}\ge 0.8\beta_0^{\binom{f}{r}}n^{f-r}.$$\COMMENT{Here we need $\gamma\ll \eps$}
By Corollary~\ref{cor:greedy cover},\COMMENT{With $\gamma$ playing the role of $\gamma$ and $0.8\beta_0^{\binom{f}{r}}$ playing the role of $\alpha$} there is a $1$-well separated $F$-packing $\cF_0^\ast$ in $\tilde{G}[H_0']-O^\ast$ covering all edges of $L^\ast$. Since $H_0^\ast$ does not contain any edges from $\tau_r^{-1}(0)$, $\cF_0^\ast$ satisfies \indcov{1}.
\endclaimproof

If $i=r-1$, we can take $\cF_0^\ast$ and complete the proof of \ref{lem:horrible:specific}. So assume that $i<r-1$ and that Lemma~\ref{lem:horrible} holds for larger values of $i$.

Suppose that for some $\ell\in[r-i-1]$, $ \cF_{\ell-1}^{\ast}$ satisfies \indcov{\ell}. Let $i':=r-\ell>i$.
We will now find a $3\sqrt{\kappa}$-well separated $F$-packing $\cF_\ell$ in $G[H_{\ell}^\ast]-\cF_{\ell-1}^{\ast(r)}-\cF_{\ell-1}^{\ast\le(r+1)}-O^\ast$ such that $\cF_\ell$ covers all edges of $H_{\ell}^\ast-\cF_{\ell-1}^{\ast(r)}$ that belong to $\tau_{r}^{-1}(\ell)$.

Then $\cF_{\ell}^\ast:=\cF_{\ell-1}^\ast\cup \cF_\ell$ covers all edges of $L^\ast$ and all $\cS$-important $e\in H_{\ell}^\ast$ with $\tau_r(e)< \ell+1$. By Fact~\ref{fact:ws}\ref{fact:ws:1}, $\cF_{\ell}^\ast$ is $(3\ell\sqrt{\kappa}+3\sqrt{\kappa})$-well separated, implying that \indcov{\ell+1} is true.

Crucially, by \ref{prop:extension types 3 newer}, all the edges of $\tau_{r}^{-1}(\ell)$ that we seek to cover in this step are $\cT^{i'}$-important. We will obtain $\cF_\ell$ as the union of $\cF_\ell^\circ$ and $\cF_\ell^\dagger$, where
\begin{enumerate}[label=(COV\arabic*)]
\item $\cF_\ell^\circ$ is $2\sqrt{\kappa}$-well separated $F$-packing in $G[H_{\ell}^\ast]-\cF_{\ell-1}^{\ast(r)}-\cF_{\ell-1}^{\ast\le(r+1)}-O^\ast$ which covers all $\cT^{i'}$-important edges of $H_{\ell}^\ast-\cF_{\ell-1}^{\ast(r)}$ except possibly some from $\tau_{\ell,r}^{-1}(\ell)$;\label{inductive cover:a}
\item $\cF_\ell^\dagger$ is a $\sqrt{\kappa}$-well separated $F$-packing in $G[H_{\ell}^\ast]-\cF_{\ell-1}^{\ast(r)}-\cF_\ell^{\circ(r)}-\cF_{\ell-1}^{\ast\le(r+1)}-\cF_\ell^{\circ\le(r+1)}-O^\ast$ which covers all $\cT^{i'}$-important edges of $H_{\ell}^\ast-\cF_{\ell-1}^{\ast(r)}-\cF_\ell^{\circ(r)}$.\label{inductive cover:b}
\end{enumerate}
Since $\cF_\ell^\dagger$ and $\cF_{\ell}^\circ$ are $(r+1)$-disjoint, $\cF_\ell:=\cF_\ell^\circ\cup \cF_\ell^\dagger$ is $3\sqrt{\kappa}$-well separated by Fact~\ref{fact:ws}\ref{fact:ws:1}. Clearly, $\cF_\ell$ covers all $\cT^{i'}$-important edges of $H_{\ell}^\ast-\cF_{\ell-1}^{\ast(r)}$, as required. We will obtain $\cF_\ell^\circ$ by using \ref{lem:horrible:all} of this lemma inductively, and $\cF_\ell^\dagger$ by an application of the Localised cover down lemma (Lemma~\ref{lem:dense jain}).

Recall that $F$-divisibility with respect to $\cT^{i'},\cU^{i'}$ was defined in Definition~\ref{def:qr divisible}.
Let $H_\ell'':=H_{\ell}'-\cF_{\ell-1}^{\ast(r)}$.
\begin{NoHyper}\begin{claim}\label{claim:cover down divisible}
$H_\ell''$ is $F$-divisible with respect to $\cT^{i'},\cU^{i'}$.
\end{claim}\end{NoHyper}

\claimproof
Let $T\in\cT^{i'}$ and $b'\In V(G)\sm T$ with $|b'|\le r-i'-1$ and $|b'\sm U_T|\ge 1$. We have to show that $Deg(F)_{i'+|b'|}\mid |H_\ell''(T\cup b')|$. Let $S:=T{\restriction_{\cS}}$ and $b:=b'\cup (T\sm S)$. Hence, $|b|=|b'|+i'-i$. Clearly, $b\In V(G)\sm S$, $|b|\le r-i-1$ and $|b\sm U_S|\ge |T\sm S| \ge 1$. Hence, since $H^\ast\cup L^\ast$ is $F$-divisible with respect to $\cS,\cU$ by assumption, we have $Deg(F)_{i+|b|}\mid |(H^\ast\cup L^\ast)(S\cup b)|$, and this implies that $Deg(F)_{i+|b|}\mid |((H^\ast\cup L^\ast)-\cF_{\ell-1}^{\ast(r)})(S\cup b)|$. It is thus sufficient to show that
$$H_\ell''(T\cup b')=((H^\ast\cup L^\ast)-\cF_{\ell-1}^{\ast(r)})(S\cup b).$$
Clearly, we have $T\cup b'=S\cup b$ and $H_\ell''\In H^\ast-\cF_{\ell-1}^{\ast(r)}$. Conversely, observe that every $e\in H^\ast\cup L^\ast$ that contains $T\cup b'$ and is not covered by $\cF_{\ell-1}^\ast$ must belong to $H_\ell''$. Indeed, since $e$ contains $T$, we have that $\tau_r(e)\le r-i'=\ell$, so $e\in H_\ell^\ast$. Moreover, by \indcov{\ell} we must have $\tau_r(e)\ge \ell$.\COMMENT{If $e\in H_{\ell-1}^\ast$, this is clear. If not, then $e\in H_\ell \cup J_\ell$ and thus the conclusion is true as well.} Hence, $\tau_r(e)=\ell$. But since $|b'\sm U_T|\ge 1$, we have $\tau_{\ell,r}(e)<\ell$. By~\eqref{def G l star}, $e\notin J_\ell$. Thus, $e\in H_\ell'-\cF_{\ell-1}^{\ast(r)}= H_\ell''$. Hence, $H_\ell''(T\cup b')=((H^\ast\cup L^\ast)-\cF_{\ell-1}^{\ast(r)})(S\cup b)$. This implies the claim.
\endclaimproof

Let $L_\ell':=H_\ell''\bigtriangleup H_\ell$. So $H_\ell''=H_\ell\bigtriangleup L_\ell'$.

\begin{NoHyper}\begin{claim}\label{claim:cover down symmetric difference}
$L_\ell'\In G_\ell^{(r)}$ and $\Delta(L_\ell')\le 3\beta_{\ell-1}n$.
\end{claim}\end{NoHyper}

\claimproof
Suppose, for a contradiction, that there is $e\in H_\ell''\bigtriangleup H_\ell$ with $e\notin G_\ell^{(r)}$. Since $H_\ell\In G_\ell^{(r)}$, we must have $e\in H_\ell''=H_{\ell}'-\cF_{\ell-1}^{\ast(r)}$. Thus, since $e$ is not covered by $\cF_{\ell-1}^\ast$, \indcov{\ell} implies that $e$ is $\cS$-unimportant or $\tau_r(e)\ge \ell$, both contradicting $e\notin G_\ell^{(r)}$.

In order to see the second part, observe that $L_\ell'= ((H_{\ell-1}^\ast\cup H_\ell)-\cF_{\ell-1}^{\ast(r)})\bigtriangleup H_\ell \In H_{\ell-1}^\ast\cup L^\ast$ since $\cF_{\ell-1}^{\ast(r)}\In L^\ast \cup H_{\ell-1}^\ast$.\COMMENT{Suppose $e\notin H_{\ell-1}^\ast\cup L^\ast$. Distinguish the cases $e\in H_\ell$ and $e\notin H_\ell$. If $e\in H_\ell$, $e$ is not covered by $\cF_{\ell-1}^{\ast(r)}$. If $e\notin H_\ell$, we have $e\notin H_{\ell-1}^\ast\cup H_\ell$ as well.} Thus, $\Delta(L_\ell')\le \Delta(H_{\ell-1}^\ast)+\Delta(L^\ast) \le 3\beta_{\ell-1}n$.
\endclaimproof

Note that Claim~\ref*{claim:cover down symmetric difference} implies that $H_\ell''\In G_\ell^{(r)}$. Let $G_{\ell,ind}:=G_\ell[H_\ell'']-\cF_{\ell-1}^{\ast\le(r+1)}-O^\ast$. By Fact~\ref{fact:ws}\ref{fact:ws:maxdeg} and \indcov{\ell}, we have that $\Delta(\cF_{\ell-1}^{\ast\le(r+1)}\cup O^\ast)\le (3\ell \sqrt{\kappa})(f-r)+\gamma n \le 2 \gamma n$. Thus, by~\eqref{sparse induction graph} and Claim~\ref*{claim:cover down symmetric difference}, $\cT^{i'},\cU^{i'},(\cP^{i'}_{r},\cP^{i'}_{f})[G_{\ell,ind}]$ form a diagonal-dominant $(\sqrt{3\beta_{\ell-1}},\rho_\ell\mu,\xi_\ell,f,r,i')$-setup for $G_{\ell,ind}$.
We can thus apply Lemma~\ref{lem:horrible}\ref{lem:horrible:all} inductively with the following objects/parameters.\COMMENT{$\gamma,\nu$ not needed for \ref{lem:horrible:all}}

\smallskip
{\footnotesize
\noindent
{
\begin{tabular}{c|c|c|c|c|c|c|c|c|c|c|c|c|c}
object/parameter & $G_{\ell,ind}$ & $n$ & $\sqrt{3\beta_{\ell-1}}$ & $\rho_\ell\mu$ & $\xi_\ell$ & $i'$ & $\cT^{i'}$ & $\cU^{i'}$ & $(\cP^{i'}_{r},\cP_{f}^{i'})[G_{\ell,ind}]$ & $\sqrt{\kappa}$ & $f$ & $r$ & $F$\\ \hline
playing the role of & $G$ & $n$ & $\eps$ & $\mu$ & $\xi$ & $i$ & $\cS$ & $\cU$& $(\cP_{r},\cP_{f})$ & $\kappa$ & $f$ & $r$ & $F$
\end{tabular}
}}\newline \vspace{0.2cm}

Since $G_{\ell,ind}^{(r)}=H_\ell''$ is $F$-divisible with respect to $\cT^{i'},\cU^{i'}$ by Claim~\ref*{claim:cover down divisible}, there exists a $2\sqrt{\kappa}$-well separated $F$-packing $\cF_\ell^\circ$ in $G_{\ell,ind}$ covering all $\cT^{i'}$-important edges of $H_\ell''$ except possibly some from $\tau_{\ell,r}^{-1}(r-i')=\tau_{\ell,r}^{-1}(\ell)$. Note that $H_\ell^\ast-H_\ell'\In J_\ell$ and that every  $\cT^{i'}$-important edge of $J_\ell$ lies in $\tau_{\ell,r}^{-1}(\ell)$. Thus $\cF_\ell^\circ$ does indeed cover all $\cT^{i'}$-important edges of $H_\ell^\ast-\cF_{\ell-1}^{\ast(r)}$ except possibly some from $\tau_{\ell,r}^{-1}(\ell)$, as required for~\ref{inductive cover:a}.

We will now use $J_\ell$ to cover the remaining $\cT^{i'}$-important edges of $H_\ell^\ast$.
Let $J_\ell':=H_{\ell}^\ast-\cF_{\ell-1}^{\ast(r)}-\cF_\ell^{\circ(r)}$.
Let $S^\ast_{i'}\in \binom{V(F)}{i'}$ be such that $F(S^\ast_{i'})$ is non-empty.

\begin{NoHyper}\begin{claim}\label{claim:localised divisible}
$J_\ell'(T)[U_T]$ is $F(S^\ast_{i'})$-divisible for every $T\in \cT^{i'}$.
\end{claim}\end{NoHyper}

\claimproof
Let $T\in \cT^{i'}$ and $b'\In U_T$ with $|b'|\le r-i'-1$. We have to show that $Deg(F(S^\ast_{i'}))_{|b'|}\mid |J_\ell'(T)[U_T](b')|$. Note that for every $e\in J_\ell'\In G_\ell^{\ast(r)}$ containing $T$, we have $\tau_{\ell,r}(e)=r-i'$. Thus, $J_\ell'(T)[U_T]$ is identical with $J_\ell'(T)$ except for the different vertex sets. It is thus sufficient to show that $Deg(F(S^\ast_{i'}))_{|b'|}\mid |J_\ell'(T\cup b')|$. By Proposition~\ref{prop:link divisibility}, we have that $Deg(F(S^\ast_{i'}))_{|b'|}=Deg(F)_{i'+|b'|}$. Let $S:=T{\restriction_{\cS}}$ and $b:=b'\cup (T\sm S)$. By assumption, $H^\ast\cup L^\ast$ is $F$-divisible with respect to $\cS,\cU$. Thus, since $S\in \cS$, $|b|\le r-i-1$ and $|b\sm U_S|\ge |T\sm S|\ge 1$, we have that $Deg(F)_{i+|b|}\mid |(H^\ast\cup L^\ast)(S\cup b)|$. This implies that $Deg(F)_{i+|b|}\mid |((H^\ast\cup L^\ast)- \cF_{\ell-1}^{\ast(r)}-\cF_\ell^{\circ(r)})(S\cup b)|$.
It is thus sufficient to prove that $J_\ell'(T\cup b')=((H^\ast\cup L^\ast)- \cF_{\ell-1}^{\ast(r)}-\cF_\ell^{\circ(r)})(S\cup b)$.\COMMENT{$i+|b|=i'+|b'|$} Clearly, $J_\ell'\In H^\ast-\cF_{\ell-1}^{\ast(r)}-\cF_\ell^{\circ(r)}$ by definition. Conversely, observe that every $e\in (H^\ast\cup L^\ast) -\cF_{\ell-1}^{\ast(r)}-\cF_\ell^{\circ(r)}$ that contains $T\cup b'$ must belong to $J_\ell'$. Indeed, since $L^\ast\In \cF_{\ell-1}^{\ast(r)}$, we have $e\in H^\ast$, and since $e$ contains $T$, we have $\tau_r(e)\le \ell$.\COMMENT{$=$ holds} Hence, $e\in H_\ell^\ast$ and thus $e\in J_\ell'$. This implies the claim.
\endclaimproof

Let $L_\ell'':=J_\ell'\bigtriangleup J_\ell$. So $J_\ell'=J_\ell\bigtriangleup L_\ell''$.

\begin{NoHyper}\begin{claim}\label{claim:dense remainder}
$L_\ell''\In G_\ell^{\ast(r)}$ and $\Delta(L_\ell'')\le 3\nu_\ell n$.
\end{claim}\end{NoHyper}

\claimproof
Suppose, for a contradiction, that there is $e\in J_\ell'\bigtriangleup J_\ell$ with $e\notin G_\ell^{\ast(r)}$. By~\eqref{def G l} and~\eqref{def G l star}, the latter implies that $e$ is $\cS$-important with $\tau_{r}(e)<\ell$ or $\cT^{i'}$-important with $\tau_{\ell,r}(e)<\ell$. However, since $J_\ell\In G_\ell^{\ast(r)}$, we must have $e\in J_\ell'- J_\ell$ and thus $e\in H_{\ell}'$ and $e\notin \cF_{\ell-1}^{\ast(r)}\cup \cF_\ell^{\circ(r)}$. In particular, $e\in H_\ell''$. Now, if $e$ was $\cS$-important with $\tau_{r}(e)<\ell$, then $e\in H_\ell'-H_\ell \In H_{\ell-1}^\ast$. But then $e$ would be covered by $\cF_{\ell-1}^{\ast}$, a contradiction. So $e$ must be $\cT^{i'}$-important with $\tau_{\ell,r}(e)<\ell$. But since $e\in H_\ell''$, $e$ would be covered by $\cF_\ell^{\circ}$ unless $\tau_{\ell,r}(e)=\ell$, a contradiction.

In order to see the second part, observe that $$L_\ell''=((H_\ell' \cup J_\ell)-\cF_{\ell-1}^{\ast(r)}-\cF_\ell^{\circ(r)})\bigtriangleup J_\ell \In H_\ell'\cup L^\ast$$ since $\cF_{\ell-1}^{\ast(r)}\cup\cF_\ell^{\circ(r)} \In H_\ell' \cup L^\ast$.\COMMENT{Suppose $e\notin H_{\ell}'\cup L^\ast$. Distinguish the cases $e\in J_\ell$ and $e\notin J_\ell$.} Thus, $\Delta(L_\ell'')\le \Delta(H_\ell') +\Delta(L^\ast) \le 3\nu_\ell n$.
\endclaimproof

Note that Claim~\ref*{claim:dense remainder} implies that $J_\ell'\In G_{\ell}^{\ast(r)}$. Let $$G_{\ell,clean}:= G_\ell^\ast[J_\ell']-\cF_{\ell-1}^{\ast\le(r+1)}-\cF_{\ell}^{\circ\le(r+1)}-O^\ast.$$ By \indcov{\ell}, \ref{inductive cover:a} and Fact~\ref{fact:ws}\ref{fact:ws:maxdeg}, we have that $$\Delta(\cF_{\ell-1}^{\ast\le(r+1)}\cup \cF_{\ell}^{\circ\le(r+1)} \cup O^\ast)\le (3\ell \sqrt{\kappa})(f-r)+ (2 \sqrt{\kappa})(f-r) + \gamma n \le 2 \gamma n.$$

Thus, by \eqref{robust dense}, Claim~\ref*{claim:localised divisible} and Claim~\ref*{claim:dense remainder}, $G_{\ell,clean}(T)[U_T]$ is an $F(S^\ast_{i'})$-divisible $(\rho_\ell,\beta_\ell^{(8^f)+1},f-i',r-i')$-supercomplex for every $T\in\cT^{i'}$.\COMMENT{$\beta_\ell^{(8^f)+1}\le 0.81\xi \beta_\ell^{(8^f)}$} Moreover, whenever there are $T\in \cT^{(i')}$ and $e\in G_{\ell,clean}^{(r)}\In G_{\ell}^{\ast(r)}$ with $T\In e$, then $|(e\sm T)\cap U_T|=\tau_{\ell,r}(e)=\ell=|e\sm T|$\COMMENT{$\Ima(\tau_{\ell,r})=[\ell]_0$} and thus $e\sm T\In U_T$.
By~\ref{prop:extension types 1 newer}, $G_{\ell,clean}\In G_{\ell}$ is $r$-exclusive with respect to $\cT^{i'}$, and by~\ref{refined focus random}, $\cU^{i'}$ is a $(\mu,\rho_{\ell},r)$-focus for $\cT^{i'}$. We can therefore apply the Localised cover down lemma (Lemma~\ref{lem:dense jain}) with the following objects/parameters.

\smallskip
{\footnotesize
\noindent
{
\begin{tabular}{c|c|c|c|c|c|c|c|c|c|c|c|c}
object/parameter & $n$ & $\rho_\ell$ & $\mu$ & $\beta_\ell^{(8^f)+1}$ & $i'$ & $G_{\ell,clean}$ & $\cT^{i'}$ & $\cU^{i'}$ & $r$ & $f$ & $F$ & $S^\ast_{i'}$\\ \hline
playing the role of & $n$ & $\rho$ & $\rho_{size}$ & $\xi$ & $i$ & $G$ & $\cS$ & $\cU$ & $r$ & $f$ & $F$ & $S^\ast$
\end{tabular}
}}\newline \vspace{0.2cm}

This yields a $\rho_{\ell}^{-1/12}$-well separated $F$-packing $\cF_\ell^\dagger$ in $G_{\ell,clean}$ covering all $\cT^{i'}$-important edges of $G_{\ell,clean}^{(r)}=J_\ell'=H_{\ell}^\ast-\cF_{\ell-1}^{\ast(r)}-\cF_\ell^{\circ(r)}$ . Thus $\cF_\ell^\dagger$ is as required in \ref{inductive cover:b}. As observed before, this completes the proof of~\indcov{\ell+1} and thus the proof of~\ref{lem:horrible:specific}.
\endproof

\lateproof{\ref{lem:horrible:all}}

Let $Y\In G^{(f)}$ and $A\in [0,1]^{(r+1)\times(f+1)}$ be such that \ref{setup:objects}--\ref{setup:intersections} hold. We assume that $G^{(r)}$ is $F$-divisible with respect to $\cS,\cU$ and that $A$ is diagonal-dominant.

\begin{NoHyper}\setcounter{claim}{5}\begin{claim}\label{claim:hoover dense}
$G$ is $(\xi-\eps,f,r)$-dense with respect to $G^{(r)}- \tau_r^{-1}(0)$.
\end{claim}\end{NoHyper}

\claimproof Let $e\in G^{(r)}$ and let $\ell'\in[r+1]$ be such that $e\in \cP_{r}(\ell')$. Suppose first that $\ell'\le i$. Then no $f$-set from $\cP_{f}(\ell')$ contains any edge from $\tau_r^{-1}(0)$ (as such an $f$-set is $\cS$-unimportant). Recall from \ref{setup:matrix regular} for $\cS,\cU,(\cP_{r},\cP_{f})$ that $G[Y]$ is $(\eps,A,f,r)$-regular with respect to $(\cP_{r},\cP_{f}[Y])$ and $\min^{\backslash\backslash r-i+1}(A)\ge \xi$. Thus,
\begin{align*}
|G[(G^{(r)}- \tau_r^{-1}(0))\cup e]^{(f)}(e)| &\ge |(Y\cap\cP_{f}(\ell'))(e)| \ge (a_{\ell',\ell'}-\eps) n^{f-r} \ge (\xi-\eps) n^{f-r}.
\end{align*}

If $\ell'>i+1$, then by~\ref{prop:facts about partition pairs:index shift} in Proposition~\ref{prop:facts about partition pairs}, no $f$-set from $\cP_{f}(f-r+\ell')$ contains any edge from $\tau_r^{-1}(0)$. Thus, we have
$$|G[(G^{(r)}- \tau_r^{-1}(0))\cup e]^{(f)}(e)|\ge (a_{\ell',f-r+\ell'}-\eps) n^{f-r} \ge (\xi-\eps) n^{f-r}.$$

If $\ell'=i+1$, then $\cP_{r}(\ell')=\tau_r^{-1}(0)$ by~\ref{prop:facts about partition pairs:index shift}. However, every $f$-set from $\tau_{f}^{-1}(f-r)=\cP_{f}(f-r+\ell')$ that contains $e$ contains no other edge from $\tau_r^{-1}(0)$. Thus, $$|G[(G^{(r)}- \tau_r^{-1}(0))\cup e]^{(f)}(e)|\ge (a_{\ell',f-r+\ell'}-\eps) n^{f-r} \ge (\xi-\eps) n^{f-r}.$$
\endclaimproof

By Claim~\ref*{claim:hoover dense}, we can choose $H^\ast\In G^{(r)}- \tau_r^{-1}(0)$ such that \ref{lem:horrible:specific} holds with $G$ playing the role of $\tilde{G}$.
Let $$H_{nibble}:=G^{(r)}-H^\ast.$$

Recall that by~\ref{setup:matrix regular}, $G[Y]$ is $(\eps,A,f,r)$-regular with respect to $(\cP_{r},\cP_{f}[Y])$, and \ref{setup:dense} implies that $G[Y]$ is $(\mu^{f}\xi,f+r,r)$-dense. Let $$G_{nibble}:=(G[Y])[H_{nibble}].$$ Using Proposition~\ref{prop:sparse noise containment}, it is easy to see that $G_{nibble}$ is $(2^{r+1}\nu,A,f,r)$-regular with respect to $(\cP_{r},\cP_{f})[G_{nibble}]$. Moreover, by Proposition~\ref{prop:noise}\ref{noise:dense}, $G_{nibble}$ is $(\mu^f\xi/2,f+r,r)$-dense. Thus, by Lemma~\ref{lem:dominant2regular}, there exists $Y^\ast\In G_{nibble}^{(f)}$ such that $G_{nibble}[Y^\ast]$ is $(\sqrt{\nu},d,f,r)$-regular for $d:=\min^{\backslash}(A)\ge \xi$ and $(0.45\mu^f\xi(\mu^f\xi/8(f+1))^{\binom{f+r}{f}},f+r,r)$-dense. Thus, by Lemma~\ref{lem:F nibble}\COMMENT{Since $\nu\ll\mu$} there is a $\kappa$-well separated $F$-packing $\cF_{nibble}$ in $G_{nibble}[Y^\ast]$ such that $\Delta(L_{nibble})\le \gamma n$, where $L_{nibble}:=G_{nibble}[Y^\ast]^{(r)}-\cF_{nibble}^{(r)}=H_{nibble}-\cF_{nibble}^{(r)}$. Since $G^{(r)}$ is $F$-divisible with respect to $\cS,\cU$, we clearly have that $H^\ast \cup L_{nibble}=G^{(r)}-\cF_{nibble}^{(r)}$ is $F$-divisible with respect to $\cS,\cU$. By Fact~\ref{fact:ws}\ref{fact:ws:maxdeg}, we have that $\Delta(\cF_{nibble}^{\le(r+1)})\le \kappa (f-r)\le \gamma n$. Thus, by~\ref{lem:horrible:specific}, there exists a $\kappa$-well separated $F$-packing $\cF^\ast$ in $G[H^\ast \cup L_{nibble}]-\cF_{nibble}^{\le(r+1)}$ which covers all edges of $L_{nibble}$, and all $\cS$-important edges of $H^\ast$ except possibly some from $\tau_{r}^{-1}(r-i)$. But then, by Fact~\ref{fact:ws}\ref{fact:ws:1}, $\cF_{nibble}\cup \cF^\ast$ is a $2\kappa$-well separated $F$-packing in $G$ which covers all $\cS$-important $r$-edges except possibly some from $\tau_{r}^{-1}(r-i)$, completing the proof.
\endproof
This completes the proof of Lemma~\ref{lem:horrible}.
\endproof

\vspace{2cm}

{\footnotesize \obeylines \parindent=0pt

Stefan Glock, Daniela K\"{u}hn, Allan Lo, Deryk Osthus
\vspace{0.3cm}
School of Mathematics
University of Birmingham
Edgbaston
Birmingham
B15 2TT
UK
}
\vspace{0.3cm}
\begin{flushleft}
{\it{E-mail addresses}:}
\tt{[sglock,d.kuhn,s.a.lo,d.osthus]@bham.ac.uk}
\end{flushleft}

\end{document}